\documentclass[11pt]{article}

\usepackage{graphicx} 
\usepackage{amsmath}%
\usepackage{amsfonts}%
\usepackage{amssymb}%
\usepackage{graphicx}
\usepackage{color}
\usepackage{amsthm}
\usepackage{tabu}

\usepackage{pifont}
\usepackage{float} 
\usepackage{subfigure}
\usepackage{tikz}
\usetikzlibrary{shapes, arrows.meta, positioning}
\usepackage[all]{xy}
\usepackage{mathrsfs}
\usepackage{bbold}
\usepackage{amscd}
\usepackage{lipsum}
\usepackage{enumitem}
\usepackage[colorlinks=true,hyperfootnotes=false,linkcolor=blue]{hyperref}

\providecommand{\U}[1]{\protect\rule{.1in}{.1in}}
\def\theenumi{\arabic{enumi}}

\def\theenumii{\alph{enumii}}
\def\p@enumii{\theenumi.}

\def\theenumiii{\arabic{enumiii}}
\def\p@enumiii{(\theenumi)(\theenumii)}

\def\p@enumiv{\p@enumiii.\theenumiii}
\usepackage{todonotes}

\newcommand{\ii }{{\rm i} }
\newcommand{\N}{{\mathbb N}}
\newcommand{\A}{{\mathcal A}}
\newcommand{\B}{{\mathcal B}}
\newcommand{\Z}{{\mathbb Z}}
\newcommand{\R}{{\mathbb R}}

\newcommand{\C}{{\mathbb C}}
\newcommand{\K}{{\mathscr K}}
\newcommand{\E}{{\mathscr E}}
\newcommand{\F}{{\mathscr F}}
\newcommand{\G}{{\mathscr G}}
\newcommand{\h}{{\mathscr H}}
\newcommand{\bigO}{{\mathcal O}}

\def\p{\partial}
\def\supp#1{{\textrm{supp}(#1)}}
\def\Im{{\rm Im}}
\def\Re{{\rm Re}}
\newtheorem{Thm}{Theorem}[section]
\newtheorem{thm}[Thm]{Theorem}
\newtheorem{coro}[Thm]{Corollary}
\newtheorem{rem}[Thm]{Remark}
\newtheorem{lem}[Thm]{Lemma}
\newtheorem{prop}[Thm]{Proposition}
\newtheorem{defi}[Thm]{Definition}

\newtheorem{assu}[Thm]{Assumption}

\newtheorem{question}[Thm]{Problem}
\theoremstyle{definition}
\newtheorem*{exa}{Example}
\def\poscalr#1#2{\langle\langle#1,#2\rangle\rangle}
\def\poscals#1#2{\langle#1,#2\rangle}
\numberwithin{equation}{section}
\newcommand{\cqfd}
{%
\mbox{}%
\nolinebreak%
\hfill%
\rule{2mm}{2mm}%
\medbreak%
\par%
}

\title{Stability of a Korteweg–de Vries equation close to critical lengths}
\author{
  Jingrui Niu\thanks{Institute for Advanced Study in Mathematics, Harbin Institute of Technology, 150001, Harbin, China., jingrui.niu@hit.edu.cn}
\and
  Shengquan Xiang\thanks{School of Mathematical Sciences, Peking University, 100871, Beijing, China., shengquan.xiang@math.pku.edu.cn}
}
\date{}

\begin{document}

\maketitle
\begin{abstract}
In this paper, we investigate the quantitative exponential stability of the Korteweg–de Vries equation on a finite interval with its length close to the critical set. Sharp decay estimates are obtained via a constructive PDE control framework. We first introduce a novel transition–stabilization approach, combining the Lebeau–Robbiano strategy with the moment method, to establish constructive null controllability for the KdV equation. This approach is then coupled with precise spectral analysis and invariant manifold theory to characterize the asymptotic behavior of the decay rate as the length of the interval approaches the set of critical lengths. Building on our classification of the critical lengths, we show that the KdV equation exhibits distinct asymptotic behaviors in neighborhoods of different types of critical lengths.
\end{abstract}

	\setcounter{tocdepth}{2}
	\tableofcontents

\section{Introduction}
Let $L>0$, we are interested in the exponential stability of the following KdV system equipped with Dirichlet boundary conditions and the Neumann boundary condition on the right:
\begin{equation}\label{eq: KdV system-stability-intro}
\left\{
\begin{array}{lll}
    \p_ty+\p_x^3y+\p_xy+ y \p_x y=0 & \text{ in }(0,T)\times(0,L), \\
     y(t,0)=y(t,L)=0&  \text{ in }(0,T),\\
     \p_xy(t,L)=0&  \text{ in }(0,T),\\
     y(0,x)=y^0(x)&\text{ in }(0,L).
\end{array}
\right.    
\end{equation}
Given $T>0$, we are also concerned with the null controllability of the related KdV system equipped with Dirichlet boundary conditions and using the Neumann boundary control on the right:
\begin{equation}\label{eq: nonlinear KdV system-control-intro}
\left\{
\begin{array}{lll}
    \p_ty+\p_x^3y+\p_xy+ y \p_x y=0 & \text{ in }(0,T)\times(0,L), \\
     y(t,0)=y(t,L)=0&  \text{ in }(0,T),\\
     \p_xy(t,L)=u(t)&  \text{ in }(0,T),\\
     y(0,x)=y^0(x)&\text{ in }(0,L),
\end{array}
\right.    
\end{equation} 
Here $y$ denotes the state, $y^0$ is the initial datum and $u$ denotes the control function. 

\vspace{2mm}

The primary objective of this paper is twofold. Firstly, we propose a control theory-based method to derive sharp decay rates for the stability of (nonlinear) KdV equations. 
\begin{figure}[h]
    \centering
\begin{tikzpicture}[
    scale=0.85,
    transform shape,
    auto,
    block/.style={
        rectangle, draw, 
        text width=6em, text centered, rounded corners, minimum height=4em
    },
    arrow/.style={
        -{Latex[width=2mm,length=3mm]}, thick
    },
    Arrow/.style={
        {Latex[width=2mm,length=3mm]}-{Latex[width=2mm,length=3mm]}, thick
    }
]

\node [block, text width=8em] (L-stability) {Stability of linearized KdV \eqref{eq: linear KdV-stability-intro}};
\node [block, below= of L-stability, text width=6em] (Nl-stability) {Stability of KdV \eqref{eq: KdV system-stability-intro}};
\node [block, right=2cm of L-stability, text width=8em] (ob) {Observability $C(T,L)$};
\node [block,above= of ob] (propagation-sing) {Propogation of compactness};
\node [block, right=3cm of propagation-sing, text width=8em] (control) {Quantitative control constrcution};
\node [block, below=of control, text width=8em] (LCKdv) {Controllability of linearized KdV \eqref{eq: linearized KdV system-control-intro}};
\node [block, below=of LCKdv, text width=8em] (CKdv) {Controllability of KdV \eqref{eq: nonlinear KdV system-control-intro}};

\draw [arrow,green](propagation-sing) -- (ob);
\draw [arrow,green](L-stability) -- (Nl-stability);
\draw [arrow,red,thick](control) -- (LCKdv);
\draw [arrow](LCKdv) -- (CKdv);
\draw [arrow,green](ob)--(L-stability) ;
\draw [Arrow,red](LCKdv) -- (ob) node[midway, above] {HUM};
\end{tikzpicture} 
    \caption{Relations among different notations}
    \label{fig:placeholder}
\end{figure}
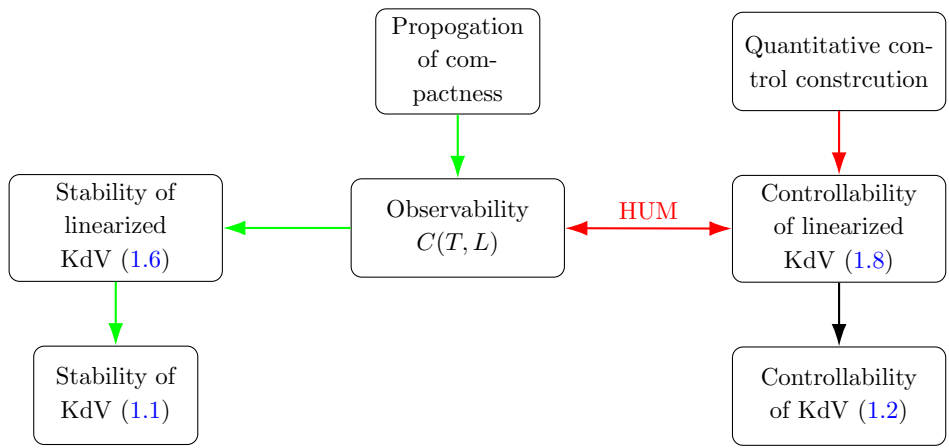

In the literature, it is classical to use observability to prove exponential stability, followed by a \textit{compactness-uniqueness method} in proving the observability (denoted by green arrows in the graph above). Due to the contradiction arguments, it is difficult to track the quantitative properties for both stability and observability. In this paper, we use the \textit{Hilbert uniqueness method}, which is a classic tool in control theory, to transform observability into a controllability problem. Then, in combination with various control ideas and several analysis techniques, we establish observability through a quantitative control construction (in red in the graph above). In general, we could expect this idea, using observability as a bridge and bringing control-theory ingredients, to apply to other PDE models. 

Moreover, our stability result also provides a precise new characterization of the exponential decay rates of the system. As $L$ approaches the critical set (see the definition in \eqref{eq: defi of critical set-intro} below), the decay rates remain uniform in a finite-codimensional invariant manifold.

\vspace{1mm}
Secondly, we introduce a novel constructive approach to tackling control problems, providing a quantitative description of an important but previously unknown constant $C(T,L)$.
This constant was initially defined in the observability inequality for linear KdV equations:
 \begin{equation}
         \int_0^L|y(0,x)|^2dx\leq C(T,L)\int_0^T |\p_x y(t, 0)|^2 dt. \notag
     \end{equation} 

Below we briefly review the literature and some related problems in Section \ref{sec: Review of the literature} and Section \ref{sec: motivations}. We provide an overview of our main results in Section \ref{sec: Statement of the results}, followed by some comments. We finish this introduction part with an outline for this article at the end.
\subsection{Review of the literature}\label{sec: Review of the literature}
The KdV equation was initially introduced by Boussinesq \cite{1877-Boussinesq} and Korteweg and de Vries \cite{KdV} as a model for the propagation of surface water waves along a channel. This equation is now commonly used to model the unidirectional propagation of small amplitude long waves in nonlinear dispersive systems. This equation also serves as a valuable nonlinear approximation model that balances weak nonlinearity and weak dispersive effects. The KdV equation has been extensively studied from various mathematical perspectives in the existing literature, including the
well-posedness, the existence and stability of solitary waves, the integrability, the
long-time behavior, etc., see e.g.~\cite{Whitham74, Kato83,Bourgain93,  CKSTT, KPV96,MV, LP15,KV-19}.

If we further focus on the stability properties of the KdV system \eqref{eq: KdV system-stability-intro}, there are also many related results. At first, by exponential stability, we refer to small data stability results, since there exists a non-trivial stationary state (see for example \cite{Doronin-Natali-2014}). Recall the so-called ``critical lengths set'' for the KdV system \eqref{eq: nonlinear KdV system-control-intro} introduced by Rosier \cite{Rosier-1997} 
\begin{equation}\label{eq: defi of critical set-intro}
     \mathcal{N}:=\{2\pi\sqrt{\frac{k^2+kl+l^2}{3}}:k,l\in\N^*  \}.
\end{equation} 
Defining the energy of this system by $E(t):= \int_0^L |y(t, x)|^2 dx$, we notice that for $f=u\equiv0$, $E(t)$ dissipates due to the absorption on the boundary  $E(0)- E(T)= 2\int_0^T |y_x(t, 0)|^2 dt$. Thus, the system is exponentially stable if and only if the following observability inequality holds
     \begin{equation}\label{eq: general Ob-intro}
         E(0)=\int_0^L|y(0,x)|^2dx\leq C(T,L)\int_0^T |y_x(t, 0)|^2 dt.
     \end{equation} 
When  $L\notin \mathcal{N}$,  Rosier proved \eqref{eq: general Ob-intro},  leading to the local exponential stability (see also \cite{MVZ,Krieger-Xiang-2021}),
\begin{equation}\label{eq: exponential stability-intro}
  E(t)\leq Ce^{-Ct}E(0), \text{ for }t\in(0,+\infty).
\end{equation}
The proofs are based on compactness arguments, and no quantitative information is provided. When $L\in\mathcal{N}$,  Rosier  showed that the linearized system
 \begin{equation}\label{eq: linear KdV-stability-intro}
\left\{
\begin{array}{lll}
    \p_t y+\p_x^3y+\p_xy=0 & \text{ in }(0,T)\times(0,L), \\
     y(t,0)=y(t,L)=\p_xy(t,L)=0&  \text{ in }(0,T),\\
     y(0,x)=y^0(x)&\text{ in }(0,L).
\end{array}
\right.
\end{equation}
 admits a family of non-trivial solutions of the form $e^{\ii\lambda t}\G_{\lambda}(x)$, where $\G_{\lambda}$ (We call it a ``\textit{Type I eigenfucntion}" and see Section \ref{sec: Eigenvalues and eigenfunctions at the critical length} for more details) is the solution to
 \begin{equation}\label{eq: type-1-eigenfunction-intro}
  \left\{
 \begin{array}{c}
      \G'''_{\lambda}+\G'_{\lambda}+\ii\lambda \G_{\lambda}=0,  \\
     \G_{\lambda}(0)=\G_{\lambda}(L)=\G'_{\lambda}(0)=\G'_{\lambda}(L)=0.   
 \end{array}
 \right.    
 \end{equation}
For critical lengths, the nonlinear term plays an important role in the local asymptotic stability of $0$. In \cite{Chu-Coron-Shang}, based on the center manifold method, Chu, Coron, and Shang first considered the special case, where $L=2k\pi$, and later considered some cases with dim $M = 2$ \cite{TCSC-2018}. Later Nguyen considered more general cases in \cite{Nguyen-2021} using the power series expansion method.
\vspace{2mm}

The controllability of the KdV system \eqref{eq: nonlinear KdV system-control-intro} has been extensively studied in the last decades, see e.g. the surveys \cite{RZ09, Cerpa14} and the references therein.  Here we only concentrate on the right Neumann control as presented in \eqref{eq: nonlinear KdV system-control-intro}.  Rosier first showed that the linearized system around $0$ in $L^2(0,L)$
\begin{equation}\label{eq: linearized KdV system-control-intro}
\left\{
\begin{array}{lll}
    \p_ty+\p_x^3y+\p_xy=0 & \text{ in }(0,T)\times(0,L), \\
     y(t,0)=y(t,L)=0&  \text{ in }(0,T),\\
     \p_xy(t,L)=u(t)&  \text{ in }(0,T),\\
     y(0,x)=y^0(x)&\text{ in }(0,L),
\end{array}
\right.    
\end{equation} 
is controllable if and only if $L\notin\mathcal{N}$, from which he deduced that \eqref{eq: nonlinear KdV system-control-intro} is locally controllable if $L\notin\mathcal{N}$. In the same paper, for $L_0\in \mathcal{N}$, he also showed decomposition of $L^2(0,L_0)=H(L_0)\oplus M(L_0)$, where $M(L_0)$ denotes the unreachable subspace for the linearized system of \eqref{eq: linearized KdV system-control-intro}  and $H$ is the reachable subspace. 
\begin{equation}\label{eq: defi-M-L-0-intro}
    M(L_0)=\mathrm{Span}\{\Re \G_{\lambda}, \Im \G_{\lambda}:\G_{\lambda} \text{ defined in \eqref{eq: type-1-eigenfunction-intro} above}\}, \text{ and }\mathrm{dim} M(L_0)=N_0,
\end{equation}
where $N_0$
is the number of different pairs of positive integers $(k_j,l_j)$ satisfying $L=2\pi\sqrt{\frac{k_j^2+k_jl_j+l_j^2}{3}}$. We shall discuss about this crucial dimension $N_0$ in Section \ref{sec: limiting analysis on different types}.

 To deal with the control problem for $L\in\mathcal{N}$, Coron and Cr\'epeau introduced the power series expansion method in \cite{CC04} and proved that the system \eqref{eq: nonlinear KdV system-control-intro} is locally controllable if $N_0=1$. Later on the controllability and stabilization of the KdV equation on critical lengths has been extensively investigated, see \cite{Cerpa07, Cerpa-Crepeau-2009, Coron-Rivas-Xiang, CKN-JEMS, NX, Nguyen-kdv-2025} and the references therein.

\subsection{Motivations}\label{sec: motivations}
In this paper, we aim to study the following questions.
\subsubsection{The constructive control approach}
Note that in \cite{Rosier-1997} Rosier used the compactness-uniqueness method to establish the observability  \eqref{eq: general Ob-intro} and the controllability of \eqref{eq: linearized KdV system-control-intro}. This method is not sufficient to provide enough information except for the existence.  The value of the observability constant has played a role in quantifying various estimates, including control costs and the rate of exponential decay etc.
\begin{question}\label{OP: constructive method}
Concerning the Neumann boundary control problem of KdV equations as presented in \eqref{eq: linearized KdV system-control-intro}, can we find a constructive approach for null controllability, thus quantitatively characterize the observability constant $C(T,L)$?
\end{question}

\subsubsection{The fast control cost}
Given $L>0$. For every $T_0>0$, the KdV system \eqref{eq: linearized KdV system-control-intro} is null controllable at time $T_0$ with some control cost $C(T,L)$. As $T\rightarrow0^+$, which corresponds to a fast control process, we expect the control cost $C(T,L)$ to blow up. Understanding the asymptotic behavior of fast controls is of great interest in itself, but it may also be applied to studying the uniform controllability in the zero-dispersion limit, e.g. see \cite{GG08}. This leads us to the following question:
\begin{question}\label{OP: C(T,L)-T-limit}
Given $L\notin \mathcal{N}$, when $T\rightarrow0^+$, how does $C(T,L)$ behave (blow up)?
\end{question}

\subsubsection{The limiting stability problem}
For both control and stability problems, we have observed different behaviors depending on whether $L\in\mathcal{N}$ or $L\notin \mathcal{N}$. Additionally, at a critical length $L\in\mathcal{N}$, the space $L^2(0,L)$ decomposes into $H\oplus M$. In the subspace $H$, the linearized KdV equations exhibit both null controllability and exponential stability. Conversely, in the subspace $M$, neither null controllability nor exponential stability is achieved due to the existence of Type I eigenfunctions of the form \eqref{eq: type-1-eigenfunction-intro}. Following Problem \ref{OP: constructive method}, it would be interesting to know
\begin{question}\label{OP: C(T,L)-L-limit}
Given $T>0$ and $L_0\in \mathcal{N}$, can we describe the asymptotic behavior of $C(T,L)$ as $L$ tends to $L_0$? Can we define a decomposition $H(L)\oplus M(L)$ such that different decay rates are observed in $H(L)$ and $M(L)$? Furthermore, how does this behavior extend to the nonlinear case?
\end{question}

\subsection{Statement of the results}\label{sec: Statement of the results}
\subsubsection{Stability results}
Firstly, our main results concern the sharp exponential stability of the KdV equations \eqref{eq: linear KdV-stability-intro}.
\begin{thm}\label{thm: main theorem linear version}
Let $L_0\in \mathcal{N}$ be a fixed critical length.  Let $I=[L_0- \delta,L_0+ \delta]$ with $\delta= \frac{\pi^2}{3L_0^2}$. For every $L\in I\setminus\{L_0\}$, the state space $L^2(0,L)$ can be decomposed as $H_{\A}(L)\oplus M_{\A}(L)$ (see Section \ref{sec: Projections and state space decomposition} for detailed definitions). Moreover, there are effective computable constants $C_0(L_0)$, $K_0(L_0)$, $m_0(L_0)$, $M_0(L_0)$, $C_u(L_0)$, and $R_{0}(L_0)$ independent of $L$ in $I$, such that   
\begin{enumerate}
    \item  for  $\forall y^0\in H_{\A}(L)$, \eqref{eq: linear KdV-stability-intro} is exponentially stable and the solution $y$ satisfies the following decay estimates
\begin{gather*}
    E(y(t))\leq e^{-C_0 t}E(y^0),\forall t\in(0, +\infty),\\
\|y^0\|^2_{L^2(0,L)} \leq  K_0e^{\frac{K_0}{\sqrt{t}}}\int_0^t|\p_x y(s,0)|^2ds, \forall t\in (0, +\infty).
\end{gather*}
    \item  for  $\forall y^0\in M_{\A}(L)$, \eqref{eq: linear KdV-stability-intro} is exponentially stable and the solution $y$ satisfies 
\begin{gather*}
    E(y(t))\leq e^{-m_0 |L-L_0|^2 t}E(y^0),\forall t\in(0, +\infty),\\
    \|y^0\|^2_{L^2(0,L)} \leq  \frac{K_0}{t|L-L_0|^2}\int_0^t|\p_x y(s,0)|^2ds, \forall t\in (0, +\infty).
\end{gather*} 
On the other hand there exists $y^0\in M_{\A}(L)$ such that the the solution satisfies 
\begin{equation*}
    E(y(t))\geq e^{-M_0 |L-L_0|^2 t}E(y^0),\forall t\in(0, +\infty).
\end{equation*} 
    \item  for  $\forall y^0\in L^2(0,L)$, \eqref{eq: linear KdV-stability-intro} is exponentially stable and the solution $y$ satisfies 
\begin{gather*}
    E(y(t))\leq C_u e^{-R_0 |L-L_0|^2 t}E(y^0),\forall t\in(0, +\infty).
\end{gather*} 
\end{enumerate}
\end{thm}

 \begin{rem}
   Note that in Theorem \ref{thm: main theorem linear version}, the constant $C(T,L)$ blows up at a rate of $\frac{1}{|L-L_0|^2}$, which is sharp. Moreover, the decomposition is sharp in the sense that the blow-up occurs only along the directions in $M_{\A}(L)$,  while observability holds uniformly in $H_{\A}(L)$. 
 \end{rem}
Our stability analysis is based on the decomposition $L^2(0,L)=M_{\A}(L)\oplus H_{\A}(L)$ related to the operator $\A$ defined by $\A\varphi:=-\varphi'''-\varphi'$ with the domain
\begin{equation*}
D(\A):=\{\varphi\in H^3(0,L): \varphi(0)=\varphi(L)=\varphi'(L)=0\}.
\end{equation*}
This provides an answer to the linear part of Problem \ref{OP: C(T,L)-L-limit}. We shall further show in Section \ref{sec: classification quasiinvariant space} that $M_{\A}(L)\sim M(L_0)+\bigO(|L-L_0|)$, where $M(L_0)$ is the unreachable subspace defined in \eqref{eq: defi-M-L-0-intro}. This result aligns well with the intuition, i.e., $C(T,L)$ blows up in $M_{\A}(L)$ when $L$ approaches the critical length $L_0$, which means we will finally lose control in the limit space $M(L_0)$.

We introduce a constructive approach to obtain the quantitative estimates. The operator $\A$ is difficult to cooperate with the moment methods, since $\A$ is neither self-adjoint nor skew-adjoint and its eigenfunctions do not form a Riesz basis. To deal with this difficulty, we benefit from another related operator $\B$ defined by $\B\varphi:=-\varphi'''-\varphi'$ with the domain
\begin{equation*}
D(\B):=\{\varphi\in H^3(0,L): \varphi(0)=\varphi(L)=0,\varphi'(0)=\varphi'(L)\}.
\end{equation*}
$\B$ is skew-adjoint with  eigenmodes $(\ii\lambda_j,\E_j)_{j\in\Z\setminus\{0\}}$ satisfying $\B\E_j=\ii\lambda_j\E_j$ (See more details in Section \ref{sec: Asymptotic behavior close to the critical lengths}). We aim to have a good {\it transition} from $\A$ to $\B$ and vice versa. We achieve this by finding a proper subspace $H_{\B}(L)$ such that $L^2(0,L)=M_{\B}(L)\oplus H_{\B}(L)$ and finding transitions between $H_{\A}(L)$ and $H_{\B}(L)$. We also expect $M_{\B}(L)\sim M(L_0)+\bigO(|L-L_0|)$\footnote{This notation is not mathematically precise, and we use it to simply provide an intuition that the two subspaces are clsoe to each other.}.

Interestingly, it turns out that the characterization of the subspace $M_{\B}(L)$ depends largely on the classification of critical lengths. We need more detailed spectral analysis of the asymptotic behaviors of the eigenvalues and eigenfunctions of $\B$ (we refer to Section \ref{sec: Asymptotic behavior close to the critical lengths}), see Section \ref{sec: classification quasiinvariant space}. To provide a brief intuition, we show some links between this characterization and the classification of critical lengths using the following typical examples.
\begin{itemize}
    \item Let $L_0=2\pi$. Then, the only possible pair is $k=l=1$. Moreover, 
    \begin{equation*}
    M(L_0)=\mathrm{Span}\{1-\cos{x}\},\;\;M_{\A}(L)=\mathrm{Span}\{1-\cos{x}+\bigO(L-2\pi)\}.
    \end{equation*}
Note that $1-\cos{x}$ is an eigenfunction of $\A$ at $L_0=2\pi$ with the eigenvalue $0$, and $(1-\cos{x})'(0)=0$. When $L$ is close to $L_0$, the related eigenfunctions of $\mathcal{A}$ and $\mathcal{B}$ can be regarded as a perturbation.  For $\B$, there are two eigenvalues $\ii\lambda_{\pm1}=\bigO(|L-2\pi|)$ close to $0$  and the associated eigenfunctions are given by  $\E_{\pm1}=\frac{1}{\sqrt{6\pi}}(1-\cos{x}\pm\sqrt{3}\ii\sin{x})+\bigO(|L-2\pi|)$ (see Fig. \ref{fig: 2pi-case-intro}). In particular, one shall notice that $\E_{1}'(L)=\E_{1}'(0)\neq0$ and this holds uniformly as $L\to2\pi$.

One can not directly define $M_{\mathcal{B}}(L)$ as $\mathrm{Span}{\E_{+1}}$ or $\mathrm{Span}{\E_{-1}}$,  due to the boundary derivatives. Indeed, $\Re\E_{+1}\sim 1-\cos{x}+\bigO(L-2\pi)$, while $\Im\E_1\sim \sin{x}+\bigO(L-2\pi)$. For the $\sin{x}$ function,
$$\B(\sin{x})=0\text{ at }L=2\pi,\text{ with } \sin'x|_{x=0}=\sin'x|_{x=2\pi}=1.
$$ 
We call this eigenfunction of Type 2, and we study more general cases in Section \ref{sec: Eigenvalues and eigenfunctions at the critical length}. 
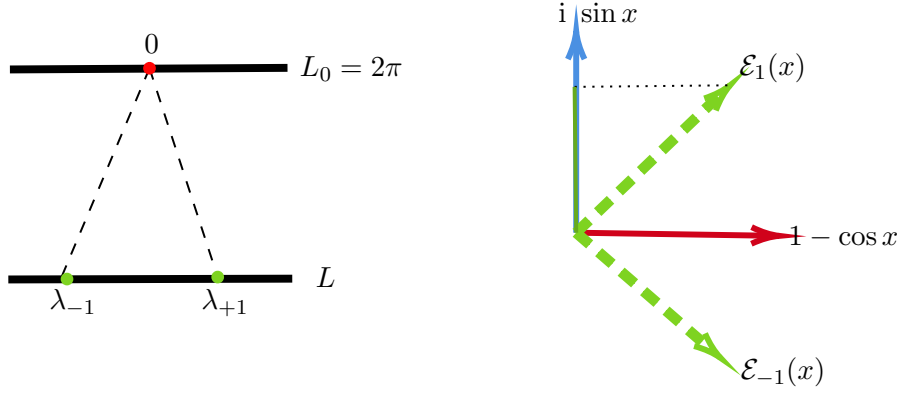
\begin{figure}
    \centering
\tikzset{every picture/.style={line width=0.75pt}} 

\begin{tikzpicture}[x=0.75pt,y=0.75pt,yscale=-0.8,xscale=0.8]

\draw [line width=3]    (82,64.5) -- (257,63.5) ;

\draw [line width=3]    (81,197) -- (260,196.5) ;

\draw  [dash pattern={on 4.5pt off 4.5pt}]  (170,63.75) -- (114,196.75) ;

\draw  [dash pattern={on 4.5pt off 4.5pt}]  (170,63.75) -- (213,195.5) ;

\draw  [draw opacity=0][fill={rgb, 255:red, 254; green, 5; blue, 5 }  ,fill opacity=1 ] (166,63.75) .. controls (166,61.54) and (167.79,59.75) .. (170,59.75) .. controls (172.21,59.75) and (174,61.54) .. (174,63.75) .. controls (174,65.96) and (172.21,67.75) .. (170,67.75) .. controls (167.79,67.75) and (166,65.96) .. (166,63.75) -- cycle ;
\draw  [draw opacity=0][fill={rgb, 255:red, 126; green, 211; blue, 33 }  ,fill opacity=1 ] (114,196.75) .. controls (114,194.54) and (115.79,192.75) .. (118,192.75) .. controls (120.21,192.75) and (122,194.54) .. (122,196.75) .. controls (122,198.96) and (120.21,200.75) .. (118,200.75) .. controls (115.79,200.75) and (114,198.96) .. (114,196.75) -- cycle ;
\draw  [draw opacity=0][fill={rgb, 255:red, 126; green, 211; blue, 33 }  ,fill opacity=1 ] (209,195.5) .. controls (209,193.29) and (210.79,191.5) .. (213,191.5) .. controls (215.21,191.5) and (217,193.29) .. (217,195.5) .. controls (217,197.71) and (215.21,199.5) .. (213,199.5) .. controls (210.79,199.5) and (209,197.71) .. (209,195.5) -- cycle ;
\draw [color={rgb, 255:red, 208; green, 2; blue, 27 }  ,draw opacity=1 ][line width=2.25]    (439,167.5) -- (563,169.44) ;
\draw [shift={(567,169.5)}, rotate = 180.9] [color={rgb, 255:red, 208; green, 2; blue, 27 }  ,draw opacity=1 ][line width=2.25]    (17.49,-5.26) .. controls (11.12,-2.23) and (5.29,-0.48) .. (0,0) .. controls (5.29,0.48) and (11.12,2.23) .. (17.49,5.26)   ;
\draw [color={rgb, 255:red, 74; green, 144; blue, 226 }  ,draw opacity=1 ][line width=2.25]    (439,167.5) -- (439,44.5) ;
\draw [shift={(439,40.5)}, rotate = 90] [color={rgb, 255:red, 74; green, 144; blue, 226 }  ,draw opacity=1 ][line width=2.25]    (17.49,-5.26) .. controls (11.12,-2.23) and (5.29,-0.48) .. (0,0) .. controls (5.29,0.48) and (11.12,2.23) .. (17.49,5.26)   ;
\draw [color={rgb, 255:red, 126; green, 211; blue, 33 }  ,draw opacity=1 ][line width=2.25]  [dash pattern={on 6.75pt off 4.5pt}]  (437.96,166.42) -- (527.74,80.34)(440.04,168.58) -- (529.82,82.5) ;
\draw [shift={(536,74.5)}, rotate = 136.21] [color={rgb, 255:red, 126; green, 211; blue, 33 }  ,draw opacity=1 ][line width=2.25]    (17.49,-5.26) .. controls (11.12,-2.23) and (5.29,-0.48) .. (0,0) .. controls (5.29,0.48) and (11.12,2.23) .. (17.49,5.26)   ;
\draw [color={rgb, 255:red, 126; green, 211; blue, 33 }  ,draw opacity=1 ][line width=2.25]  [dash pattern={on 6.75pt off 4.5pt}]  (439.98,166.37) -- (523.43,238.81)(438.02,168.63) -- (521.47,241.08) ;
\draw [shift={(530,246.5)}, rotate = 220.96] [color={rgb, 255:red, 126; green, 211; blue, 33 }  ,draw opacity=1 ][line width=2.25]    (17.49,-5.26) .. controls (11.12,-2.23) and (5.29,-0.48) .. (0,0) .. controls (5.29,0.48) and (11.12,2.23) .. (17.49,5.26)   ;
\draw  [dash pattern={on 0.84pt off 2.51pt}]  (438,75.5) -- (536,74.5) ;
\draw [color={rgb, 255:red, 102; green, 170; blue, 30 }  ,draw opacity=1 ][line width=1.5]    (438,75.5) -- (439,167.5) ;

\draw (165,40.4) node [anchor=north west][inner sep=0.75pt]    {$0$};
\draw (105,201.4) node [anchor=north west][inner sep=0.75pt]    {$\lambda _{-1}$};
\draw (201,201.4) node [anchor=north west][inner sep=0.75pt]    {$\lambda _{+1}$};
\draw (273,187.4) node [anchor=north west][inner sep=0.75pt]    {$L$};
\draw (263,54.4) node [anchor=north west][inner sep=0.75pt]    {$L_{0} =2\pi $};
\draw (426,20.4) node [anchor=north west][inner sep=0.75pt]    {$\mathrm{i} \ \sin x$};
\draw (571,161.4) node [anchor=north west][inner sep=0.75pt]    {$1-\cos x$};
\draw (540,52.4) node [anchor=north west][inner sep=0.75pt]    {$\mathcal{E}_{1}( x)$};
\draw (541,244.4) node [anchor=north west][inner sep=0.75pt]    {$\mathcal{E}_{-1}( x)$};

\end{tikzpicture}
    \caption{$2\pi$ case: asymptotic behaviors of eigenvalues and eigenfunctions}
    \label{fig: 2pi-case-intro}
\end{figure}

In this case, to get rid of the Type 2 eigenmodes of $\B$, we define 
\begin{equation*}
M_{\B}(L)\sim \mathrm{Span}\{\Re\E_{+1}\}=M(L_0)+\bigO(|L-2\pi|).
\end{equation*}
Note that, however, the Type 2 eigenfunctions are useful in defining the transition maps. We leave this part for more explanation in Remark \ref{rem: rho}.
    \item Now we turn to a different case $L_0=2\pi\sqrt{\frac{7}{3}}$ with $k=2,l=1$. We have $\dim M(L_0)=2$ and
    \begin{equation*}
    M(L_0)=\mathrm{Span}\{\Re\G_c,\Im\G_c,\text{ where }\A(\G_c)=\ii\frac{20}{21\sqrt{21}}\G_c\},\;M_{\A}(L)\sim M(L_0)+\bigO(|L-2\pi\sqrt{\frac{7}{3}}|).
    \end{equation*}
Note that $\G_c'(0)=0$. In this case, things are more direct. There are two eigenvalues $\ii\lambda_{\pm1}$ of $\B$ such that $\ii\lambda_{\pm1}=\pm\ii\frac{20}{21\sqrt{21}}+\bigO((L-2\pi\sqrt{\frac{7}{3}})^2)$ and their associated eigenfunctions $\E_{\pm1}\sim \G_{\pm c}+\bigO(|L-2\pi\sqrt{\frac{7}{3}}|)$. Hence, here we define
\begin{equation*}
M_{\B}(L)=\mathrm{Span}\{\Re\E_1,\Im\E_1,\text{ where }\B(\E_1)=\ii\lambda_1\E_1\} \sim M(L_0)+\bigO(|L-2\pi\sqrt{\frac{7}{3}}|).
\end{equation*}
In this case, there is no Type 2 eigenmode for $\B$ at $L=2\pi\sqrt{\frac{7}{3}}$, see Section \ref{sec: Asymptotic behavior close to the critical lengths}.
\begin{figure}
    \centering
\tikzset{every picture/.style={line width=0.75pt}} 

\begin{tikzpicture}[x=0.75pt,y=0.75pt,yscale=-0.8,xscale=0.8]

\draw [line width=3]    (106,43) -- (401,43) ;
\draw [line width=3]    (109,176) -- (404,176) ;
\draw  [dash pattern={on 4.5pt off 4.5pt}]  (144,46.5) -- (170,177) ;
\draw  [draw opacity=0][fill={rgb, 255:red, 254; green, 5; blue, 5 }  ,fill opacity=1 ] (140,42.5) .. controls (140,40.29) and (141.79,38.5) .. (144,38.5) .. controls (146.21,38.5) and (148,40.29) .. (148,42.5) .. controls (148,44.71) and (146.21,46.5) .. (144,46.5) .. controls (141.79,46.5) and (140,44.71) .. (140,42.5) -- cycle ;
\draw  [draw opacity=0][fill={rgb, 255:red, 254; green, 5; blue, 5 }  ,fill opacity=1 ] (240,41.75) .. controls (240,39.54) and (241.79,37.75) .. (244,37.75) .. controls (246.21,37.75) and (248,39.54) .. (248,41.75) .. controls (248,43.96) and (246.21,45.75) .. (244,45.75) .. controls (241.79,45.75) and (240,43.96) .. (240,41.75) -- cycle ;
\draw  [draw opacity=0][fill={rgb, 255:red, 126; green, 211; blue, 33 }  ,fill opacity=1 ] (166,177) .. controls (166,174.79) and (167.79,173) .. (170,173) .. controls (172.21,173) and (174,174.79) .. (174,177) .. controls (174,179.21) and (172.21,181) .. (170,181) .. controls (167.79,181) and (166,179.21) .. (166,177) -- cycle ;
\draw  [draw opacity=0][fill={rgb, 255:red, 254; green, 5; blue, 5 }  ,fill opacity=1 ] (345,42.5) .. controls (345,40.29) and (346.79,38.5) .. (349,38.5) .. controls (351.21,38.5) and (353,40.29) .. (353,42.5) .. controls (353,44.71) and (351.21,46.5) .. (349,46.5) .. controls (346.79,46.5) and (345,44.71) .. (345,42.5) -- cycle ;
\draw  [draw opacity=0][fill={rgb, 255:red, 126; green, 211; blue, 33 }  ,fill opacity=1 ] (304,176.5) .. controls (304,174.29) and (305.79,172.5) .. (308,172.5) .. controls (310.21,172.5) and (312,174.29) .. (312,176.5) .. controls (312,178.71) and (310.21,180.5) .. (308,180.5) .. controls (305.79,180.5) and (304,178.71) .. (304,176.5) -- cycle ;
\draw [fill={rgb, 255:red, 252; green, 3; blue, 3 }  ,fill opacity=1 ] [dash pattern={on 4.5pt off 4.5pt}]  (349,42.5) -- (308,172.5) ;
\draw [color={rgb, 255:red, 208; green, 2; blue, 27 }  ,draw opacity=1 ][line width=2.25]  [dash pattern={on 2.53pt off 3.02pt}]  (466,131.5) -- (575,131.98) ;
\draw [shift={(579,132)}, rotate = 180.25] [color={rgb, 255:red, 208; green, 2; blue, 27 }  ,draw opacity=1 ][line width=2.25]    (17.49,-5.26) .. controls (11.12,-2.23) and (5.29,-0.48) .. (0,0) .. controls (5.29,0.48) and (11.12,2.23) .. (17.49,5.26)   ;
\draw [color={rgb, 255:red, 208; green, 2; blue, 27 }  ,draw opacity=1 ][line width=2.25]    (466,131.5) -- (572.67,60.22) ;
\draw [shift={(576,58)}, rotate = 146.25] [color={rgb, 255:red, 208; green, 2; blue, 27 }  ,draw opacity=1 ][line width=2.25]    (17.49,-5.26) .. controls (11.12,-2.23) and (5.29,-0.48) .. (0,0) .. controls (5.29,0.48) and (11.12,2.23) .. (17.49,5.26)   ;
\draw [color={rgb, 255:red, 74; green, 144; blue, 226 }  ,draw opacity=1 ][line width=2.25]  [dash pattern={on 6.75pt off 4.5pt}]  (464.96,130.42) -- (554.74,44.34)(467.04,132.58) -- (556.82,46.5) ;
\draw [shift={(563,38.5)}, rotate = 136.21] [color={rgb, 255:red, 74; green, 144; blue, 226 }  ,draw opacity=1 ][line width=2.25]    (17.49,-5.26) .. controls (11.12,-2.23) and (5.29,-0.48) .. (0,0) .. controls (5.29,0.48) and (11.12,2.23) .. (17.49,5.26)   ;
\draw [color={rgb, 255:red, 74; green, 144; blue, 226 }  ,draw opacity=1 ][line width=2.25]  [dash pattern={on 6.75pt off 4.5pt}]  (466.98,130.37) -- (550.43,202.81)(465.02,132.63) -- (548.47,205.08) ;
\draw [shift={(557,210.5)}, rotate = 220.96] [color={rgb, 255:red, 74; green, 144; blue, 226 }  ,draw opacity=1 ][line width=2.25]    (17.49,-5.26) .. controls (11.12,-2.23) and (5.29,-0.48) .. (0,0) .. controls (5.29,0.48) and (11.12,2.23) .. (17.49,5.26)   ;
\draw [color={rgb, 255:red, 208; green, 2; blue, 27 }  ,draw opacity=1 ][line width=2.25]    (471,133.5) -- (566.59,191.91) ;
\draw [shift={(570,194)}, rotate = 211.43] [color={rgb, 255:red, 208; green, 2; blue, 27 }  ,draw opacity=1 ][line width=2.25]    (17.49,-5.26) .. controls (11.12,-2.23) and (5.29,-0.48) .. (0,0) .. controls (5.29,0.48) and (11.12,2.23) .. (17.49,5.26)   ;
\draw [color={rgb, 255:red, 208; green, 2; blue, 27 }  ,draw opacity=1 ][line width=2.25]  [dash pattern={on 2.53pt off 3.02pt}]  (466,131.5) -- (466.95,61) ;
\draw [shift={(467,57)}, rotate = 90.77] [color={rgb, 255:red, 208; green, 2; blue, 27 }  ,draw opacity=1 ][line width=2.25]    (17.49,-5.26) .. controls (11.12,-2.23) and (5.29,-0.48) .. (0,0) .. controls (5.29,0.48) and (11.12,2.23) .. (17.49,5.26)   ;

\draw (125,15.4) node [anchor=north west][inner sep=0.75pt]    {$\lambda _{-c}$};
\draw (341,16.4) node [anchor=north west][inner sep=0.75pt]    {$\lambda _{c}$};
\draw (239,18.4) node [anchor=north west][inner sep=0.75pt]    {$0$};
\draw (158,185.4) node [anchor=north west][inner sep=0.75pt]    {$\lambda _{-1}$};
\draw (295,180.4) node [anchor=north west][inner sep=0.75pt]    {$\lambda _{+1}$};
\draw (90,167.4) node [anchor=north west][inner sep=0.75pt]    {$L$};
\draw (8,25.4) node [anchor=north west][inner sep=0.75pt]    {$L_{0} =2\pi \sqrt{\frac{7}{3}}$};
\draw (444,22.4) node [anchor=north west][inner sep=0.75pt]    {$\mathrm{i\ Im}\mathcal{G}_{c}$};
\draw (543,11.4) node [anchor=north west][inner sep=0.75pt]    {$\mathcal{E}_{1}$};
\draw (538,222.4) node [anchor=north west][inner sep=0.75pt]    {$\mathcal{E}_{-1}$};
\draw (597,122.4) node [anchor=north west][inner sep=0.75pt]    {$\mathrm{Re}\mathcal{G}_{c}$};
\draw (577,44.4) node [anchor=north west][inner sep=0.75pt]    {$\mathcal{G}_{c}$};
\draw (575,189.4) node [anchor=north west][inner sep=0.75pt]    {$\mathcal{G}_{-c}$};
\end{tikzpicture}
    \caption{$2\pi\sqrt{7/3}$ case: asymptotic behaviors of eigenvalues and eigenfunctions}
    \label{fig: 2pi7-3-intro}
\end{figure}
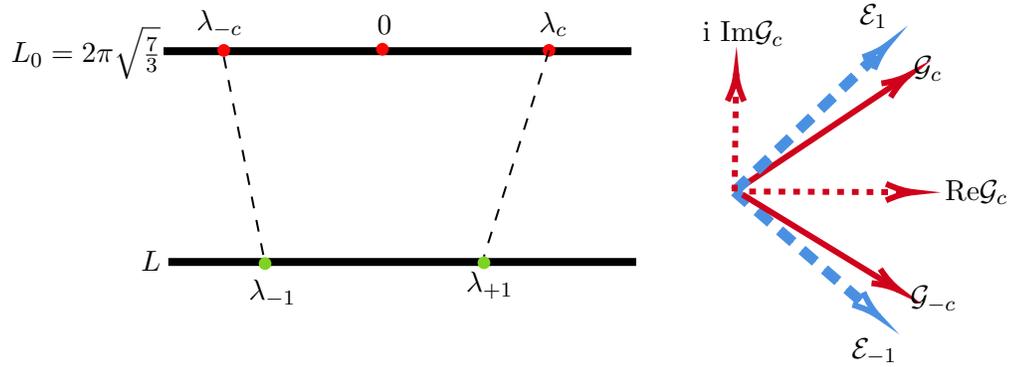
    \item The next example concerns the case $L_0=2\pi\sqrt{7}$ with $k=4,l=1$. We also have $\dim M(L_0)=2$ and 
    \begin{equation*}
    M(L_0)=\mathrm{Span}\{\Re\G_c,\Im\G_c,\text{ where }\A(\G_c)=\ii\frac{6\sqrt{7}}{49}\G_c\},\; M_{\A}(L) \sim M(L_0)+\bigO(|L-2\pi\sqrt{\frac{7}{3}}|).
    \end{equation*}
Note that $\G_c'(0)=0$. Even though this case looks quite similar to the second one, there are some distinct phenomena. We have four eigenvalues $\ii\lambda_{\pm2},\ii\lambda_{\pm1}$ of $\B$ such that 
$\ii\lambda_{1}=\ii\frac{6\sqrt{7}}{49}+\bigO(|L-2\pi\sqrt{7}|)$, $\ii\lambda_{2}=\ii\frac{6\sqrt{7}}{49}+\bigO(|L-2\pi\sqrt{7}|)$, and $\lambda_{-1}=-\lambda_1,\lambda_{-2}=-\lambda_2$.  Let $\{\E_j\}_{j=\pm1,\pm2}$ be the associated eigenfunctions ($\B(\E_j)=\ii\lambda\E_j$). To give a rough idea, we present Figure \ref{fig: 2pi-7-intro} in analog of previous two cases.
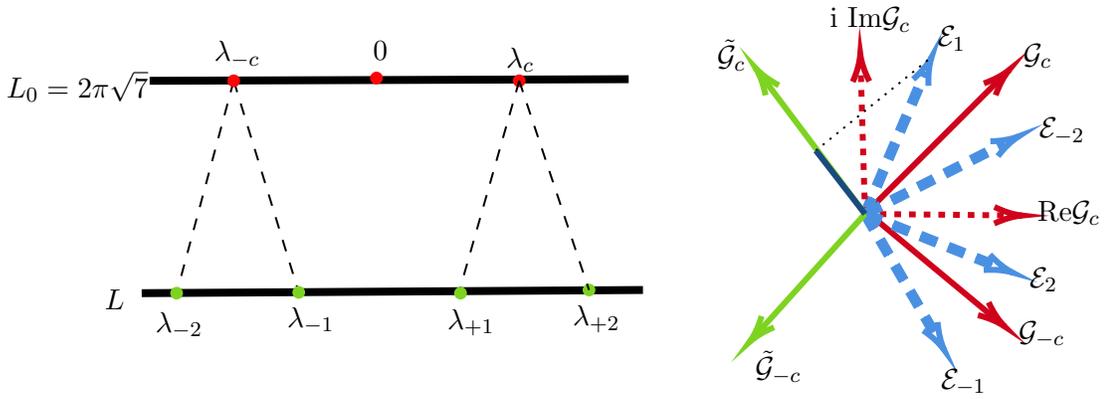
\begin{figure}
    \centering
\tikzset{every picture/.style={line width=0.75pt}} 

\begin{tikzpicture}[x=0.75pt,y=0.75pt,yscale=-0.8,xscale=0.8]

\draw [line width=3]    (95,53.5) -- (397,52.5) ;
\draw [line width=3]    (90,187.5) -- (406,185.5) ;
\draw [fill={rgb, 255:red, 252; green, 3; blue, 3 }  ,fill opacity=1 ] [dash pattern={on 4.5pt off 4.5pt}]  (148,53.5) -- (112,187.5) ;
\draw  [dash pattern={on 4.5pt off 4.5pt}]  (148,53.5) -- (189,187) ;
\draw  [draw opacity=0][fill={rgb, 255:red, 254; green, 5; blue, 5 }  ,fill opacity=1 ] (144,53.5) .. controls (144,51.29) and (145.79,49.5) .. (148,49.5) .. controls (150.21,49.5) and (152,51.29) .. (152,53.5) .. controls (152,55.71) and (150.21,57.5) .. (148,57.5) .. controls (145.79,57.5) and (144,55.71) .. (144,53.5) -- cycle ;
\draw  [draw opacity=0][fill={rgb, 255:red, 254; green, 5; blue, 5 }  ,fill opacity=1 ] (234,51.75) .. controls (234,49.54) and (235.79,47.75) .. (238,47.75) .. controls (240.21,47.75) and (242,49.54) .. (242,51.75) .. controls (242,53.96) and (240.21,55.75) .. (238,55.75) .. controls (235.79,55.75) and (234,53.96) .. (234,51.75) -- cycle ;
\draw  [draw opacity=0][fill={rgb, 255:red, 126; green, 211; blue, 33 }  ,fill opacity=1 ] (108,187.5) .. controls (108,185.29) and (109.79,183.5) .. (112,183.5) .. controls (114.21,183.5) and (116,185.29) .. (116,187.5) .. controls (116,189.71) and (114.21,191.5) .. (112,191.5) .. controls (109.79,191.5) and (108,189.71) .. (108,187.5) -- cycle ;
\draw  [draw opacity=0][fill={rgb, 255:red, 126; green, 211; blue, 33 }  ,fill opacity=1 ] (185,187) .. controls (185,184.79) and (186.79,183) .. (189,183) .. controls (191.21,183) and (193,184.79) .. (193,187) .. controls (193,189.21) and (191.21,191) .. (189,191) .. controls (186.79,191) and (185,189.21) .. (185,187) -- cycle ;
\draw  [draw opacity=0][fill={rgb, 255:red, 254; green, 5; blue, 5 }  ,fill opacity=1 ] (324,53.5) .. controls (324,51.29) and (325.79,49.5) .. (328,49.5) .. controls (330.21,49.5) and (332,51.29) .. (332,53.5) .. controls (332,55.71) and (330.21,57.5) .. (328,57.5) .. controls (325.79,57.5) and (324,55.71) .. (324,53.5) -- cycle ;
\draw  [draw opacity=0][fill={rgb, 255:red, 126; green, 211; blue, 33 }  ,fill opacity=1 ] (287,187.5) .. controls (287,185.29) and (288.79,183.5) .. (291,183.5) .. controls (293.21,183.5) and (295,185.29) .. (295,187.5) .. controls (295,189.71) and (293.21,191.5) .. (291,191.5) .. controls (288.79,191.5) and (287,189.71) .. (287,187.5) -- cycle ;
\draw  [draw opacity=0][fill={rgb, 255:red, 126; green, 211; blue, 33 }  ,fill opacity=1 ] (368,185.5) .. controls (368,183.29) and (369.79,181.5) .. (372,181.5) .. controls (374.21,181.5) and (376,183.29) .. (376,185.5) .. controls (376,187.71) and (374.21,189.5) .. (372,189.5) .. controls (369.79,189.5) and (368,187.71) .. (368,185.5) -- cycle ;
\draw [fill={rgb, 255:red, 252; green, 3; blue, 3 }  ,fill opacity=1 ] [dash pattern={on 4.5pt off 4.5pt}]  (328,53.5) -- (291,183.5) ;
\draw  [dash pattern={on 4.5pt off 4.5pt}]  (328,53.5) -- (372,185.5) ;
\draw [color={rgb, 255:red, 208; green, 2; blue, 27 }  ,draw opacity=1 ][line width=2.25]  [dash pattern={on 2.53pt off 3.02pt}]  (546,137.5) -- (641,138.46) ;
\draw [shift={(645,138.5)}, rotate = 180.58] [color={rgb, 255:red, 208; green, 2; blue, 27 }  ,draw opacity=1 ][line width=2.25]    (17.49,-5.26) .. controls (11.12,-2.23) and (5.29,-0.48) .. (0,0) .. controls (5.29,0.48) and (11.12,2.23) .. (17.49,5.26)   ;
\draw [color={rgb, 255:red, 208; green, 2; blue, 27 }  ,draw opacity=1 ][line width=2.25]    (546,137.5) -- (634.17,49.33) ;
\draw [shift={(637,46.5)}, rotate = 135] [color={rgb, 255:red, 208; green, 2; blue, 27 }  ,draw opacity=1 ][line width=2.25]    (17.49,-5.26) .. controls (11.12,-2.23) and (5.29,-0.48) .. (0,0) .. controls (5.29,0.48) and (11.12,2.23) .. (17.49,5.26)   ;
\draw [color={rgb, 255:red, 74; green, 144; blue, 226 }  ,draw opacity=1 ][line width=2.25]  [dash pattern={on 6.75pt off 4.5pt}]  (544.62,136.91) -- (582.71,47.12)(547.38,138.09) -- (585.48,48.29) ;
\draw [shift={(588,38.5)}, rotate = 112.99] [color={rgb, 255:red, 74; green, 144; blue, 226 }  ,draw opacity=1 ][line width=2.25]    (17.49,-5.26) .. controls (11.12,-2.23) and (5.29,-0.48) .. (0,0) .. controls (5.29,0.48) and (11.12,2.23) .. (17.49,5.26)   ;
\draw [color={rgb, 255:red, 74; green, 144; blue, 226 }  ,draw opacity=1 ][line width=2.25]  [dash pattern={on 6.75pt off 4.5pt}]  (547.3,136.75) -- (592.27,214.1)(544.7,138.25) -- (589.68,215.61) ;
\draw [shift={(596,223.5)}, rotate = 239.83] [color={rgb, 255:red, 74; green, 144; blue, 226 }  ,draw opacity=1 ][line width=2.25]    (17.49,-5.26) .. controls (11.12,-2.23) and (5.29,-0.48) .. (0,0) .. controls (5.29,0.48) and (11.12,2.23) .. (17.49,5.26)   ;
\draw [color={rgb, 255:red, 208; green, 2; blue, 27 }  ,draw opacity=1 ][line width=2.25]    (551,139.5) -- (630.92,205.94) ;
\draw [shift={(634,208.5)}, rotate = 219.74] [color={rgb, 255:red, 208; green, 2; blue, 27 }  ,draw opacity=1 ][line width=2.25]    (17.49,-5.26) .. controls (11.12,-2.23) and (5.29,-0.48) .. (0,0) .. controls (5.29,0.48) and (11.12,2.23) .. (17.49,5.26)   ;
\draw [color={rgb, 255:red, 208; green, 2; blue, 27 }  ,draw opacity=1 ][line width=2.25]  [dash pattern={on 2.53pt off 3.02pt}]  (546,137.5) -- (543.12,40.5) ;
\draw [shift={(543,36.5)}, rotate = 88.3] [color={rgb, 255:red, 208; green, 2; blue, 27 }  ,draw opacity=1 ][line width=2.25]    (17.49,-5.26) .. controls (11.12,-2.23) and (5.29,-0.48) .. (0,0) .. controls (5.29,0.48) and (11.12,2.23) .. (17.49,5.26)   ;
\draw [color={rgb, 255:red, 74; green, 144; blue, 226 }  ,draw opacity=1 ][line width=2.25]  [dash pattern={on 6.75pt off 4.5pt}]  (545.32,136.16) -- (633.4,91.67)(546.68,138.84) -- (634.75,94.35) ;
\draw [shift={(643,88.5)}, rotate = 153.2] [color={rgb, 255:red, 74; green, 144; blue, 226 }  ,draw opacity=1 ][line width=2.25]    (17.49,-5.26) .. controls (11.12,-2.23) and (5.29,-0.48) .. (0,0) .. controls (5.29,0.48) and (11.12,2.23) .. (17.49,5.26)   ;
\draw [color={rgb, 255:red, 74; green, 144; blue, 226 }  ,draw opacity=1 ][line width=2.25]  [dash pattern={on 6.75pt off 4.5pt}]  (546.55,136.11) -- (630.26,169.41)(545.45,138.89) -- (629.15,172.2) ;
\draw [shift={(639,174.5)}, rotate = 201.7] [color={rgb, 255:red, 74; green, 144; blue, 226 }  ,draw opacity=1 ][line width=2.25]    (17.49,-5.26) .. controls (11.12,-2.23) and (5.29,-0.48) .. (0,0) .. controls (5.29,0.48) and (11.12,2.23) .. (17.49,5.26)   ;
\draw [color={rgb, 255:red, 126; green, 211; blue, 33 }  ,draw opacity=1 ][line width=2.25]    (546,137.5) -- (479.4,48.7) ;
\draw [shift={(477,45.5)}, rotate = 53.13] [color={rgb, 255:red, 126; green, 211; blue, 33 }  ,draw opacity=1 ][line width=2.25]    (17.49,-5.26) .. controls (11.12,-2.23) and (5.29,-0.48) .. (0,0) .. controls (5.29,0.48) and (11.12,2.23) .. (17.49,5.26)   ;
\draw [color={rgb, 255:red, 126; green, 211; blue, 33 }  ,draw opacity=1 ][line width=2.25]    (546,137.5) -- (475.68,215.53) ;
\draw [shift={(473,218.5)}, rotate = 312.03] [color={rgb, 255:red, 126; green, 211; blue, 33 }  ,draw opacity=1 ][line width=2.25]    (17.49,-5.26) .. controls (11.12,-2.23) and (5.29,-0.48) .. (0,0) .. controls (5.29,0.48) and (11.12,2.23) .. (17.49,5.26)   ;
\draw  [dash pattern={on 0.84pt off 2.51pt}]  (515,97.5) -- (588,38.5) ;
\draw [color={rgb, 255:red, 26; green, 75; blue, 128 }  ,draw opacity=1 ][line width=2.25]    (515,97.5) -- (546,137.5) ;

\draw (134,26.4) node [anchor=north west][inner sep=0.75pt]    {$\lambda _{-c}$};
\draw (318,30.4) node [anchor=north west][inner sep=0.75pt]    {$\lambda _{c}$};
\draw (234,26.4) node [anchor=north west][inner sep=0.75pt]    {$0$};
\draw (180,194.4) node [anchor=north west][inner sep=0.75pt]    {$\lambda _{-1}$};
\draw (97,197.4) node [anchor=north west][inner sep=0.75pt]    {$\lambda _{-2}$};
\draw (360,193.4) node [anchor=north west][inner sep=0.75pt]    {$\lambda _{+2}$};
\draw (281,195.4) node [anchor=north west][inner sep=0.75pt]    {$\lambda _{+1}$};
\draw (65,183.4) node [anchor=north west][inner sep=0.75pt]    {$L$};
\draw (3,45.4) node [anchor=north west][inner sep=0.75pt]    {$L_{0} =2\pi \sqrt{7}$};
\draw (522,5.4) node [anchor=north west][inner sep=0.75pt]    {$\mathrm{i\ Im}\mathcal{G}_{c}$};
\draw (590,17.4) node [anchor=north west][inner sep=0.75pt]    {$\mathcal{E}_{1}$};
\draw (592,233.4) node [anchor=north west][inner sep=0.75pt]    {$\mathcal{E}_{-1}$};
\draw (653,127.4) node [anchor=north west][inner sep=0.75pt]    {$\mathrm{Re}\mathcal{G}_{c}$};
\draw (644,28.4) node [anchor=north west][inner sep=0.75pt]    {$\mathcal{G}_{c}$};
\draw (641,203.4) node [anchor=north west][inner sep=0.75pt]    {$\mathcal{G}_{-c}$};
\draw (654,74.9) node [anchor=north west][inner sep=0.75pt]    {$\mathcal{E}_{-2}$};
\draw (648,169.4) node [anchor=north west][inner sep=0.75pt]    {$\mathcal{E}_{2}$};
\draw (451,25.4) node [anchor=north west][inner sep=0.75pt]    {$\tilde{\mathcal{G}}_{c}$};
\draw (475,221.9) node [anchor=north west][inner sep=0.75pt]    {$\tilde{\mathcal{G}}_{-c}$};
\end{tikzpicture}
    \caption{$2\pi\sqrt{7}$ case: asymptotic behaviors of eigenvalues and eigenfunctions}
    \label{fig: 2pi-7-intro}
\end{figure}
Here we have 
\begin{align*}
\G_c=a_1\E_1+b_1\E_2+\bigO(|L-2\pi\sqrt{7}|),\widetilde{\G}_c=a_2\E_1+b_2\E_2+\bigO(|L-2\pi\sqrt{7}|),
\end{align*}
where $\widetilde{\G}_c$ is a Type 2 eigenfunction of $\B$ at $L=2\pi\sqrt{7}$ with $\widetilde{\G}_c'(0)=\widetilde{\G}_c'(2\pi\sqrt{7})\neq0$. Therefore, like $2\pi$ case, 
\begin{equation*}
M_{\B}(L)=\mathrm{Span}\{\Re(a_1\E_1+b_1\E_2),\Im(a_1\E_1+b_1\E_2)\}\sim M(L_0)+\bigO(|L-2\pi\sqrt{7}|).
\end{equation*}
We notice that for $\dim M(L_0)=2$, we also have there two distinguished situations like $k=2,l=1$ and $k=4,l=1$.
    \item  At last, we quickly mention an example of $L_0=14\pi$ with $k=11,l=2$, or $k=l=7$. Here $\dim M(L_0)=3$. One can see this as a combination of the first and third examples. For simplicity, we do not mention many details here and refer to Section \ref{sec: Asymptotic behavior close to the critical lengths}. 
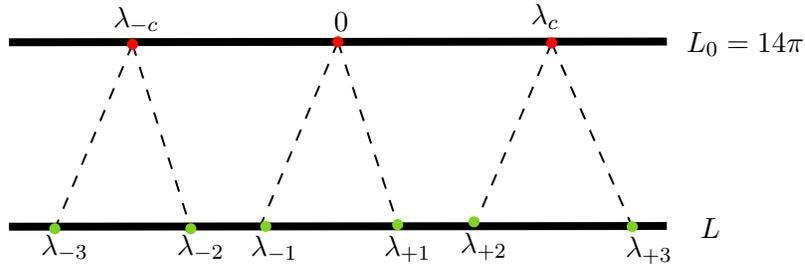
\begin{figure}
    \centering
\tikzset{every picture/.style={line width=0.75pt}} 

\begin{tikzpicture}[x=0.75pt,y=0.75pt,yscale=-0.7,xscale=0.7]

\draw [line width=3]    (101,51) -- (575,50.5) ;
\draw [line width=3]    (101,184) -- (575,183.5) ;
\draw [fill={rgb, 255:red, 252; green, 3; blue, 3 }  ,fill opacity=1 ] [dash pattern={on 4.5pt off 4.5pt}]  (190,52.5) -- (134,185.5) ;
\draw  [dash pattern={on 4.5pt off 4.5pt}]  (190,52.5) -- (232,185.5) ;
\draw  [dash pattern={on 4.5pt off 4.5pt}]  (338,50.75) -- (282,183.75) ;
\draw  [dash pattern={on 4.5pt off 4.5pt}]  (492,51.5) -- (436,184.5) ;
\draw  [dash pattern={on 4.5pt off 4.5pt}]  (338,50.75) -- (381,182.5) ;
\draw  [dash pattern={on 4.5pt off 4.5pt}]  (492,51.5) -- (550,184) ;
\draw  [draw opacity=0][fill={rgb, 255:red, 254; green, 5; blue, 5 }  ,fill opacity=1 ] (186,52.5) .. controls (186,50.29) and (187.79,48.5) .. (190,48.5) .. controls (192.21,48.5) and (194,50.29) .. (194,52.5) .. controls (194,54.71) and (192.21,56.5) .. (190,56.5) .. controls (187.79,56.5) and (186,54.71) .. (186,52.5) -- cycle ;
\draw  [draw opacity=0][fill={rgb, 255:red, 254; green, 5; blue, 5 }  ,fill opacity=1 ] (334,50.75) .. controls (334,48.54) and (335.79,46.75) .. (338,46.75) .. controls (340.21,46.75) and (342,48.54) .. (342,50.75) .. controls (342,52.96) and (340.21,54.75) .. (338,54.75) .. controls (335.79,54.75) and (334,52.96) .. (334,50.75) -- cycle ;
\draw  [draw opacity=0][fill={rgb, 255:red, 254; green, 5; blue, 5 }  ,fill opacity=1 ] (488,51.5) .. controls (488,49.29) and (489.79,47.5) .. (492,47.5) .. controls (494.21,47.5) and (496,49.29) .. (496,51.5) .. controls (496,53.71) and (494.21,55.5) .. (492,55.5) .. controls (489.79,55.5) and (488,53.71) .. (488,51.5) -- cycle ;
\draw  [draw opacity=0][fill={rgb, 255:red, 126; green, 211; blue, 33 }  ,fill opacity=1 ] (130,185.5) .. controls (130,183.29) and (131.79,181.5) .. (134,181.5) .. controls (136.21,181.5) and (138,183.29) .. (138,185.5) .. controls (138,187.71) and (136.21,189.5) .. (134,189.5) .. controls (131.79,189.5) and (130,187.71) .. (130,185.5) -- cycle ;
\draw  [draw opacity=0][fill={rgb, 255:red, 126; green, 211; blue, 33 }  ,fill opacity=1 ] (228,185.5) .. controls (228,183.29) and (229.79,181.5) .. (232,181.5) .. controls (234.21,181.5) and (236,183.29) .. (236,185.5) .. controls (236,187.71) and (234.21,189.5) .. (232,189.5) .. controls (229.79,189.5) and (228,187.71) .. (228,185.5) -- cycle ;
\draw  [draw opacity=0][fill={rgb, 255:red, 126; green, 211; blue, 33 }  ,fill opacity=1 ] (282,183.75) .. controls (282,181.54) and (283.79,179.75) .. (286,179.75) .. controls (288.21,179.75) and (290,181.54) .. (290,183.75) .. controls (290,185.96) and (288.21,187.75) .. (286,187.75) .. controls (283.79,187.75) and (282,185.96) .. (282,183.75) -- cycle ;
\draw  [draw opacity=0][fill={rgb, 255:red, 126; green, 211; blue, 33 }  ,fill opacity=1 ] (377,182.5) .. controls (377,180.29) and (378.79,178.5) .. (381,178.5) .. controls (383.21,178.5) and (385,180.29) .. (385,182.5) .. controls (385,184.71) and (383.21,186.5) .. (381,186.5) .. controls (378.79,186.5) and (377,184.71) .. (377,182.5) -- cycle ;
\draw  [draw opacity=0][fill={rgb, 255:red, 126; green, 211; blue, 33 }  ,fill opacity=1 ] (432,180.5) .. controls (432,178.29) and (433.79,176.5) .. (436,176.5) .. controls (438.21,176.5) and (440,178.29) .. (440,180.5) .. controls (440,182.71) and (438.21,184.5) .. (436,184.5) .. controls (433.79,184.5) and (432,182.71) .. (432,180.5) -- cycle ;
\draw  [draw opacity=0][fill={rgb, 255:red, 126; green, 211; blue, 33 }  ,fill opacity=1 ] (546,184) .. controls (546,181.79) and (547.79,180) .. (550,180) .. controls (552.21,180) and (554,181.79) .. (554,184) .. controls (554,186.21) and (552.21,188) .. (550,188) .. controls (547.79,188) and (546,186.21) .. (546,184) -- cycle ;

\draw (174,23.4) node [anchor=north west][inner sep=0.75pt]    {$\lambda _{-c}$};
\draw (475,21.4) node [anchor=north west][inner sep=0.75pt]    {$\lambda _{c}$};
\draw (333,27.4) node [anchor=north west][inner sep=0.75pt]    {$0$};
\draw (122,187.4) node [anchor=north west][inner sep=0.75pt]    {$\lambda _{-3}$};
\draw (220,187.4) node [anchor=north west][inner sep=0.75pt]    {$\lambda _{-2}$};
\draw (273,188.4) node [anchor=north west][inner sep=0.75pt]    {$\lambda _{-1}$};
\draw (369,188.4) node [anchor=north west][inner sep=0.75pt]    {$\lambda _{+1}$};
\draw (424,186.4) node [anchor=north west][inner sep=0.75pt]    {$\lambda _{+2}$};
\draw (543,189.4) node [anchor=north west][inner sep=0.75pt]    {$\lambda _{+3}$};
\draw (597,175.4) node [anchor=north west][inner sep=0.75pt]    {$L$};
\draw (587,41.4) node [anchor=north west][inner sep=0.75pt]    {$L_{0} =14\pi $};

\end{tikzpicture}
    \caption{$14\pi$ case: asymptotic behaviors of eigenvalues}
    \label{fig: 14pi-intro}
\end{figure}    
\end{itemize}

Our next main result concerns the stability result for the nonlinear KdV system \eqref{eq: KdV system-stability-intro}.
\begin{thm}\label{thm: nonlinear stability result}
Let $L_0\in \mathcal{N}$ be a fixed critical length.  Let $I=[L_0- \delta,L_0+ \delta]$ with $\delta= \frac{\pi^2}{3L_0^2}$.  There exist several positive constants $\epsilon_0$, $m_0$, $M_0$, and $C_0$ such that for every $L\in I\setminus\{L_0\}$, there are two connected $C^1$ manifolds $\mathcal{M},\mathcal{H}\subset L^2(0,L)$ satisfy the following. $\mathcal{H}$ and $\mathcal{M}$ are two invariant manifolds under the nonlinear KdV flow, with $\mathcal{M}$ of finite dimension and $\mathcal{H}$ of finite codimension. Let $y$ be any solution to \eqref{eq: KdV system-stability-intro} with $\|y^0\|_{L^2(0,L)}<\epsilon_0$. Then we have 
\begin{enumerate}
    \item $T_0\mathcal{H}=H_{\A}(L),T_0\mathcal{M}=M_{\A}(L)$. (Recall that $H_{\A}(L)$ and $M_{\A}(L)$ are defined in Theorem \ref{thm: main theorem linear version});
    \item For any $y^0\in\mathcal{H}$, $E(y(t))\leq e^{-C_0 t}E(y^0),\forall t\in(0, +\infty)$;
    \item For any $y^0\in\mathcal{M}$, $E(y(t))\leq e^{-m_0(L-L_0)^2t}E(y^0), \forall t\in(0, +\infty)$.
\end{enumerate}
\end{thm}
Theorem \ref{thm: main theorem linear version} is related to Problem  \ref{OP: C(T,L)-L-limit}.  As reviewed in Section \ref{sec: Review of the literature}, previous stability results focused on establishing the existence of exponential decay but lacked information on quantifying the decay rates.   
This result provides a quantitative characterization of the exponential decay rates.

Furthermore, we analyze the limiting problem when $L$ approaches the critical set $\mathcal{N}$. Through a state space decomposition $L^2(0,L)=H_{\A}(L)\oplus M_{\A}(L)$, we give a detailed description of the decay rates. 
\begin{rem}
We identify the uniform directions, namely, $H_{\A}(L)$, in which the decay rates remain uniform with respect to $L$ throughout the limiting process. Conversely, in $M_{\A}(L)$, we observe a completely different behavior, with the decay rates vanishing as $\mathrm{dist}(L,\mathcal{N})^2$ when $L$ gets close to $\mathcal{N}$.
\end{rem}
Informally, this provides a perspective on understanding the different behaviors of asymptotic stability for $L\notin\mathcal{N}$ and $L\in\mathcal{N}$. 

For the nonlinear version, Theorem \ref{thm: nonlinear stability result} demonstrates the existence of two invariant manifolds, $\mathcal{H}$ and $\mathcal{M}$. As important tools in nonlinear systems, there is a substantial body of literature on center and invariant manifolds in various settings. We refer \cite{Magal-Ruan, Carr, Nakanishi-Schlag-book} to the abstract theory in Banach spaces and also some other important models \cite{Krieger-Schlag-JAMS, Krieger-Nakanishi-Schlag-2014}. 

For our nonlinear KdV system \eqref{eq: KdV system-stability-intro}, leveraging the center manifold results from \cite{Minh-Wu2004}, we prove the existence and smoothness of an invariant manifold for \eqref{eq: KdV system-stability-intro} with $L$ near the critical set $\mathcal{N}$. Consequently, we obtain Theorem \ref{thm: nonlinear stability result} for the original system \eqref{eq: KdV system-stability-intro}. 

\begin{rem}
As we know, for $L\in\mathcal{N}$, we can only expect polynomial decay for the asymptotic stability of nonlinear KdV equations (see, for example, \cite{Chu-Coron-Shang}). However, as we presented in Theorem \ref{thm: nonlinear stability result}, we observe exponential stability for $L\notin\mathcal{N}$ with a quantitative estimate of the decay rates. Through the different characteristics of these two invariant manifolds, we see that in $\mathcal{M}$, the decay rates become progressively smaller, ultimately losing the exponential decay property at the critical length. This provides us with a perspective for understanding the transition from exponential decay to polynomial decay.
\end{rem}

In the proof of Theorem \ref{thm: main theorem linear version}, we introduce a constructive approach, \textit{transition-stabilization method} (see details in Section \ref{sec: A transition-stabilization method} and Section \ref{sec: Part II: Sharp stability analysis}). This approach allows us to explicitly construct the null-control function $u$ for the linearized control problem \eqref{eq: linearized KdV system-control-intro} and obtain quantitative estimates simultaneously. To apply this method, we need to explore the information on the spectrum of the stationary operator.

\subsubsection{Control results}
By applying this constructive approach, we summarize the following result
\begin{thm}\label{thm: control-cost}
Given $L\notin\mathcal{N}$, there exists a constant $\K$, independent of the control time $T>0$, such that $\forall y^0\in L^2(0,L)$, we are able to construct explicitly a function $u(t) \in L^2(0,T)$ such that the solution $y$ to the system \eqref{eq: linearized KdV system-control-intro} satisfies $y(T,\cdot)= 0$, and 
\begin{equation}\label{eq: est-control-L-infty-intro}
 \|u\|_{L^{\infty}(0,T)}\leq \K e^{\frac{\K}{\sqrt{T}}}\|y^{0}\|_{L^2(0,L)},\forall T>0. 
\end{equation}
\end{thm}
Theorem \ref{thm: control-cost} provides  an answer to Problem \ref{OP: constructive method}. This result is a concise summary, with detailed information on the concrete construction of the control $u$ available in Proposition \ref{prop: iteration schemes}. According to the result given in \cite[Proposition 3.1]{GG08}, one should expect that for KdV equations, the cost of fast controls is bounded by $C e^{\frac{C}{\sqrt{T}}}$ due to the
weights used in the Carleman estimates. In \cite{Lissy-2014}, Lissy proved the optimality of the power of $T$ in $C e^{\frac{C}{\sqrt{T}}}$. However, in \cite{GG08,Lissy-2014}, they did not deal with the right Neumann boundary control as in \eqref{eq: linearized KdV system-control-intro}.  Note that, very recently and independently, Nguyen also studied the fast control problem for this model and obtained the same result \cite{Nguyen-cost}. His proof is based on a novel semigroup approach for unbounded operators that he recently developed in \cite{Nguyen-sta-cocv, Nguyen-cost-SH}, and it also exploits the idea of transitioning between two operators.

\subsection{Organization of the paper}
\subsection*{Notations}
For the sake of clarity and ease of reference, here we introduce some basic notations used throughout the paper (see Table \ref{tab: basic notations}). Our aim is to provide a coherent framework that supports the elucidation of the solutions, properties, and implications of the KdV equations.
\begin{table}[h]
    \centering
    \begin{tabular}{|p{5cm}|p{5cm}|p{5cm}|}
\hline
 \textbf{Parameters and variables} & \textbf{Solutions and special functions}& \textbf{Constants.} \\
\hline
$T>0$: fixed time; & $y,z,w$: solutions to KdV systems & $C,\mathcal{K},r$: constants  \\
\hline
$L$: a non-critical length
& $u,v$: control functions& $C=C(L)$: specify the significant dependence on parameters  \\
\hline
$L_0$: a critical length & $\{\phi_j\}, \{\vartheta_j\}$: the {\it bi-orthogonal family} (see Section \ref{sec: bi-orthogonal family} for details) & $N_0$: dimension of the  unreachable subspace \\
\hline
$(t,x)$: the time variable and the spatial variable  & $h$:  the {\it modulated functions} (defined in Section \ref{sec: preliminary}) & $N_E(L)$: dimension of the elliptic subspace (see Definition \ref{defi: elliptic/hyperbolic index and subspaces}).  \\
\hline
\end{tabular}
    \caption{Basic notations}
    \label{tab: basic notations}
\end{table}
We also want to introduce the following \textbf{stationary KdV operators.} We investigate KdV operators with different boundary conditions, $\mathcal{A}, \mathcal{B}$. 
\begin{enumerate}
    \item We define $\B$ via 
\begin{equation}\label{eq: defi-B-op}
\begin{array}{c}
\B\varphi:=-\varphi'''-\varphi', \text{ with the domain }\\
D(\B):=\{\varphi\in H^3(0,L): \varphi(0)=\varphi(L)=0,\varphi'(0)=\varphi'(L)\}.
\end{array}
\end{equation}
Under this definition, we know that $\B$ is skew-adjoint, with its spectrum $\{\ii\lambda_j\}_{j\in\Z}\subset\ii\R$ and associated eigenfunctions denoted by $\E_j$ (i.e., $\B\E_j=\ii\lambda_j\E_j$, where $\E_j(0)=\E_j(L)=0,\E_j'(0)=\E_j'(L)$).
\item We define $\A$ via
\begin{equation}\label{eq: defi-A-op}
\begin{array}{c}
\A\varphi:=-\varphi'''-\varphi', \text{ with the domain }\\
D(\A):=\{\varphi\in H^3(0,L): \varphi(0)=\varphi(L)=\varphi'(L)=0\}.
\end{array}
\end{equation}
Under this definition, we know that $\A$ is neither self-adjoint nor skew-adjoint. We denote its eigenmodes by $(\zeta,\F_{\zeta})\in\C\times L^2((0,L);\C)$, where  $\A\F_{\zeta}=\zeta\F_{\zeta}$, with $\F_{\zeta}(0)=\F_{\zeta}(L)=\F_{\zeta}'(L)=0$.
\item In particular, by \cite{Rosier-1997}, we know the existence of critical lengths $L_0\in\mathcal{N}$. At the critical length $L_0=2\pi\sqrt{\frac{k^2+kl+l^2}{3}}$, we are interested in a finite number of special eigenmodes of $\A$ that satisfy the following (see \cite{Rosier-1997}):
\begin{equation*}
\begin{array}{c}
\A\G=\ii\lambda_{c}\G, \text{ with }\lambda_{c}=\frac{(2k+l)(k-l)(2l+k)}{3\sqrt{3}(k^2+kl+l^25)^{\frac{3}{2}}},\\
\G(0)=\G(L)=\G'(L)=0.
\end{array}
\end{equation*}
Indeed, we point out that here one also has $\G'(0)=0$.
\end{enumerate}

\section{Strategy of the proof}
\subsection{Outlines}
Our proof can be divided into three stages. 
\begin{enumerate}
    \item \textbf{Stage 1: establish the \textit{transition-stabilization method}. }\label{it: 1}In this stage, we introduce our transition-stabilization approach. In application, we prove a weak form of Theorem \ref{thm: main theorem linear version} and Theorem \ref{thm: control-cost}.
    The main difficulty here is the ``bad" spectral behaviors of the stationary operator $\A$, thus the classic moment method fails. 

   We divide the interval $(0,T)$ into $\cup_{n=1}^{\infty}(T_{n-1},T_n)$, and then construct an iterative scheme where, in each $(T_{n-1},T_n)$, design a control function $u_n$ to achieve a quantitative estimate of the energy decay 
\begin{gather*}
\|y(T_n)\|_{L^2(0,L)}\leq C_n\|y(T_{n-1})\|_{L^2(0,L)}, \textrm{ with } \lim_{t\rightarrow T}\|y(t)\|_{L^2(0,L)}=0.
\end{gather*}

    \item \textbf{Stage 2:  sharp stability analysis. }In this stage, we mainly perform a quantitative and sharp analysis of the exponential stability of the linear KdV system \eqref{eq: linear KdV-stability-intro} and refine the transition-stabilization method to prove Theorem \ref{thm: main theorem linear version}. \\
    The key result in Theorem \ref{thm: main theorem linear version} is the uniform decay rate in $H_{\A}(L)$. Thanks to the HUM in a projective space, it suffices to prove a null controllability result in $H_{\A}(L)$. To refine our transition-stabilization method, the main difficulty here is to characterize a related subspace $H_{\B}(L)$. It is reduced to characterize a finite-dimensional ``quasi-invariant'' subspace $M_{\B}(L)$, which satisfies $L^2(0,L)=M_{\B}(L)\oplus H_{\B}(L)$.
We additionally introduce two transition maps.

    \item \textbf{Stage 3: invariant manifolds for nonlinear KdV. } Based on the results in Theorem \ref{thm: main theorem linear version} concerning the linearized KdV equation, we adapt the invariant manifold arguments for nonlinear KdV to obtain Theorem \ref{thm: nonlinear stability result}. 
\subsection{Stage 1: establish the transition-stabilization method}\label{sec: details-stage-1}
The central part of our proof is constructing a null control for \eqref{eq: linearized KdV system-control-intro}. Concerning this null control problem, there are many classical and typical attempts such as the moment method \cite{Tucsnak-Weiss-2009,LRZ-2010}, Carleman estimates, and harmonic analysis tools \cite{nwx-measure}, etc. For example, the moment method is based on the orthonormal basis and good spectral analysis.  However,  the eigenfunctions of the stationary operator $\A$, as well as generalized eigenfunctions, do not form a basis, since the operator is neither self-adjoint nor skew-adjoint. In fact, they are not even complete in $L^2(0,L)$ (see \cite[Appendix]{xiang-2019-Sicon}). This fact becomes an obstacle to constructing a null control for \eqref{eq: linearized KdV system-control-intro}.

Therefore, we first ``regularize"  $\A$ to a skew-adjoint operator $\B$. This simple observation is the key of the analysis. 
Introduce the auxiliary system associated with the stationary operator $\B$
\begin{equation}\label{eq: z-system-sketch-proof}
\left\{
\begin{array}{lll}
    \p_tz+\p_x^3z+\p_xz=0 , \\
     z(t,0)=z(t,L)=0, \\
     \p_xz(t,L)-\p_xz(t,0)=v(t),\\
     z|_{t=0}=z^0.
\end{array}
\right.
\end{equation}
Thanks to the good spectral properties of $\B$ and the moment method, we can construct the control function $v$ explicitly steering any initial state $z^0$ to $0$. We expect that if $y^0=z^0$, taking $u(t)=\p_xz(t,0)+v(t)$ could help us drive $y^0$ to $0$. But this simple direct approach is not enough 
\begin{equation*}
    z^0=y^0\in L^2(0,L)\not\Rightarrow \partial_x z(t, L)=\p_xz(t,0)+v(t)\in L^2(0,T).
\end{equation*}
Only the transition from $\A$ to $\B$ does not directly solve the problem. 
Then we are inspired by the strategy of iteratively decreasing the energy, as discussed in works like  Lebeau-Robbiano strategy \cite{Lebeau-Robbiano-1995}, and the finite time stabilization \cite{CoronNguyen, Coron-Xiang-2021, Xiang-scl, xiang2020quantitative},
\begin{equation*}
\text{Quantitative enhanced energy dissipation} \xrightarrow{\text{fast iteration}} \text{Null controllability}.
\end{equation*}
This becomes the \textit{transition-stabilization method}, which incorporates various techniques such as smoothing effects, the moment method, and high-frequency energy dissipation. 
Specifically, when focusing on the constructive approach in $(0, T_0)$, this \textit{transition-stabilization method} is composed by four steps. The package during the time interval  $(0,T_0)$ works as follows:\\

\begin{tikzpicture}[
    scale=0.9,
    transform shape,
    auto,
    block/.style={
        rectangle, draw, 
        text width=9em, text centered, rounded corners, minimum height=4em
    },
    arrow/.style={
        -{Latex[width=2mm,length=3mm]}, thick
    },
    Arrow/.style={
        {Latex[width=2mm,length=3mm]}-{Latex[width=2mm,length=3mm]}, thick
    }
]

\node (initial)[block, text width=3em] {$y_0$ };
\node (kato) [block, right= 3cm of initial] {$y(\frac{T_0}{2})\in D(\A^2)$};
\node (z) [block, above=of kato] {$z^0\in D(\B^2)$};
\node (h)[block, below=of kato]{$c^1 h_{\mu}+c^2 h_{2\mu}$};
\node (0)[block, right= 4.4cm of z,text width=2em]{0};
\node (size)[block,right=3cm of h]{$\bigO(\|y^0\|_{L^2}e^{-\mu\frac{T_0}{2}})$};
\node (final)[block, right= 3cm of kato]{$y(T_0)$};

\draw [arrow] (initial) -- (kato) node[midway, above] {Smoothing effects} node[midway, below] {Step 1};
\draw [arrow] (kato) -- (z) node[midway, left] {Step 2};
\draw [arrow] (kato) -- (h) node[midway, left] {Step 2};
\draw [arrow] (z)--(0) node[midway, above] {Constructive null-control} node[midway, below] {Step 3};
\draw [arrow] (h)--(size) node[midway, above] {Free flow} node[midway, below] {Step 3};
\draw[arrow](0)--(final) ;
\draw[arrow](size)--(final);
\end{tikzpicture}\\

\noindent\textbf{Step 1: Regularization. }Based on a priori estimates of the intermediate system \eqref{eq: z-system-sketch-proof} (details in Section \ref{sec: estimates for the control}), we obtain
\begin{gather*}
    z^0\in L^2(0,L)\Rightarrow \p_xz(t,0)+v(t)\in H^{-2}(0,T).
\end{gather*}
Notice that improving the regularity of $z^0$ can ensure $\p_xz(t,0)+v(t)\in L^2$. More precisely,
\begin{gather*}
    z^0\in D(\B^2)\subset H^6(0,L)\Rightarrow \p_xz(t,0)+v(t)\in L^2(0,T).
\end{gather*}
Meanwhile, recall the smoothing effects of free KdV flow w.r.t. $\A$ (see Proposition \ref{prop: Xiang-Krieger gain 1}). This motivates us to split $[0,T_0]=[0,\frac{T_0}{2}]\cup[\frac{T_0}{2},T_0]$. In $[0,\frac{T_0}{2}]$, using free KdV flow and due to the smoothing effects, we obtain $y(\frac{T_0}{2})\in D(\A^2)$ satisfying
\[
\|y(\frac{T_0}{2})\|_{H^6(0,L)}\leq \frac{C}{T_0^3}\|y^0\|_{L^2};
\]
\noindent\textbf{Step 2: Transition. }After Step 1, we observe that $y(\frac{T_0}{2})\in D(\A^2)\subset H^6(0,L)$ is almost a well-prepared initial state for \eqref{eq: z-system-sketch-proof} except that the boundary conditions do not match. Hence, we need to transition from $D(\A^2)$ to $ D(\B^2)$, with a focus on reconciling the differing boundary conditions between $\A$ and $\B$. Employ the following decomposition:
\begin{equation*}
        y(\frac{T_0}{2})=z^0+c_1h_{\mu}+c_2h_{2\mu},
    \end{equation*}
where $z^0\in D(\B^2)$ and $h_{\mu}$ and $h_{2\mu}$ are two modulated functions defined in \eqref{eq: static KdV with nonvanishing boundary condition}. $h_{\mu}$ and $h_{2\mu}$ are designed to compensate for the boundary difference. Then consider
\begin{equation*}
\left\{
\begin{array}{l}
    \p_t z+\p_x^3z+\p_xz=0 \\
     z(t,0)=z(t,L)=0\\
     \p_xz(t,L)-\p_xz(t,0)=v(t)\\
     z\left|_{t=\frac{T_0}{2}}\right.=z^0
\end{array}
\right.
\end{equation*}
\begin{equation*}
\begin{array}{cc}
\left\{
\begin{array}{l}
    \p_t z_{\mu}+\p_x^3z_{\mu}+\p_xz_{\mu}=0 \\
     z_{\mu}(t,0)=z_{\mu}(t,L)=0\\
     \p_xz_{\mu}(t,L)=c_1e^{-\mu(t-\frac{T_0}{2})}h'_{\mu}(L)\\
     z_{\mu}\left|_{t=\frac{T_0}{2}}\right.=c_1h_{\mu}
\end{array}
\right. \;\;\;\; \;\;\;\;\;\;\;\; \;\;\;\; \;\;\;\; &\left\{
\begin{array}{l}
    \p_t z_{2\mu}+\p_x^3z_{2\mu}+\p_xz_{2\mu}=0 \\
     z_{2\mu}(t,0)=z_{2\mu}(t,L)=0\\
     \p_xz_{2\mu}(t,L)=c_2e^{-2\mu(t-\frac{T_0}{2})}h'_{2\mu}(L)\\
     z_{2\mu}\left|_{t=\frac{T_0}{2}}\right.=c_2h_{2\mu}
\end{array}
\right.
\end{array}
\end{equation*}
and let $y=z+z_{\mu}+z_{2\mu}$. We easily verify that $y$ is a solution to \eqref{eq: linearized KdV system-control-intro} in $[\frac{T_0}{2},T_0]$ with 
\begin{equation}\label{eq: form-u-control}
u(t)=\p_xz(t,L)+c_1e^{-\mu(t-\frac{T_0}{2})}h'_{\mu}(L)+c_2e^{-2\mu(t-\frac{T_0}{2})}h'_{2\mu}(L).  
\end{equation}

\noindent\textbf{Step 3: Dissipation. }For the modulated compensate terms, the free solutions $z_{\mu}=c_1e^{-\mu(t-\frac{T_0}{2})}h_{\mu}$ and $z_{2\mu}=c_2e^{-2\mu(t-\frac{T_0}{2})}h_{2\mu}$ have fast decay. Therefore, we only need to focus on the construction of $v$ and $z$. For \eqref{eq: z-system-sketch-proof}, we use the moment method to construct $v$ explicitly steering $z^0$ to $0$ with quantitative energy estimates\footnote{Using the moment method to construct a null control is a standard process, as outlined in \cite[Appendix]{Beauchard-Laurent}. However, deriving quantitative estimates is a more delicate issue. For further details, we refer to Section \ref{sec: A transition-stabilization method}.}. Then we use $u(t)$ in \eqref{eq: form-u-control} as a control such that the solution $y$ to \eqref{eq: linearized KdV system-control-intro} satisfies that $y|_{t=0}=y^0$ and $y|_{t=T_0}=y(T_0)=c_1e^{-\mu\frac{T_0}{2}}h_{\mu}+c_2e^{-\mu T_0}h_{2\mu}$ with
\begin{gather*}
\|y(T_0)\|_{L^2(0,L)}\leq C\frac{e^{-\mu\frac{T_0}{2}}}{T_0^3}\|y^0\|_{L^2},\\
\|u\|_{L^2(0,T_0)}\leq \frac{C}{|L-L_0|}\frac{e^{\frac{C}{\sqrt{T}}}\mu^{\frac{5}{2}}+e^{-\frac{\mu^{\frac{1}{3}}}{2}L}}{T_0^3}\|y^0\|_{L^2(0,L)}.
\end{gather*}

\noindent\textbf{Step 4: Iteration. }Repeat the preceding three steps in each interval $[T_{n-1},T_n]$. After a good choice of the parameters $\mu_n$ and $T_n$, we conclude that $\|y(T)\|_{L^2}=\lim_{n\rightarrow\infty}\|y(T_n)\|_{L^2}=0$.
\begin{rem}
We emphasize the following two points.
\begin{enumerate}
    \item This transition-stabilization scheme is specifically designed to address the null control problem both quantitatively and constructively. With the help of quantitative estimates, we can track the asymptotic behavior as $T\rightarrow0^+$ and thereby prove Theorem \ref{thm: control-cost}.
    \item Conversely, tracking the asymptotic behavior as $L\rightarrow\mathcal{N}$ presents a more complex challenge. A weak version of Theorem \ref{thm: main theorem linear version} can be proven directly using the transition-stabilization method. However, proving the sharp version of the theorem requires a more delicate analysis on the eigenmodes of $\B$ (see Section 7).
\end{enumerate}
\end{rem}
\subsection{Stage 2:  sharp stability analysis}\label{sec: details-stage-2}
In this stage, we prove Theorem \ref{thm: main theorem linear version}.  More precisely, we characterize two subspaces $M_{\A}(L)$ and $H_{\A}(L)$ such that $L^2(0,L)=M_{\A}(L)\oplus H_{\A}(L)$. In $H_{\A}(L)$, we prove exponential stability with a uniform decay rate for $L$ near $L_0$, while in $M_{\A}(L)$, exponential stability holds with a decay rate$\sim|L-L_0|^2$.

We aim to prove the sharp quantitative exponential stability for \eqref{eq: linear KdV-stability-intro} by establishing the corresponding observability. Since $M_{\A}(L)$ is finite-dimensional, we perform a direct approach to prove observability in $M_{\A}(L)$. Hence, we concentrate on the uniform observability in $H_{\A}(L)$. Thanks to the duality argument (see Section \ref{sec: preliminary}), establishing observability is reduced to a null controllability problem for \eqref{eq: linearized KdV system-control-intro} in $H_{\A}(L)$. 

Naturally, we expect to apply our transition-stabilization method to solve this null control problem. Recall that in the transition-stabilization package, 
\begin{gather*}
\forall y^0\in L^2(0,L), \textrm{ we construct a transition from }D(\A^2) \textrm{ to }D(\B^2).
\end{gather*}
In analog, we expect to find a proper subspace $H_{\B}(L)$ as well as a transition for $y^0\in H_{\A}(L)$:
\begin{gather*}
 D(\A^2)\cap H_{\A}(L) \longrightarrow  D(\B^2)\cap H_{\B}(L).
\end{gather*}
It is natural to consider $L^2(0,L)=M_{\B}(L)\oplus H_{\B}(L)$. At $L_0\in\mathcal{N}$, it is well-known that $M(L_0)$ is the unreachable space for \eqref{eq: linearized KdV system-control-intro}. As $L\rightarrow L_0$, we expect the finite-dimensional subspaces $M_{\A}(L)$ and $M_{\B}(L)$ as perturbations of $M(L_0)$. However, we will face three main questions:
\begin{itemize}
    \item how to characterize the subspace $M_{\B}(L)$?
    \item   how to define the transition on $D(\A^2)\cap H_{\A}(L)$?
    \item how to construct the null-control in the subspace $H_{\A}(L)$?
\end{itemize}
 A key aspect of deriving uniform quantitative estimates lies in characterizing the quasi-invariant subspace $M_{\B}(L)$. For $\A$,  due to its spectral properties, $M_{\A}(L)$ is spanned by eigenfunctions of $\A$ and can be approximated by $M(L_0)+\bigO(|L-L_0|)$. However, the situation for $\B$ is considerably more complex, as it turns out to depend on the novel classification of the critical lengths set $\mathcal{N}$ \cite{NX}.  For a detailed classification of the eigenmodes of $\B$ and the precise definition of  \textit{quasi-invariant subspace} $M_{\B}(L)$, we refer to  Section \ref{sec: Sharp spectral analysis and classification of critical lengths} and Section \ref{sec: quasi-invariant subsapce}. \\

The second question is relatively straightforward. Similar to the approach discussed in Section \ref{sec: details-stage-1}, we expect the transition can be achieved through a combination of modulated functions, $h_{\mu}$. We introduce $\mathcal{T}_c$, defined in Section \ref{sec: Projections and state space decomposition}, as follows. For $\forall y^0\in H_{\A}(L)\cap D(\A^2)$, we construct $z^0$ by
\begin{equation}\label{eq: 1-proj-intro}
    z^0=y^0-\sum_{j=1}^{N_0+2}c_jh_{\mu_j}\in H_{\B}(L)\cap D(\B^2),
\end{equation}
which achieves the transition between $y^0\in H_{\A}\rightarrow z^0\in H_{\B}$. It remains to determine the exact modulated functions and to characterize the coefficients. \\

However, even with a well-prepared quasi-invariant subspace $M_{\B}(L)$ and a transition map from $H_{\A}(L)$ to $H_{\B}(L)$, these elements are not sufficient to construct a null control in $H_{\A}(L)$. It is still necessary to introduce {\it another transition map}: $\mathcal{T}_{\varrho}$. \\
Indeed, utilizing the same transition-stabilization package in $[0,T_0]$, after a transition $\mathcal{T}_c$, we construct a control such that the final state of $z$ is $0$. Therefore, using $u(t)=\p_xz(t,L)+$``modulated part", we obtain
\[
y^0\in H_{\A}(L)\xrightarrow[]{u(t)}y(T_0)=\sum_{j=1}^{N_0+2}c_je^{-\mu_j\frac{T_0}{2}}h_{\mu_j}\notin H_{\A}(L).
\]
This disrupts the interaction scheme.  As a consequence, we need to introduce an additional transition map $\mathcal{T}_{\varrho}$ to adjust our method. Specifically, this transition map $\mathcal{T}_{\varrho}$ generates a well-chosen final target $z^{T_0}=\sum_{m}\varrho_m E_m$, where $\{E_m\}$ are $N_0$ well-prepared directions in $H_{\B}(L)$ (as detailed in Section \ref{sec: Projections and state space decomposition}), such that
\[
z^{T_0}+\sum_{j=1}^{N_0+2}c_je^{-\mu_j T_0}h_{\mu_j}\in H_{\A}(L),
\]
with $c_j$ being the exactly same as \eqref{eq: 1-proj-intro}. This transition is specifically designed to bring the final state $y(T_0)$ back into the same space as the initial state, $H_{\A}(L)$, through the adjustment of $z^{T_0}$. 
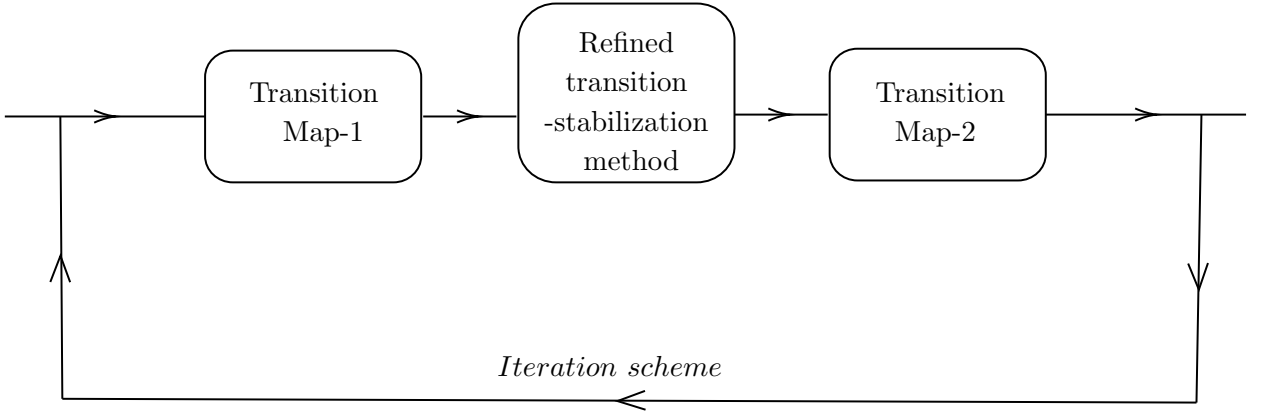
\begin{figure}[h]
    \centering
\tikzset{every picture/.style={line width=0.75pt}} 

\begin{tikzpicture}[x=0.75pt,y=0.75pt,yscale=-0.95,xscale=1
]

\draw   (106,45) .. controls (106,37.27) and (112.27,31) .. (120,31) -- (201,31) .. controls (208.73,31) and (215,37.27) .. (215,45) -- (215,87) .. controls (215,94.73) and (208.73,101) .. (201,101) -- (120,101) .. controls (112.27,101) and (106,94.73) .. (106,87) -- cycle ;
\draw    (216,66) -- (242,66) ;
\draw [shift={(244,66)}, rotate = 180] [color={rgb, 255:red, 0; green, 0; blue, 0 }  ][line width=0.75]    (10.93,-3.29) .. controls (6.95,-1.4) and (3.31,-0.3) .. (0,0) .. controls (3.31,0.3) and (6.95,1.4) .. (10.93,3.29)   ;
\draw   (264,25) .. controls (264,14.51) and (272.51,6) .. (283,6) -- (354,6) .. controls (364.49,6) and (373,14.51) .. (373,25) -- (373,82) .. controls (373,92.49) and (364.49,101) .. (354,101) -- (283,101) .. controls (272.51,101) and (264,92.49) .. (264,82) -- cycle ;
\draw    (244,66) -- (263,66) ;
\draw    (5,66) -- (59,66) ;
\draw [shift={(61,66)}, rotate = 180] [color={rgb, 255:red, 0; green, 0; blue, 0 }  ][line width=0.75]    (10.93,-3.29) .. controls (6.95,-1.4) and (3.31,-0.3) .. (0,0) .. controls (3.31,0.3) and (6.95,1.4) .. (10.93,3.29)   ;
\draw    (61,66) -- (106,66) ;
\draw    (530,65) -- (584,65) ;
\draw [shift={(586,65)}, rotate = 180] [color={rgb, 255:red, 0; green, 0; blue, 0 }  ][line width=0.75]    (10.93,-3.29) .. controls (6.95,-1.4) and (3.31,-0.3) .. (0,0) .. controls (3.31,0.3) and (6.95,1.4) .. (10.93,3.29)   ;
\draw    (586,65) -- (631,65) ;
\draw    (373,65) -- (399,65) ;
\draw [shift={(401,65)}, rotate = 180] [color={rgb, 255:red, 0; green, 0; blue, 0 }  ][line width=0.75]    (10.93,-3.29) .. controls (6.95,-1.4) and (3.31,-0.3) .. (0,0) .. controls (3.31,0.3) and (6.95,1.4) .. (10.93,3.29)   ;
\draw   (421,44) .. controls (421,36.27) and (427.27,30) .. (435,30) -- (516,30) .. controls (523.73,30) and (530,36.27) .. (530,44) -- (530,86) .. controls (530,93.73) and (523.73,100) .. (516,100) -- (435,100) .. controls (427.27,100) and (421,93.73) .. (421,86) -- cycle ;
\draw    (401,65) -- (420,65) ;
\draw    (33,66) -- (34,216) ;
\draw    (34,216) -- (606,218) ;
\draw    (608.5,65) -- (606,218) ;
\draw   (611.8,143.86) -- (607.2,158) -- (601.8,144.14) ;
\draw   (328.14,221.8) -- (314,217.2) -- (327.86,211.8) ;
\draw   (27.94,153.95) -- (33.06,140) -- (37.94,154.05) ;

\draw (269,15.4) node [anchor=north west][inner sep=0.75pt]    {$ \begin{array}{c}
 \text{Refined}\\
 \text{transition}\\
\text{-stabilization}\\
\text{ method}
\end{array}$};
\draw (436,42) node [anchor=north west][inner sep=0.75pt]    {$ \begin{array}{l}
\text{Transition}\\
\ \ \text{Map-2}
\end{array}$};
\draw (252,191.4) node [anchor=north west][inner sep=0.75pt]    {$Iteration\ scheme$};
\draw (120,42) node [anchor=north west][inner sep=0.75pt]  [font=\normalsize] [align=left]{$ \begin{array}{c}
\text{Transition}\\
\ \ \text{Map-1}
\end{array}$};
\end{tikzpicture}
    \caption{Refined transition-stabilization method}
        \label{fig: refined method}
\end{figure}
Consequently, in this revised transition-stabilization method, we concentrate on the following six steps:

\noindent\textbf{Step 0: Characterization of $M_{\B}(L)$. }At the beginning, we shall construct our quasi-invariant subspaces $M_{\B}(L)$ and $H_{\B}(L)$. 

\noindent\textbf{Step 1: Regularization. }Same as before, in $[0,\frac{T_0}{2}]$, due to the smoothing effects, for $y^0\in H_{\A}(L)$, we obtain $y(\frac{T_0}{2})\in H_{\A}(L)\cap D(\A^2)$.

\noindent\textbf{Step 2: Transition map-1. } After Step 1, using the first transition projection, we obtain
\begin{equation*}
         z^0=y(\frac{T_0}{2})-\sum_{j=1}^{N_0+2}c_jh_{\mu_j}\in H_{\B}(L)\cap D(\B^2),
\end{equation*}
and we also prove the coefficients $\{c_j\}_{1\leq j\leq N_0+2}$ are uniformly bounded. As before, we consider
\begin{gather*}
    y=z+\sum_{j=1}^{N_0+2}z_{\mu_j},\;
    u(t)=\p_x z(t,L)+\sum_{j=1}^{N_0+2}c_j e^{-\mu_j(t-\frac{T_0}{2})}h'_{\mu_j}(L);
\end{gather*}

\noindent\textbf{Step 3: Stabilization. }For the modulated terms, we use the free solutions $c_je^{-\mu_j(t-\frac{T_0}{2})}h_{\mu_j}$ for $1\leq j\leq N_0+2$. Therefore, at $t=T_0$,
\[
\|\sum_{j=1}^{N_0+2}c_je^{-\mu_j(t-\frac{T_0}{2})}h_{\mu_j}\|_{L^2}\sim \bigO(e^{-\min{\mu_j}\frac{T_0}{2}})\|y^0\|_{L^2}.
\]
Hence, it suffices to focus on the construction of $z$ and $v$.

\noindent\textbf{Step 4: Transition map-2. } After that, using the second transition projection, we obtain a well-prepared  final state $z^{T_0}=\sum_{m}\varrho_m E_m$ for \eqref{eq: z-system-sketch-proof} such that
\begin{equation*}
         z^{T_0}+\sum_{j=1}^{N_0+2}c_je^{-\mu_j T_0}h_{\mu_j}\in H_{\A}(L), \textrm{ with estimates }\varrho_m=\bigO(e^{-C T_0}).
\end{equation*}
where $c_j$ are the same as \eqref{eq: 1-proj-intro}. Using a revised moment method, we construct $v$ explicitly steering $z^0$ to $z^{T_0}$. \\
Combining Step 3 and Step 4, in $[0,T_0]$, we use
\(
u(t)=\p_xz(t,L)+\sum_{j=1}^{N_0+2}c_je^{-\mu_j(t-\frac{T_0}{2})}h'_{\mu_j}(L)
\)
as a control function such that the solution $y=z+\sum_{j=1}^{N_0+2}z_{\mu_j}$ satisfies that $y|_{t=0}=y^0\in H_{\A}(L)$ with
\begin{gather*}
y|_{t=T_0}=z^{T_0}+\sum_{j=1}^{N_0+2}c_je^{-\mu_j T_0}h_{\mu_j}\in H_{\A}(L),\\
\|y|_{t=T_0}\|_{L^2(0,L)}\leq C\frac{e^{-\mu_0\frac{T_0}{2}}}{T_0^3}\|y^0\|_{L^2},\\
\|u\|_{L^2(0,T_0)}\leq \frac{e^{\frac{C}{\sqrt{T}}}\mu_0^{4}+e^{-\frac{\mu_0^{\frac{1}{3}}}{4}L}}{T_0^3}\|y^0\|_{L^2(0,L)};
\end{gather*}
with a constant $\mu_0>0$.
\begin{rem}\label{rem: rho}
Note that we are always able to choose good parameters $\{\varrho_m\}$ such that $z^{T_0}$ is a real-valued function. Later in Section \ref{sec: classification quasiinvariant space} and Section \ref{sec: Part II: Sharp stability analysis}, we shall see that the directions $E_m$ are just the real parts and imaginary parts of a special linear combination of eigenfunctions of $\B$. More precisely, 
\begin{itemize}
    \item If $N_0=1$, this means that we are in a case $\mathcal{N}^1$ (See Definition \ref{def:new:classification}). Then, $z^{T_0}\sim2\ii\sin{kx}$, we refer to Subsection \ref{sec: Around Type I unreachable pair} for more details. Here $z^{T_0}$ is approximated by the Type 2 eigenfunctions of $\B$ (See Section \ref{sec: Eigenvalues and eigenfunctions at the critical length}).
    \item If $N_0\geq 3$ odd, then we are in  $\mathcal{N}^3$, apart from the direction $\sin{kx}$, we refer to Subsection \ref{sec: Around Type II unreachable pair} for construction of other directions and Subsection \ref{sec: limiting analysis on different types}. Here, roughly speaking, $E_m$ is approximated by the real parts and imaginary parts of Type 2 eigenfunctions of $\B$.
    
    We take an example of $L_0=14\pi$ with $N_0=3$. There are three Type 2 eigenfunctions $\widetilde{\G}_{\pm1}$, and $\sin{7x}$, which can be approximated by 
    \begin{equation*}
    \widetilde{\G}_1=a_2\E_2+b_2\E_3+\bigO(|L-14\pi|), 2\ii\sin{7x}=\E_{-1}-\E_1+\bigO(|L-14\pi|).
    \end{equation*}
Then, in this situation, we have $\{E_1,E_2,E_3\}=\{\Re(a_2\E_2+b_2\E_3),\Im(a_2\E_2+b_2\E_3),\E_{-1}-\E_1\}$.
    \item If $N_0$ is even, we have two different cases: $\mathcal{N}^2$ (see more details in Subsection \ref{sec: Around Type III unreachable pair}) and $\mathcal{N}^3$ (see more details in Subsection \ref{sec: Around Type II unreachable pair}). For $\mathcal{N}^2$ case, $E_m$ is approximated by eigenfunctions in the hyperbolic regime (see Proposition \ref{prop: asymptotic eigenvalues} and Subsection \ref{sec: limiting analysis on different types}); while $\mathcal{N}^2$ case, $E_m$ is approximated by Type 2 eigenfunctions of $\B$.

    We only take $L_0=2\pi\sqrt{7}$ as an example. There are two Type 2 eigenfunctions $\widetilde{\G}_{\pm1}$, which can be approximated by 
    \begin{equation*}
\widetilde{\G}_1=a_2\E_1+b_2\E_2+\bigO(|L-14\pi|).
    \end{equation*}
Then, in this situation, we have $\{E_1,E_2\}=\{\Re(a_2\E_1+b_2\E_2),\Im(a_2\E_1+b_2\E_2)\}$.
\end{itemize}
\end{rem}

\noindent\textbf{Step 5: Iteration. }Repeat the preceding three steps in each interval $[T_{n-1},T_n]$, and we obtain the final target by $\|y(T)\|_{L^2}=\lim_{n\rightarrow\infty}\|y(T_n)\|_{L^2}=0$.
\subsection{Stage 3: invariant manifolds for nonlinear KdV}
In this stage, we prove Theorem \ref{thm: nonlinear stability result}. Our proof builds upon the framework established by Chu, Coron, and Shang \cite{Chu-Coron-Shang}, but adapts their approach for $L$ near $L_0\in\mathcal{N}$ but $L\not\in\mathcal{N}$. The proof proceeds in two main steps:
\begin{enumerate}
    \item \textit{Construction of Invariant Manifolds:} Inspired by \cite{Chu-Coron-Shang}, after a smooth truncation to ensure the nonlinear perturbation is globally Lipschitz near the origin, we decompose $L^2(0,L)=H\oplus M$, where $H$ ( or $M$) is infinite-dimensional (or finite-dimensional) stable subspace. Then, we define a stable manifold $\mathcal{H}$ (tangent to $H$) and a finite-dimensional invariant manifold $\mathcal{M}$ (tangent to $M$). For initial data on $\mathcal{H}$, the dynamic is governed by the strongly dissipative linear semigroup, yielding a fast, uniform exponential decay bounded by $e^{-C_0 t}$.
    
    \item \textit{Simplified Dynamics on $\mathcal{M}$:} Unlike the critical case ($L=L_0$) studied in \cite{Chu-Coron-Shang}, our case retains a weak linear exponential dissipation proportional to $(L-L_0)^2$. Combined with the asymptotic behaviors $\bigO(|L-L_0|^2)$ of the eigenfunctions, this simplification allows us to directly close the energy estimate using a differential inequality of the form
    $$
    \frac{d}{dt}E \lesssim -(L-L_0)^2 E+\bigO(|L-L_0|^3),
    $$
    rigorously establishing the $e^{-m_0(L-L_0)^2 t}$ exponential decay in $\mathcal{M}$.
\end{enumerate}
\end{enumerate}

\section{Preliminary}\label{sec: preliminary}
In this section, we review some basic properties of the KdV system. There is a large literature on this classic topic. Here we mainly refer to \cite{coron-book, chapouly-2009, Krieger-Xiang-2021, Rosier-1997}.

\subsection*{Smoothing effects}\label{sec: smoothing effects}
In this part, we would like to introduce the smoothing effects related to the operator $\A$ (Recall its definition in \eqref{eq: defi-A-op}). 
We denote by $S(t)$ the corresponding semi-group of $\A$. Now we focus on the smoothing effects of the linear KdV flow generated by $S(t)$. In Rosier's paper \cite[Proposition 3.2]{Rosier-1997}, he observed the following Kato type regularizing effects of the solution: $ \|y\|_{L^2((0,T);H^1(0,L))}\lesssim \|y^0\|_{L^2(0,L)}$.

In fact, we are able to develop several properties concerning the smoothing effects of $S(t)$. Due to some compatibility issues, we need to define the following adapted Sobolev spaces (one can also see \cite{Krieger-Xiang-2021})
\begin{equation*}
\begin{aligned}
H^0_{(0)}(0,L)&:=L^2(0,L),\; \;\;
H^1_{(0)}(0,L):=\{\varphi\in H^1(0,L):\varphi(0)=\varphi(L)=0\};\\
H^i_{(0)}(0,L)&:=\{\varphi\in H^i(0,L):\varphi(0)=\varphi(L)=\varphi'(L)=0\}, \textrm{ for } i= 2,3;\\
H^4_{(0)}(0,L)&:=\{\varphi\in H^4(0,L)\cap H^3_{(0)}(0,L):(\A\varphi)(0)=(\A\varphi)(L)=0\};\\
H^j_{(0)}(0,L)&:=\{\varphi\in H^j(0,L)\cap H^3_{(0)}(0,L):(\A\varphi)(0)=(\A\varphi)(L)=(\A\varphi)'(L)=0\}, \textrm{ for } j= 5,6.
\end{aligned}    
\end{equation*}
with the same Sobolev norm: $\|\varphi\|^2_{H^k(0,L)}:=\int_0^L\left(|\varphi^{(k)}(x)|^2+|\varphi(x)|^2\right)dx$.
\begin{prop}{\cite[Lemma 2.2]{Krieger-Xiang-2021}}\label{prop: Xiang-Krieger gain 1}
For $k\in\{0,1,2,3,4,5,6\}$, if $y^0\in H^k_{(0)}(0,L)$, then the flow $S(t)y^0$ stays in $C([0,T];H^k_{(0)}(0,L))\cap L^2((0,T);H^{k+1}_{(0)}(0,L))$. Moreover, there exist constants $C^k_0(L)$ and $\Tilde{C}_k(L)$ independent of the choice of the initial data $y^0\in H^k_{(0)}(0,L)$ and $T\in(0,L]$ such that
\begin{gather*}
 \|S(t)y^0\|_{C([0,T];H^k_{(0)}(0,L))}+ \|S(t)y^0\|_{L^2([0,T];H^{k+1}_{(0)}(0,L))} \leq C^k_0\|y^0\|_{H^k_{(0)}(0,L)}, \\
\|S(t)y^0\|_{H^k_{(0)}(0,L))} \leq \frac{\Tilde{C}_k(L)}{t^{\frac{k}{2}}}\|y^0\|_{L^2(0,L)},  \forall t\in (0,T],T\in (0, L].
\end{gather*}
\end{prop}
\subsection*{Duality Argument: Hilbert Uniqueness Method}\label{sec: duality arguments and HUM}
In this section, we recall the Hilbert Uniqueness Method (HUM), introduced by J.-L. Lions \cite{HUM, Lions-1988}. This method presents the duality between the controllability and the observability; see also \cite{dolecki-russell-1977}. We also introduce the equivalence between  exponential stability and  observability for KdV equations. Thus, in order to give a quantitative exponential stability, we need a quantitative observability inequality. To establish a quantitative observability inequality, thanks to the duality arguments or HUM, it suffices to prove a quantitative controllability result. 

Introduce the following notations and definitions. Fix the $L^2((0,L)\times(0,T);\C)$, $L^2((0,L);\C)$ and $L^2((0,T);\C)$ with the duality relation respectively
\begin{equation}\label{eq: duality relations}
\begin{aligned}
    \poscalr{u}{v}_{(0,L)\times(0,T)}&:=\int_0^T\int_0^L u(t,x)\overline{v(T-t,L-x)}dxdt,\forall u,v\in L^2((0,L)\times(0,T);\C)\\
    \poscalr{u}{v}_{(0,L)}&:=\int_0^L u(x)\overline{v(L-x)}dx, \forall u,v\in L^2(0,L;\C),\\
     \poscalr{u}{v}_{(0,T)}&:=\int_0^T u(t)\overline{v(T-t)}dt, \forall u,v\in L^2(0,T;\C).
\end{aligned}    
\end{equation}
\begin{rem}
In general, this duality relation is different from the usual $L^2-$inner product. 
\end{rem}
We consider a controlled KdV system in $(0,L)$ with a Neumann control acting on the right endpoint \eqref{eq: linearized KdV system-control-intro}.
Define the energy function as $E(y(t)):=\int_0^L|y(t,x)|^2dx$. The following KdV system is called the adjoint system of \eqref{eq: linearized KdV system-control-intro}  under the duality relation \eqref{eq: duality relations}:
\begin{equation}\label{eq: adjoint linear KdV-HUM}
\left\{
\begin{array}{lll}
    \p_tw+\p_x^3w+\p_xw=0 & \text{ in }(0,T)\times(0,L), \\
     w(t,0)=w(t,L)= \p_xw(t,L)=0&  \text{ in }(0,T),\\
     w(0,x)=w^0(x)&\text{ in }(0,L).
\end{array}
\right.
\end{equation}
Using the definition of the operator $\A$, we can denote the above system by 
\begin{equation}\label{eq: controlled linear KdV-HUM-A_0}
\left\{
\begin{array}{lll}
    \p_t w-\A w=0 & \text{ in }(0,T)\times(0,L), \\
     w(0,x)=w^0(x)&\text{ in }(0,L).
\end{array}
\right.
\end{equation} 
To derive exponential stability, we introduce the following duality argument. 
\begin{defi}[Symmetric finite co-dimensional projector]\label{defi: finite-co-dimensional projector}
Let $\varphi_1,\cdots,\varphi_n\in L^2(0,L;\C)$ be eigenfunctions of the operator $\A$. Let $H$ be a subspace defined by 
\begin{equation}
H:=\{u\in L^2(0,L;\C):\poscalr{u}{\varphi_j}_{(0,L)}=\poscalr{u}{\overline{\varphi}_j}_{(0,L)}=0,j=1,2,\cdots,n\}
\end{equation}
Then $L^2(0,L;\C)= H\oplus \rm{Span}_{\C}\{\varphi_1,\cdots,\varphi_n,\overline{\varphi}_1,\cdots,\overline{\varphi}_n\}$, with a canonical projector 
\begin{equation}
\begin{aligned}
    \Pi_H:L^2(0,L;\C)&\rightarrow H,\\ 
        u&\mapsto \Tilde{u}=\Pi_H(u).
\end{aligned}
\end{equation}
\end{defi}
Based on the definition of the projector $\Pi_H$, we have the following generalized null controllability.
\begin{defi}[Projective null controllability]\label{defi: quantitative null controllability}Let $H$ and $\Pi_H$ be the same as in Definition \ref{defi: finite-co-dimensional projector}.  
We say the system \eqref{eq: linearized KdV system-control-intro} is quantitatively null controllable if and only if there exists a control function $f\in L^2(0,T)$ such that for $\forall y^0\in L^2(0,L)$, the solution $y$ to the system \eqref{eq: linearized KdV system-control-intro} satisfies that $\Pi_Hy(T)=0$, where
\begin{equation}
    \|u\|_{L^2(0,T)}\leq C\|y^0\|_{L^2(0,L)},
\end{equation}
with a constant $C$.
\end{defi}
In analogy with the quantitative projective null controllability defined in Definition \ref{defi: quantitative null controllability}, we have the following definition for Quantitative Projective Observability.
\begin{defi}[Projective observability]\label{defi: Quantitative Observability}
Let $H$ and $\Pi_H$ be the same as in Definition \ref{defi: finite-co-dimensional projector}. We say the system \eqref{eq: adjoint linear KdV-HUM} is quantitatively observable if and only if for $\forall w^0\in H$, the solution $w$ to the system \eqref{eq: adjoint linear KdV-HUM} satisfies the quantitative observability
\begin{equation}\label{eq: quantitative ob-HUM}
\|S(T)w^0\|^2_{L^2(0,L)}\leq C^2\int_0^T|\p_xw(t,0)|^2dt,  
\end{equation}
with a constant $C$. Here $S(t)$ is the semi-group generated by the operator $\A$.
\end{defi}
Last but not least, we define the exponential stability for the system \eqref{eq: adjoint linear KdV-HUM}.
\begin{defi}[Projective exponential stability]\label{defi: quantitative exponential stability}
We say the system \eqref{eq: adjoint linear KdV-HUM} is exponentially stable if and only if there exist two effectively computable constants $C_1$ and $C_2$ such that for $\forall w^0\in H$ the solution $w$ to the system \eqref{eq: adjoint linear KdV-HUM} satisfies that
\begin{equation}
    E(w(t))\leq C_1e^{-C_2t}E(w^0) \forall t\in (0, +\infty).
\end{equation}
\end{defi}
We are now in a position to show the relations among these definitions. All the proofs can be found in Appendix \ref{sec: proof in duality arguments}.
\begin{enumerate}
\item \textbf{Stability\&Observability. }
    In this part, we give a detailed description of the relations between exponential stability and observability. In the following proposition, we prove that there exists $T_0>0$ such that for $T>T_0$, exponential stability implies observability.
\begin{prop}\label{prop: stability to ob}
Assume that the system \eqref{eq: adjoint linear KdV-HUM} is exponentially stable. Let $ w^0\in H$ and $w$ be a solution to the system \eqref{eq: adjoint linear KdV-HUM}. Then, there exists a constant $C$ such that the system \eqref{eq: adjoint linear KdV-HUM} is quantitatively observable with 
\begin{equation*}
\|S(T)w^0\|^2_{L^2(0,L)}\leq C^2\int_0^T|\p_xw(t,0)|^2dt,  
\end{equation*}
where $S(t)$ is the semi-group generated by the operator $\A$.
\end{prop}
On the other hand, in this article, we use the observability to prove the exponential stability. 
\begin{prop}\label{prop: ob to stability}
Assume that the system \eqref{eq: adjoint linear KdV-HUM} is quantitatively observable. Let $ w^0\in H$ and $w$ be a solution to the system \eqref{eq: adjoint linear KdV-HUM}. Then, there exists a constant $C$ such that the system \eqref{eq: adjoint linear KdV-HUM} is exponentially stable.
\end{prop}
\item \textbf{Hilbert Uniqueness Method: Observability\&Controllability. }In this part, we prove that the quantitative controllability of the system \eqref{eq: linearized KdV system-control-intro} is equivalent to the quantitative observability of the system \eqref{eq: adjoint linear KdV-HUM} with the same effectively computable constant $C$. We introduce the following lemma to show that null controllability is equivalent to trajectory controllability.
\begin{lem}\label{lem: trajectoory controllability}
Let $H$ and $\Pi_H$ be the same as in Definition \ref{defi: finite-co-dimensional projector}. The system \eqref{eq: linearized KdV system-control-intro} is null controllable in the sense of Definition \ref{defi: quantitative null controllability} with a control function $f$ and $\|f\|_{L^2(0,T)}\leq C\|y^0\|_{L^2(0,L)}$
if and only if there exists a control function $\Hat{f}\in L^2(0,T)$ such that for $\forall y^0\in L^2(0,L)$, the solution $\Hat{y}$ to the system 
\begin{equation}\label{eq: controlled linear KdV-HUM-zero-initial}
\left\{
\begin{array}{lll}
    \p_t\Hat{y}+\p_x^3\Hat{y}+\p_x\Hat{y}=0 & \text{ in }(0,T)\times(0,L), \\
     \Hat{y}(t,0)=\Hat{y}(t,L)=0&  \text{ in }(0,T),\\
     \p_x\Hat{y}(t,L)=\Hat{f}(t)&\text{ in }(0,T),\\
     \Hat{y}(0,x)=0&\text{ in }(0,L).
\end{array}
\right.
\end{equation} 
satisfies that $\Pi_H\Hat{y}(T)=\Pi_HS(T)y^0$, where $S(t)$ is the corresponding semi-group generated by $\A$. In addition, with a same constant $C$, $\Hat{f}$ satisfies $\|\Hat{f}\|_{L^2(0,T)}\leq C\|y^0\|_{L^2(0,L)}$. 
\end{lem}
\begin{prop}\label{prop: control and ob}
Let $H$ and $\Pi_H$ be the same as in Definition \ref{defi: finite-co-dimensional projector}. The following statements are equivalent
\begin{enumerate}
    \item the system \eqref{eq: linearized KdV system-control-intro} is null controllable in the sense of Definition \ref{defi: quantitative null controllability} with a control function $f$ and
\begin{equation}\label{eq: control-cost-HUM}
\|f\|_{L^2(0,T)}\leq C\|y^0\|_{L^2(0,L)}    
\end{equation}
\item for the same constant $C$, we have 
\begin{equation}\label{eq: ob-cost-HUM}
    \|S(T)w^0\|_{L^2(0,L)}\leq C\|\p_x w(t,0)\|_{L^2(0,T)}, \forall w^0\in H,
\end{equation}
where $w$ is a solution to the system \eqref{eq: adjoint linear KdV-HUM}.
\end{enumerate}
\end{prop}
\end{enumerate}

\subsection*{Modulated functions}
In this section, we fix a parameter $\mu>0$. We aim to give a detailed description of the solutions to the following system
\begin{equation}\label{eq: static KdV with nonvanishing boundary condition}
\left\{
\begin{aligned}
     &h'''+h'=\mu h(x),  \\
     &h(0)=h(L)=0,\\
     &h'(L)-h'(0)=1.
\end{aligned}
\right.
\end{equation}
\begin{lem}\label{lem: existence of modulated functions}
For each $\mu>0$, there exists a unique solution to the system \eqref{eq: static KdV with nonvanishing boundary condition}.
\end{lem}
We put the proof into the Appendix \ref{sec: modulated functions-appendix}.
Now we analyze the asymptotic behavior of the solution $h_{\mu}$ as $\mu\rightarrow+\infty$. For the details of proof, we refer to the Appendix \ref{sec: modulated functions-appendix}. 
\begin{prop}\label{prop: h_mu-est-uniform}
For $\mu>0$, the solution $h_{\mu}$ to the system \eqref{eq: static KdV with nonvanishing boundary condition} satisfies the following properties:
\begin{enumerate}
    \item $\|h_{\mu}\|_{L^{\infty}(0,L)}\sim \frac{1}{\mu^{\frac{1}{3}}}$, as $\mu\rightarrow+\infty$.
    \item $\|h'_{\mu}\|_{L^{\infty}(0,L)}\lesssim 1$, and $\|h'_{\mu}\|_{L^{2}(0,L)}\sim \mu^{-\frac{1}{6}}$ as $\mu\rightarrow+\infty$.
    \item $\lim_{\omega\rightarrow+\infty}h_{\mu}'(0)=1,\quad \lim_{\omega\rightarrow+\infty}h_{\mu}'(L)=0$.
\end{enumerate}
\end{prop}

\section{Spectral analysis of the stationary operator}\label{sec: spectral analysis of A}
In this section, we shall provide detailed spectral information of the stationary operators $\B$ (defined by \eqref{eq: defi-B-op}) and $\A$ (defined by \eqref{eq: defi-A-op}). To enhance readability, it is important to note that only the asymptotic expansion for eigenvalues and eigenfunctions, along with their estimates, will be used in proving the main results. If preferred, this section can be treated as a black box up to these statements. All the technical details in this section can be found in Appendix \ref{sec: appendix-spectrum-A-B}.

\subsection{Eigenvalues and eigenfunctions for $\B$}
 We start by analyzing the eigenvalues and eigenfunctions of the operator $\B$. We shall first investigate the spectral information at a fixed length $L$, where both $L\notin\mathcal{N}$ and $L\in\mathcal{N}$ are involved. Then we analyze the limiting process as $L$ tends to $\mathcal{N}$.

\subsubsection{Eigenmodes at a critical length}\label{sec: Eigenvalues and eigenfunctions at the critical length}

Before we state our result on the asymptotic behaviors near the critical length, it is natural to provide a brief description of the eigenmodes at the critical length. By \cite{Rosier-1997}, we are interested in the eigenmodes that could generate the unreachable subspace $M(L_0)$.
\begin{prop}\label{prop: type-1-2}
For a fixed critical length $L_0\in \mathcal{N}$,  there exists a finite number of pairs $\{k_m,l_m\}\subset \N^*\times \N^*$, with $k_m\geq l_m$ such that $L_0=2\pi\sqrt{\frac{k_m^2+k_m l_m+l_m^2}{3}}$. Moreover, $\{\ii\lambda_{c,m},\G_m\}$ satisfies the eigenvalue problem
\begin{equation}\label{eq: eigenvalue problem with critical length}
\left\{
\begin{array}{ll}
     \G'''+\G'+\ii\lambda_{c}\G=0,  & \text{ in }(0,L_0),\\
     \G(0)=\G(L_0)=\G'(0)-\G'(L_0)=0.& 
\end{array}
\right.       
\end{equation}
with explicit formulas:
\begin{gather}
\lambda_{c,m}:=\lambda_{c,m}(k_m,l_m)=\frac{(2k_m+l_m)(k_m-l_m)(2l_m+k_m)}{3\sqrt{3}(k_m^2+k_m l_m+l_m^2)^{\frac{3}{2}}}\label{eq: defi of lambda_c}\\
\G_m(x):=\pi\sqrt{\frac{2}{3L_0^3}}(-l_m e^{\ii x\frac{\sqrt{3} (2 k_m + l_m) }{3\sqrt{k_m^2 + k_m l_m + l_m^2}}}  -k_m e^{-\ii x\frac{\sqrt{3} (k_m + 2 l_m) }{3 \sqrt{k_m^2 + k_m l_m + l_m^2}}}  + (k_m+l_m)e^{\ii x\frac{\sqrt{3} (-k_m + l_m)}{ 3\sqrt{k_m^2 + k_m l_m + l_m^2}}}).\label{eq: exact formula for critical eigenfunctions}
\end{gather}
Moreover, if $k_m-l_m\equiv0\mod{3}$, we observe another type of eigenfunctions:
\begin{equation}\label{eq: defi of tilde-G}
\Tilde{\G}_m(x):= \frac{1}{\sqrt{2L_0}}\left(e^{\ii x\frac{\sqrt{3} (2 k_m + l_m) }{3\sqrt{k_m^2 + k_m l_m + l_m^2}}}  - e^{-\ii x\frac{\sqrt{3} (k_m + 2 l_m) }{3 \sqrt{k_m^2 + k_m l_m + l_m^2}}}\right).   
\end{equation}
Additionally, $\G_m$ and $\Tilde{\G}_m$ are linearly independent, while $\Tilde{\G}_m'(0)=\Tilde{\G}_m'(L_0)=\frac{\ii\sqrt{2}\pi(k_m+l_m)}{L_0^{3/2}}\neq0$.
\end{prop}

We call these $\G_m$ the \textit{Type 1} eigenfunctions.  As is mentioned in \cite[Remark 3.6]{Rosier-1997}, for $L_0\in\mathcal{N}$, the eigenfunction $\G_m$ can generate a unreachable state for \eqref{eq: linearized KdV system-control-intro}. That is the reason that we call this type of eigenfunctions unreachable directions. We denote by $N_0$ the number of Type 1 eigenfunctions, which is also the dimension of the unreachable subspace.

We emphasize that for $k_m-l_m\not\equiv0\mod{3}$, $\G_m$ is the only solution to \eqref{eq: eigenvalue problem with critical length}. When $k_m-l_m\equiv0\mod{3}$, we call the eigenfunctions $\Tilde{\G}_m$ the \textit{Type 2 eigenfunctions}. Moreover, if we find a Type 2 eigenfunction $\widetilde{\G}_m$, then simple observation implies that any linear combination of $\widetilde{\G}_m$ and $\G_m$ is also of Type 2, where $\G_m$ is defined by \eqref{eq: exact formula for critical eigenfunctions}. Therefore, if the Type 2 eigenfunctions exist, we can always find a normalized Type 2 eigenfunction that is orthogonal to $\G_m$.

\begin{exa}
Let us take $k=l=1$ as an example. In such case, $L_0=2\pi\in\mathcal{N}$. $\G(x)=2(1-\cos{x})$ represents the Type $1$ eigenfunction, while $\Tilde{\G}(x)=2\ii \sin{x}$ represents the Type $2$ eigenfunction.    
\end{exa}

\subsubsection{Eigenmodes for $\B$ a fixed non-critical length $L$}
It is well-known that $\B$ is a skew-adjoint operator with compact resolvent (see \cite{coron-lv2014}). Furthermore, if $L\notin \mathcal{N}$, all the eigenvalues $\{\ii\lambda_j\}_{j\in\Z}$ are simple and $\{\lambda_j\}_{j\in\Z}$ could be organized in the way $\cdots< \lambda_{-2}< \lambda_{-1}<0<\lambda_{1}< \lambda_{2}<\cdots$. Let us denote the normalized eigenfunction corresponding to the eigenvalue $\ii\lambda_j$ by $\E_j$ with $\|\E_j\|_{L^2(0,L)}=1$. Then $\{\E_j\}_{j\in\Z\backslash\{0\}}$ forms an orthonormal basis of $L^2(0,L)$. In particular, we point out that $\E_j(0)=\E_j(L)=0$ and $\E'_j(0)=\E'_j(L)\neq0$ for $L\notin\mathcal{N}$. In the following proposition, we would like to analyze the asymptotic behaviors of the eigenvalues and eigenfunctions as $|j|$ tends to infinity for a fixed non-critical length.

\begin{prop}\label{prop: asymptotic eigenvalues}
Let $\{(\ii \lambda_j, \E_j)\}_{j\in\Z\setminus\{0\}}$ be the eigenmodes of the operator $\B$:
\begin{equation*}
\left\{
\begin{array}{l}
     \E'''+\E'+\ii\lambda\E=0,  \\
     \E(0)=\E(L)=\E'(0)-\E'(L)=0.
\end{array}
\right.   
\end{equation*}
Then they satisfy 
\begin{gather*}
    \lambda_j=(\frac{2j\pi}{L})^3+\frac{40\pi^2}{3}j^2+O(j)\text{ as }j\rightarrow+\infty,\\
\lambda_j=(\frac{2j\pi}{L})^3+\frac{8\pi^2}{3}j^2+O(j)\text{ as }j\rightarrow-\infty, \\
\|\E'_j\|_{L^{\infty}(0,L)}=\bigO(|j|).
\end{gather*}
\end{prop}
To solve the eigenvalue problem for $\B$, we consider the characteristic equation $ (\ii\xi)^3+\ii\xi+\ii \lambda=0$. Introducing a parameter $\tau\in\R$, let $\lambda=2\tau(4\tau^2-1)$. Therefore, the three roots read as $\xi_1=-\tau+\ii\sqrt{3\tau^2-1},\xi_2=-\tau-\ii\sqrt{3\tau^2-1},\xi_3=2\tau$. We distinguish two different regimes:
\begin{itemize}
    \item \textbf{Elliptic regime}:  $3\tau^2-1<0$.
    \item \textbf{Hyperbolic regime}:  $3\tau^2-1>0$.
\end{itemize}
In the proof, we find finite ($2N_L$) elliptic eigenvalues and infinite hyperbolic eigenvalues. We shall see more different features for eigenvalues and eigenfunctions in these two regimes later.
\begin{proof}[Proof of Proposition \ref{prop: asymptotic eigenvalues}]
After the computation of eigenvalues of $\B$, we continue to explore the localization of these eigenvalues and the general form of the associated eigenfunctions. We are now in a position to complete the proof of Proposition \ref{prop: asymptotic eigenvalues}.

We distinguish three different cases.
\begin{enumerate}
\item \textit{Case 1: Elliptic Regime i.e. $3\tau^2-1<0$}.\\
    In this case, we obtain the eigenfunctions in the following form:
\begin{equation}
\E(x)=r_1 e^{-\ii(\tau+\sqrt{1-3\tau^2})x}+r_2 e^{\ii(\sqrt{1-3\tau^2}-\tau)x}+r_3e^{2\ii \tau x}.
\end{equation}
In addition, we are able to compute the derivative of $\E$ in the following form:
\begin{equation}
\E'(x)=-\ii r_1(\tau+\sqrt{1-3\tau^2})e^{-\ii(\tau+\sqrt{1-3\tau^2})x}+\ii r_2(\sqrt{1-3\tau^2}-\tau) e^{\ii(\sqrt{1-3\tau^2}-\tau)x}+2\ii\tau r_3e^{2\ii \tau x}.
\end{equation}
Combining with the boundary conditions, the coefficients $r_1$, $r_2$, and $r_3$ satisfy the following equations:
\begin{align*}
r_2=r_1\frac{ e^{2\ii \tau L}-e^{-\ii(\tau+\sqrt{1-3\tau^2})L}}{e^{\ii(\sqrt{1-3\tau^2}-\tau)L}-e^{2\ii \tau L}},\quad r_3=-r_1\left(1+\frac{ e^{2\ii \tau L}-e^{-\ii(\tau+\sqrt{1-3\tau^2})L}}{e^{\ii(\sqrt{1-3\tau^2}-\tau)L}-e^{2\ii \tau L}}\right).
\end{align*}
Simplifying the above equations, we know that $\tau$ must satisfy the following equation: 
\begin{equation}\label{eq: t-L defined eq-1}
     2\sqrt{1-3\tau^2}\cos{(2\tau L)}-(\sqrt{1-3\tau^2}+3\tau)\cos{((\sqrt{1-3\tau^2}-\tau)L)}+(3\tau-\sqrt{1-3\tau^2})\cos{((\sqrt{1-3\tau^2}+\tau)L)}=0.
\end{equation}
The number of parameters $\tau$ satisfying the equation \eqref{eq: t-L defined eq-1} is finite and depends on $L$. A simple observation shows that if $\tau$ satisfies \eqref{eq: t-L defined eq-1}, then $-\tau$ also does. And we also notice that there exist two trivial solutions $\tau=\pm \frac{\sqrt{3}}{6}$, such that for any $L>0$,
{\footnotesize
\begin{align*}
 &2\sqrt{1-3\tau^2}\cos{(2\tau L)}-(\sqrt{1-3\tau^2}+3\tau)\cos{((\sqrt{1-3\tau^2}-\tau)L)}+(3\tau-\sqrt{1-3\tau^2})\cos{((\sqrt{1-3\tau^2}+\tau)L)}|_{\tau=\frac{\sqrt{3}}{6}}\\
 =&\sqrt{3}\cos{(\frac{\sqrt{3}}{3} L)}-(\frac{\sqrt{3}}{2}+\frac{\sqrt{3}}{2})\cos{(\frac{\sqrt{3}}{2}-\frac{\sqrt{3}}{6}) L}+0\cdot\cos{(\frac{\sqrt{3}}{2}+\frac{\sqrt{3}}{6}) L}
 \equiv0.\\
  &2\sqrt{1-3\tau^2}\cos{(2\tau L)}-(\sqrt{1-3\tau^2}+3\tau)\cos{((\sqrt{1-3\tau^2}-\tau)L)}+(3\tau-\sqrt{1-3\tau^2})\cos{((\sqrt{1-3\tau^2}+\tau)L)}|_{\tau=-\frac{\sqrt{3}}{6}}\\
 =&\sqrt{3}\cos{(\frac{\sqrt{3}}{3} L)}-0\cdot\cos{(\frac{\sqrt{3}}{2}+\frac{\sqrt{3}}{6}) L}-(\frac{\sqrt{3}}{2}+\frac{\sqrt{3}}{2})\cdot\cos{(\frac{\sqrt{3}}{2}-\frac{\sqrt{3}}{6}) L}
 \equiv0.
\end{align*} }
Now we plug $\tau=\frac{\sqrt{3}}{6}$ into the equations to verify the boundary condition. After simplifying the equations, we obtain
\begin{equation*}
\left\{
\begin{array}{l}
     r_1+r_2+r_3=0,  \\
     r_1 e^{-\ii\frac{2\sqrt{3}}{3}L}+r_2 e^{\ii\frac{\sqrt{3}}{3}L}+r_3e^{\ii \frac{\sqrt{3}}{3} L}=0,\\
    -\ii\frac{2\sqrt{3}}{3} r_1+\ii \frac{\sqrt{3}}{3} r_2+\ii\frac{\sqrt{3}}{3} r_3     =-\ii \frac{2\sqrt{3}}{3}r_1 e^{-\ii\frac{2\sqrt{3}}{3}L}+\ii\frac{\sqrt{3}}{3} r_2 e^{\ii\frac{\sqrt{3}}{3}L}+\ii\frac{\sqrt{3}}{3} r_3e^{\ii \frac{\sqrt{3}}{3} L}.
\end{array}
\right.    
\end{equation*}
This implies that $r_1=0$ and $r_2+r_3=0$. Hence, in this case, $\E(x)\equiv0$, which implies that the candidate function fails to be an eigenfunction. Thus, we exclude these two trivial solutions for $\tau$. Therefore, we deduce that in this case, we find $2N_L$ eigenvalues (see more details in Lemma \ref{lem: invariance of elliptic subspace} below) 
$$\{\lambda_{-N_L},\cdots,\lambda_{-1},\lambda_1,\cdots,\lambda_{N_L}\}.$$ Based on the equation above, we obtain the form of $\E_j(x)$, $|j|\leq N_L$,
{\footnotesize
\begin{equation*}
\E_j(x)= \alpha_j\frac{2\ii\left(e^{\ii\tau_j(2L-x)}\sin{(x\sqrt{1-3\tau_j^2})}+e^{-\ii\tau_j(L+x)}\sin{(\sqrt{1-3\tau_j^2}(L-x))}-e^{\ii\tau_j(2x-L)}\sin{(L\sqrt{1-3\tau_j^2})}\right)}{e^{\ii(\sqrt{1-3\tau_j^2}-\tau_j)L}-e^{2\ii \tau_j L}},
\end{equation*}
}
where $\alpha_j(L)$ is a normalized constant such that $\|\E_j\|_{L^2(0,L)}=1$.
\item \textit{Case 2: $3\tau^2-1=0$}.\\ 
In this case, we have $ \xi_1=-\frac{\sqrt{3}}{3},\xi_2=-\frac{\sqrt{3}}{3},\xi_3=\frac{2\sqrt{3}}{3}$, or $\xi_1=\frac{\sqrt{3}}{3},\xi_2=\frac{\sqrt{3}}{3},\xi_3=-\frac{2\sqrt{3}}{3}$. We obtain the eigenfunctions in the form (the other case is similar) $\E(x)=r_1 e^{-\ii\frac{\sqrt{3}}{3}x}+r_2 e^{-\ii\frac{\sqrt{3}}{3}x}+r_3e^{\ii \frac{2\sqrt{3}}{3} x}$. To verify the boundary condition, we know that 
{\footnotesize
\begin{equation*}
\left\{
\begin{array}{l}
     \E(0)=r_1+r_2+r_3=0,  \\
     \E(L)=r_1 e^{-\ii\frac{\sqrt{3}}{3}L}+r_2 e^{-\ii\frac{\sqrt{3}}{3}L}+r_3e^{\ii \frac{2\sqrt{3}}{3} L}=0,\\
     \E'(0)=-\ii\frac{\sqrt{3}}{3}r_1 -\ii\frac{\sqrt{3}}{3}r_2 +\frac{2\sqrt{3}}{3}r_3
     =-\ii\frac{\sqrt{3}}{3}r_1 e^{-\ii\frac{\sqrt{3}}{3}L}-\ii\frac{\sqrt{3}}{3}r_2 e^{-\ii\frac{\sqrt{3}}{3}L}+\frac{2\sqrt{3}}{3}r_3e^{\ii \frac{2\sqrt{3}}{3} L}=\E'(L)
\end{array}
\right.
\end{equation*}
}
The first two equations imply that 
$r_3=0$ and $r_1+r_2=0$. Hence, in this case, $\E'(0)=\E'(L)=0$, which implies that the candidate function cannot satisfy the boundary conditions. As a consequence, the candidate function fails to be an eigenfunction.
\item \textit{Case 3: Hyperbolic Regime i.e. $3\tau^2-1>0$ }.\\
This case is quite similar to the first case, for the details, one can check in the Appendix. This case can also be found in \cite{Cerpa-Crepeau-2009,coron-lv2014}.
\end{enumerate}

In the remaining case $3\tau^2-1>0$, we obtain the eigenfunction $\E$ in the following form:
\begin{equation}
\E(x)=e^{-\ii\tau x}\left[r_1\cosh{(\sqrt{3\tau^2-1}x)}+r_2\sinh{(\sqrt{3\tau^2-1}x)}\right]+r_3e^{2\ii \tau x}.
\end{equation}
In addition, we are able to compute the derivative of $\E$ in the following form:
\begin{equation}
\begin{aligned}
\E'(x)&=-\ii\tau e^{-\ii\tau x}\left[r_1\cosh{(\sqrt{3\tau^2-1}x)}+r_2\sinh{(\sqrt{3\tau^2-1}x)}\right]+2\ii \tau r_3e^{2\ii \tau x}  \\
+&e^{-\ii\tau x}\left[(\sqrt{3\tau^2-1}r_1\sinh{(\sqrt{3\tau^2-1}x)}+\sqrt{3\tau^2-1}r_2\cosh{(\sqrt{3\tau^2-1}x)}\right].
\end{aligned}
\end{equation}
Combining with the boundary conditions, the coefficients $r_1$, $r_2$, and $r_3$ satisfy the following equations:
\begin{equation*}
\left\{
\begin{array}{l}
     \E(0)=r_1+r_3=0,  \\
     \E(L)=e^{-\ii\tau L}\left[r_1\cosh{(\sqrt{3\tau^2-1}L)}+r_2\sinh{(\sqrt{3\tau^2-1}L)}\right]+r_3e^{2\ii \tau L}=0,\\
     \E'(0)=-\ii\tau r_1+2\ii \tau r_3+\sqrt{3\tau^2-1}r_2\\
     =-\ii\tau e^{-\ii\tau L}\left[r_1\cosh{(\sqrt{3\tau^2-1}L)}+r_2\sinh{(\sqrt{3\tau^2-1}L)}\right]+2\ii \tau r_3e^{2\ii \tau L}  \\
+e^{-\ii\tau L}\left[\sqrt{3\tau^2-1}r_1\sinh{(\sqrt{3\tau^2-1}L)}+\sqrt{3\tau^2-1}r_2\cosh{(\sqrt{3\tau^2-1}L)}\right]=\E'(L).
\end{array}
\right.
\end{equation*}
From the first two equations, we obtain
\begin{align*}
r_2&=r_1\frac{ e^{2\ii \tau L}-e^{-\ii\tau L}\cosh{\sqrt{3\tau^2-1})L}}{\sinh{\sqrt{3\tau^2-1})L}},\\
r_3&=-r_1.
\end{align*}
Simplifying the above equations, we know that $\tau$ must satisfy the following equation: 
\begin{equation}\label{eq: t-L defined eq-2}
     \sqrt{3\tau^2-1}\cos{(2\tau L)}-3\tau \sin{(\tau L)}\sinh{(\sqrt{3\tau^2-1}L)}-\sqrt{3\tau^2-1}\cos{(\tau L)}\cosh{(\sqrt{3\tau^2-1}L)}=0.
\end{equation}
As $|\tau|\rightarrow\infty$, we know that
\begin{equation*}
\sinh{(\sqrt{3\tau^2-1}L)}\sim\cosh{(\sqrt{3\tau^2-1}L)}\sim \frac{1}{2}e^{\sqrt{3\tau^2-1}L}\sim \frac{1}{2}e^{|\tau|L}.
\end{equation*}
This implies that
\begin{align*}
e^{\sqrt{3\tau^2-1}L}&=\frac{\cos{2\tau L}}{2\cos(\tau L-\frac{\pi}{3})}+\bigO(1), \text{ as }\tau\rightarrow+\infty,\\
e^{\sqrt{3\tau^2-1}L}&=\frac{\cos{2\tau L}}{2\cos(\tau L+\frac{\pi}{3})}+\bigO(1), \text{ as }\tau\rightarrow-\infty.
\end{align*}
Hence, for $|j|$ large enough, there exists a unique solution $\tau_{N_L+j}$ (respectively $\tau_{-N_L-j}$) in each interval $[\frac{j\pi}{L},\frac{(j+1)\pi}{L})$ (respectively $[-\frac{(j+1)\pi}{L},-\frac{j\pi}{L})$)
\begin{align*}
\tau_{N_L+j}=\frac{j\pi}{L}+\frac{5\pi}{6}+O(\frac{1}{j})\text{ as }j\rightarrow+\infty,\\
\tau_{-N_L-j}=-\frac{j\pi}{L}+\frac{\pi}{6}+O(\frac{1}{j})\text{ as }j\rightarrow+\infty.
\end{align*}
As a consequence, the eigenvalue has the following asymptotic expansion:
\begin{align*}
\lambda_{N_L+j}=(\frac{2j\pi}{L})^3+\frac{40\pi^2}{3}j^2+O(j)\text{ as }j\rightarrow+\infty,\\
\lambda_{-N_L-j}=-(\frac{2j\pi}{L})^3+\frac{8\pi^2}{3}j^2+O(j)\text{ as }j\rightarrow+\infty.
\end{align*}
and the associated eigenfunction $\E_{N_L+j}(x)=$
\begin{equation*}
\alpha_j(L)e^{-\ii\tau_{N_L+j} x}\left[\cosh{(\sqrt{3\tau_{N_L+j}^2-1}x)}+\frac{e^{3\ii\tau_{N_L+j} L}-\cosh{(\sqrt{3\tau_{N_L+j}^2-1}L)}}{\sinh{(\sqrt{3\tau_{N_L+j}^2-1}L)}}\sinh{(\sqrt{3\tau_{N_L+j}^2-1}x)}\right]-\alpha_je^{2\ii \tau_{N_L+j} x},
\end{equation*}
where $\alpha_j(L)$ is a normalized constant such that $\|\E_{N_L+j}\|_{L^2}=1$. In fact, 
\begin{equation*}
\left(\int_0^L|\E_{N_L+j}(x)-\alpha_j\left[e^{-\ii\tau_{N_L+j} x}\frac{2e^{3\ii\tau_{N_L+j} L}+e^{\sqrt{3\tau_{N_L+j}^2-1}(L-x)}}{2\sinh{(\sqrt{3\tau_{N_L+j}^2-1}L)}}-e^{2\ii \tau_{N_L+j} x}\right]|^2dx\right)^{\frac{1}{2}}=\bigO(|\alpha_j|e^{-\sqrt{3\tau_{N_L+j}^2-1}L})
\end{equation*}
which implies that $\alpha_j\rightarrow\frac{1}{\sqrt{L}}$. Moreover, the derivative of $\E_{N_L+j}$ is in the following form
\begin{align*}
|\E'_{N_L+j}(x)|&\lesssim|\tau_{N_L+j}||\alpha_j|\left[|\cosh{(\sqrt{3\tau_{N_L+j}^2-1}x)}+\frac{e^{3\ii\tau_{N_L+j} L}-\cosh{(\sqrt{3\tau_{N_L+j}^2-1}L)}}{\sinh{(\sqrt{3\tau_{N_L+j}^2-1}L)}}\sinh{(\sqrt{3\tau_{N_L+j}^2-1}x)}|\right.\\
+&\left.|\sinh{(\sqrt{3\tau_{N_L+j}^2-1}x)}+\frac{e^{3\ii\tau_{N_L+j} L}-\cosh{(\sqrt{3\tau_{N_L+j}^2-1}L)}}{\sinh{(\sqrt{3\tau_{N_L+j}^2-1}L)}}\cosh{(\sqrt{3\tau_{N_L+j}^2-1}x)}|+1\right]\\
&\lesssim|\tau_{N_L+j}||\alpha_j|\left[|\frac{\sinh{(\sqrt{3\tau_{N_L+j}^2-1}(L-x))}+e^{3\ii\tau_{N_L+j} L}\sinh{(\sqrt{3\tau_{N_L+j}^2-1}x)}}{\sinh{(\sqrt{3\tau_{N_L+j}^2-1}L)}}|\right.\\
+&\left.|\frac{e^{3\ii\tau_{N_L+j} L}\cosh{(\sqrt{3\tau_{N_L+j}^2-1}x)}-\cosh{(\sqrt{3\tau_{N_L+j}^2-1}(L-x))}}{\sinh{(\sqrt{3\tau_{N_L+j}^2-1}L)}}|+1\right]\\
&\lesssim|\tau_{N_L+j}||\alpha_j|\left(e^{-\sqrt{3\tau_{N_L+j}^2-1}x}+e^{\sqrt{3\tau_{N_L+j}^2-1}(x-L)}+1\right).
\end{align*}
Since $x\in(0,L)$, we obtain $\|\E'_{N_L+j}\|_{L^{\infty(0,L)}}\lesssim|\tau_{N_L+j}|\lesssim |j|$. In particular, at $x=L$, we obtain that
\begin{align*}
\E'_{N_L+j}(L)&=-3\ii\tau_{N_L+j}\alpha_je^{2\ii\tau_{N_L+j} L}
+\alpha_j(\sqrt{3\tau_{N_L+j}^2-1})e^{-\ii\tau_{N_L+j} L}\left[\sinh{(\sqrt{3\tau_{N_L+j}^2-1}L)}\right.\\
&\left.+\frac{e^{3\ii\tau_{N_L+j} L}-\cosh{(\sqrt{3\tau_{N_L+j}^2-1}L)}}{\sinh{(\sqrt{3\tau_{N_L+j}^2-1}L)}}\cosh{(\sqrt{3\tau_{N_L+j}^2-1}L)}\right]\\
&=\alpha_j(-3\ii\tau_{N_L+j}+\sqrt{3\tau_{N_L+j}^2-1}L)e^{2\ii\tau_{N_L+j} L}+\alpha_j\bigO(e^{-\sqrt{3\tau_{N_L+j}^2-1}L})\\
&\sim |j|.
\end{align*} 
\end{proof}
\begin{lem}\label{lem: invariance of elliptic subspace}
Let  $L_0\in\mathcal{N}$ be a fixed critical length. Suppose that $0<\delta<\frac{L_0}{2}$ is sufficiently small. Let $I=[L_0-\delta,L_0+\delta]$ be a small compact interval such that $I\cap \mathcal{N}=\{L_0\}$. For any $L\in I\setminus \{L_0\}$, in the elliptic regime, there are $2N_L$ eigenfunctions (we refer to Proposition \ref{prop: asymptotic eigenvalues}). The number of eigenfunctions in the elliptic regime and $N_L$ are both invariant for $L\in I$.
\end{lem}
\begin{proof}
Let us denote the stationary KdV operator $\B$ in $(0,L)$ by $\B(L)$. Under the condition of $I$, for $\forall L\in I$, $\B(L)=\B(L_0)+(L-L_0)\mathcal{R}(L)$, with $\mathcal{R}(L)=\frac{L_0^2+L_0L+L^2}{L_0^3}\B(L_0)-\frac{L(L+L_0)}{L_0^3}\p_x$ and $|L-L_0|\leq\delta$. For $\forall L\in I$, by Definition \ref{defi: elliptic/hyperbolic index and subspaces}, $L^2(0,L)=U_E(L)\oplus U_H(L)$ and $U_E(L)$ is finite-dimensional. Following \cite[Chapter 4 \textsection 3.4, Theorem 3.16 and Chapter 4 \textsection 3.5]{Kato-book}, we know that $N_E(L):=dim(U_E(L))$ is invariant for $L\in I$. Since $U_E(L)$ is spanned by eigenfunctions in the elliptic regime of $\B(L)$, the number of eigenfunctions in the elliptic regime of $\B(L)$, i.e. $2N_L$, is invariant for $L\in I$.  
\end{proof}
\begin{rem}
As we presented in Section \ref{sec: Eigenvalues and eigenfunctions at the critical length}, $N_0$, the dimension of the unreachable subspace, is related to the number of eigenfunctions of Type 1 in the elliptic regime. We shall see the explicit relation between $N_E(L)$ and $N_0$ in Proposition \ref{prop: Index set M-E}.
\end{rem}
\begin{rem}
Lemma \ref{lem: invariance of elliptic subspace} provides invariant quantities of the limiting process as $L$ gets close to the critical set $\mathcal{N}$. However, following this approach, it is hard to obtain a quantitative explicit description of the eigenvalues and eigenfunctions. Therefore, we adopt a different method to derive quantitative estimates and asymptotic behaviors of the eigenmodes.
\end{rem}

\begin{defi}\label{defi: elliptic/hyperbolic index and subspaces}
We define two index sets 
\begin{gather*}
    \Lambda_E:=\{j\in\Z:\ii\lambda_j \text{ is in the elliptic regime}\}\\
    \Lambda_H:=\{j\in\Z:\ii\lambda_j \text{ is in the hyperbolic regime}\}.
\end{gather*}
Furthermore, we define the elliptic subspace and the hyperbolic subspace by 
\begin{gather*}
    U_E(L):=Span\{\E_j: j\in \Lambda_E\},\;\;\; 
    U_H(L):=Span\{\E_j: j\in\Lambda_H\}.
\end{gather*}
In particular, we know that $N_E(L):=dim(U_E(L))=2N_L<\infty$.
\end{defi}
\begin{rem}\label{rem: link-elliptic-unreacheable}
Recall the unreachable eigenvalues $\ii\lambda_{c,m}$ in \eqref{eq: defi of lambda_c} and we notice that $\ii\lambda_{c,m}$ is in the elliptic regime. This motivates us to investigate the relation between ``the elliptic subspace" and ``the unreachable space".  More details are presented in Section \ref{sec: limiting analysis on different types}. 
\end{rem}

\subsubsection{Asymptotic behaviors for $\B$ as $L$ approaches $\mathcal{N}$}\label{sec: Asymptotic behavior close to the critical lengths}
In this section, we would like to analyze the quantitative asymptotic behaviors of eigenmodes:
\begin{equation*}
\{(\ii \lambda_j, \E_j)(L)\}_{j\in\Z\setminus\{0\}} \textrm{ as } L \textrm{ tends to } L_0.
\end{equation*}
There are many results in this section, which can be sorted as follows:
\begin{enumerate}
    \item For the hyperbolic regime, asymptotic behaviors of eigenvalues are described in Proposition \ref{prop: Asymp in L}, while the uniform estimates of eigenfunctions are presented in Proposition \ref{prop: High-frequency behaviors: uniform estimates} and Proposition \ref{prop: Low-frequency behaviors: uniform estimates};
    \item For the elliptic regime, the asymptotic expansion of eigenvalues is presented in Proposition \ref{prop: Asymp in L-low}, while the singularity of eigenfunctions is described in Proposition \ref{prop: Low-frequency behaviors: Singular limits}. 
\end{enumerate}
Before we state our asymptotic propositions, we introduce the following assumption~:

\smallskip {\bf (C)} \textit{Let  $L_0=2\pi\sqrt{\frac{k^2+kl+l^2}{3}}$ be a fixed critical length. Suppose that $0<\delta<\frac{L_0}{2}$ is sufficiently small. Let $I=[L_0-\delta,L_0+\delta]$ be a small compact interval such that $I\cap \mathcal{N}=\{L_0\}$.}
\vspace{1em}

$\bullet$ \textbf{Hyperbolic regime.} To analyze the asymptotic behaviors as $L$ approaches the critical length $L_0$, we begin with the following proposition, which describes the eigenvalue asymptotic behaviors in the hyperbolic regime as $|j|\rightarrow+\infty$.
\begin{prop}[Hyperbolic eigenvalue localization]\label{prop: Asymp in L}
Let $I$ satisfy the condition {\bf (C)}. Then for every $L\in I$, the sequence  $\{\lambda_j(L)\}_{j\in\Z\backslash\{0\}}$ uniformly satisfies 
    \begin{align*}
        \lambda_{j}=(\frac{2j\pi}{L})^3+O(j^2)\text{ as }j\rightarrow+\infty,\;
\lambda_{j}=(\frac{2j\pi}{L})^3+O(j^2)\text{ as }j\rightarrow-\infty.
    \end{align*}
\end{prop}
\begin{proof}[Proof of Proposition \ref{prop: Asymp in L}]
The uniform asymptotic behavior of the eigenvalues is quite similar to that in Proposition \ref{prop: asymptotic eigenvalues}. We know that $(\tau,L)$ satisfies the equation,
\begin{equation*}
     \sqrt{3\tau^2-1}\cos{(2\tau L)}-3\tau \sin{(\tau L)}\sinh{(\sqrt{3\tau^2-1}L)}-\sqrt{3\tau^2-1}\cos{(\tau L)}\cosh{(\sqrt{3\tau^2-1}L)}=0.
\end{equation*}
As $|\tau|\rightarrow\infty$, we know that
\begin{equation*}
\frac{1}{4}e^{\sqrt{3\tau^2-1}L}\leq\sinh{(\sqrt{3\tau^2-1}L)}\leq \frac{1}{2}e^{\sqrt{3\tau^2-1}L},\frac{1}{4}e^{\sqrt{3\tau^2-1}L}\leq\cosh{(\sqrt{3\tau^2-1}L)}\leq \frac{1}{2}e^{\sqrt{3\tau^2-1}L}
\end{equation*}
Therefore, we obtain, uniformly in $L$,
\begin{align*}
e^{\sqrt{3\tau^2-1}L}&=\frac{\cos{2\tau L}}{2\cos(\tau L-\frac{\pi}{3})}+\bigO(1),\text{ as }\tau\rightarrow+\infty,\;
e^{\sqrt{3\tau^2-1}L}&=\frac{\cos{2\tau L}}{2\cos(\tau L+\frac{\pi}{3})}+\bigO(1), \text{ as }|\tau|\rightarrow+\infty.
\end{align*}
Hence, for $|j|$ large enough, the eigenvalue has the following asymptotic expansion uniformly in $L$:
\begin{align*}
\lambda_{j}=(\frac{2j\pi}{L})^3+O(j^2)\text{ as }j\rightarrow+\infty,\;
\lambda_{j}=(\frac{2j\pi}{L})^3+O(j^2)\text{ as }j\rightarrow-\infty.
\end{align*}
\end{proof}
The following lemma gives a description on the eigenvalue bounds.
\begin{lem}\label{lem: bound of eigenvalues}
Let $I$ satisfy the condition {\bf (C)}. For any integer $J>0$ there exists an effectively constant $K$ such that for any $L\in I$, any $|j|\leq J$, the eigenvalues  $\lambda_j$  of $\B$ satisfies $ |\lambda_j|\leq K$. In particular, for any $L\in I\backslash\{L_0\}$, any $k\in \N^*$, we know that $|\lambda_{-N_L-k}|=|\lambda_{N_L+k}|\geq \frac{2\sqrt{3}}{9}$.
\end{lem}
\begin{proof}
For our punctured interval $I\setminus\{L_0\}$, we have the decomposition $I\setminus\{L_0\}=I^1_J\cup I^2_J$, with 
\begin{equation*}
I^1_J=\{L\in I\setminus\{L_0\} \text{ such that }N_L< J\}\quad I^2_J=\{L\in I\setminus\{L_0\} \text{ such that }N_L\geq J\}.
\end{equation*}
Therefore, $I=I^1_J\cup I^2_J\cup\{L_0\}$.
\begin{enumerate}
    \item First, for the critical length $L_0\in I$, by the proof of Lemma \ref{lem: critical eigenvalues} in Appendix \ref{sec: Proof of Lemma-lem: critical eigenvalues}, we know that $|\lambda_{c}(k_j,l_j)|<\frac{2\sqrt{3}}{9}$.
    \item For any $L\in I^2_J$, $|j|\leq J\leq N_L$, we deduce that $3\tau_j^2<1$. Hence, we know that $|\lambda_j|<\frac{2\sqrt{3}}{9}$.
    \item For any $L\in I^1_J$, we know that $|\lambda_j|<\frac{2\sqrt{3}}{9}$ holds for $|j|\leq N_L$. For $N_L<|j|\leq J$, we are in the hyperbolic regime, i.e., $3\tau_j^2>1$. Due to $\lambda_j=2\tau_j(4\tau_j^2-1)$, we have
    \begin{equation*}
        |\lambda_j|=2|\tau_j|(\tau_j^2+3\tau_j^2-1)>2|\tau_j|^3>2(\frac{\sqrt{3}}{3})^3=\frac{2\sqrt{3}}{9}.
    \end{equation*}
    However, by the proof of Proposition \ref{prop: asymptotic eigenvalues}, there are finite $\tau_{N_L+k}$ (respectively $\tau_{-N_L-k}$), $1\leq k\leq J-N_L$, such that $\tau_{N_L+k}\in [\frac{k\pi}{L},\frac{(k+1)\pi}{L})$ (respectively $\tau_{-N_L-k}\in [-\frac{(k+1)\pi}{L},-\frac{k\pi}{L})$). Hence, 
    \begin{equation*}
        |\lambda_{N_L+k}|=|2\tau_{N_L+k}(4\tau_{N_L+k}^2-1)|\leq |8(\frac{(k+1)\pi}{L})^3|\leq \frac{8J^3\pi^3}{(L_0-\delta)^3}.
    \end{equation*}
\end{enumerate}
In summary, let us define a constant $K=K(I,J)=\frac{8J^3\pi^3}{(L_0-\delta)^3}+\frac{2\sqrt{3}}{9}$. Then, for any $L\in I$, any $|j|\leq J$, the eigenvalues $\lambda_j$ of $\B$ satisfies $|\lambda_j|\leq K:=\frac{8J^3\pi^3}{(L_0-\delta)^3}+\frac{2\sqrt{3}}{9}$.
\end{proof}
\begin{rem}
Combining with Lemma \ref{lem: bound of eigenvalues}, one can use the eigenvalue bound $\frac{2\sqrt{3}}{9}$ or equivalently the index $N_L$ to divide the spectrum of $\B$ into elliptic/hyperbolic regimes. We can use either of them for convenience.
\begin{figure}[h]
\centering
\includegraphics[width=0.7\textwidth]{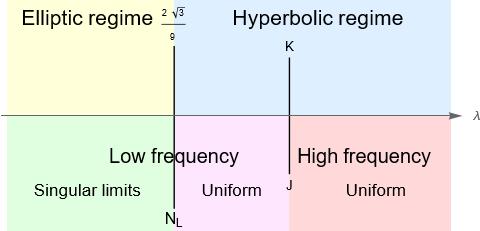} 
\caption{Spectral Division of $\B$}
\label{fig: spectrum divide} 
\end{figure}
\end{rem}
\begin{prop}[High-frequency behaviors of hyperbolic eigenfunctions: uniform estimates]\label{prop: High-frequency behaviors: uniform estimates}
Let $I$ satisfy the condition {\bf (C)}. There exists an  integer $J$ and a  constant $\gamma$ such that for any $L\in I$, any eigenfunction $\E_j$ with $|j|>J$ of the operator $\B$ satisfies that 
  \begin{equation}
        |\E'_{j}(0)|= |\E'_{j}(L)|\geq \gamma|j|.
  \end{equation}
\end{prop}
\begin{proof}[Proof of Proposition \ref{prop: High-frequency behaviors: uniform estimates}]
Let $J>2N_0+J_0(L_0)$ with some $J_0\geq \frac{L_0}{\pi}-1$. We know that for $|j|>J$, $|\lambda_j|>\frac{2\sqrt{3}}{9}$ Furthermore, we know that
\begin{align*}
\E'_{N_L+j}(L)&=-3\ii\tau_{N_L+j}\alpha_je^{2\ii\tau_{N_L+j} L}
+\alpha_j(\sqrt{3\tau_{N_L+j}^2-1})e^{-\ii\tau_{N_L+j} L}\left[\sinh{(\sqrt{3\tau_{N_L+j}^2-1}L)}\right.\\
&\left.+\frac{e^{3\ii\tau_{N_L+j} L}-\cosh{(\sqrt{3\tau_{N_L+j}^2-1}L)}}{\sinh{(\sqrt{3\tau_{N_L+j}^2-1}L)}}\cosh{(\sqrt{3\tau_{N_L+j}^2-1}L)}\right]\\
&=-3\ii\tau_{N_L+j}\alpha_je^{2\ii\tau_{N_L+j} L}
+\alpha_j(\sqrt{3\tau_{N_L+j}^2-1})e^{-\ii\tau_{N_L+j} L}\left[\frac{e^{3\ii\tau_{N_L+j} L}\cosh{(\sqrt{3\tau_{N_L+j}^2-1}L)}-1}{\sinh{(\sqrt{3\tau_{N_L+j}^2-1}L)}}\right].
\end{align*}
Hence, we obtain the estimates
\begin{align*}
|e^{2\ii\tau_{N_L+j} L}\E'_{N_L+j}(L)|&=\left|-3\ii\tau_{N_L+j}\alpha_j
+\alpha_j\sqrt{3\tau_{N_L+j}^2-1}\frac{\cosh{(\sqrt{3\tau_{N_L+j}^2-1}L)}-e^{-3\ii\tau_{N_L+j} L}}{\sinh{(\sqrt{3\tau_{N_L+j}^2-1}L)}}\right|\\
&\geq \left|\alpha_j\sqrt{3\tau_{N_L+j}^2-1}\frac{\cosh{(\sqrt{3\tau_{N_L+j}^2-1}L)}-\cos{3\tau_{N_L+j} L}}{\sinh{(\sqrt{3\tau_{N_L+j}^2-1}L)}}\right|\\
\end{align*}
Since $N_L+j\geq J>2N_0+J_0(L_0)$, we deduce that $j>J_0(L_0)$. This implies that $\tau_{N_L+j}\geq\frac{(J_0+1)\pi}{L}>\frac{2}{3}>\frac{\sqrt{3}}{3}$. Therefore, for any $j>J_0(L_0)$, $\tau_{N_L+j}>\frac{2}{3}>\frac{\sqrt{3}}{3}$, and there exists a constant $\gamma$ such that
\begin{equation*}
\left|\alpha_j\frac{\cosh{(\sqrt{3\tau_{N_L+j}^2-1}L)}-\cos{3\tau_{N_L+j} L}}{\sinh{(\sqrt{3\tau_{N_L+j}^2-1}L)}}\right|>\epsilon_*.    
\end{equation*}
Let $\gamma=\epsilon_*\frac{\sqrt{3}\pi}{3L_0}$. Since $\frac{\sqrt{3\tau_{N_L+j}^2-1}}{|j|}\geq \frac{\sqrt{3(\frac{j\pi}{L})^2-1}}{|j|}\geq \frac{\sqrt{3(\frac{j\pi}{L})^2-\frac{9}{4}(\frac{j\pi}{L})^2}}{|j|}\geq \frac{\sqrt{3}\pi}{3L_0}$, we conclude that 
\begin{equation*}
        |\E'_{j}(0)|= |\E'_{j}(L)|\geq \gamma|j|.
  \end{equation*}
\end{proof} 
\begin{prop}[Low-frequency behaviors of hyperbolic eigenfunctions: uniform estimates]\label{prop: Low-frequency behaviors: uniform estimates}
 Let $I$ satisfy the condition {\bf (C)}. Let  $K>\frac{2\sqrt{3}}{9}$. There exists $\gamma= \gamma(K)>0$ such that for any $L\in I$ and for any eigenfunction $\E_j$ associated to $\lambda_j$ satisfying $|\lambda_j|\leq K$ and $j\notin\Lambda_E$, we obtain $|\E_j'(0)|= |\E_j'(L)|\geq \gamma$.
\end{prop}

\begin{proof}[Sketch of the proof of Proposition \ref{prop: Low-frequency behaviors: uniform estimates}]
Here we only state the sketch of the proof and more technical details can be found in Appendix \ref{sec: Remainders of the proof-uniform-low}. Indeed, this proof is inspired by the proof of \cite[Proposition 1.2]{Krieger-Xiang-2021}.
We argue by contradiction. Suppose that there exists $j_0\notin \Lambda_E$ and $|\lambda_{j_0}|< K$ such that $|\E'_{j_0}(0)|=|\E'_{j_0}(L)|\leq\gamma$. For simplicity, we denote by $\Lambda=\lambda_{j_0}$ in this proof. We first extend the function $\E_{j_0}$ trivially past the endpoints of the interval $[0,L]$, we obtain a function 
$f(x)=\E_{j_0}(x)$, for $x\in[0,L]$, and $f(x)=0$, for $x\notin[0,L]$. Then, the extended function $f$, via Fourier transform, further satisfies the equation:
\begin{equation}
\Hat{f}(\xi)\cdot((\ii\xi)^3+\ii \xi+\ii\Lambda)=2\alpha-2\beta e^{-\ii\xi L}+\ii\theta\xi(1-e^{-\ii\xi L}),
\end{equation}
where $\alpha=\E_{j_0}''(0),\beta=\E_{j_0}''(L),\theta=\E_{j_0}'(0)=\E_{j_0}'(L)$. In addition, we know that $|\theta|\leq \gamma$. By Palay-Wiener theorem, it is easy to see that $\Hat{f}$ is a holomorphic function when we extend $\xi$ to complex values, as $f$ is compactly supported in $\R$. Therefore, away from the zeros of the polynomial $(\ii\xi)^3+\ii \xi+\ii\Lambda$, we have the following expression
\begin{equation}
    \Hat{f}(\xi)=\ii\frac{2\alpha-2\beta e^{-\ii\xi L}+\ii\theta\xi(1-e^{-\ii\xi L})}{\xi^3-\xi-\Lambda}.
\end{equation}
Let $\gamma_*=\frac{R^{\frac{5}{2}}}{(1+18e^{12 L_0 R})^{\frac{1}{2}}}$. Thus, provided that $\gamma<\frac{(2\pi-1)R^{\frac{3}{2}}}{\sqrt{3(3+2e^{12R L_0})}}$, we can deduce that
$|\alpha|+|\beta|>\gamma_*$ similarly as in \cite{Krieger-Xiang-2021} (See more details in Appendix \ref{sec: Remainders of the proof-uniform-low}).  Then, let $\varepsilon_*=\varepsilon_*(K)=\frac{1}{2}e^{-\frac{2}{L_0}\sqrt{\frac{3}{4}(1+3/2K)^{\frac{2}{3}}-1}}$.
Similarly as in \cite{Krieger-Xiang-2021}  (See more details in Appendix \ref{sec: Remainders of the proof-uniform-low}), we can verify that $\varepsilon_*|\beta|\leq |\alpha| \leq \varepsilon_*^{-1}|\beta|$ and $\min\{|\alpha|,|\beta|\}\geq \frac{\varepsilon_*}{\varepsilon_*+1}\gamma_*$. Thanks to this fact, all the roots of $2\alpha-2\beta e^{-\ii L\xi}$ are of the form: $\mu_0+\frac{2\pi n}{L},\text{ with }|\Re{\mu_0}|\leq \frac{\pi}{L},\Im{\mu_0}\leq \frac{1}{L}\ln{\frac{1}{\varepsilon_*}}$.
Using Cauchy's argument principle, provided that $\gamma$ satisfies $48(1+\varepsilon_*)\gamma\left(\frac{96 R e^{2RL_0}(1+e^{2R L_0})}{r\varepsilon_*^2\gamma_*L_0}+\frac{1+e^{2R L_0}+ 2R L_0 e^{2R L_0}}{\varepsilon_*\gamma_*L_0}\right)<1$, all zeros of the numerator $2\alpha-2\beta e^{-\ii\xi L}+\ii\theta\xi(1-e^{-\ii\xi L})$ in $D_R$ are of the form $\mu_0+\frac{2k\pi}{L}+\bigO(r),k\in\Z$, where $|\mu_0|\leq \frac{2\pi}{L}$ and $r\leq\min\{\frac{\pi}{L},R\}$. Because all zeros of the denominator should also be the zeros of the numerator, assuming that $\xi_0^3-\xi_0=\Lambda$, we know that the roots of the denominator are of the form
\begin{equation*}
     \xi_1=\xi_0, \; 
    \xi_2=\xi_0+\frac{2k\pi}{L}+2\bigO(r), \;
    \xi_3=\xi_0+\frac{2(k+l)\pi}{L}+2\bigO(r),
\end{equation*}
where $k,l$ are positive integers. 
Since $\xi_1+\xi_2+\xi_3=0$,
we obtain $3\xi_0+(2k+l)\frac{2\pi}{L}+4\bigO(r)=0$.
Therefore, we know that $|\xi_0+(2k+l)\frac{2\pi}{3L}|\leq \frac{4}{3}r$. Since $j\notin \Lambda_E$, we know that $\min_{k,l}|\xi_0+(2k+l)\frac{2\pi}{3L}|>0$.
In particular, if we require that $\frac{4}{3}r< \min_{k,l}|\xi_0+(2k+l)\frac{2\pi}{3L}|$. Then, we obtain a contradiction, which implies that there exists a constant $\gamma=\gamma(K)>0$ such that for any eigenfunction $\E_j$ associated to $\lambda_j$ satisfying $|\lambda_j|\leq K$ and $j\notin\Lambda_E$, we obtain $|\E_j'(0)|= |\E_j'(L)|\geq \gamma$.
\end{proof}
$\bullet$ \textbf{Elliptic regime.} In the following proposition, we study the eigenvalue asymptotic behaviors in the elliptic regime.
\begin{prop}[Eigenvalue asymptotic behaviors: elliptic regime]\label{prop: Asymp in L-low}
Let $I$ satisfy the condition {\bf (C)}. Then for every $L\in I$, and $j\in \Lambda_E(L)$, there exists a critical eigenvalue $\ii\lambda_c(L_0,j)$ such that $\lambda_j(L)$ has the following asymptotic expansion 
    \begin{align*}
       \lambda_j= \lambda_c(L_0,j)+O(|L- L_0|).
    \end{align*}
\end{prop}
\begin{proof}
Before we present our proof, we note that a more precise correspondence of the index is stated in Proposition \ref{prop: Index set M-E}.\\
In this proof, we omit some calculation details and put them into the Appendix \ref{sec: Remainders of the proof of Proposition-prop: Asymp in L-low}. For $|j|\leq N_L$,  we recall the equation \eqref{eq: t-L defined eq-1}. 
\begin{equation*}
     2\sqrt{1-3\tau_j^2}\cos{(2\tau_j L)}-(\sqrt{1-3\tau_j^2}+3\tau_j)\cos{((\sqrt{1-3\tau_j^2}-\tau_j)L)}+(3\tau_j-\sqrt{1-3\tau_j^2})\cos{((\sqrt{1-3\tau_j^2}+\tau_j)L)}=0.
\end{equation*}
We define a function $F$ by 
\begin{equation}\label{eq: defi of function F(t,L)}
F(t,L)=2\sqrt{1-3t^2}\cos{(2t L)}-(\sqrt{1-3t^2}+3t)\cos{((\sqrt{1-3t^2}-t)L)}+(3t-\sqrt{1-3t^2})\cos{((\sqrt{1-3t^2}+t)L)}. 
\end{equation}
For the critical length $L_0=2\pi\sqrt{\frac{k^2+kl+l^2}{3}}$ with $k\geq l$, there are three different roots for $2\tau_c(4\tau_c^2-1)=\lambda_c(k,l)$. We denote by 
\begin{equation}
    \tau_{c,1}=\frac{\pi}{L_0}\frac{2k+l}{3},\tau_{c,2}=\tau_{c,1}-\frac{k\pi}{L_0},\tau_{c,3}=\tau_{c,2}-\frac{l\pi}{L_0}
\end{equation}
Here we take $\frac{\pi}{L_0}\frac{2k+l}{3}$ for example. Other cases can be treated similarly. In particular, we observe that $F(\frac{\pi}{L_0}\frac{2k+l}{3}, L_0)=0$. 
 Then we look at the first derivative of $F$ at the point $(\frac{\pi}{L_0}\frac{2k+l}{3}, L_0)$ (for general explicit formulas of $\p_t F$ and $\p_L F$, one can refer to the Appendix \ref{sec: Remainders of the proof of Proposition-prop: Asymp in L-low}). There are different cases. 
\begin{enumerate}
    \item \textbf{Case 1:} $\exists m\in \N$ such that $k-l=3m$. In this case, we know that $\nabla F(\frac{\pi}{L_0}\frac{2k+l}{3}, L_0)=0$. For general explicit formulas, one can refer to Appendix \ref{sec: Remainders of the proof of Proposition-prop: Asymp in L-low}. Thus, we obtain the Hessian $\nabla^2 F$ at the point $(\frac{\pi}{L_0}\frac{2k+l}{3}, L_0)$ as follows 
    \begin{equation*}
   \nabla^2 F(\frac{\pi}{L_0}\frac{2k+l}{3}, L_0)=\left(
   \begin{array}{cc}
        24\pi\frac{(k^2+kl)L_0}{l}&\frac{8\pi^2(k-l)(k+2l)(2k+l)}{3lL_0}  \\
        \frac{8\pi^2(k-l)(k+2l)(2k+l)}{3lL_0}& -\frac{8\pi^3kl(k+l)}{L_0^3}
   \end{array}
   \right).
    \end{equation*}
    Now we know that the determinant of $\nabla^2 F(\frac{\pi}{L_0}\frac{2k+l}{3}, L_0)$ is 
    \begin{equation*}
        \det{\nabla^2 F}(\frac{\pi}{L_0}\frac{2k+l}{3}, L_0)=-\frac{128\pi^4(-2k^2-2k l+l^2)(k^2+4k l+l^2)(-k^2+2k l+2l^2)}{9l^2L_0^2}.
    \end{equation*}
    Since $k=3m+l$ with $m\in\Z$ and $m\geq0$, it is easy to check that $k^2+4k l+l^2>0$ and $(-2k^2-2k l+l^2)<0$. In addition, $-k^2+2k l+2l^2=3(l^2-3m^2)$. Since there is no solution $(m,l)\in\N^*\times\N^*$ such that $l^2-3m^2=0$, we deduce that the Hessain $\nabla^2 F$ is non-degenerate at the point $(\frac{\pi}{L_0}\frac{2k+l}{3}, L_0)$.

    In summary, at the point $(\frac{\pi}{L_0}\frac{2k+l}{3}, L_0)$, we have that 
    \begin{equation*}
F(\frac{\pi}{L_0}\frac{2k+l}{3}, L_0)=\p_tF(\frac{\pi}{L_0}\frac{2k+l}{3}, L_0)=\p_LF(\frac{\pi}{L_0}\frac{2k+l}{3}, L_0)=0,
    \end{equation*}
    while the Hessian $\nabla^2F(\frac{\pi}{L_0}\frac{2k+l}{3}, L_0)$ is non-degenerate. By Morse Lemma, we know that there exists a neighborhood $U$ of $(\frac{\pi}{L_0}\frac{2k+l}{3}, L_0)$ and a change of coordinates $\kappa: B_r((0,0))\rightarrow U$ defined in a neighborhood of $(0,0)$, with 
    \begin{gather*}
\kappa(0,0)=(\frac{\pi}{L_0}\frac{2k+l}{3}, L_0),\nabla\kappa=Id, \text{ such that }\\
F(t,L)=\frac{1}{2}\langle \kappa^{-1}(t,L),\nabla^2F(\frac{\pi}{L_0}\frac{2k+l}{3}, L_0)\kappa^{-1}(t,L)\rangle \text{ holds in }U.
    \end{gather*}
After simple computation, we have the following expansion :
\begin{equation*}
\left|\tau_j(L)-\frac{\pi}{L_0}\frac{2k+l}{3}+\frac{(k-l)(k+2l)(2k+l)}{12\pi k(k+l)(k^2+kl+l^2)}(L-L_0)\right|=\frac{\sqrt{k^2+kl+l^2}}{6\pi k(k+l)}\left|L-L_0\right|+\bigO((L-L_0)^2),
\end{equation*}
which implies that $\tau^{\pm}_{j}(L)=\frac{\pi}{L_0}\frac{2k+l}{3}-\frac{(k-l)(k+2l)(2k+l)}{12\pi k(k+l)(k^2+kl+l^2)}(L-L_0)\pm\frac{\sqrt{k^2+kl+l^2}}{6\pi k(k+l)}\left|L-L_0\right|+\bigO((L-L_0)^2)$. Hence, using $\lambda_j(L)=2\tau_j(4\tau_j^2-1)$, $\tau^+_j$ and $\tau^-_j$ generate two eigenvalues $\ii\lambda^{\pm}_{j}$ of $\B$ that approach the same critical eigenvalues $\ii\lambda_{c}(k,l)$. More precisely, we have the following asymptotic expansion:
\begin{equation}\label{eq: pm-eigenvalues}
\lambda^{\pm}_j(L)=\lambda_{c}(k,l)-\frac{(k-l)(k+2l)(2k+l)}{2\pi (k^2+kl+l^2)^2}(L-L_0)\pm\frac{\left|L-L_0\right|}{\pi \sqrt{k^2+kl+l^2}}+\bigO((L-L_0)^2).
\end{equation}
From the preceding formula one observes that the first order does not vanish, thanks to the following 
\begin{lem}\label{lem: critical eigenvalues}
$    (k-l)(k+2l)(2k+l)\neq 2 (k^2+kl+l^2)^{\frac{3}{2}}$.
\end{lem}
One can find its proof in Appendix \ref{sec: Proof of Lemma-lem: critical eigenvalues}.

\item \textbf{Case 2: }$\exists m\in \N$ such that $k-l=3m+1$, $(\p_t F)(\frac{\pi}{L_0}\frac{2k+l}{3}, L_0)=4\sqrt{3}k\pi\neq0$. In this case, we also observe that $\p_L F(\frac{\pi}{L_0}\frac{2k+l}{3}, L_0)=0$. Moreover, $\p^2_L F(\frac{\pi}{L_0}\frac{2k+l}{3}, L_0)=\frac{4l(2k+l)^2\pi^3}{9L_0^3}\neq0$ (see details in the Appendix \ref{sec: Remainders of the proof of Proposition-prop: Asymp in L-low}). By implicit function theorem, we know that there exists a neighborhood of $L=L_0$ such that $\tau_j=\tau_j(L)$ and $F(\tau_j(L),L)=0$ holds near $L=L_0$. After simple computation, we have the following expansion :
\begin{equation}
\tau_j(L)=\frac{\pi}{L_0}\frac{2k+l}{3}-\frac{l(2k+l)^2\pi}{216k^2L_0^3}(L-L_0)^2+\bigO((L-L_0)^3).
\end{equation}
Hence, $\lambda_j(L)=\lambda_{c}(L_0)-\frac{l(k+l)(2k+l)^2}{27kL_0^5}(L-L_0)^2+\bigO((L-L_0)^3)$.
\item \textbf{Case 3: }$\exists m\in \N$ such that $k-l=3m-1$. This case is the same to the second case. We do not repeat the procedure.
\end{enumerate}
\end{proof}

The asymptotic behaviors for the elliptic eigenfunctions are very different from the hyperbolic ones. They could vanish on the boundary. More precisely, we have the following proposition.
\begin{prop}[Low-frequency behaviors: vanishing limits]\label{prop: Low-frequency behaviors: Singular limits}
Let $I$ satisfy the condition {\bf (C)}. Let $j\in\Lambda_E$. There are two cases: 
\begin{enumerate}
    \item If $k\not\equiv l\mod 3$, $\E_j(x)=\G_j(x)+\bigO(|L-L_0|)$ and $|\E'_j(0)|=|\E'_j(L)|=\bigO(L-L_0)$. 
    \item If $k\equiv l\mod 3$, $\E^{\pm}_j(x)=\left( \frac{\sqrt{3} L_0\G_j(x)-(2\pi(k - l)\pm\sqrt{3} L_0) \widetilde{\G}_j(x)}{\sqrt{6L_0^2\pm 2\sqrt{3}\pi L_0(k - l)}}\right)+\bigO(|L-L_0|)$. Here $\E_j^{\pm}$ denotes the associated eigenfunctions for $\ii\lambda_j^{\pm}$ (defined in \eqref{eq: pm-eigenvalues}) and $\G_j,\widetilde{\G}_j$ are in the form of \eqref{eq: exact formula for critical eigenfunctions} and \eqref{eq: defi of tilde-G}. Moreover, $(\E_j^{\pm})'(0)=(\E_j^{\pm})'(L)=\epsilon(L_0)+\bigO(L-L_0)$.
\end{enumerate}
\end{prop}
\begin{proof}[Proof of Proposition \ref{prop: Low-frequency behaviors: Singular limits}]
As we presented in the proof of Proposition \ref{prop: Asymp in L-low}, there are two different cases for the eigenvalues of the operator $\B$. 
\begin{enumerate}
    \item \textbf{Case 1:} $k\not\equiv l\mod 3$. The eigenvalues $\{\lambda_j(L)\}_{j\in \Lambda_E}$ near $L=L_0$ satisfy that 
    \begin{equation*}
        \tau_j(L)=\frac{\pi}{L_0}\frac{2k+l}{3}-\frac{l(2k+l)^2\pi}{216k^2L_0^3}(L-L_0)^2+\bigO((L-L_0)^3).
    \end{equation*}
    Recall the expression of $\E_j(x)$ and $\E_j'(L)$, we obtain
    {\footnotesize
    \begin{equation*}
    \begin{aligned}
    \E_j(x)&= \alpha_j\frac{2\ii\left(e^{\ii\tau_j(2L-x)}\sin{(x\sqrt{1-3\tau_j^2})}+e^{-\ii\tau_j(L+x)}\sin{(\sqrt{1-3\tau_j^2}(L-x))}-e^{\ii\tau_j(2x-L)}\sin{(L\sqrt{1-3\tau_j^2})}\right)}{e^{\ii(\sqrt{1-3\tau_j^2}-\tau_j)L}-e^{2\ii \tau_j L}},\\
    \E_j'(L)&=\alpha_j\frac{2\ii \left(-e^{-2\ii L \tau_j} \sqrt{1 - 3\tau_j^2} + e^{\ii L \tau_j} \sqrt{1 - 3\tau_j^2} \cos\left(L \sqrt{1 - 3\tau_j^2}\right) - 3\ii e^{\ii L \tau_j} \tau_j \sin\left(L \sqrt{1 - 3\tau_j^2}\right)\right)}{-e^{2\ii L \tau_j} + e^{\ii L(-\tau_j + \sqrt{1 - 3\tau_j^2})}}.    
    \end{aligned}
    \end{equation*}
    }
By inserting $\tau_j(L)=\frac{\pi}{L_0}\frac{2k+l}{3}-\frac{l(2k+l)^2\pi}{216k^2L_0^3}(L-L_0)^2+\bigO((L-L_0)^3)$ into the first equation, and after a painstaking computation, we arrive at a new formulation which can be revealed as follows, for two normalized eigenfunctions $\E_j$ and $\G_j$:
\begin{equation}
\E_j(x)=\G_j(x)+\ii(L-L_0)g_j(x)+\bigO((L-L_0)^2),   
\end{equation}
where $g_j$ is a uniformly bounded function. Then, we use the expression of $\E_j'(L)$, and let $\tau_j(L)=\frac{\pi}{L_0}\frac{2k+l}{3}+c_j(k,l)(L-L_0)^2+\bigO((L-L_0)^3)$. After a painful but similar  computation, we obtain an asymptotic formula for $\E_j'(L)$ as follows:
\begin{equation*}
    \E_j'(L)=-\alpha_j\frac{e^{-\frac{2}{3} \ii (k - l) \pi} \left(2 k l^2 (k + l) \pi^4 + 3 \ii c_j(k,l) L_0^5\right)}{k l \pi^2 L_0^2}(L-L_0)+\bigO((L-L_0)^2),
\end{equation*}
which implies that $|\E_j'(L)|\sim |L-L_0|$.
\item \textbf{Case 2:} $k\equiv l\mod 3$. 
We write $k=l+3m$ with $m\in \N$. The eigenvalues $\{\lambda_j(L)\}_{j\in \Lambda_E}$ near $L=L_0$ satisfy that 
    \begin{align*}
     \tau^{\pm}_{j}(L)&=\frac{\pi}{L_0}\frac{2k+l}{3}-\frac{(k-l)(k+2l)(2k+l)}{12\pi k(k+l)(k^2+kl+l^2)}(L-L_0)\pm\frac{\sqrt{k^2+kl+l^2}}{6\pi k(k+l)}\left|L-L_0\right|+\bigO((L-L_0)^2).
    \end{align*}
    By inserting $\tau^-_j(L)=\frac{\pi}{L_0}\frac{2k+l}{3}-\frac{(k-l)(k+2l)(2k+l)}{12\pi k(k+l)(k^2+kl+l^2)}(L-L_0)-\frac{\sqrt{k^2+kl+l^2}}{6\pi k(k+l)}\left|L-L_0\right|+\bigO((L-L_0)^2)$ into the expression of $\E_j(x)$, without loss of generality, we consider $L>L_0$, and after a painstaking computation, we arrive at a new formulation which can be revealed as follows:
\begin{align*}
\E^{\pm}_j(x)=\frac{\sqrt{3} L_0}{\sqrt{6L_0^2\pm 2\sqrt{3}\pi L_0(k - l)}}\G_j(x)-\frac{2\pi(k - l)\pm\sqrt{3} L_0}{\sqrt{6L_0^2\pm 2\sqrt{3}\pi L_0(k - l)}}\widetilde{\G}_j(x)+\Tilde{g}_j(x)(L-L_0)+\bigO((L-L_0)^2), 
\end{align*}
where $\Tilde{g}_j$ is a uniformly bounded function.
Then, we use the expression of $\E_j(x)=\alpha_j\Tilde{f}_j(x)+\alpha_j\Tilde{g}_j(x)(L-L_0)+\bigO((L-L_0)^2)$ to derive $\E_j'(L)$ as follows: 
\begin{equation*}
    \E_j'(L)=-\ii\frac{\sqrt{2}\pi(k+l)}{L_0^{\frac{3}{2}}}\frac{2\pi(k - l)\pm\sqrt{3} L_0}{\sqrt{6L_0^2\pm 2\sqrt{3}\pi L_0(k - l)}}+\bigO(L-L_0).
\end{equation*}
\end{enumerate}
\end{proof}
\begin{rem}\label{rem: rotation of eigenmodes}
As we mentioned in Proposition \ref{prop: type-1-2}, $\widetilde{\G}_j=\frac{1}{\sqrt{2L_0}}\left(e^{\ii x\frac{\sqrt{3} (2 k_j + l_j) }{3\sqrt{k_j^2 + k_j l_j + l_j^2}}}  - e^{-\ii x\frac{\sqrt{3} (k_j + 2 l_j) }{3 \sqrt{k_j^2 + k_j l_j + l_j^2}}}\right)$ is just a typical example of Type 2 eigenfunctions. In fact, any linear combination of $\widetilde{\G}_j$ and $\G_j$ can be a Type 2 eigenfunction. We notice that 
\begin{equation*}
\int_0^{L_0}\G_j(x)\overline{\widetilde{\G}_j(x)}\mathrm{d}x=\frac{\pi(k_j-l_j)}{\sqrt{3}L_0},
\end{equation*}
which is nonzero when $k_j\neq l_j$. For later application, by a standard Gram--Schmidt process, we obtain a normalized Type 2 eigenfunction that is orthogonal to $\G_j$ defined by 
\begin{equation}\label{eq: defi-orthogonal-G}
\widetilde{G}_j:=- \frac{k_j-l_j}{\sqrt{3}(k_j+l_j)}\G_j + \frac{L_0}{\pi(k_j+l_j)}\widetilde{\G}_j
\end{equation}
Using $\G_j$ and $G_j$, we have asymptotic expansions for $\E_j^{\pm}$ as follows:
\begin{equation*}
\begin{pmatrix}
\E_j^+\\
\E_j^-
\end{pmatrix}=
\begin{pmatrix}
C^+_1&C^+_2\\
C^-_1&C^-_2
\end{pmatrix}
\begin{pmatrix}
\G_j\\
G_j
\end{pmatrix}+\bigO(|L-L_0|),
\end{equation*}
where we define the coefficients via
\begin{gather*}
C_{1}^{\pm}:=-\frac{-2\pi^2(k_j^2 + 4k_jl_j + l_j^2) \pm \sqrt{3}\pi L_0(k_j-l_j)}{\sqrt{3} L_0 \sqrt{6L_0^2 \pm 2\sqrt{3}\pi L_0(k_j-l_j)}},\;\;C^{\pm}_{2}:=- \frac{\pi(k_j+l_j)(2\pi(k_j-l_j) \pm \sqrt{3}L_0)}{L_0 \sqrt{6L_0^2 \pm 2\sqrt{3}\pi L_0(k_j-l_j)}}
\end{gather*}
Interestingly, if we define an angle $\theta_j=\theta(k_j,l_j)\in (\frac{\pi}{3},\frac{\pi}{2})$ by $\cos{\theta}=\frac{\pi(k_j-l_j)}{\sqrt{3}L_0}$, which can be understood as the angle between $\G_j$ and $\widetilde{\G}_j$. Then, we have
\begin{equation*}
\begin{pmatrix}
\E_j^+\\
\E_j^-
\end{pmatrix}=
\begin{pmatrix}
-\cos{\frac{3\theta_j}{2}}&-\sin{\frac{3\theta_j}{2}}\\
\sin{\frac{3\theta_j}{2}}&-\cos{\frac{3\theta_j}{2}}
\end{pmatrix}
\begin{pmatrix}
\G_j\\
G_j
\end{pmatrix}+\bigO(|L-L_0|),
\end{equation*}
which implies that our perturbed eigenfunctions $\E_j^{\pm}$ are the rotation of a Type 1 eigenfunction $\G_j$ and its orthogonal Type 2 eigenfunction $\widetilde{G}_j$ and the rotation angle is determined by $\G_j$ the choice of typical Type 2 eigenfunction $\widetilde{\G}_j$. We put the computation details in Appendix \ref{sec: rotation structure of eigenmodes}.
\end{rem}

\subsection{Eigenvalues and eigenfunctions for $\A$}
In this section, we aim to provide some basic spectral information for the operator $\A$. After describing the spectrum of $\A$ for $L\notin\mathcal{N}$ (Sec. \ref{sec: Eigenmodes of A at a non-critical length}), we present an asymptotic analysis for eigenvalues and eigenfunctions of $\A$ as $L\rightarrow\mathcal{N}$ in Sec. \ref{sec: Asymptotic behaviors for A}.
\subsubsection{Eigenmodes of $\A$ at a non-critical length}\label{sec: Eigenmodes of A at a non-critical length}
We consider the eigenvalue problem:
\begin{equation}\label{eq: eigenvalue problem with Neumann}
\left\{
\begin{array}{l}
     \F'''(x)+\F'(x)+\zeta \F(x)=0,x\in(0,L),  \\
     \F(0)=\F(L)= \F'(L)=0.
\end{array}
\right.
\end{equation} 
At first glance at the eigenvalue problem, we notice that the eigenvalues of the operator $\A$  are symmetrically distributed about the real axis and located on the left side of the imaginary axis. 
\begin{lem}\label{lem: basic prop of eigen of bad op}
Let $\F$ be a normalized eigenfunction of $\A$ associated with an eigenvalue $\zeta$. Then, the following properties hold
\begin{enumerate}
    \item  $\Re \zeta\leq 0$. In particular, if $L\notin \mathcal{N}$, $\Re \zeta< 0$. 
    \item If $L\notin\mathcal{N}$, there is a unique eigenfunction associated with each eigenvalue.
    \item If $L\notin\mathcal{N}$ and $\Im \lambda= 0$, then there are infinitely many eigenvalues.
\end{enumerate}
\end{lem}
We omit its proof and one can find proof details in Appendix \ref{sec: proof-basic-properties-A}.

\subsubsection{Asymptotic behaviors for $\A$ as $L\rightarrow\mathcal{N}$}\label{sec: Asymptotic behaviors for A}
In Section \ref{sec: Asymptotic behavior close to the critical lengths}, we present a detailed asymptotic analysis of the eigenvalues of the operator $\B$ near the critical lengths. However, one cannot expect a similarly detailed analysis for the operator $\A$. To the best of our knowledge, the localization of the eigenvalues for the operator $\A$ remains unclear. We only have some partial information for the spectrum of $\A$. In this part, we concentrate on the analysis of the eigenvalues and eigenfunctions associated with the operator $\A$ in the interval $(0,L)$, with $L$ very close to the critical length $L_0\in\mathcal{N}$. More precisely, we are interested in the eigenmodes $(\zeta,\F_{\zeta})\in\C\times L^2(0,L)$, which are solutions to the eigenvalue problem \eqref{eq: eigenvalue problem with Neumann} for $L$ near $L_0$, with $(\zeta,\F_{\zeta})$ is a perturbation of the eigenmodes $(\ii\lambda_c,\G_c)$ (defined in defined in  \eqref{eq: defi of lambda_c} and \eqref{eq: eigenvalue problem with critical length}). 
\begin{prop}\label{prop: asymptotic expansion for A0}
Let $I$ satisfy the condition {\bf (C)}. Then for every $L\in I$, the eigenmodes $(\zeta_j,\F_{\zeta_j})$ have the following asymptotic expansion: 
\begin{gather*}
     \zeta_j= \ii\lambda_{c,j}(L_0)+O((L- L_0)^2),  \\
       |\F_{\zeta_j}'(0)|=\bigO(|L-L_0|).
\end{gather*}
\end{prop}
\begin{proof}[Proof of Proposition \ref{prop: asymptotic expansion for A0}]
We shall use the same trick as before. Let $\ii\zeta=2\tau(4\tau^2-1)$. Then the three roots of $ (\ii\xi)^3+\ii\xi+\zeta=0$, read as $\xi_1=\tau+\sqrt{1-3\tau^2},\xi_2=\tau-\sqrt{1-3\tau^2},\xi_3=-2\tau$.
Thus, $\F(x)=r_1 e^{\ii(\tau+\sqrt{1-3\tau^2})x}+r_2 e^{\ii(\tau-\sqrt{1-3\tau^2})x}+r_3e^{-2\ii\tau x}$. Using boundary conditions, we obtain a equation for $(\tau,L)\in \C\times I$ as follows:
\begin{equation*}
-\sqrt{1-3\tau^2} \cos{3L\tau} + \sqrt{1-3\tau^2}\cos{L\sqrt{1-3\tau^2}}- \ii \sqrt{1-3\tau^2}\sin{3L\tau} + 3\ii \tau\sin{L \sqrt{1-3\tau^2}}=0.
\end{equation*}
After simplification, the eigenfunctions are in the form
\begin{equation}\label{eq: general form for eigenfunction-bad op}
\F(x)=r_1\frac{2\ii\left(e^{-2 \ii \tau x} \sin{L\sqrt{1-3\tau^2}}-e^{\ii \tau x} \sin{\sqrt{1 - 3 \tau^2} (L - x)} - e^{-3 \ii L \tau + \ii \tau x} \sin{\sqrt{1 - 3 \tau^2} x}\right)}{-e^{-3 \ii L \tau} + e^{-\ii L \sqrt{1 - 3 \tau^2}}}
\end{equation}
Let $\tau=t_r+\ii t_i$. Define the following two functions
\begin{align*}
G_r(t_r,t_i,L)&=\Re{\left(-\sqrt{1-3\tau^2} \cos{3L\tau} + \sqrt{1-3\tau^2}\cos{L\sqrt{1-3\tau^2}}- \ii \sqrt{1-3\tau^2}\sin{3L\tau} + 3\ii \tau\sin{L \sqrt{1-3\tau^2}}\right)},\\
G_i(t_r,t_i,L)&=\Im{\left(-\sqrt{1-3\tau^2} \cos{3L\tau} + \sqrt{1-3\tau^2}\cos{L\sqrt{1-3\tau^2}}- \ii \sqrt{1-3\tau^2}\sin{3L\tau} + 3\ii \tau\sin{L \sqrt{1-3\tau^2}}\right)}.
\end{align*}
Then, we define the function $G:\R^2\times I\rightarrow \R^2$ by $G(t_r,t_i,L):=\left(
\begin{array}{c}
    G_r(t_r,t_i,L)\\
    G_i(t_r,t_i,L)
\end{array}\right)$. 
For the Jacobian matrix of $G$ with respect to $(\tau_r,\tau_i)$ at the point $(\frac{\pi}{L_0}\frac{2k+l}{3},0,L_0)$, 
\begin{align*}
J_{G,\tau}(\frac{\pi}{L_0}\frac{2k+l}{3},0,L_0)=\left(
\begin{array}{cc}
    \frac{\p G_r}{\p t_r} &\frac{\p G_r}{\p t_i}  \\
     \frac{\p G_i}{\p t_r}& \frac{\p G_i}{\p t_i}
\end{array}
\right)|_{t_r=\frac{\pi}{L_0}\frac{2k+l}{3},t_i=0,L=L_0}=\left(
\begin{array}{cc}
    0&-4\pi\frac{k^2+kl+l^2}{l}  \\
    4\pi\frac{k^2+kl+l^2}{l}& 0
\end{array}
\right).    
\end{align*}
This implies that $J_{G,\tau}(\frac{\pi}{L_0}\frac{2k+l}{3},0,L_0)$ is invertible. Thus, we deduce that there exists a neighborhood $(L_0-\delta, L_0+\delta)$, with $\delta>0$ sufficiently small, such that there exists a unique continuously differentiable map $(t_r(\cdot),t_i(\cdot)):(L_0-\delta, L_0+\delta)\rightarrow\R^2$ such that $(t_r(L_0),t_i(L_0))=(\frac{\pi}{L_0}\frac{2k+l}{3},0)$ and $G(t_r(L),t_i(L),L)=0$ for all $L\in(L_0-\delta, L_0+\delta)$. In addition, since $\frac{\p G_r}{\p L}(\frac{\pi}{L_0}\frac{2k+l}{3},0,L_0)=\frac{\p G_i}{\p L}(\frac{\pi}{L_0}\frac{2k+l}{3},0,L_0)=0$, the first derivatives for $t_r$ and $t_i$ vanish at the point $L=L_0$. Thus, $t_r$ and $t_i$ have the following expansions
\begin{align*}
t_r(L)=\frac{\pi}{L_0}\frac{2k+l}{3}+\bigO((L-L_0)^3),\;t_i(L)=(-1)^{l+1}\frac{\pi^2kl^2(k+l)}{2(k^2+kl+l^2)}(L-L_0)^2+\bigO((L-L_0)^3).
\end{align*}
As a consequence, we deduce that
\begin{equation}
\zeta=-\ii\frac{(2k+l)(k-l)(2l+k)}{3\sqrt{3}(k^2+k l+l^2)^{\frac{3}{2}}}+(-1)^{l+1}\frac{9\sqrt{3}k^2l^2(k+l)^2 }{8\pi(k^2+kl+l^2)^{\frac{7}{2}}}(L-L_0)^2+\bigO((L-L_0)^3).
\end{equation}
We plug $\tau=\frac{\pi}{L_0}\frac{2k+l}{3}+(-1)^{l+1}\frac{\pi^2kl^2(k+l)}{2(k^2+kl+l^2)}(L-L_0)^2+\bigO((L-L_0)^3)$ in the formula \eqref{eq: general form for eigenfunction-bad op}, then we derive the following expansion for the eigenfunction
\begin{equation}
\F_{\zeta}(x)=r_1\G_{c}(x) +  r_1 \Tilde{f}(x)(L-L_0)+\bigO((L-L_0)^2),
\end{equation}
where $\G_{c}(x)=\frac{1}{k}e^{-\frac{\ii (2k + l) x}{\sqrt{3}\sqrt{k^2 + kl + l^2}}}\left(l -(k+l) e^{\frac{\ii \sqrt{3} k x}{\sqrt{k^2 + kl + l^2}}} + k e^{\frac{\ii \sqrt{3} (k + l) x}{\sqrt{k^2 + kl + l^2}}}\right)$ is the eigenfunction associated with $\lambda_{c}=-\ii\frac{(2k+l)(k-l)(2l+k)}{3\sqrt{3}(k^2+k l+l^2)^{\frac{3}{2}}}$ at the critical length $L_0$.\footnote{we have an explicit formula for $\Tilde{f}$ see in Appendix \ref{sec: proof of expansion of A0}} Moreover, 
\begin{equation*}
|\F'_{\zeta}(0)|\sim|\frac{3l(k+l)(k^2+kl+l^2-6\ii(-1)^lk(k+l))}{2(k^2+kl+l^2)^2}
r_1(L-L_0)|\sim |L-L_0|.
\end{equation*}
For the computation details, we put them into the Appendix \ref{sec: proof of expansion of A0}.
\end{proof}

\section{Part I: A transition-stabilization method}\label{sec: A transition-stabilization method}
In this section, we provide a detailed description of our method: \textit{transition-stabilization method}. In application, we prove Theorem \ref{thm: control-cost} and a weak version of Theorem \ref{thm: main theorem linear version}.

\subsection{The intermediate system I: construction}\label{sec: construction of the control}
In this sequel, we aim to construct the control function $v\in C^1(0,T)$ explicitly for the following intermediate system 
\begin{equation}\label{eq: KdV boundary derivative difference}
\left\{
\begin{array}{lll}
    \p_tz+\p_x^3z+\p_xz=0 & \text{ in }(0,T)\times(0,L), \\
     z(t,0)=z(t,L)=0&  \text{ in }(0,T),\\
     \p_xz(t,L)-\p_xz(t,0)=v(t)&\text{ in }(0,T),\\
     z(0,x)=z^0(x)&\text{ in }(0,L),
\end{array}
\right.
\end{equation}
with the constraints $v(0)=v(T)=0$.
A useful and typical technique to transform an inhomogeneous  boundary problem into a homogenous one is to define a new function $\Tilde{z}(t,x)=z(t,x)-v(t)h(x)$, where $h$ solves the equation 
\begin{equation}\label{eq: defi of h}
\left\{
\begin{array}{ll}
    h'''+h'=0, & \text{ in }(0,L), \\
     h(0)=h(L)=0,\\
     h'(L)-h'(0)=1.
\end{array}
\right.    
\end{equation}
In addition, we know that $h$ solves the equation \eqref{eq: defi of h}. We are able to obtain the exact expression of $h$.
\begin{lem}\label{lem: formula h}
For $L\notin\mathcal{N}$, there exists a unique solution to the equation \eqref{eq: defi of h}. Moreover, the unique solution is 
$h(x)=-\frac{e^{\ii L}+1}{2\ii (e^{\ii L}-1)}+\frac{e^{\ii x}}{2\ii (e^{\ii L}-1)}+\frac{e^{\ii(L-x)}}{2\ii (e^{\ii L}-1)}$. Moreover, $h\in C^{\infty}(0,L)$. $h^{(2m)}(L)=h^{(2m)}(0)=0$, and $h^{(2m+1)}(L)-h^{(2m+1)}(0)=(-1)^m$, $\forall m\in\N$.
\end{lem}
The proof is straightforward and we put it in the Appendix \ref{sec: modulated functions-appendix}.  
As a consequence, we know that $\Tilde{z}$ satisfies the following control system:
\begin{equation}\label{eq: KdV source control}
\left\{
\begin{array}{lll}
    \p_t\Tilde{z}+\p_x^3\Tilde{z}+\p_x\Tilde{z}=-v'\otimes h & \text{ in }(0,\infty)\times(0,L), \\
     \Tilde{z}(t,0)=\Tilde{z}(t,L)=0&  \text{ in }(0,\infty),\\
     \p_x\Tilde{z}(t,L)-\p_x\Tilde{z}(t,0)=0&\text{ in }(0,\infty),\\
     \Tilde{z}(0,x)=z^0(x)&\text{ in }(0,L),
\end{array}
\right.
\end{equation}
According to Duhamel's formula, we write the solution $\Tilde{z}$ in the form:
\begin{equation}\label{eq: duhamel}
    \Tilde{z}(t,x)=e^{t\B}z^0(x)-\int_0^te^{(t-s)\B}v'(s)h(x)ds.
\end{equation}
Since $\{\E_j\}_{j\in\Z\backslash\{0\}}$ forms an orthonormal basis of $L^2(0,L)$, we write the initial data $z^0(x)=\sum_{j\in\Z\backslash\{0\}}z_j^0\E_j(x)$ and $h(x)=\sum_{j\in\Z\backslash\{0\}}h_j\E_j(x)$. Using these expansions, we obtain 
\begin{align*}
e^{t\B}z^0=\sum_{j\in\Z\backslash\{0\}}z_j^0e^{\ii\lambda_j t}\E_j(x),\;
\int_0^te^{(t-s)\B}v'(s)h(x)ds=\sum_{j\in\Z\backslash\{0\}}\int_0^te^{\ii(t-s)\lambda_j}v'(s)h_j\E_j(x)ds.
\end{align*}
For the latter term, we integrate by parts, and thanks to $v(0)=0$, we obtain 
\begin{equation*}
    \int_0^te^{\ii(t-s)\lambda_j}v'(s)h_j\E_j(x)ds=h_jv(t)\E_j(x)-\ii h_j\lambda_j\int_0^te^{\ii(t-s)\lambda_j}v(s)\E_j(x)ds.
\end{equation*} 
Then the solution \eqref{eq: duhamel} is in the following form:
\begin{equation}\label{eq: formula for tilde_z}
    \Tilde{z}(t,x)=\sum_{j\in\Z\backslash\{0\}}\left(z_j^0e^{\ii\lambda_j t}-h_jv(t)-\ii h_j\lambda_j\int_0^te^{\ii(t-s)\lambda_j}v(s)ds\right)\E_j(x).
\end{equation}
Now we apply the moment method. At $t=T$, null controllability implies that $\Tilde{z}(T,x)=0$. By $v(T)=0$, we simplify the condition into 
\begin{equation}\label{eq: moment equation of v}
z_j^0-\ii h_j\lambda_j\int_0^Te^{-\ii s\lambda_j}v(s)ds=0.
\end{equation}
Using this family of conditions, we are able to construct the control function $v$ based on the bi-orthogonal family.
\subsubsection{Bi-orthogonal family}\label{sec: bi-orthogonal family}
In this sequel, we construct a family of functions that is bi-orthogonal to $\{e^{-\ii s\lambda_j}\}_{j\in\Z\backslash\{0\}}$. This type of construction has been developed in several different settings, for example, Schr\"odinger equations and heat equations in \cite{Tenenbaum-Tucsnak2007}, KdV equations and fractional Schr\"odinger equations in \cite{Lissy-2014}.
\begin{prop}\label{prop: biorthogonal family}
There exists a family of functions $\{\phi_j\}_{j\in\Z\backslash\{0\}}$ such that 
\begin{enumerate}
    \item $\supp{\phi_j}\subset [-\frac{T}{2},\frac{T}{2}],\forall j\in\Z\backslash\{0\}$;\label{eq: compact support of phi_j}
    \item $\int_{-\frac{T}{2}}^{\frac{T}{2}}\phi_j(s)e^{-\ii \lambda_ks}ds=\delta_{jk},\forall j,k\in\Z\backslash\{0\}$;\label{eq: biorthorgonal}
    \item For $N\in \N$, there exists a constant $K=K(N)$ such that $\|\phi^{(m)}_j\|_{L^{\infty}(\R)}\leq Ce^{\frac{K}{\sqrt{T}}}|\lambda_j|^m,\forall j\in\Z\backslash\{0\}, \forall m\in\{0,1,\cdots,N\}$. Here the constant $C$ appearing in the inequality might depend on $N$ but not on $T$ and $j$.\label{eq: bound for derivative of phi}
\end{enumerate}
\end{prop}

\begin{coro}\label{coro: uniform est for bi-family}
Let $I$ satisfy the condition {\bf (C)}. Then for every $L\in I\setminus\{L_0\}$,
\begin{enumerate}
    \item if $k\not\equiv l\mod 3$, all estimates in Lemma \ref{prop: biorthogonal family} are uniform in $L$.
    \item if $k\equiv l\mod 3$, for $j\notin\Lambda_E$, the estimates are uniform in $L$. Moreover, 
    \begin{equation}
    \|\phi^{(m)}_{\sigma^{\pm}(q)}\|_{L^{\infty}(\R)}\leq \frac{C}{|\lambda_{\sigma^+(q)}-\lambda_{\sigma^-(q)}|}e^{\frac{K}{\sqrt{T}}}|\lambda_j|^m,\quad |q|\leq N_0.   
    \end{equation}
\end{enumerate}
\end{coro}
\begin{rem}
Our construction can be seen as an extension of the classic one, with a particular focus on tracking the $L-$dependence. Proposition \ref{prop: biorthogonal family} provides the existence and estimates for the bi-orthogonal family. If preferred, this part can be skipped temporarily. The readers can come back and review its proof later when they encounter proofs that rely on the results derived herein. 

The notations $\sigma^{\pm}(q)$ for $q\in\Lambda_{E}$ will be specified in Proposition \ref{prop: Index set M-E} and Appendix \ref{sec: proof of corollary index-lables}. In fact, as one noticed in Proposition \ref{prop: Asymp in L-low}, for $k\equiv l\mod 3$, every elliptic eigenvalue $\ii\lambda_c(k,l)$ at the critical length will split to two eigenvalues $\ii\lambda_{\pm}$ at non-critical lengths. This $\sigma^{\pm}$ is just a relabeling of the indices.
\end{rem}

\subsubsection{Formal construction of the control function: moment method}\label{sec: Formal construction of the control function: moment method}
Now we could construct our control thanks to the family $\{\phi_j\}_{j\in\Z\backslash\{0\}}$. Let
\begin{equation}\label{eq: defi of control v}
    v(t)=\sum_{k\in\Z\backslash\{0\}}\frac{z_k^0}{\ii h_k\lambda_k}e^{\ii \frac{T}{2}\lambda_k}\phi_k(t-\frac{T}{2}).
\end{equation}
By Proposition \ref{prop: biorthogonal family}, it is easy to verify that \eqref{eq: moment equation of v} holds for this $v$. 
As a consequence, we are also able to construct the solution $\Tilde{z}$ formally thanks to the family $\{\phi_j\}_{j\in\Z\backslash\{0\}}$. We plug \eqref{eq: defi of control v} into \eqref{eq: formula for tilde_z} and due to \eqref{eq: moment equation of v}, $\Tilde{z}(t,x)$ is simplified into 
\begin{equation*}
\Tilde{z}(t,x)=\sum_{j\in\Z\backslash\{0\}}\left(-h_jv(t)+ \sum_{k\in\Z\backslash\{0\}}e^{\ii \frac{T}{2}\lambda_k}\frac{h_j\lambda_jz_k^0}{ h_k\lambda_k}\int_t^Te^{\ii(t-s)\lambda_j}\phi_k(s-\frac{T}{2})ds\right)\E_j(x).   
\end{equation*}
In addition, by $z(t,x)=\Tilde{z}(t,x)+v(t)h(x)$ and $h(x)=\sum_{j\in\Z\backslash\{0\}}h_j\E_j(x)$, we obtain the formal solution $z(t,x)$ to the KdV system \eqref{eq: KdV boundary derivative difference}.
\begin{equation}\label{eq: formal sol z}
z(t,x)=\sum_{j,k\in\Z\backslash\{0\}}\left(e^{\ii \frac{T}{2}\lambda_k}\frac{h_j\lambda_jz_k^0}{ h_k\lambda_k}\int_t^Te^{\ii(t-s)\lambda_j}\phi_k(t-\frac{T}{2})ds\right)\E_j(x)
\end{equation}
Moreover, it is easy to verify that $\p_x z(t,L)$ has the following expansion 
\begin{equation}
\p_xz(t,L)=\sum_{j,k\in\Z\backslash\{0\}}\left(e^{\ii \frac{T}{2}\lambda_k}e^{\ii t \lambda_j}\frac{h_j\lambda_jz_k^0}{h_k\lambda_k}\int_t^Te^{-\ii s\lambda_j}\phi_k(s-\frac{T}{2})ds\right)\E'_j(L).
\end{equation}
To derive uniform estimates in $L$ later, we define 
\begin{align}
v_b(t)&=\sum_{k\notin\Lambda_E}\frac{z_k^0}{\ii h_k\lambda_k}e^{\ii \frac{T}{2}\lambda_k}\phi_k(t-\frac{T}{2}),\;
v_s(t)=\sum_{k\in\Lambda_E}\frac{z_k^0}{\ii h_k\lambda_k}e^{\ii \frac{T}{2}\lambda_k}\phi_k(t-\frac{T}{2}),\\
z_b(t,x)&=\sum_{j\in\Z\backslash\{0\},k\notin\Lambda_E}\left(e^{\ii \frac{T}{2}\lambda_k}\frac{h_j\lambda_jz_k^0}{ h_k\lambda_k}\int_t^Te^{\ii(t-s)\lambda_j}\phi_k(t-\frac{T}{2})ds\right)\E_j(x),\label{eq: z-uniform}\\
z_s(t,x)&=\sum_{j\in\Z\backslash\{0\},k\in\Lambda_E}\left(e^{\ii \frac{T}{2}\lambda_k}\frac{h_j\lambda_jz_k^0}{ h_k\lambda_k}\int_t^Te^{\ii(t-s)\lambda_j}\phi_k(t-\frac{T}{2})ds\right)\E_j(x).\label{eq: z-singular}
\end{align}

\subsection{The intermediate system II: a priori estimates}\label{sec: estimates for the control}
In this section, we aim to give an appropriate bound of the $L^2-$norm of $\p_xz(t,L)$. Before introducing the main estimate, we first estimate the size of $h_j$ as $|j|$ tends to $\infty$.
\begin{lem}\label{lem: size of h_j}
$h_j=-\frac{\overline{\E'_j(L)}}{\ii\lambda_j}$. Moreover, 
$|h_j|\sim |j|^{-2}$, as $|j|\rightarrow+\infty$.
\end{lem}
\begin{proof}
In fact, $h_j=\int_0^Lh(x)\overline{\E_j(x)}dx$. Since $\E_j(x)$ is the eigenfunction of the operator $\B$ corresponding to the eigenvalue $\ii\lambda_j$, we know that $\ii\lambda_j\E_j=\B\E_j=-\E'''_j-\E'_j$. Hence, by integrating by parts, we obtain
\begin{align*}
h_j&=\int_0^Lh(x)\overline{\E_j(x)}dx
   =\frac{1}{\ii \lambda_j}\int_0^Lh(x)\overline{(\E'''_j(x)+\E'_j(x))}dx=-\frac{h'(L)\overline{\E'_j(L)}-h'(0)\overline{\E'_j(0)}}{\ii \lambda_j}.
\end{align*}
Other boundary terms vanish because $h(0)=h(L)=\E_j(0)=\E_j(L)=0$. Since $h'(L)-h'(0)=1$ and $\E_j(L)=\E_j(0)$, we know $h'(L)\overline{\E'_j(L)}-h'(0)\overline{\E'_j(0)}=\overline{\E'_j(L)}$. Therefore, we obtain $h_j=-\frac{\overline{\E'_j(L)}}{\ii\lambda_j}$. 
By Proposition \ref{prop: asymptotic eigenvalues}, we know that $|\E'_j(L)|\sim|j|$ and $|\lambda_j|\sim|j|^3$. Therefore, $|h_j|=\frac{|\E'_j(L)|}{|\lambda_j|}\sim|j|^{-2}$.
\end{proof}
In addition, we know that $h$ solves the equation \eqref{eq: defi of h}. We are able to obtain the exact expression of $h$ (See Lemma \ref{lem: formula h}).

\begin{lem}\label{lem: H^m-estimate of v}
Suppose that $z^0\in D(\B^2)$. Then the control function $v$ defined in \eqref{eq: defi of control v} satisfies $v(0)=v(T)=0$ and $v\in C^2(0,T)$. Moreover, we have the following estimates
\begin{align}\label{eq: L^infty of v}
\|v^{(m)}\|_{L^{\infty}(0,T)}\lesssim e^{\frac{K}{\sqrt{T}}}\|z^0\|_{H^{3m}(0,L)},m=0,1,2.
\end{align}
\end{lem}
\begin{proof}
By our assumption $z^0\in H^{6}(0,L)$ and $z^0(x)=\sum_{j\in\Z\backslash\{0\}}z_j^0\E_j(x)$, it is easy to verify that 
\begin{equation*}
    \sum_{j\in\Z\backslash\{0\}}|\lambda_j|^{2m}|z^0_j|^2\lesssim \|z^0\|^2_{H^{3m}(0,L)}.
\end{equation*}
By the definition \eqref{eq: defi of control v}, $v(t)=\sum_{k\in\Z\backslash\{0\}}\frac{z_k^0}{\ii h_k\lambda_k}e^{\ii \frac{T}{2}\lambda_k}\phi_k(t-\frac{T}{2})$. Therefore, we know that 
\begin{equation*}
\|v\|_{L^{\infty}(0,T)}\leq \sum_{j\in\Z\backslash\{0\}}\left|\frac{z_j^0}{ h_j\lambda_j}\right|\|\phi_j(\cdot-\frac{T}{2})\|_{L^{\infty}(0,T)}.
\end{equation*}
Applying Lemma \ref{lem: size of h_j}, we obtain $|h_j\lambda_j|=|\E'_j(L)|\sim |j|$. By Proposition \ref{prop: biorthogonal family}, $\|\phi_j\|_{L^{\infty}(-\frac{T}{2},\frac{T}{2})}\lesssim e^{\frac{K}{\sqrt{T}}}$. Combining these two estimates, we are able to show that
\begin{equation*}
\|v\|_{L^{\infty}(0,T)}\lesssim \sum_{j\in\Z\backslash\{0\}}\frac{|z_j^0|}{ |j|}e^{\frac{K}{\sqrt{T}}}\lesssim e^{\frac{K}{\sqrt{T}}}\left(\sum_{j\in\Z\backslash\{0\}}|z_j^0|^2\right)^{\frac{1}{2}} .   
\end{equation*}
In conclusion, we obtain $\|v\|_{L^{\infty}(0,T)}\lesssim e^{\frac{K}{\sqrt{T}}}\|z^0\|_{L^2(0,L)}$, 
where the implicit constant is independent of $T$. Moreover, for $m=1,2$ and for the $m-$th derivative of $v$, we have the following similar estimate:
\begin{equation*}
\|v^{(m)}\|_{L^{\infty}(0,T)}\leq \sum_{j\in\Z\backslash\{0\}}\left|\frac{z_j^0}{ h_j\lambda_j}\right|\|\phi^{(m)}_j(\cdot-\frac{T}{2})\|_{L^{\infty}(0,T)}\lesssim e^{\frac{K}{\sqrt{T}}}\|z^0\|_{H^{3m}(0,L)}.    
\end{equation*}
Here all implicit constants are independent of $T$. By Proposition \ref{prop: biorthogonal family}, $\phi_j \in C^{2}(0,T)$. Hence, we deduce that $v\in C^{2}(0,T)$. Thanks to the compact support of $\phi_j$, we know that $v(0)=v(T)=0$. 
\end{proof}
\begin{coro}
Let $I$ satisfy the condition {\bf (C)}. Then for every $L\in I\setminus\{L_0\}$, uniformly in $L$, 
\begin{align}
\|v_b^{(m)}\|_{L^{\infty}(0,T)}&\lesssim e^{\frac{K}{\sqrt{T}}}\|z^0\|_{H^{3m}(0,L)},m=0,1,2,\label{eq: L^infty of derivative v-uniform}\\
\|v_s^{(m)}\|_{L^{\infty}(0,T)}&\lesssim \frac{e^{\frac{K}{\sqrt{T}}}}{|L-L_0|}\|z^0\|_{H^{3m}(0,L)},m=0,1,2.\label{eq: L^infty of derivative v-signular}
\end{align}
\end{coro}
\begin{proof}
We first look at the uniform estimates for $v_b$. The proof of these two estimates is similar to what we did in the proof of Lemma \ref{lem: H^m-estimate of v}. We directly use that  
\begin{equation*}
\|v_b\|_{L^{\infty}(0,T)}\leq \sum_{j\notin\mathcal{M}_E}\left|\frac{z_j^0}{ h_j\lambda_j}\right|\|\phi_j(\cdot-\frac{T}{2})\|_{L^{\infty}(0,T)}.
\end{equation*}
Applying Lemma \ref{lem: size of h_j}, we obtain $|h_j\lambda_j|=|\E'_j(L)|$. By Proposition \ref{prop: High-frequency behaviors: uniform estimates}, we know that for $|j|>J=\max N_L+J_0(L_0)$, $|\E'_{j}(0)|= |\E'_{j}(L)|\geq \gamma|j|$. 
By Proposition \ref{prop: Low-frequency behaviors: uniform estimates}, since $j\notin\Lambda_E$, we obtain that $|h_j\lambda_j|=|\E'_j(L)|>\gamma$. Hence,
\begin{equation*}
\|v_b\|_{L^{\infty}(0,T)}\leq \sum_{|j|\leq J,j\notin\Lambda_E}\left|\frac{z_j^0}{ \gamma}\right|\|\phi_j(\cdot-\frac{T}{2})\|_{L^{\infty}(0,T)}+\sum_{|j|\geq J}\left|\frac{z_j^0}{ \gamma|j|}\right|\|\phi_j(\cdot-\frac{T}{2})\|_{L^{\infty}(0,T)}.
\end{equation*}
Using Corollary \ref{coro: uniform est for bi-family}, we know that for $j\notin\Lambda_E$, uniformly in $L$, $\|\phi^{(m)}_j\|_{L^{\infty}(\R)}\leq C e^{\frac{K}{\sqrt{T}}}|\lambda_j|^m$. 
Therefore, we obtain uniform estimates for $v_b$ as follows
\begin{equation*}
    \|v_b\|_{L^{\infty}(0,T)}\lesssim e^{\frac{K}{\sqrt{T}}}\|z^0\|_{L^2(0,L)}.
\end{equation*}
The same arguments hold for uniform estimates of $m-$order derivatives. We shall not repeat it. For the estimates of $v_s$, we have two situations as usual. 
\begin{enumerate}
    \item If $k\not\equiv l\mod 3$, by Proposition \ref{prop: Low-frequency behaviors: Singular limits}, we know that $|\E_j'(L)|=\bigO(L-L_0)$ for $j\in\Lambda_E$. Thus, for $j\in \Lambda_E$, applying Lemma \ref{lem: size of h_j}, we obtain $|h_j\lambda_j|=|\E'_j(L)|=\bigO(L-L_0)$. In this situation, by Corollary \ref{coro: uniform est for bi-family}, $\|\phi^{(m)}_j\|_{L^{\infty}(\R)}\leq C e^{\frac{K}{\sqrt{T}}}|\lambda_j|^m$ holds uniformly in $L$. Thus, we obtain 
    \begin{equation*}
     \|v_s\|_{L^{\infty}(0,T)}\leq \sum_{j\in\Lambda_E}\left|\frac{z_j^0}{ |L-L_0|}\right|\|\phi_j(\cdot-\frac{T}{2})\|_{L^{\infty}(0,T)}\lesssim \frac{e^{\frac{K}{\sqrt{T}}}}{ |L-L_0|}\|z_0\|_{L^2(0,L)} .
    \end{equation*}
    \item If $k\equiv l\mod 3$, in this situation, by Corollary \ref{coro: uniform est for bi-family}, we know that for $j\in\Lambda_E$, 
    \begin{equation*}
    \|\phi^{(m)}_{\sigma^{\pm}(q)}\|_{L^{\infty}(\R)}\leq \frac{C}{|\lambda_{\sigma^+(q)}-\lambda_{\sigma^-(q)}|}e^{\frac{K}{\sqrt{T}}}|\lambda_j|^m,\quad |q|\leq N_0.   
    \end{equation*}
    Using Proposition \ref{prop: Index set M-E}, we know that $|\lambda_{\sigma^+(q)}-\lambda_{\sigma^-(q)}|=\bigO(L-L_0)$. By Proposition \ref{prop: Low-frequency behaviors: Singular limits}, we know that $|\E_j'(L)|>\gamma>0$ for $j\in\Lambda_E$. Therefore, we conclude that
    \begin{equation*}
     \|v_s\|_{L^{\infty}(0,T)}\leq \sum_{j\in\Lambda_E}\left|\frac{z_j^0}{ \gamma}\right|\|\phi_j(\cdot-\frac{T}{2})\|_{L^{\infty}(0,T)}\lesssim \frac{e^{\frac{K}{\sqrt{T}}}}{ |L-L_0|}\|z_0\|_{L^2(0,L)} .
    \end{equation*}
\end{enumerate}
The same arguments hold for uniform estimates of $m-$order derivatives.
\end{proof}
Now we turn to the estimates of the solution $z$.
\begin{lem}\label{lem: H^m-estimate tilde_z}
Suppose that $z^0\in D(\B^{2})$. There exists a unique solution $\Tilde{z}\in C([0,T];L^2(0,L))$ to the equation \eqref{eq: KdV source control}. Moreover, $\Tilde{z}$ has the following expansion:
\begin{equation}\label{eq: expansion of tilde-z}
\Tilde{z}(t,x)=\sum_{j\in\Z\backslash\{0\}}\left(-h_jv(t)+ \sum_{k\in\Z\backslash\{0\}}e^{\ii \frac{T}{2}\lambda_k}\frac{h_j\lambda_jz_k^0}{ h_k\lambda_k}\int_t^Te^{\ii(t-s)\lambda_j}\phi_k(s-\frac{T}{2})ds\right)\E_j(x).    
\end{equation}
In addition, $\p_x \Tilde{z}(t,L)\in L^2(0,T)$ and for fixed $L\notin\mathcal{N}$
\begin{equation}\label{eq: H^m-estimate of derivative of tilde-z}
\|\p_x\Tilde{z}(\cdot,L)\|_{L^2(0,T)}\lesssim e^{\frac{2K}{\sqrt{T}}}\|z^0\|_{H^{6}}.  
\end{equation}
\end{lem}
\begin{proof}
By the assumption $z^0\in D(\B^{2})$ and Lemma \ref{lem: H^m-estimate of v}, we know that $v\in C^{2}(0,T)$. Combining with Lemma \ref{lem: formula h}, we deduce that $v'\otimes h\in L^1((0,T);L^2(0,L))$. Thus, there exists a unique solution $\Tilde{z}\in C([0,T];L^2(0,L))$ to the equation \eqref{eq: KdV source control}. Therefore, we obtain the expansion in the basis $\{\E_j\}_{j\in\Z\backslash\{0\}}$ of $L^2(0,L)$:
\begin{align*}
\Tilde{z}(t,x)&=\sum_{j\in\Z\backslash\{0\}}\left(-h_jv(t)+ \sum_{k\in\Z\backslash\{0\}}e^{\ii \frac{T}{2}\lambda_k}\frac{h_j\lambda_jz_k^0}{ h_k\lambda_k}\int_t^Te^{\ii(t-s)\lambda_j}\phi_k(s-\frac{T}{2})ds\right)\E_j(x).   
\end{align*} 
Integrating by parts, we also obtain  $\Tilde{z}(t,x)=\sum_{j,k\in\Z\backslash\{0\}}e^{\ii \frac{T}{2}\lambda_k}\frac{h_jz_k^0}{\ii h_k\lambda_k}\int_t^Te^{\ii(t-s)\lambda_j}\phi'_k(s-\frac{T}{2})ds\E_j(x)$.
Therefore, taking derivative and integrating by parts, we obtain 
 \begin{equation*}
\p_x\Tilde{z}(t,x)
=\sum_{j,k\in\Z\backslash\{0\}}e^{\ii \frac{T}{2}\lambda_k}\left(-\frac{h_jz_k^0}{ h_k\lambda_k\lambda_j}\phi'_k(t-\frac{T}{2})-\frac{h_jz_k^0}{ h_k\lambda_k\lambda_j}\int_t^Te^{\ii(t-s)\lambda_j}\phi''_k(s-\frac{T}{2})ds\right)\E'_j(x).
\end{equation*}
We write $\p_x\Tilde{z}(t,x)=I_1+I_2$, with
\begin{align*}
I_1=-\sum_{j,k\in\Z\backslash\{0\}}e^{\ii \frac{T}{2}\lambda_k}\frac{h_jz_k^0}{ h_k\lambda_k\lambda_j}\phi'_k(t-\frac{T}{2})\E'_j(x), \\
I_2=-\sum_{j,k\in\Z\backslash\{0\}}e^{\ii \frac{T}{2}\lambda_k}\frac{h_jz_k^0}{ h_k\lambda_k\lambda_j}\int_t^Te^{\ii(t-s)\lambda_j}\phi''_k(s-\frac{T}{2})ds\E'_j(x).
\end{align*}
We estimate them respectively. For $I_1$, 
\begin{align*}
|I_1|\leq \sum_{j,k\in\Z\backslash\{0\}}\left|e^{\ii \frac{T}{2}\lambda_k}\frac{h_jz_k^0}{ h_k\lambda_k\lambda_j}\phi'_k(t-\frac{T}{2})\right||\E'_j(x)|
     \lesssim  \sum_{j,k\in\Z\backslash\{0\}}\frac{|j|^{-2}|z_k^0|}{ |k||j|^3}|\phi'_k(t-\frac{T}{2})||\E'_j(x)|.
\end{align*}
By Proposition \ref{prop: asymptotic eigenvalues} and Proposition \ref{prop: biorthogonal family}, we know that $|\E'_j(x)|\lesssim |j|$ and $|\phi'_k(t-\frac{T}{2})|\lesssim e^{\frac{K}{\sqrt{T}}}|\lambda_k|$. Hence,
\begin{equation}\label{eq: estimate I_1}
|I_1|     \lesssim e^{\frac{K}{\sqrt{T}}}\sum_{j,k\in\Z\backslash\{0\}}\frac{|\lambda_k||z_k^0|}{ |j|^4}\cdot\frac{1}{|k|}
     \lesssim e^{\frac{K}{\sqrt{T}}}\left(\sum_{k\in\Z\backslash\{0\}}|\lambda_k|^2|z_k^0|^2\right)^{\frac{1}{2}}\lesssim e^{\frac{K}{\sqrt{T}}}\|z^0\|_{H^3(0,L)}
\end{equation}
For $I_2$, we have
\begin{align*}
|I_2|&\lesssim\sum_{j,k\in\Z\backslash\{0\}}\left|e^{\ii \frac{T}{2}\lambda_k}\frac{h_jz_k^0}{ h_k\lambda_k\lambda_j}\int_t^Te^{\ii(t-s)\lambda_j}\phi''_k(s-\frac{T}{2})ds\right||\E'_j(x)|\\
     &\lesssim \sum_{j,k\in\Z\backslash\{0\}}\frac{|j|^{-2}|z_k^0|}{ |k||j|^3}\int_t^T|\phi''_k(s-\frac{T}{2})|ds|\E'_j(x)|\\
     &\lesssim \sum_{j,k\in\Z\backslash\{0\}}\frac{|z_k^0|}{ |k||j|^5}\int_t^T|\phi''_k(s-\frac{T}{2})|ds|\E'_j(x)|.
\end{align*}
Again applying Proposition \ref{prop: asymptotic eigenvalues} and Proposition \ref{prop: biorthogonal family}, here we use that $|\phi''_k(t-\frac{T}{2})|\lesssim e^{\frac{K}{\sqrt{T}}}|\lambda_k|^2$. Therefore, we deduce that
\begin{equation}\label{eq: estimate I_2}
|I_2|     \lesssim Te^{\frac{K}{\sqrt{T}}}\sum_{j,k\in\Z\backslash\{0\}}\frac{|\lambda_k|^2|z_k^0|}{ |j|^4}\cdot\frac{1}{|k|}
     \lesssim Te^{\frac{K}{\sqrt{T}}}\left(\sum_{k\in\Z\backslash\{0\}}|\lambda_k|^4|z_k^0|^2\right)^{\frac{1}{2}}\lesssim Te^{\frac{K}{\sqrt{T}}}\|z^0\|_{H^6(0,L)}.
\end{equation}
Combining the two estimates \eqref{eq: estimate I_1} and \eqref{eq: estimate I_2} above, we know that  
\begin{equation*}
\|\p_x\Tilde{z}\|_{L^{\infty}((0,T)\times(0,L))}\lesssim e^{\frac{2K}{\sqrt{T}}}\|z^0\|_{H^6(0,L)}.
\end{equation*}
In particular, we have $\|\p_x\Tilde{z}(\cdot,L)\|_{L^{\infty}(0,T)}\lesssim e^{\frac{2K}{\sqrt{T}}}\|z^0\|_{H^6(0,L)}$.
\end{proof}

By our construction, $z(t,x)=\Tilde{z}(t,x)+v(t)h(x)$. Then we have the following lemma.
\begin{lem}
Suppose that $z^0\in D(\B^{2})$. There exists a unique solution $z\in C([0,T];L^2(0,L))$ to the equation \eqref{eq: KdV boundary derivative difference}. Moreover, $z$ has the following expansion:
\begin{equation}\label{eq: expansion of z}
z(t,x)=\sum_{j\in\Z\backslash\{0\}}\sum_{k\in\Z\backslash\{0\}}e^{\ii \frac{T}{2}\lambda_k}\frac{h_j\lambda_jz_k^0}{ h_k\lambda_k}\int_t^Te^{\ii(t-s)\lambda_j}\phi_k(s-\frac{T}{2})ds\E_j(x).    
\end{equation}
In addition, $\p_x z(t,L)\in L^2(0,T)$ and for fixed $L\notin\mathcal{N}$
\begin{equation}\label{eq: H^m-estimate of derivative of z}
\|\p_x z(\cdot,L)\|_{L^2(0,T)}\lesssim e^{\frac{2K}{\sqrt{T}}}\|z^0\|_{H^{6}}.  
\end{equation}
\end{lem}
\begin{proof}
The proof is just a combination of the estimates for $v$ and the estimates for $\Tilde{z}$. By Lemma \ref{lem: formula h}, Lemma \ref{lem: H^m-estimate of v}, and Lemma \ref{lem: H^m-estimate tilde_z}, it is easy to deduce that $z\in C([0,T];L^2(0,L))$. The expansion \eqref{eq: expansion of z} is a direct result thanks to the expansion \eqref{eq: expansion of tilde-z} of $\Tilde{z}$ and $z(t,x)=\Tilde{z}(t,x)+v(t)h(x)$. The estimate \eqref{eq: H^m-estimate of derivative of z} follows the estimate \eqref{eq: H^m-estimate of derivative of tilde-z}.  
\end{proof}
\begin{coro}\label{coro: uniform and singular estimates for derivative-z}
Let $I$ satisfy the condition {\bf (C)}. Then for every $L\in I\setminus\{L_0\}$, uniformly in $L$, 
\begin{align}
\|\p_xz_b(\cdot,L)\|_{L^{\infty}(0,T)}&\lesssim e^{\frac{2K}{\sqrt{T}}}\|z^0\|_{H^{6}} ,\label{eq: H^m of z'(L)-uniform}\\
\|\p_xz_s(\cdot,L)\|_{L^{\infty}(0,T)}&\lesssim \frac{e^{\frac{K}{\sqrt{T}}}}{|L-L_0|}\|z^0\|_{H^{6}(0,L)}.\label{eq: H^m of z'(L)-signular}
\end{align}
\end{coro}
\begin{proof}
 As we defined in \eqref{eq: z-uniform}, we know that
\begin{equation}\label{eq: z-uniform-derivative}
\p_x z_b(t,L)=\sum_{j\in\Z\backslash\{0\},k\notin\Lambda_E}\left(e^{\ii \frac{T}{2}\lambda_k}\frac{h_j\lambda_jz_k^0}{ h_k\lambda_k}\int_t^Te^{\ii(t-s)\lambda_j}\phi_k(t-\frac{T}{2})ds\right)\E'_j(L).    
\end{equation}
If $k\neq l$, there exists a constant $\epsilon_0$ such that $|\lambda_j|>\epsilon_0$, for any $j\in\Z\setminus\{0\}$. After integrating by parts, we write $\p_x z_b(t,L)=\Tilde{I}_1+\Tilde{I}_2$, with 
\begin{align*}
\Tilde{I}_1=-\sum_{j\in\Z\backslash\{0\},k\notin\Lambda_E}e^{\ii \frac{T}{2}\lambda_k}\frac{h_jz_k^0}{ h_k\lambda_k\lambda_j}\phi'_k(t-\frac{T}{2})\E'_j(L),\\
\Tilde{I}_2=-\sum_{j\in\Z\backslash\{0\},k\notin\Lambda_E}e^{\ii \frac{T}{2}\lambda_k}\frac{h_jz_k^0}{ h_k\lambda_k\lambda_j}\int_t^Te^{\ii(t-s)\lambda_j}\phi''_k(s-\frac{T}{2})ds\E'_j(L).   
\end{align*}
For the term $\Tilde{I}_1$, we know that 
\begin{align*}
|\Tilde{I}_1|&\leq \sum_{j\in\Z\backslash\{0\},k\notin\mathcal{M}_E}|\frac{h_jz_k^0}{ h_k\lambda_k\lambda_j}|\|\phi'_k(\cdot-\frac{T}{2})\|_{L^{\infty}(\R)}|\E'_j(L)| \\
&\leq \sum_{j\in\Z\backslash\{0\},k\notin\mathcal{M}_E}\frac{|z_k^0|}{ |\E'_k(L)||\lambda_j|^2}\|\phi'_k(\cdot-\frac{T}{2})\|_{L^{\infty}(\R)}|\E'_j(L)|^2.
\end{align*}
Since $k\notin\Lambda_E$, using Corollary \ref{coro: uniform est for bi-family}, we know that for $j\notin\Lambda_E$, uniformly in $L$, $\|\phi^{(m)}_k\|_{L^{\infty}(\R)}\leq C e^{\frac{K}{\sqrt{T}}}|\lambda_k|^m$. 
Moreover, by Proposition \ref{prop: High-frequency behaviors: uniform estimates}, we know that for $|k|>J=\max N_L+J_0(L_0)$, $|\E'_{k}(0)|= |\E'_{k}(L)|\geq \gamma|k|$. 
By Proposition \ref{prop: Low-frequency behaviors: uniform estimates}, since $k\notin\Lambda_E$, we obtain that $|\E'_k(L)|>\gamma$. Hence,
\begin{align*}
|\Tilde{I}_1|
&\lesssim e^{\frac{K}{\sqrt{T}}}\left(\sum_{j\in\Z\backslash\{0\},k\notin\Lambda_E}\frac{|\lambda_k z_k^0|}{ |\gamma||\lambda_j|^2}|\E'_j(L)|^2+ \sum_{j\in\Z\backslash\{0\},k>J}\frac{|\lambda_kz_k^0|}{\gamma |k||\lambda_j|^2}|\E'_j(L)|^2\right)\\
&\lesssim e^{\frac{K}{\sqrt{T}}}\|z^0\|_{H^3(0,L)}\sum_{j\in\Z\backslash\{0\}}\frac{1}{\gamma |\lambda_j|^2}|\E'_j(L)|^2.
\end{align*}
By Proposition \ref{prop: Asymp in L}, we know that $\lambda_{j}=(\frac{2j\pi}{L})^3+O(j^2)$ uniformly in $L$ and $|\E'_j(L)|\lesssim |j|$. Thus, we obtain uniform estimates for 
\begin{equation*}
    |\Tilde{I}_1|\lesssim e^{\frac{K}{\sqrt{T}}}\|z^0\|_{H^3(0,L)}.
\end{equation*}
For the term $\Tilde{I}_2$, the procedure is similar and we obtain
\begin{align*}
|\Tilde{I}_2|
 \lesssim  e^{\frac{K}{\sqrt{T}}}\left(\sum_{j\in\Z\backslash\{0\},k\notin\Lambda_E}\frac{|\lambda_k|^2| z_k^0|}{ |\gamma||\lambda_j|^2}|\E'_j(L)|^2+ \sum_{j\in\Z\backslash\{0\},k>J}\frac{|\lambda_k|^2|z_k^0|}{\gamma |k||\lambda_j|^2}|\E'_j(L)|^2\right).
\end{align*}
Using again Proposition \ref{prop: High-frequency behaviors: uniform estimates}, Proposition \ref{prop: Low-frequency behaviors: uniform estimates} and Proposition \ref{prop: Asymp in L}, we obtain $\Tilde{I}_2\lesssim e^{\frac{K}{\sqrt{T}}}\|z^0\|_{H^6(0,L)}$. 
If $k=l$, let $j_0$ denote the index of the eigenvalue that satisfies $\lambda_{j_0}=\bigO(L-L_0)$. Thus, $\lambda_{-j_0}=\bigO(L-L_0)$. In the expression \eqref{eq: z-uniform-derivative} of $\p_xz_b(t,L)$, we look at the term 
\begin{align*}
&\;\;\; \; |\sum_{k\notin\Lambda_E}\left(e^{\ii \frac{T}{2}\lambda_k}\frac{h_{j_0}\lambda_{j_0}z_k^0}{ h_k\lambda_k}\int_t^Te^{\ii(t-s)\lambda_{j_0}}\phi_k(t-\frac{T}{2})ds\right)\E'_{j_0}(L)| \\ 
&\leq T\sum_{k\notin\Lambda_E}\frac{|h_{j_0}\lambda_{j_0}||z_k^0|}{| h_k\lambda_k|}\|\phi_k(\cdot-\frac{T}{2})\|_{L^{\infty}(\R)}|\E'_{j_0}(L)|.
\end{align*}
Since $k\notin\Lambda_E$, using Corollary \ref{coro: uniform est for bi-family}, we know that for $k\notin\Lambda_E$, uniformly in $L$, $\|\phi^{(m)}_k\|_{L^{\infty}(\R)}\leq C e^{\frac{K}{\sqrt{T}}}|\lambda_k|^m$. 
Moreover, by Proposition \ref{prop: High-frequency behaviors: uniform estimates} and Proposition \ref{prop: Low-frequency behaviors: uniform estimates}, we know
\begin{align*}
&\;\;\; \; |\sum_{k\notin\Lambda_E}\left(e^{\ii \frac{T}{2}\lambda_k}\frac{h_{j_0}\lambda_{j_0}z_k^0}{ h_k\lambda_k}\int_t^Te^{\ii(t-s)\lambda_{j_0}}\phi_k(t-\frac{T}{2})ds\right)\E'_{j_0}(L)|\\
&\lesssim Te^{\frac{K}{\sqrt{T}}}\sum_{k\notin\Lambda_E}\frac{|z_k^0|}{| \gamma|}|\E'_{j_0}(L)|^2+\sum_{k>J}\frac{|z_k^0|}{| \gamma k|}|\E'_{j_0}(L)|^2\\
&\lesssim e^{\frac{K}{\sqrt{T}}}\|z_0\|_{L^2(0,L)}.
\end{align*}
We could deal with all other terms using integration by parts to derive uniform estimates. Consequently, we conclude that 
\begin{equation*}
\|\p_xz_b(\cdot,L)\|_{L^{\infty}(0,T)}\lesssim e^{\frac{2K}{\sqrt{T}}}\|z^0\|_{H^{6}}
\end{equation*}
Next we look at $\p_xz_s(t,L)$,
\begin{equation*}
\p_x z_s(t,L)=\sum_{j\in\Z\backslash\{0\},k\in\Lambda_E}\left(e^{\ii \frac{T}{2}\lambda_k}\frac{h_j\lambda_jz_k^0}{ h_k\lambda_k}\int_t^Te^{\ii(t-s)\lambda_j}\phi_k(t-\frac{T}{2})ds\right)\E'_j(L).    
\end{equation*}
After integrating by parts, we write $\p_x z_s(t,L)=\Tilde{J}_1+\Tilde{J}_2$, with 
\begin{align*}
\Tilde{J}_1&=-\sum_{j\in\Z\backslash\{0\},k\in\Lambda_E}e^{\ii \frac{T}{2}\lambda_k}\frac{h_jz_k^0}{ h_k\lambda_k\lambda_j}\phi'_k(t-\frac{T}{2})\E'_j(L),\\
\Tilde{J}_2&=-\sum_{j\in\Z\backslash\{0\},k\in\Lambda_E}e^{\ii \frac{T}{2}\lambda_k}\frac{h_jz_k^0}{ h_k\lambda_k\lambda_j}\int_t^Te^{\ii(t-s)\lambda_j}\phi''_k(s-\frac{T}{2})ds\E'_j(L).   
\end{align*}
We prove the estimates in two situations.
\begin{enumerate}
    \item If $k\not\equiv l\mod 3$, by Proposition \ref{prop: Low-frequency behaviors: Singular limits}, we know that $|\E_k'(L)|=\bigO(L-L_0)$ for $k\in\Lambda_E$. Thus, for $k\in \Lambda_E$, we obtain $|h_k\lambda_k|=|\E'_k(L)|=\bigO(L-L_0)$. In this situation, by Corollary \ref{coro: uniform est for bi-family}, $\|\phi^{(m)}_k\|_{L^{\infty}(\R)}\leq C e^{\frac{K}{\sqrt{T}}}|\lambda_k|^m$ holds uniformly in $L$. Thus, 
    \begin{equation*}
        |\Tilde{J}_1|\lesssim \frac{e^{\frac{K}{\sqrt{T}}}}{|L-L_0|}\|z^0\|_{H^{3}(0,L)},\; |\Tilde{J}_2|\lesssim \frac{e^{\frac{K}{\sqrt{T}}}}{|L-L_0|}\|z^0\|_{H^{6}(0,L)}.
    \end{equation*}
\item If $k\equiv l\mod 3$ and $k\neq l$, in this situation, by Corollary \ref{coro: uniform est for bi-family}, we know that for $j\in\Lambda_E$, 
    \begin{equation*}
    \|\phi^{(m)}_{\sigma^{\pm}(q)}\|_{L^{\infty}(\R)}\leq \frac{C}{|\lambda_{\sigma^+(q)}-\lambda_{\sigma^-(q)}|}e^{\frac{K}{\sqrt{T}}}|\lambda_j|^m,\quad |q|\leq N_0.   
    \end{equation*}
    Using Proposition \ref{prop: Index set M-E}, we know that $|\lambda_{\sigma^+(q)}-\lambda_{\sigma^-(q)}|=\bigO(L-L_0)$. By Proposition \ref{prop: Low-frequency behaviors: Singular limits}, we know that $|\E_j'(L)|>\gamma>0$ for $j\in\Lambda_E$. Therefore,
    \begin{align*}
     |\Tilde{J}_1|
     &\lesssim e^{\frac{K}{\sqrt{T}}}\sum_{j\in\Z\backslash\{0\},k\in\Lambda_E}\frac{|\lambda_kz_k^0|}{\gamma |L-L_0||\lambda_j|^2}|\E'_j(L)|^2
     \lesssim\frac{e^{\frac{K}{\sqrt{T}}}}{|L-L_0|}\|z^0\|_{H^{3}(0,L)}.
    \end{align*}
Similarly, we also have $|\Tilde{J}_2|\lesssim \frac{e^{\frac{K}{\sqrt{T}}}}{|L-L_0|}\|z^0\|_{H^{6}(0,L)}$.
\item If $k=l$, we only need to analyze the term $\sum_{k\in\Lambda_E}\left(e^{\ii \frac{T}{2}\lambda_k}\frac{h_{j_0}\lambda_{j_0}z_k^0}{ h_k\lambda_k}\int_t^Te^{\ii(t-s)\lambda_{j_0}}\phi_k(s-\frac{T}{2})ds\right)\E'_{j_0}(L)$. It is easy to verify that
\begin{align*}
& \;\;\;\; |\sum_{k\in\Lambda_E}\left(e^{\ii \frac{T}{2}\lambda_k}\frac{h_{j_0}\lambda_{j_0}z_k^0}{ h_k\lambda_k}\int_t^Te^{\ii(t-s)\lambda_{j_0}}\phi_k(s-\frac{T}{2})ds\right)\E'_{j_0}(L)| \\
&\leq T\sum_{k\in\mathcal{M}_E}\frac{|z_k^0|}{ |\E_k'(L)|}\|\phi_k(\cdot-\frac{T}{2})ds\|_{L^{\infty}(\R)}|\E'_{j_0}(L)|^2\\
&\lesssim T e^{\frac{K}{\sqrt{T}}}\sum_{k\in\Lambda_E}\frac{|z_k^0|}{ \gamma |L-L_0|}|\E'_{j_0}(L)|^2\\
&\lesssim\frac{e^{\frac{K}{\sqrt{T}}}}{|L-L_0|}\|z^0\|_{L^2(0,L)}.
\end{align*}
\end{enumerate}
As a consequence, we conclude that $\|\p_xz_s(\cdot,L)\|_{L^{\infty}(0,T)}\lesssim \frac{e^{\frac{K}{\sqrt{T}}}}{|L-L_0|}\|z^0\|_{H^{6}(0,L)}$.
\end{proof}
Armed with all the estimates in this section, we could prove the following lemma.
\begin{lem}\label{lem: control-toy}
Let $T>0$.Let $I$ satisfy the condition {\bf (C)} and $z^0\in D(\B^2)$. There exists a unique solution $z\in C([0,T];L^2(0,L))$ to the equation \eqref{eq: KdV boundary derivative difference} 
such that $z(T,x)\equiv0$. Then for every $L\in I\setminus\{L_0\}$, there are two constants $K_1$ and $K_2$, independent of $L$, such that 
$\p_x z(\cdot,L)\in L^2(0,T)$ and 
\begin{equation}\label{eq: L^2-estimate of derivative of z-toy}
\|\p_x z(\cdot,L)\|_{L^{\infty}(0,T)}\leq \frac{K_1}{|L-L_0|}e^{\frac{2K}{\sqrt{T}}}\|z^0\|_{H^{6}(0,L)}.
\end{equation}
and for any $t\in(0,T]$, we have the following estimate
\begin{equation}\label{eq: L^2-estimate of z(t)-toy}
\|z(t,\cdot)\|_{L^2(0,L)}\leq \frac{K_2}{|L-L_0|} e^{\frac{2K}{\sqrt{T}}}\|z^0\|_{H^{3}(0,L)}
\end{equation}
\end{lem}
\begin{proof}
Recall that $z(t,x)=v(t)h(x)+\sum_{j,k\in\Z\backslash\{0\}}e^{\ii \frac{T}{2}\lambda_k}\frac{h_jz_k^0}{\ii h_k\lambda_k}\int_t^Te^{\ii(t-s)\lambda_j}\phi'_k(s-\frac{T}{2})ds\E_j(x)$.
Hence, for any $t\in (0,T)$, 
\begin{align*}
\|z(t,\cdot)\|_{L^2(0,L)}&\leq |v(t)|\|h\|_{L^2(0,L)}+\sum_{j,k\in\Z\backslash\{0\}}|e^{\ii \frac{T}{2}\lambda_k}\frac{h_jz_k^0}{\ii h_k\lambda_k}\int_t^Te^{\ii(t-s)\lambda_j}\phi'_k(s-\frac{T}{2})ds|\\
&\leq \|v\|_{L^{\infty}(0,T)}\|h\|_{L^2(0,L)}+T\sum_{j,k\in\Z\backslash\{0\}}\frac{|h_j||z_k^0|}{|\E_k'(L)|}\|\phi'_k\|_{L^{\infty}(\R)}.
\end{align*}
We write that $I_1+I_2=\sum_{j,k\in\Z\backslash\{0\}}\frac{|h_j||z_k^0|}{|\E_k'(L)|}\|\phi'_k\|_{L^{\infty}(\R)}$, with
\begin{equation*}
    I_1=\sum_{j\in\Z\backslash\{0\}}\sum_{k\notin\Lambda_E}\frac{|h_j||z_k^0|}{|\E_k'(L)|}\|\phi'_k\|_{L^{\infty}(\R)},\;
I_2=\sum_{j\in\Z\backslash\{0\}}\sum_{k\in\Lambda_E}\frac{|h_j||z_k^0|}{|\E_k'(L)|}\|\phi'_k\|_{L^{\infty}(\R)}.
\end{equation*}
By Proposition \ref{prop: High-frequency behaviors: uniform estimates} and Proposition \ref{prop: Low-frequency behaviors: uniform estimates}, we know
\begin{align*}
I_1&\lesssim e^{\frac{K}{\sqrt{T}}}\sum_{j\in\Z\backslash\{0\},k\notin\Lambda_E}\frac{|h_j||z_k^0|}{\gamma}|\lambda_k|+\sum_{j\in\Z\backslash\{0\},k>J}\frac{|\E'_j(L)||z_k^0|}{\gamma|k|}|\lambda_k|
\lesssim e^{\frac{K}{\sqrt{T}}}\|z^0\|_{H^3(0,L)}\sum_{j}\frac{|\E_j'(L)|}{|\lambda_j|}.\\
I_2&\lesssim \frac{e^{\frac{K}{\sqrt{T}}}}{|L-L_0|}\|z^0\|_{H^3(0,L)}\sum_{j}\frac{|\E_j'(L)|}{|\lambda_j|}
\end{align*}
If $L_0\neq 2l\pi$, then $\sum_{j}\frac{|\E_j'(L)|}{|\lambda_j|}<\infty$ is uniformly bounded. Thus, we obtain $\|z(t,\cdot)\|_{L^2(0,L)}\leq \frac{K_2}{|L-L_0|}e^{\frac{2K}{\sqrt{T}}}\|z^0\|_{H^3(0,L)}$. 
If $L_0=2l\pi$, let $j_0$ denote the index of the eigenvalue that satisfies $\lambda_{j_0}=\bigO(L-L_0)$. Thus, $\lambda_{-j_0}=\bigO(L-L_0)$. We only need to analyze the term
$$\sum_{k\in\Z\backslash\{0\}}e^{\ii \frac{T}{2}\lambda_k}\frac{h_{j_0}z_k^0}{\ii h_k\lambda_k}\int_t^Te^{\ii(t-s)\lambda_{j_0}}\phi'_k(s-\frac{T}{2})ds\E_{j_0}(x).$$ Integrating by parts, 
\begin{align*}
&\;\;\;\; \sum_{k\in\Z\backslash\{0\}}e^{\ii \frac{T}{2}\lambda_k}\frac{h_{j_0}z_k^0}{\ii h_k\lambda_k}\int_t^Te^{\ii(t-s)\lambda_{j_0}}\phi'_k(s-\frac{T}{2})ds\E_{j_0}(x) \\
&=\sum_{k\in\Z\backslash\{0\}}e^{\ii \frac{T}{2}\lambda_k}\frac{\lambda_{j_0}h_{j_0}z_k^0}{ h_k\lambda_k}\int_t^Te^{\ii(t-s)\lambda_{j_0}}\phi_k(s-\frac{T}{2})ds\E_{j_0}(x)\\
&-\sum_{k\in\Z\backslash\{0\}}e^{\ii \frac{T}{2}\lambda_k}\frac{\lambda_{j_0}h_{j_0}z_k^0}{ h_k\lambda_k}\phi_k(t-\frac{T}{2})\E_{j_0}(x)
\end{align*}
For the term $\sum_{k\in\Z\backslash\{0\}}e^{\ii \frac{T}{2}\lambda_k}\frac{\lambda_{j_0}h_{j_0}z_k^0}{ h_k\lambda_k}\phi_k(t-\frac{T}{2})\E_{j_0}(x)$, we perform similar estimates, 
\begin{align*}
\|\sum_{k\in\Z\backslash\{0\}}e^{\ii \frac{T}{2}\lambda_k}\frac{\lambda_{j_0}h_{j_0}z_k^0}{ h_k\lambda_k}\phi_k(t-\frac{T}{2})\E_{j_0}(\cdot)\|_{L^2(0,L)}&\leq \sum_{k\in\Z\backslash\{0\}}|\frac{|\E_{j_0}'(L)||z_k^0|}{ |\E'_k(L)|}\|\phi_k(\cdot)\|_{L^{\infty}(\R)}\lesssim \frac{K_2}{|L-L_0|}e^{\frac{2K}{\sqrt{T}}}\|z^0\|_{H^3(0,L)}.
\end{align*}
Moreover, for the term $\sum_{k\in\Z\backslash\{0\}}e^{\ii \frac{T}{2}\lambda_k}\frac{\lambda_{j_0}h_{j_0}z_k^0}{ h_k\lambda_k}\int_t^Te^{\ii(t-s)\lambda_{j_0}}\phi_k(s-\frac{T}{2})ds\E_{j_0}(x)$, we obtain 
\begin{align*}
\|\sum_{k\in\Z\backslash\{0\}}e^{\ii \frac{T}{2}\lambda_k}\frac{\lambda_{j_0}h_{j_0}z_k^0}{ h_k\lambda_k}\int_t^Te^{\ii(t-s)\lambda_{j_0}}\phi_k(s-\frac{T}{2})ds\E_{j_0}(\cdot)\|_{L^2(0,L)}&\leq T\sum_{k\in\Z\backslash\{0\}}\frac{|\E_{j_0}'(L)||z_k^0|}{ |\E'_k(L)|}\|\phi_k(\cdot-\frac{T}{2})\|_{L^{\infty}(\R)}\\
&\lesssim \frac{K_2}{|L-L_0|}e^{\frac{2K}{\sqrt{T}}}\|z^0\|_{H^3(0,L)}.
\end{align*}
In summary, $\|\p_x z(\cdot,L)\|_{L^{\infty}(0,T)}\leq \frac{K_1}{|L-L_0|}e^{\frac{2K}{\sqrt{T}}}\|z^0\|_{H^{6}(0,L)},\text{ and }
    \|z(t,\cdot)\|_{L^2(0,L)}\leq \frac{K_2}{|L-L_0|} e^{\frac{2K}{\sqrt{T}}}\|z^0\|_{H^{3}(0,L)}$.
\end{proof}
\begin{coro}
Under the same assumptions in Lemma \ref{lem: control-toy}, we obtain following estimates for $z_b$ and $z_s$ at $t\in(0,T]$, 
\begin{equation*}
\|z_b(t,\cdot)\|_{L^2(0,L)}\leq K_2 e^{\frac{2K}{\sqrt{T}}}\|z^0\|_{H^{3}(0,L)},\; \|z_s(t,\cdot)\|_{L^2(0,L)}\leq \frac{K_2}{|L-L_0|} e^{\frac{2K}{\sqrt{T}}}\|z^0\|_{H^{3}(0,L)}.  
\end{equation*}
\end{coro}

\subsection{Iteration scheme}
In this section, we perform an infinite-iteration scheme based on quantitative rapid stabilization to obtain a fast cost estimate. See also \cite{xiang2020quantitative} for the heat equations, \cite{Xiang-NS} for Navier-Stokes equations, and \cite{Coron-Xiang-2025} for the heat flow.
\begin{prop}\label{prop: estimates on fixed interval-iteration preparation}
  Let $T>0$, and two positive parameters $\mu_2>\mu_1>1$. Let $I$ satisfy the condition {\bf (C)}. For every $L\in I\setminus\{L_0\}$, and every $y^0\in L^2(0, L)$, there exists a function $u\in L^2(0, T)$ satisfying $u= u_1+ u_2+u_3$ in $(0, T)$, and $u_1(t)=u_2(t)=u_3(t)=0, \forall t\in (0, T/2)$ and
\begin{gather*}
\|u_1\|_{L^{\infty}(0, T)}\leq \frac{1}{|L-L_0|}\frac{e^{\frac{2\sqrt{2}K}{\sqrt{T}}}}{T^3}\frac{\mu_2^{\frac{7}{2}}}{\mu_2-\mu_1}\|y^0\|_{L^2(0,L)}\\
\|u_j\|_{L^{\infty}(0, T)}\leq Ce^{-\frac{\mu_1^{\frac{1}{3}}}{2}L}\frac{1}{T^3}\frac{\mu_2}{(\mu_2-\mu_1)}\|y^0\|_{L^2(0,L)},j=2,3,
\end{gather*}
such that the unique solution $y$ of \eqref{eq: linearized KdV system-control-intro} satisfies 
$y(t)=
\left\{
\begin{array}{ll}
    S(t) y^0, & t\in (0, T/2), \\
     y_1(t)+ y_2(t)+y_3(t),& t\in (T/2, T),
\end{array}
\right.
$ where $y_j(t)$ solves the equation \eqref{eq: defi-y_j-eqaution}. Furthermore, 
\begin{gather*}
\|y_1(t, \cdot)\|_{L^2(0, L)}\leq \frac{C}{|L-L_0|} \frac{e^{\frac{2\sqrt{2}K}{\sqrt{T}}}}{T^3} \frac{C\mu_2^{\frac{7}{2}}}{(\mu_2-\mu_1)}\|y^0\|_{L^2(0,L)};\\
\|y_j(t, \cdot)\|_{L^2(0, L)}\leq C\frac{e^{-\mu_1(t-\frac{T}{2})}}{T^3}\frac{\mu_{2}+1}{\mu_1^{\frac{1}{2}}|\mu_1-\mu_2|}\|y^0\|_{L^2(0,L)},j=2,3.
\end{gather*} 
All constants appearing in this proposition are independent of $L$.
\end{prop}
\begin{proof}
For simplicity, set $\mathcal{P}=\p_x^3+\p_x$. $\mathcal{P}$ is a differential operator. We first split our time interval $[0,T]$ into two parts $[0,\frac{T}{2}]$ and $[\frac{T}{2},T]$. In the first part, we only use the free KdV flow \eqref{eq: linear KdV-stability-intro}. 
By smoothing effects, we know that $y(\frac{T}{2},\cdot)\in D(\A^2)\subset H^6(0,L)$ and have the following estimate
\begin{equation}\label{eq: smoothing estimate-toy}
\|y(\frac{T}{2},\cdot)\|_{H^6(0,L)}+|\p_x y(\frac{T}{2},0)|+|\p_x\mathcal{P} y(\frac{T}{2},0)|\leq \frac{C_1}{T^3}\|y^0\|_{L^2(0,L)}.    
\end{equation}
Then we consider the second time interval $[\frac{T}{2},T]$. We notice that our initial datum $y(\frac{T}{2},\cdot)\in H^6(0,L)$. As we have presented in the previous section (see Section \ref{sec: preliminary}), given two positive parameters $\mu_1$ and $\mu_2$, we can construct two modulated functions $h_{\mu_1}$ and $h_{\mu_2}$.
For $j=1,2$, $h_{\mu_j}$ solves the following system respectively:
\begin{equation*}
\left\{
\begin{aligned}
     &h'''+h'=\mu_j h,\text{ in }(0,L)  \\
     &h(0)=h(L)=0,\\
     &h'(L)-h'(0)=1.
\end{aligned}
\right.    
\end{equation*}
And we have the estimates for $h_{\mu_j}$, $\|h_{\mu_j}\|_{L^2(0,L)}\leq C^h_j|\mu_j|^{-\frac{1}{2}}$ and $\|h_{\mu_j}\|_{H^6(0,L)}\leq C^h_j|\mu_j|^{\frac{5}{2}}$. 
Using the functions $h_{\mu_1}$ and $h_{\mu_2}$, we could write $y(\frac{T}{2},x)$ into three parts by 
\begin{gather*}
y(\frac{T}{2},x)=z^0_{\frac{T}{2}}(x)+F_1(\mu_1,\mu_2)h_{\mu_1}(x)+F_2(\mu_1,\mu_2)h_{\mu_2}(x),\text{ where }\\
F_1(\mu_1,\mu_2)=\frac{\mu_2}{\mu_1-\mu_2}\p_x y(\frac{T}{2},0)+\frac{1}{\mu_2-\mu_1}(\p_x\mathcal{P}y)(\frac{T}{2},0),\\
    F_2(\mu_1,\mu_2)=\frac{\mu_1}{\mu_2-\mu_1}\p_x y(\frac{T}{2},0)+\frac{1}{\mu_1-\mu_2}(\p_x\mathcal{P}y)(\frac{T}{2},0).
\end{gather*}
By \eqref{eq: smoothing estimate-toy}, the following estimates hold for $F_1$ and $F_2$
\begin{equation}\label{eq: es F1 and F2}
|F_1(\mu_1,\mu_2)|\leq C_1\frac{\mu_2+1}{|\mu_1-\mu_2|T^3}\|y^0\|_{L^2(0,L)},\quad 
|F_2(\mu_1,\mu_2)|\leq C_1\frac{\mu_1+1}{|\mu_1-\mu_2|T^3}\|y^0\|_{L^2(0,L)}.   
\end{equation}
By the definitions of $F_1$ and $F_2$, we can verify that $z^0_{\frac{T}{2}}\in D(\B^2)\subset H^6(0,L)$. Indeed,
\begin{align*}
\p_x z^0_{\frac{T}{2}}(L)-\p_xz^0_{\frac{T}{2}}(0)&=\p_x y(\frac{T}{2},L)-\p_xy(\frac{T}{2},0)+F_1(\mu_1,\mu_2)h'_{\mu_1}(0)+F_2(\mu_1,\mu_2)h'_{\mu_2}(0)\\
-&F_1(\mu_1,\mu_2)h'_{\mu_1}(L)-F_2(\mu_1,\mu_2)h'_{\mu_2}(L)\\
&=-\p_xy(\frac{T}{2},0)+\p_xy(\frac{T}{2},0)\\
&=0.
\end{align*}
Similarly, we also have $\p_x\mathcal{P} z^0_{\frac{T}{2}}(L)-\p_x\mathcal{P}z^0_{\frac{T}{2}}(0)=0$. Using Lemma \ref{lem: control-toy}, we are able to construct a continuous function $v$ such that $z$ is a solution to \eqref{eq: KdV boundary derivative difference} in $(\frac{T}{2},T)\times(0,L)$ with the initial condition $z(\frac{T}{2},x)=z^0_{\frac{T}{2}}(x)$ and 
 $z(T,x)\equiv0$. Furthermore, we have
\begin{align*}
\|\p_x z(\cdot,L)\|_{L^{\infty}(\frac{T}{2},T)}\leq  \frac{K_3}{|L-L_0|}e^{\frac{2\sqrt{2}K}{\sqrt{T}}}\|z^0_{\frac{T}{2}}\|_{H^{6}},\\
\|z(t,\cdot)\|_{L^2(\frac{T}{2},T)}\leq  \frac{K_2}{|L-L_0|}e^{\frac{2\sqrt{2}K}{\sqrt{T}}}\|z^0_{\frac{T}{2}}\|_{H^{6}}.
\end{align*}
Using the estimates \eqref{eq: smoothing estimate-toy} and \eqref{eq: es F1 and F2}, 
\begin{align*}
\|z^0_{\frac{T}{2}}\|_{H^{6}}&\leq  \|y(\frac{T}{2},\cdot)\|_{H^6(0,L)}+|F_1(\mu_1,\mu_2)|\|h_{\mu_1}\|_{H^6(0,L)}+|F_2(\mu_1,\mu_2)|\|h_{\mu_2}\|_{H^6(0,L)}\\
&\leq \frac{C_1}{T^3}\left(1+\frac{\mu_2+1}{|\mu_1-\mu_2|}C^h_2\mu_1^{\frac{5}{2}}+
\frac{\mu_1+1}{|\mu_1-\mu_2|}C^h_2\mu_2^{\frac{5}{2}}\right)\|y^0\|_{L^2(0,L)}\\
&\leq \frac{C_1\left(|\mu_1-\mu_2|+C^h_2\mu_1\mu_2^{\frac{5}{2}}+C^h_2\mu_2^{\frac{5}{2}}+C^h_2\mu_2\mu_1^{\frac{5}{2}}+C^h_2\mu_1^{\frac{5}{2}}\right)}{|\mu_1-\mu_2|T^3}\|y^0\|_{L^2(0,L)}.
\end{align*}
Therefore, we obtain
\begin{align*}
\|\p_x z(\cdot,L)\|_{L^{\infty}(\frac{T}{2},T)}&\leq \frac{K_3}{|L-L_0|}e^{\frac{2\sqrt{2}K}{\sqrt{T}}}\frac{C_1\left(|\mu_1-\mu_2|+C^h_2\mu_1\mu_2^{\frac{5}{2}}+C^h_2\mu_2^{\frac{5}{2}}+C^h_2\mu_2\mu_1^{\frac{5}{2}}+C^h_2\mu_1^{\frac{5}{2}}\right)}{|\mu_1-\mu_2|T^3}\|y^0\|_{L^2(0,L)},\\
\|z(t,\cdot)\|_{L^{2}(0,L)}&\leq \frac{K_2}{|L-L_0|}e^{\frac{2\sqrt{2}K}{\sqrt{T}}}\frac{C_1\left(|\mu_1-\mu_2|+C^h_2\mu_1\mu_2^{\frac{5}{2}}+C^h_2\mu_2^{\frac{5}{2}}+C^h_2\mu_2\mu_1^{\frac{5}{2}}+C^h_2\mu_1^{\frac{5}{2}}\right)}{|\mu_1-\mu_2|T^3}\|y^0\|_{L^2(0,L)}.
\end{align*}
For the other part, for $j=1,2$, we define a function $z_{\mu_j}(t,x)$ by $z_{\mu_j}(t,x)=e^{-\mu_j(t-\frac{T}{2})}F_j(\mu_1,\mu_2)h_{\mu_j}(x)$. Then $z_{\mu_j}$ satisfies the equation with initial conditions $z_{\mu_j}(\frac{T}{2},x)=F_j(\mu_1,\mu_2)h_{\mu_j}(x)$
\begin{equation*}
\left\{
\begin{array}{lll}
    \p_tz_{\mu_j}+\p_x^3z_{\mu_j}+\p_xz_{\mu_j}=-\mu_j z_{\mu_j}+\mu_j z_{\mu_j}=0 & \text{ in }(\frac{T}{2},T)\times(0,L), \\
     z_{\mu_j}(t,0)=z_{\mu_j}(t,L)=0&  \text{ in }(\frac{T}{2},T),\\
     \p_xz_{\mu_j}(t,L)=e^{-\mu_j(t-\frac{T}{2})}F_j(\mu_1,\mu_2)h'_{\mu_j}(L)&\text{ in }(\frac{T}{2},T),
\end{array}
\right.
\end{equation*}
Then it is easy to see $\|z_{\mu_j}(t,\cdot)\|_{L^2(0,L)}
\leq C_1C^h_1e^{-\mu_j(t-\frac{T}{2})}\frac{\mu_{3-j}+1}{\mu_j^{\frac{1}{2}}|\mu_1-\mu_2|T^3}\|y^0\|_{L^2(0,L)}$. 
And the control cost\\ $\|\p_xz_{\mu_j}(\cdot,L)\|_{L^{\infty}(\frac{T}{2},T)}
\leq C_2e^{-\frac{\mu_j^{\frac{1}{3}}}{2}L}\frac{\mu_{3-j}+1}{|\mu_1-\mu_2|T^3}\|y^0\|_{L^2(0,L)}$. 
We set 
\begin{equation}
\begin{array}{ll}
  u_1(t)=\left\{
  \begin{array}{ll}
     0  &t\in[0,\frac{T}{2}), \\
     \p_x z(t,L)  &t\in[\frac{T}{2},T],
  \end{array}
\right.   &u_{j+1}(t)=\left\{
  \begin{array}{ll}
     0  &t\in[0,\frac{T}{2}), \\
     \p_x z_{\mu_j}(t,L)  &t\in[\frac{T}{2},T], 
  \end{array}
\right. j=1,2.
\end{array}
\end{equation}
 We consider the solution $y_j$ to 
\begin{equation}\label{eq: defi-y_j-eqaution}
\left\{
\begin{array}{lll}
    \p_ty_j+\p_x^3y_j+\p_xy_j=0 & \text{ in }(\frac{T}{2},T)\times(0,L), \\
     y_j(t,0)=y_j(t,L)=0&  \text{ in }(\frac{T}{2},T),\\
     \p_x y_j(t,L)= u_j(t)&\text{ in }(\frac{T}{2},T),\\
     y_j(\frac{T}{2},x)=y_j^0(x)&\text{ in }(0,L),
\end{array}
\right.    
\end{equation}
where 
$y_1^0(x)=z^0_{\frac{T}{2}}(x)$, $y_2^0(x)=F_1(\mu_1,\mu_2)h_{\mu_1}(x)$, and $y_3^0(x)=F_2(\mu_1,\mu_2)h_{\mu_2}(x)$. We have the following properties:
\begin{enumerate}
    \item $y(\frac{T}{2},x)=y_1^0(x)+y_2^0(x)+y_3^0(x)$;
    \item By the uniqueness, we know that $y_1(t,x)=z(t,x)$, $y_2(t,x)=z_{\mu_1}(t,x)$, and $y_3(t,x)=z_{\mu_2}(t,x)$ in $(\frac{T}{2},T)\times(0,L)$.
\end{enumerate}
Now let us consider the solutions in the time interval $[0,T]$. Define
\begin{equation*}
  Y(t,x)=\left\{
  \begin{array}{ll}
     y(t,x)  &t\in[0,\frac{T}{2}]\times(0,L), \\
     y_1(t,x)+y_2(t,x)+y_3(t,x)  &t\in[\frac{T}{2},T]\times(0,L).
  \end{array}
\right. 
\end{equation*}
Then $Y$ solves the equation \eqref{eq: linearized KdV system-control-intro} with $u(t)=u_1(t)+u_2(t)+u_3(t)$.
Indeed, $Y\in C([0,T],L^2(0,L))$ and in particular, $Y$ is continuous at the time $t=\frac{T}{2}$. For $t\in[0,\frac{T}{2}]$, by energy estimates, $\|Y(t,\cdot)\|_{L^2(0,L)}=\|y(t,\cdot)\|_{L^2(0,L)}\leq \|y^0\|_{L^2(0,L)}$. 
For $t\in[\frac{T}{2},T]$, we collect the estimates above,
\begin{align}
\|y_1(t,\cdot)\|_{L^2(0,L)}&\leq\frac{K_2}{|L-L_0|}e^{\frac{2\sqrt{2}K}{\sqrt{T}}}\frac{C_1\left(|\mu_1-\mu_2|+C^h_2\mu_1\mu_2^{\frac{5}{2}}+C^h_2\mu_2^{\frac{5}{2}}+C^h_2\mu_2\mu_1^{\frac{5}{2}}+C^h_2\mu_1^{\frac{5}{2}}\right)}{|\mu_1-\mu_2|T^3}\|y^0\|_{L^2(0,L)}\label{eq: L^2-y_1-single}\\
\|y_2(t,\cdot)\|_{L^2(0,L)}&\leq C_1C^h_1\frac{e^{-\mu_1(t-\frac{T}{2})}(\mu_{2}+1)}{\mu_1^{\frac{1}{2}}|\mu_1-\mu_2|T^3}\|y^0\|_{L^2(0,L)},\;
\|y_3(t,\cdot)\|_{L^2(0,L)}\leq C_1C^h_1\frac{e^{-\mu_2(t-\frac{T}{2})}(\mu_{1}+1)}{\mu_2^{\frac{1}{2}}|\mu_1-\mu_2|T^3}\|y^0\|_{L^2(0,L)}.\notag
\end{align}
As for the control cost
\begin{align}
\|u_1\|_{L^{\infty}(0,T)}&\leq \frac{K_3}{|L-L_0|}e^{\frac{2\sqrt{2}K}{\sqrt{T}}}\frac{C_1\left(|\mu_1-\mu_2|+C^h_2\mu_1\mu_2^{\frac{5}{2}}+C^h_2\mu_2^{\frac{5}{2}}+C^h_2\mu_2\mu_1^{\frac{5}{2}}+C^h_2\mu_1^{\frac{5}{2}}\right)}{|\mu_1-\mu_2|T^3}\|y^0\|_{L^2(0,L)},\label{eq: cost estimate-w-1}\\
\|u_2\|_{L^{\infty}(0,T)}&\leq C_2e^{-\frac{\mu_1^{\frac{1}{3}}}{2}L}\frac{\mu_{2}+1}{|\mu_1-\mu_2|T^3}\|y^0\|_{L^2(0,L)},\notag\\
\|u_3\|_{L^{\infty}(0,T)} 
&\leq C_2e^{-\frac{\mu_2^{\frac{1}{3}}}{2}L}\frac{\mu_{1}+1}{|\mu_1-\mu_2|T^3}\|y^0\|_{L^2(0,L)}.
\end{align}
 Thanks to the condition $\mu_2>\mu_1>0$, we derive the desired estimates outlined in the proposition.
\end{proof}
\begin{coro}\label{coro: simplified estimates for iteration scheme}
Under the assumptions of Proposition \ref{prop: estimates on fixed interval-iteration preparation}. If we choose $\mu_2=2\mu_1$ and $\mu_1>1$, we are able to simplify the estimates above. There exists a constant $\mathcal{C}$ such that the solution to \eqref{eq: linearized KdV system-control-intro} satisfies 
\begin{equation}\label{eq: simplified inequality for y and u}
\begin{aligned}
\|y(t,\cdot)\|_{L^2(0,L)}&\leq \frac{\mathcal{C}}{|L-L_0|}\frac{e^{\frac{2\sqrt{2}K}{\sqrt{T}}}\mu_1^{\frac{5}{2}}+ e^{-\mu_1(t-\frac{T}{2})}}{T^3} \|y^0\|_{L^2(0,L)},t\in (\frac{T}{2},T)\\
\|u\|_{L^{\infty}(0,T)}&\leq \frac{\mathcal{C}}{|L-L_0|}\frac{e^{\frac{2\sqrt{2}K}{\sqrt{T}}}\mu_1^{\frac{5}{2}}+e^{-\frac{\mu_1^{\frac{1}{3}}}{2}L}}{T^3}\|y^0\|_{L^2(0,L)}.
\end{aligned}
\end{equation}
\end{coro}
\begin{proof}
We plug $\mu_2=2\mu_1$ into the estimates \eqref{eq: L^2-y_1-single}, \eqref{eq: cost estimate-w-1},  then $|\mu_1-\mu_2|=\mu_1$ and we can find a constant $\mathcal{C}=\mathcal{C}(L_0,C_1,C_2,C_2^h,\mu_1)$ such that estimates \eqref{eq: simplified inequality for y and u} hold. As for the explicit computation, one can refer to Appendix \ref{sec: Proof of Corollary coro: simplified estimates for iteration scheme}.
\end{proof}
Based on the previous estimates, we could construct the following iteration schemes.

\begin{prop}[Iteration schemes]\label{prop: iteration schemes}
Let $T>0$.  Let $I$ satisfy the condition {\bf (C)}. For every $L\in I\setminus\{L_0\}$, there exists a constant $\epsilon=\epsilon(T,L)\sim\left|\frac{1}{\sqrt{T}\ln{|L-L_0|}}\right|>0$ such that for $\forall y^0\in L^2(0,L)$, there exists a function $u(t) \in L^2(0,T)$ such that the solution $y$ to the system \eqref{eq: linearized KdV system-control-intro}   
satisfies $\lim_{t\rightarrow T^-}\|y(t, \cdot)\|_{L^2(0, L)}= 0$, 
    and there exists a constant $\mathcal{K}$ such that
\begin{equation}\label{eq: est-control-L-infty}
 \|u\|_{L^{\infty}(0,T)}\leq \frac{\mathcal{K}}{|L-L_0|^{1+\epsilon}}\|y^{0}\|_{L^2(0,L)}.   
\end{equation}
In particular, $\epsilon(T,L)$ tends to $0$ as $L$ approaches $L_0$.
\end{prop}
\begin{proof}
For every $L\in I\setminus \{L_0\}$, without loss of generality, we set $T\in(0,1)$. We define the constant $\epsilon$ by $\epsilon(T,L)=\left|\frac{\mathcal{K}}{\sqrt{T}\ln{|L-L_0|}}\right|$, with $\mathcal{K}$ to be specified later. Suppose that $T=2^{-n_0}$. Let $T_n=2^{-n_0}(1-2^{-n})$ and $\mathcal{I}_n=[T_{n-1},T_n)$, $n\in\N$. Let us also take a constant $Q>0$ that is independent of $T$ and $L$. And $Q$ will be fixed later on. Now our concerned time interval $[0,T)$ has a partition $[0,T)=\bigcup_{n=1}^{\infty}\mathcal{I}_n$.
We fix our choice of $\mu_{1,n}= Q 2^{\frac{3}{2}(n_0+ n)}, n= 1,2,...$ and $\mu_{2,n}=2\mu_{1,n}$. On each time interval $\mathcal{I}_n$, we construct the control function $w_n(t)\in L^2(\mathcal{I}_n)$ the unique solution $y_n$ of the Cauchy problem with the initial condition $y_n(T_{n-1},x)=y^{n-1}(x)$,
\begin{equation*}
\left\{
\begin{array}{lll}
    \p_ty_n+\p_x^3y_n+\p_xy_n=0 & \text{ in }\mathcal{I}_n\times(0,L), \\
     y_n(t,0)=y_n(t,L)=0&  \text{ in }\mathcal{I}_n,\\
     \p_x y_n(t,L)= u^n(t)&\text{ in }\mathcal{I}_n,
\end{array}
\right.
\end{equation*}
satisfies 
\begin{align*}
\|y(t,\cdot)\|_{L^2(0,L)}&\leq \|y^{n-1}\|_{L^(0,L)},t\in(T_{n-1},\frac{T_{n-1}+T_n}{2}],\\
\|y_n(t,\cdot)\|_{L^2(0,L)}&\leq \frac{\mathcal{C}}{|L-L_0|}\frac{e^{\frac{2\sqrt{2}K}{\sqrt{T_n-T_{n-1}}}}\mu_{1,n}^{\frac{5}{2}}+ e^{-\mu_{1,n}(t-\frac{T_{n-1}+T_n}{2})}}{(T_{n-1}-T_n)^3} \|y^{n-1}\|_{L^2(0,L)},t\in (\frac{T_{n-1}+T_n}{2},T_{n})\\
\|u^n\|_{L^{\infty}(\mathcal{I}_n)}&\leq \frac{\mathcal{C}}{|L-L_0|}\frac{e^{\frac{2\sqrt{2}K}{\sqrt{T_n-T_{n-1}}}}\mu_{1,n}^{\frac{5}{2}}+e^{-\frac{\mu_{1,n}^{\frac{1}{3}}}{2}L}}{(T_n-T_{n-1})^3}\|y^{n-1}\|_{L^2(0,L)}.
\end{align*}
This construction implies that the solution $y$ to the equation \eqref{eq: linearized KdV system-control-intro} is defined by $y(t,x)|_{\mathcal{I}_n}=y_n(t,x)$. In the meanwhile, the control function $u(t)|_{\mathcal{I}_n}=u^n(t)$. Consider the estimate at $t=T_{n}$, we obtain
\begin{equation*}
\|y^n\|_{L^2(0,L)}=\|y^{n}(T_n,\cdot)\|_{L^2(0,L)}\leq \frac{\mathcal{C}}{|L-L_0|}\frac{ e^{-\frac{(T_n-T_{n-1})\mu_{1,n}}{2}}}{(T_n-T_{n-1})^3} \|y^{n-1}\|_{L^2(0,L)}.
\end{equation*}
Using that $T_n-T_{n-1}=2^{-n}$, we simplify the estimates above
\begin{align}
\|y(t,\cdot)\|_{L^2(0,L)}&\leq \|y^{n-1}\|_{L^(0,L)},t\in(T_{n-1},\frac{T_{n-1}+T_n}{2}],\label{eq: L^2-y-first half}\\
\|y_n(t,\cdot)\|_{L^2(0,L)}&\leq \frac{\mathcal{C}}{|L-L_0|}\left(e^{2\sqrt{2}K2^{\frac{n_0+n}{2}}}Q^{\frac{5}{2}}2^{\frac{15}{4}(n_0+n)}+ 1\right)2^{3(n_0+n)}\|y^{n-1}\|_{L^2(0,L)},t\in (\frac{T_{n-1}+T_n}{2},T_{n})\label{eq: L^2-y-second half}\\
\|u^n\|_{L^{\infty}(\mathcal{I}_n)}&\leq \frac{\mathcal{C}}{|L-L_0|}\left(e^{2\sqrt{2}K2^{\frac{n_0+n}{2}}}Q^{\frac{5}{2}}2^{\frac{15}{4}(n_0+n)}+e^{-\frac{Q^{\frac{1}{3}}2^{\frac{n_0+n}{2}}}{2}L}\right)2^{3(n_0+n)}\|y^{n-1}\|_{L^2(0,L)},\label{eq: L^2-w-all}\\
\|y^n\|_{L^2(0,L)}&\leq \frac{\mathcal{C}}{|L-L_0|}e^{-Q2^{\frac{n_0+n}{2}-1}}2^{3(n_0+n)} \|y^{n-1}\|_{L^2(0,L)}\label{eq: L^2 for y^n}.
\end{align}
Using the inequality \eqref{eq: L^2 for y^n}, by induction, we obtain 
\begin{equation*}
\|y^n\|_{L^2(0,L)}\leq \frac{\mathcal{C}^n}{|L-L_0|^n}e^{-\frac{Q}{2}\sum_{j=0}^{n-1}2^{\frac{n_0+n-j}{2}}}2^{3\sum_{j=0}^{n-1}(n_0+n-j)} \|y^{0}\|_{L^2(0,L)},   
\end{equation*}
which implies that
\begin{equation}\label{eq: decay for y^n}
\|y^n\|_{L^2(0,L)}\leq \frac{\mathcal{C}^n}{|L-L_0|^n}e^{-\frac{Q2^{\frac{n_0}{2}}(2^{\frac{n}{2}}-1)}{2-\sqrt{2}}}2^{3n_0 n+\frac{3n(1+n)}{2}} \|y^{0}\|_{L^2(0,L)},   
\end{equation}
With the help of the good choice of $Q$ (see Appendix \ref{sec: app-choice-Q}), we simplify the estimates for $n>1$,
\begin{align}
\|y_n(t,\cdot)\|_{L^2(0,L)}&\leq  e^{-\frac{Q2^{\frac{n_0}{2}}(2^{\frac{n-1}{2}}-1)}{2(2-\sqrt{2})}}\|y^{0}\|_{L^2(0,L)},t\in(T_{n-1},\frac{T_{n-1}+T_n}{2}],\\
\|y_n(t,\cdot)\|_{L^2(0,L)}&\leq (Q^{\frac{5}{2}}+1)e^{-\frac{Q2^{\frac{n_0}{2}}(2^{\frac{n-1}{2}}-1)}{2(2-\sqrt{2})}} \|y^{0}\|_{L^2(0,L)},t\in (\frac{T_{n-1}+T_n}{2},T_{n})\\
\|u^n\|_{L^{\infty}(\mathcal{I}_n)}&\leq 2e^{-\frac{Q2^{\frac{n_0}{2}}(2^{\frac{n-1}{2}}-1)}{2(2-\sqrt{2})}} \|y^{0}\|_{L^2(0,L)}\label{eq: final estimate for w_n}. 
\end{align}
Therefore, we know that for any $L\in I\setminus\{L_0\}$, $\lim_{t\rightarrow T^{-}}\|y(t,\cdot)\|_{L^2(0,L)}=0$.
Moreover, for $n=1$, we have the following estimates
\begin{align*}
\|y_1(t,\cdot)\|_{L^2(0,L)}&\leq \|y^{0}\|_{L^2(0,L)},t\in(0,\frac{1}{2^{n_0+2}}],\\
\|y_1(t,\cdot)\|_{L^2(0,L)}&\leq \frac{\mathcal{C}}{|L-L_0|}\left(e^{2\sqrt{2}K2^{\frac{n_0}{2}}}Q^{\frac{5}{2}}2^{\frac{15}{4}(n_0)}+ 1\right)2^{3(n_0+1)} \|y^{0}\|_{L^2(0,L)},t\in (\frac{1}{2^{n_0+2}},\frac{1}{2^{n_0+1}})\\
\|\p_xy_1(\cdot,L)\|_{L^{\infty}(\mathcal{I}_1)}&\leq \frac{\mathcal{C}}{|L-L_0|}\left(e^{2\sqrt{2}K2^{\frac{n_0}{2}}}Q^{\frac{5}{2}}2^{\frac{15}{4}(n_0)}+e^{-\frac{Q^{\frac{1}{3}}2^{\frac{n_0}{2}}}{2}L}\right)2^{3(n_0+1)} \|y^{0}\|_{L^2(0,L)} . 
\end{align*}
Since $T=2^{-n_0}$, there exists a constant $\mathcal{K}$ such that  
\begin{equation*}
\|\p_xy_1(\cdot,L)\|_{L^{\infty}(\mathcal{I}_1)}\leq \frac{\mathcal{K}}{|L-L_0|}e^{\frac{\mathcal{K}}{\sqrt{T}}}\|y^{0}\|_{L^2(0,L)}    
\end{equation*}
Combing with the inequality \eqref{eq: final estimate for w_n}, we know that
\begin{equation*}
    \|u\|_{L^{\infty}(0,T)}\leq \frac{\mathcal{K}}{|L-L_0|}e^{\frac{\mathcal{K}}{\sqrt{T}}}\|y^{0}\|_{L^2(0,L)}\leq \frac{\mathcal{K}}{|L-L_0|^{1+\epsilon}}\|y^{0}\|_{L^2(0,L)} .
\end{equation*}
\end{proof}

\subsection{Quantitative observability and exponential stability}
In this section, we first prove Theorem \ref{thm: control-cost}.
\begin{proof}[Proof of Theorem \ref{thm: control-cost}]
Applying Proposition \ref{prop: iteration schemes}, for $\forall y^0\in L^2(0,L)$, there exists a control function $u\in L^2(0,T)$ such that the solution $y$ to the equation \eqref{eq: linearized KdV system-control-intro} achieve null controllability and 
\begin{equation*}
\|u\|_{L^{\infty}(0,T)}\leq \frac{\mathcal{K}e^{\frac{\mathcal{K}}{\sqrt{T}}}}{|L-L_0|}\|y^{0}\|_{L^2(0,L)}.    
\end{equation*}
Given $L\notin\mathcal{N}$, let $\mathrm{d}:=\mathrm{dist}(L,\mathcal{N})>0$. Therefore, defining $\mathscr{K}:=\max\{\frac{\mathcal{K}}{\mathrm{d}},\mathcal{K}\}$,
\begin{equation*}
\|u\|_{L^{\infty}(0,T)}\leq \mathscr{K}e^{\frac{\mathscr{K}}{\sqrt{T}}}\|y^{0}\|_{L^2(0,L)}.    
\end{equation*}
\end{proof}
Besides Theorem \ref{thm: control-cost}, we are also able to prove a weak version of Theorem \ref{thm: main theorem linear version} as follows.
\begin{thm}\label{thm: main result with L-L0}
Let $T>0$. Let $I$ satisfy the condition {\bf (C)}. For every $L\in I\setminus\{L_0\}$, there exist two effectively computable constants $\epsilon=\epsilon(T,L)=\left|\frac{\mathcal{K}}{\sqrt{T}\ln{|L-L_0|}}\right|$ and $C=C(T,L)=\frac{T\mathcal{K}}{|L-L_0|^{1+\epsilon}}$ such that for $\forall y^0\in L^2(0,L)$, the following quantitative observability inequality 
\begin{equation}\label{eq: quantitative-ob-main}
   \|y^0\|^2_{L^2(0,L)} \leq  C(T,L)^2\int_0^T|\p_x y(t,0)|^2dt
\end{equation}
holds for any solution $y$ to the KdV system \eqref{eq: linear KdV-stability-intro}. In addition, for $T>0$, there exists $R_e=R_e(L)=\ln{(1+\frac{2|L-L_0|^{2+2\epsilon}}{\mathcal{K}^2})}$ such that for $\forall y^0\in L^2(0,L)$, the KdV system \eqref{eq: linear KdV-stability-intro} 
\begin{equation*}
    E(y(t))\leq e^{-t R_e(L)}E(y^0),\forall t\in(0,\infty).
\end{equation*}
\end{thm}
\begin{proof}
By the Hilbert Uniqueness Method, it suffices to analyze the solution $\Tilde{y}$ to the controlled KdV system \eqref{eq: linearized KdV system-control-intro} with initial data $\Tilde{y}^0(x)$ and control $u(t)$.
The quantitative observability \eqref{eq: quantitative-ob-main} is equivalent to the following estimate
\begin{equation}
    \|u\|_{L^2(0,T)}\leq C(T,L)\|\Tilde{y}^0\|_{L^2(0,L)}.
\end{equation}
Using Proposition \ref{prop: iteration schemes} and the estimate \eqref{eq: est-control-L-infty}, we obtain that 
\begin{equation*}
 \|u\|_{L^2(0,T)}\leq T \|u\|_{L^{\infty}(0,T)}\leq \frac{T\mathcal{K}}{|L-L_0|^{1+\epsilon}}\|\Tilde{y}^0\|_{L^2(0,L)},
\end{equation*}
which ensures the quantitative observability \eqref{eq: quantitative-ob-main}. Following the standard approach, the exponential stability is a direct consequence of the observability \eqref{eq: quantitative-ob-main}. Without loss of generality, we set $T=1$. Using the fact that
\begin{equation*}
\frac{d}{dt}E(y(t))=-2 |\p_xy(t,0)|^2, 
\end{equation*}
we obtain $E(y(1))-E(y^0)=-2\int_0^1|\p_xy(t,0)|^2dt$. By the observability \eqref{eq: quantitative-ob-main}, we know that 
\begin{align*}
E(y^0)-E(y(1))\geq \frac{2}{\frac{\mathcal{K}^2}{|L-L_0|^{2+2\epsilon}}}\|S(1)y^0\|^2_{L^2(0,L)}
=\frac{2|L-L_0|^{2+2\epsilon}}{\mathcal{K}^2}E(y(1)),
\end{align*}
which implies that $E(y(1))\leq e^{-\ln{(1+\frac{2|L-L_0|^{2+2\epsilon}}{\mathcal{K}^2})}}E(y^0)$. Then for any $t\in(0,+\infty)$, we  deduce that
\begin{equation*}
    E(y(t))\leq e^{-t\ln{(1+\frac{2|L-L_0|^{2+2\epsilon}}{\mathcal{K}^2})}}E(y^0):=e^{-tR_e(L)}E(y^0).
\end{equation*}
In particular, $R_e(L)\sim |L-L_0|^{2+2\epsilon}$.
\end{proof}

\begin{rem}
At this point, we have established an exponential stability result of the KdV system \eqref{eq: adjoint linear KdV-HUM} with the exponential decay rate $R_e(L)\sim |L-L_0|^{2+2\epsilon}$. Here we notice that there is a loss of order $\epsilon$. However, we know that $\lim_{L\rightarrow L_0}\epsilon(L)=0$. This motivates us to do a more careful analysis of the exponential stability to obtain a sharp exponential decay rate $R_e(L)\sim |L-L_0|^{2}$. Furthermore, we can obtain a decomposition of $L^2(0,L)=H_{\A}(L)\oplus M_{\A}(L)$ such that on $\h$ we obtain exponential stability with a uniform decay rate, while on $M_{\A}(L)$ we obtain exponential stability with a decay rate $\sim|L-L_0|^2$. 
\end{rem}

\section{Classification of critical lengths, invariant manifolds}\label{sec: classification quasiinvariant space}
From the definition of the critical length, $L_0=2\pi\sqrt{\frac{k^2+k l+l^2}{3}}$ if $L_0\in\mathcal{N}$. We introduce the following classification index
\begin{equation}
\mathcal{I}_C(L_0):=3(\frac{L_0}{2\pi})^2\in\N^*,\forall L_0\in\mathcal{N},
\end{equation}
This motivates us to consider the following simple Diophantine Equation 
\begin{equation}\label{integersolution}
    a^2+ab+b^2=n,
\end{equation}
or equivalently
\begin{equation}\label{eq:klIC}
    k^2+ kl+ l^2= \mathcal{I}_C(L_0).
\end{equation}
The solutions of the preceding algebraic equations lead to description of $N_0$, i.e., the dimension of the unreachable subspace.  

Depending on the solutions we can classify different critical lengths, and  distinguish the solutions $(k, l)$ for each given length.  The details  are presented in Section \ref{sec: Sharp spectral analysis and classification of critical lengths}.

As we noticed in Section \ref{sec: spectral analysis of A}, $k-l\equiv0\mod{3}$ is a crucial condition to distinguish different asymptotic behaviors of eigenvalues and eigenfunctions. For $L\notin\mathcal{N}$ close to $L_0\in\mathcal{N}$, we conduct a comparative analysis of the asymptotic properties in Section \ref{sec: limiting analysis on different types} related to different classes of $L_0$.

\subsection{Classification of critical lengths and unreachable pairs}\label{sec: Sharp spectral analysis and classification of critical lengths}
For each $L_0\in\mathcal{N}$,  Equation \eqref{eq:klIC} may admit more than one solution, we have the following natural classification for those solutions $(k,l)$ (We also refer to \cite{NX} for more details).
\begin{defi}\label{defi: types of unreachable pairs}
Let $L_0\in \mathcal{N}$. We define the following sets for the unreachable pairs $(k,l)$:
\begin{align*}
    \mathcal{S}_1(L_0)&:=\{(k,l) \textrm{ solution of } \eqref{eq:klIC}: k=l\}, \\
    \mathcal{S}_2(L_0)&:=\{(k,l) \textrm{ solution of } \eqref{eq:klIC}: k\equiv l\mod{3}, \; k\neq l\}\\
    \mathcal{S}_3(L_0)&:=\{(k,l) \textrm{ solution of } \eqref{eq:klIC}: k\not\equiv l\mod{3}\}.
\end{align*}
\end{defi}
One easily observe that 
\begin{equation*}
     \mathcal{S}_3(L_0)\cap  \mathcal{S}_1(L_0)=  \mathcal{S}_3(L_0)\cap  \mathcal{S}_2(L_0)= \emptyset.
\end{equation*}
Later, we call $\mathcal{S}_1(L_0)$ ($\mathcal{S}_2(L_0),\mathcal{S}_3(L_0)$) the Type I (Type II, Type III) unreachable pairs, respectively. 
Let $L_0\in\mathcal{N}$. As we presented in Section \ref{sec: Eigenvalues and eigenfunctions at the critical length},  for a pair $(k,l)$,  we obtain two eigenvalues $\pm\ii\lambda_{c}=\pm\ii\frac{(2k+l)(k-l)(2l+k)}{3\sqrt{3}(k^2+k l+l^2)^{\frac{3}{2}}}$. Their associated eigenfunctions satisfy
\begin{equation*}
\left\{
\begin{array}{ll}
     \G'''+\G'+\ii\lambda_{c}\G=0,  & \text{ in }(0,L_0),\\
     \G(0)=\G(L_0)=\G'(0)-\G'(L_0)=0.& 
\end{array}
\right.       
\end{equation*}
As noticed in Proposition \ref{prop: type-1-2}, when $(k,l)\in \mathcal{S}_1\cup \mathcal{S}_2$, we observe two types of eigenfunctions $\G_c$ and $\Tilde{\G}_c$ (see their explicit formulas in \eqref{eq: exact formula for critical eigenfunctions} and \eqref{eq: defi of tilde-G}). We distinguish different types of $(k,l)$ by their different behaviors concerning eigenmodes as follows. 
\begin{table}[]
\begin{tabular}{|c|c|c|c|}
\hline
 &$\mathcal{S}_1$ & $\mathcal{S}_2$ & $\mathcal{S}_3$ \\
\hline
$0$ is  eigenvalue&YES & NO & NO \\
\hline
Eigenfunctions & Two eigenfucntions  & Two eigenfucntions & A unique eigenfucntion\\
for $\B$ w.r.t $\ii\lambda_{c}$& $\G_{c}$ and $\Tilde{\G}_{c}$&  $\G_{c}$ and $\Tilde{\G}_{c}$& $\G_{c}$\\
\hline
\end{tabular}
     \caption{Eigenmodes for different types of $(k,l)$}
     \label{tab: (k,l)-type}
 \end{table}

As noticed in \cite{NX}, based on the features of the integer pairs $(k,l)$ solving \eqref{eq:klIC}, we classify the critical lengths under three cases:
\begin{defi}\label{def:new:classification}
We have the following types of critical lengths,
\begin{align*}
    \mathcal{N}^1&:=\{L_0\in \mathcal{N}: \textrm{ there exists only one pair $(k, l)$ solving \eqref{eq:klIC}, which belongs to } \mathcal{S}_1(L_0)\}, \\
    \mathcal{N}^2&:= \{L_0\in \mathcal{N}: \textrm{ all solutions $(k, l)$ of \eqref{eq:klIC}  belong to } \mathcal{S}_3(L_0)\}\\
       \mathcal{N}^3&:=
    \{L_0\in \mathcal{N}: \textrm{ there exists pair $(k, l)$ solving \eqref{eq:klIC}, which belongs to } \mathcal{S}_2(L_0)\}.
\end{align*}
\end{defi}
Clearly, these sets are disjoint and 
\begin{equation}
    \mathcal{N}=\mathcal{N}^1\cup \mathcal{N}^2\cup\mathcal{N}^3. 
\end{equation}
\begin{rem}
We notice that for $L_0\in\mathcal{N}$
\begin{itemize}
    \item $L_0\in\mathcal{N}^1$, then $\mathcal{I}_C(L_0)\equiv0\mod{3}$, and $\mathcal{N}^1=2\pi\N^*$.
    \item $L_0\in\mathcal{N}^2$, then $\mathcal{I}_C(L_0)\not\equiv0\mod{3}$.
    \item $L_0\in\mathcal{N}^3$, then $\mathcal{I}_C(L_0)\equiv0\mod{3}$. However, $\mathcal{N}^3$ may contain lengths $L_0$ with  $\sqrt{\frac{\mathcal{I}_C(L_0)}{3}}\not\in \N^*$.
\end{itemize}
\end{rem}
\begin{exa}
For each case above, we could choose a model to catch a glimpse of its features. We choose $L_0=2\pi$ as a model for $\mathcal{N}^1$, $L_0=2\pi\sqrt{\frac{7}{3}}$ for $\mathcal{N}^2$, and $L_0=2\pi\sqrt{7}$ for $\mathcal{N}^3$.
\end{exa}
Then we have the following proposition to describe the characteristics of dimensions of unreachable subspaces, which can be found in \cite[Proposition 2.9]{NX} and we omit its proof here.
\begin{prop}
For any $d\in\N^*$, there are infinitely many $L_0\in \mathcal{N}$ such that the dimension of the unreachable subspace at $L_0$ is exactly $d$.
\end{prop}
Based on this classification, in \cite{NX}, we also obtained a negative result on the small-time local controllability of a KdV system for critical lengths in $\mathcal{N}^3$.

\subsection{Classifiaction of elliptic eigenmodes}\label{sec: limiting analysis on different types}
In the preceding section, we distinguish three different types of critical lengths. This distinction allows us to classify the asymptotic behaviors of elliptic eigenmodes as $L\rightarrow L_0\in\mathcal{N}$. Indeed, we shall characterize the relation between the unreachable space $M(L_0)$ and the elliptic subspace $U_E(L)$ (see also Remark \ref{rem: link-elliptic-unreacheable}). We point out that in general, these two spaces are of different dimensions (see Proposition \ref{prop: Index set M-E}) and contain different directions (see two examples in Section \ref{sec: quasi-invariant subsapce}). To provide a better approximation of $M(L_0)$, we introduce a well-prepared subspace $M_{\B}(L)$, which is the \textit{quasi-invariant subspace} for $\B$ (defined explicitly in \eqref{eq: defi-quasi-space-B}). \\
\begin{figure}
    \centering
    \includegraphics[width=0.5\linewidth]{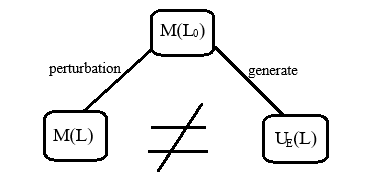}
    \caption{Relations among spaces}
    \label{fig: quasi-invariant}
\end{figure}
However, we have already noticed different asymptotic behaviors of elliptic eigenmodes corresponding to different types of $(k,l)$ in Proposition \ref{prop: Asymp in L-low}, \ref{prop: Low-frequency behaviors: Singular limits}, and \ref{prop: asymptotic expansion for A0}. Therefore, it is more convenient for us to classify elliptic eigenmodes using different types of $(k,l)$ as a criterion. 

\subsubsection{Dimension of the elliptic subspace}

Firstly, we characterize the elliptic index set $\Lambda_{E}$, which directly describes the dimension of the elliptic subspace.
\begin{prop}\label{prop: Index set M-E}
Let $I$ satisfy the condition {\bf (C)} and $N_0$ be the dimension of the unreachable subspace. Then, the elliptic index set $\Lambda_E=\Lambda_E(L_0)\subset[- N_0,N_0]\cap\Z$ (defined in Definition \ref{defi: elliptic/hyperbolic index and subspaces}) such that 
\begin{enumerate}
    \item If $L_0\in \mathcal{N}^1$, then $\Lambda_E=\{-N_0,\cdots,-1,1,\cdots,N_0\}$ with $N_L=N_0$. Moreover, for each critical eigenvalue $\ii\lambda_{c,m},1\leq |m|\leq \frac{N_0-1}{2}$, there are two perturbed elliptic eigenvalues $\ii\lambda_{2m}$ and $\ii\lambda_{2m+1}$ such that
 $\lim_{L\rightarrow L_0}\lambda_{2m}(L)=\lim_{L\rightarrow L_0}\lambda_{2m+1}(L)=\lambda_{c,m}$. Additionally, for eigenvalue $0$,  $\lim_{L\rightarrow L_0}\lambda_{-1}(L)=\lim_{L\rightarrow L_0}\lambda_{1}(L)=0$.
\begin{figure}[h]
    \centering
\tikzset{every picture/.style={line width=0.75pt}} 

\begin{tikzpicture}[x=0.75pt,y=0.75pt,yscale=-0.65,xscale=0.65]

\draw [line width=3]    (101,51) -- (575,50.5) ;
\draw [line width=3]    (101,184) -- (575,183.5) ;
\draw [fill={rgb, 255:red, 252; green, 3; blue, 3 }  ,fill opacity=1 ] [dash pattern={on 4.5pt off 4.5pt}]  (190,52.5) -- (134,185.5) ;
\draw  [dash pattern={on 4.5pt off 4.5pt}]  (190,52.5) -- (232,185.5) ;
\draw  [dash pattern={on 4.5pt off 4.5pt}]  (338,50.75) -- (282,183.75) ;
\draw  [dash pattern={on 4.5pt off 4.5pt}]  (492,51.5) -- (436,184.5) ;
\draw  [dash pattern={on 4.5pt off 4.5pt}]  (338,50.75) -- (381,182.5) ;
\draw  [dash pattern={on 4.5pt off 4.5pt}]  (492,51.5) -- (550,184) ;
\draw  [draw opacity=0][fill={rgb, 255:red, 254; green, 5; blue, 5 }  ,fill opacity=1 ] (186,52.5) .. controls (186,50.29) and (187.79,48.5) .. (190,48.5) .. controls (192.21,48.5) and (194,50.29) .. (194,52.5) .. controls (194,54.71) and (192.21,56.5) .. (190,56.5) .. controls (187.79,56.5) and (186,54.71) .. (186,52.5) -- cycle ;
\draw  [draw opacity=0][fill={rgb, 255:red, 254; green, 5; blue, 5 }  ,fill opacity=1 ] (334,50.75) .. controls (334,48.54) and (335.79,46.75) .. (338,46.75) .. controls (340.21,46.75) and (342,48.54) .. (342,50.75) .. controls (342,52.96) and (340.21,54.75) .. (338,54.75) .. controls (335.79,54.75) and (334,52.96) .. (334,50.75) -- cycle ;
\draw  [draw opacity=0][fill={rgb, 255:red, 254; green, 5; blue, 5 }  ,fill opacity=1 ] (488,51.5) .. controls (488,49.29) and (489.79,47.5) .. (492,47.5) .. controls (494.21,47.5) and (496,49.29) .. (496,51.5) .. controls (496,53.71) and (494.21,55.5) .. (492,55.5) .. controls (489.79,55.5) and (488,53.71) .. (488,51.5) -- cycle ;
\draw  [draw opacity=0][fill={rgb, 255:red, 126; green, 211; blue, 33 }  ,fill opacity=1 ] (130,185.5) .. controls (130,183.29) and (131.79,181.5) .. (134,181.5) .. controls (136.21,181.5) and (138,183.29) .. (138,185.5) .. controls (138,187.71) and (136.21,189.5) .. (134,189.5) .. controls (131.79,189.5) and (130,187.71) .. (130,185.5) -- cycle ;
\draw  [draw opacity=0][fill={rgb, 255:red, 126; green, 211; blue, 33 }  ,fill opacity=1 ] (228,185.5) .. controls (228,183.29) and (229.79,181.5) .. (232,181.5) .. controls (234.21,181.5) and (236,183.29) .. (236,185.5) .. controls (236,187.71) and (234.21,189.5) .. (232,189.5) .. controls (229.79,189.5) and (228,187.71) .. (228,185.5) -- cycle ;
\draw  [draw opacity=0][fill={rgb, 255:red, 126; green, 211; blue, 33 }  ,fill opacity=1 ] (282,183.75) .. controls (282,181.54) and (283.79,179.75) .. (286,179.75) .. controls (288.21,179.75) and (290,181.54) .. (290,183.75) .. controls (290,185.96) and (288.21,187.75) .. (286,187.75) .. controls (283.79,187.75) and (282,185.96) .. (282,183.75) -- cycle ;
\draw  [draw opacity=0][fill={rgb, 255:red, 126; green, 211; blue, 33 }  ,fill opacity=1 ] (377,182.5) .. controls (377,180.29) and (378.79,178.5) .. (381,178.5) .. controls (383.21,178.5) and (385,180.29) .. (385,182.5) .. controls (385,184.71) and (383.21,186.5) .. (381,186.5) .. controls (378.79,186.5) and (377,184.71) .. (377,182.5) -- cycle ;
\draw  [draw opacity=0][fill={rgb, 255:red, 126; green, 211; blue, 33 }  ,fill opacity=1 ] (432,180.5) .. controls (432,178.29) and (433.79,176.5) .. (436,176.5) .. controls (438.21,176.5) and (440,178.29) .. (440,180.5) .. controls (440,182.71) and (438.21,184.5) .. (436,184.5) .. controls (433.79,184.5) and (432,182.71) .. (432,180.5) -- cycle ;
\draw  [draw opacity=0][fill={rgb, 255:red, 126; green, 211; blue, 33 }  ,fill opacity=1 ] (546,184) .. controls (546,181.79) and (547.79,180) .. (550,180) .. controls (552.21,180) and (554,181.79) .. (554,184) .. controls (554,186.21) and (552.21,188) .. (550,188) .. controls (547.79,188) and (546,186.21) .. (546,184) -- cycle ;

\draw (174,23.4) node [anchor=north west][inner sep=0.75pt]    {$\lambda _{c,-1}$};
\draw (475,21.4) node [anchor=north west][inner sep=0.75pt]    {$\lambda _{c,+1}$};
\draw (333,27.4) node [anchor=north west][inner sep=0.75pt]    {$0$};
\draw (122,187.4) node [anchor=north west][inner sep=0.75pt]    {$\lambda _{-3}$};
\draw (220,187.4) node [anchor=north west][inner sep=0.75pt]    {$\lambda _{-2}$};
\draw (273,188.4) node [anchor=north west][inner sep=0.75pt]    {$\lambda _{-1}$};
\draw (369,188.4) node [anchor=north west][inner sep=0.75pt]    {$\lambda _{+1}$};
\draw (424,186.4) node [anchor=north west][inner sep=0.75pt]    {$\lambda _{+2}$};
\draw (543,189.4) node [anchor=north west][inner sep=0.75pt]    {$\lambda _{+3}$};
\draw (597,175.4) node [anchor=north west][inner sep=0.75pt]    {$L$};
\draw (587,41.4) node [anchor=north west][inner sep=0.75pt]    {$L_{0} \in \mathcal{N}^{1}$};
\end{tikzpicture}
\caption{Index relation for $L\rightarrow L_0\in \mathcal{N}^1$}
\label{fig: index relation-1}
\end{figure}
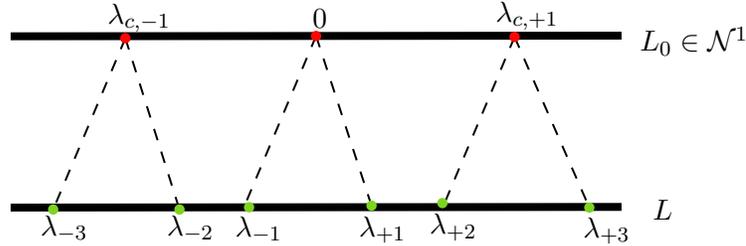
    \item If $L_0\in \mathcal{N}^3$, then $\Lambda_E=\{-N_0,\cdots,-1,1,\cdots,N_0\}$ with $N_L=N_0$. Moreover, 
 for each critical eigenvalue $\ii\lambda_{c,m},1\leq |m|\leq \frac{N_0}{2}$, there are two perturbed elliptic eigenvalues $\ii\lambda_{2m-1}$ and $\ii\lambda_{2m}$ such that
 $\lim_{L\rightarrow L_0}\lambda_{2m-1}(L)=\lim_{L\rightarrow L_0}\lambda_{2m}(L)=\lambda_{c,m}$.
\begin{figure}[h]
    \centering
\tikzset{every picture/.style={line width=0.75pt}} 
\begin{tikzpicture}[x=0.75pt,y=0.75pt,yscale=-0.65,xscale=0.65]

\draw [line width=3]    (101,51) -- (575,50.5) ;
\draw [line width=3]    (101,184) -- (575,183.5) ;
\draw [fill={rgb, 255:red, 252; green, 3; blue, 3 }  ,fill opacity=1 ] [dash pattern={on 4.5pt off 4.5pt}]  (204,51.5) -- (134,185.5) ;
\draw  [dash pattern={on 4.5pt off 4.5pt}]  (204,51.5) -- (264,186) ;
\draw  [draw opacity=0][fill={rgb, 255:red, 254; green, 5; blue, 5 }  ,fill opacity=1 ] (200,51.5) .. controls (200,49.29) and (201.79,47.5) .. (204,47.5) .. controls (206.21,47.5) and (208,49.29) .. (208,51.5) .. controls (208,53.71) and (206.21,55.5) .. (204,55.5) .. controls (201.79,55.5) and (200,53.71) .. (200,51.5) -- cycle ;
\draw  [draw opacity=0][fill={rgb, 255:red, 254; green, 5; blue, 5 }  ,fill opacity=1 ] (334,50.75) .. controls (334,48.54) and (335.79,46.75) .. (338,46.75) .. controls (340.21,46.75) and (342,48.54) .. (342,50.75) .. controls (342,52.96) and (340.21,54.75) .. (338,54.75) .. controls (335.79,54.75) and (334,52.96) .. (334,50.75) -- cycle ;
\draw  [draw opacity=0][fill={rgb, 255:red, 126; green, 211; blue, 33 }  ,fill opacity=1 ] (130,185.5) .. controls (130,183.29) and (131.79,181.5) .. (134,181.5) .. controls (136.21,181.5) and (138,183.29) .. (138,185.5) .. controls (138,187.71) and (136.21,189.5) .. (134,189.5) .. controls (131.79,189.5) and (130,187.71) .. (130,185.5) -- cycle ;
\draw  [draw opacity=0][fill={rgb, 255:red, 126; green, 211; blue, 33 }  ,fill opacity=1 ] (260,186) .. controls (260,183.79) and (261.79,182) .. (264,182) .. controls (266.21,182) and (268,183.79) .. (268,186) .. controls (268,188.21) and (266.21,190) .. (264,190) .. controls (261.79,190) and (260,188.21) .. (260,186) -- cycle ;
\draw  [draw opacity=0][fill={rgb, 255:red, 254; green, 5; blue, 5 }  ,fill opacity=1 ] (468,51.5) .. controls (468,49.29) and (469.79,47.5) .. (472,47.5) .. controls (474.21,47.5) and (476,49.29) .. (476,51.5) .. controls (476,53.71) and (474.21,55.5) .. (472,55.5) .. controls (469.79,55.5) and (468,53.71) .. (468,51.5) -- cycle ;
\draw  [draw opacity=0][fill={rgb, 255:red, 126; green, 211; blue, 33 }  ,fill opacity=1 ] (398,185.5) .. controls (398,183.29) and (399.79,181.5) .. (402,181.5) .. controls (404.21,181.5) and (406,183.29) .. (406,185.5) .. controls (406,187.71) and (404.21,189.5) .. (402,189.5) .. controls (399.79,189.5) and (398,187.71) .. (398,185.5) -- cycle ;
\draw  [draw opacity=0][fill={rgb, 255:red, 126; green, 211; blue, 33 }  ,fill opacity=1 ] (524,186) .. controls (524,183.79) and (525.79,182) .. (528,182) .. controls (530.21,182) and (532,183.79) .. (532,186) .. controls (532,188.21) and (530.21,190) .. (528,190) .. controls (525.79,190) and (524,188.21) .. (524,186) -- cycle ;
\draw [fill={rgb, 255:red, 252; green, 3; blue, 3 }  ,fill opacity=1 ] [dash pattern={on 4.5pt off 4.5pt}]  (472,51.5) -- (402,185.5) ;
\draw  [dash pattern={on 4.5pt off 4.5pt}]  (472,51.5) -- (532,186) ;

\draw (174,23.4) node [anchor=north west][inner sep=0.75pt]    {$\lambda _{c,-1}$};
\draw (475,21.4) node [anchor=north west][inner sep=0.75pt]    {$\lambda _{c,+1}$};
\draw (333,27.4) node [anchor=north west][inner sep=0.75pt]    {$0$};
\draw (251,186.4) node [anchor=north west][inner sep=0.75pt]    {$\lambda _{-1}$};
\draw (120,187.4) node [anchor=north west][inner sep=0.75pt]    {$\lambda _{-2}$};
\draw (515,188.4) node [anchor=north west][inner sep=0.75pt]    {$\lambda _{+2}$};
\draw (389,189.4) node [anchor=north west][inner sep=0.75pt]    {$\lambda _{+1}$};
\draw (597,175.4) node [anchor=north west][inner sep=0.75pt]    {$L$};
\draw (587,41.4) node [anchor=north west][inner sep=0.75pt]    {$L_{0} \in \mathcal{N}^{3}$};
\end{tikzpicture}
\caption{Index relation for $L\rightarrow L_0\in \mathcal{N}^3$}
\label{fig: index relation-2}
\end{figure}
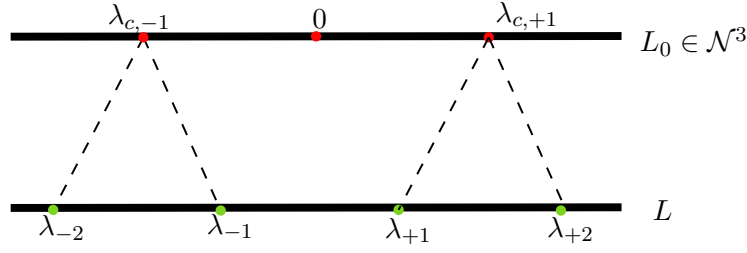

    \item If $L_0\in \mathcal{N}^2$, then $\Lambda_E=\{-\frac{N_0}{2},\cdots,-1,1.\cdots,\frac{N_0}{2}\}$ with $N_L=\frac{N_0}{2}$. Moreover, for each critical eigenvalue $\ii\lambda_{c,m},1\leq |m|\leq \frac{N_0-1}{2}$, there is a unique perturbed elliptic eigenvalue $\ii\lambda_{m}$ such that 
 $\lim_{L\rightarrow L_0}\lambda_m(L)=\lambda_{c,m}$, for $1\leq |j|\leq \frac{N_0}{2}$.
 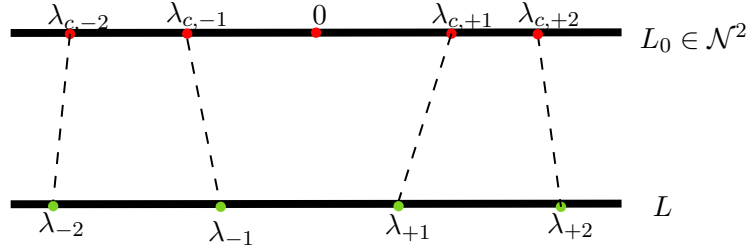
\begin{figure}[h]
\centering
\tikzset{every picture/.style={line width=0.75pt}} 

\begin{tikzpicture}[x=0.75pt,y=0.75pt,yscale=-0.65,xscale=0.65]

\draw [line width=3]    (101,51) -- (575,50.5) ;
\draw [line width=3]    (101,184) -- (575,183.5) ;
\draw  [dash pattern={on 4.5pt off 4.5pt}]  (238,55.5) -- (264,186) ;
\draw  [draw opacity=0][fill={rgb, 255:red, 254; green, 5; blue, 5 }  ,fill opacity=1 ] (234,51.5) .. controls (234,49.29) and (235.79,47.5) .. (238,47.5) .. controls (240.21,47.5) and (242,49.29) .. (242,51.5) .. controls (242,53.71) and (240.21,55.5) .. (238,55.5) .. controls (235.79,55.5) and (234,53.71) .. (234,51.5) -- cycle ;
\draw  [draw opacity=0][fill={rgb, 255:red, 254; green, 5; blue, 5 }  ,fill opacity=1 ] (334,50.75) .. controls (334,48.54) and (335.79,46.75) .. (338,46.75) .. controls (340.21,46.75) and (342,48.54) .. (342,50.75) .. controls (342,52.96) and (340.21,54.75) .. (338,54.75) .. controls (335.79,54.75) and (334,52.96) .. (334,50.75) -- cycle ;
\draw  [draw opacity=0][fill={rgb, 255:red, 126; green, 211; blue, 33 }  ,fill opacity=1 ] (130,185.5) .. controls (130,183.29) and (131.79,181.5) .. (134,181.5) .. controls (136.21,181.5) and (138,183.29) .. (138,185.5) .. controls (138,187.71) and (136.21,189.5) .. (134,189.5) .. controls (131.79,189.5) and (130,187.71) .. (130,185.5) -- cycle ;
\draw  [draw opacity=0][fill={rgb, 255:red, 126; green, 211; blue, 33 }  ,fill opacity=1 ] (260,186) .. controls (260,183.79) and (261.79,182) .. (264,182) .. controls (266.21,182) and (268,183.79) .. (268,186) .. controls (268,188.21) and (266.21,190) .. (264,190) .. controls (261.79,190) and (260,188.21) .. (260,186) -- cycle ;
\draw  [draw opacity=0][fill={rgb, 255:red, 254; green, 5; blue, 5 }  ,fill opacity=1 ] (439,51.5) .. controls (439,49.29) and (440.79,47.5) .. (443,47.5) .. controls (445.21,47.5) and (447,49.29) .. (447,51.5) .. controls (447,53.71) and (445.21,55.5) .. (443,55.5) .. controls (440.79,55.5) and (439,53.71) .. (439,51.5) -- cycle ;
\draw  [draw opacity=0][fill={rgb, 255:red, 126; green, 211; blue, 33 }  ,fill opacity=1 ] (398,185.5) .. controls (398,183.29) and (399.79,181.5) .. (402,181.5) .. controls (404.21,181.5) and (406,183.29) .. (406,185.5) .. controls (406,187.71) and (404.21,189.5) .. (402,189.5) .. controls (399.79,189.5) and (398,187.71) .. (398,185.5) -- cycle ;
\draw  [draw opacity=0][fill={rgb, 255:red, 126; green, 211; blue, 33 }  ,fill opacity=1 ] (524,186) .. controls (524,183.79) and (525.79,182) .. (528,182) .. controls (530.21,182) and (532,183.79) .. (532,186) .. controls (532,188.21) and (530.21,190) .. (528,190) .. controls (525.79,190) and (524,188.21) .. (524,186) -- cycle ;
\draw [fill={rgb, 255:red, 252; green, 3; blue, 3 }  ,fill opacity=1 ] [dash pattern={on 4.5pt off 4.5pt}]  (443,51.5) -- (402,181.5) ;
\draw  [dash pattern={on 4.5pt off 4.5pt}]  (510,52) -- (528,186) ;
\draw  [draw opacity=0][fill={rgb, 255:red, 254; green, 5; blue, 5 }  ,fill opacity=1 ] (506,52) .. controls (506,49.79) and (507.79,48) .. (510,48) .. controls (512.21,48) and (514,49.79) .. (514,52) .. controls (514,54.21) and (512.21,56) .. (510,56) .. controls (507.79,56) and (506,54.21) .. (506,52) -- cycle ;
\draw  [draw opacity=0][fill={rgb, 255:red, 254; green, 5; blue, 5 }  ,fill opacity=1 ] (143,51.75) .. controls (143,49.54) and (144.79,47.75) .. (147,47.75) .. controls (149.21,47.75) and (151,49.54) .. (151,51.75) .. controls (151,53.96) and (149.21,55.75) .. (147,55.75) .. controls (144.79,55.75) and (143,53.96) .. (143,51.75) -- cycle ;
\draw  [dash pattern={on 4.5pt off 4.5pt}]  (147,51.75) -- (134,181.5) ;

\draw (219,24.4) node [anchor=north west][inner sep=0.75pt]    {$\lambda _{c,-1}$};
\draw (425,25.4) node [anchor=north west][inner sep=0.75pt]    {$\lambda _{c,+1}$};
\draw (333,27.4) node [anchor=north west][inner sep=0.75pt]    {$0$};
\draw (252,194.4) node [anchor=north west][inner sep=0.75pt]    {$\lambda _{-1}$};
\draw (120,187.4) node [anchor=north west][inner sep=0.75pt]    {$\lambda _{-2}$};
\draw (515,188.4) node [anchor=north west][inner sep=0.75pt]    {$\lambda _{+2}$};
\draw (389,189.4) node [anchor=north west][inner sep=0.75pt]    {$\lambda _{+1}$};
\draw (597,175.4) node [anchor=north west][inner sep=0.75pt]    {$L$};
\draw (587,41.4) node [anchor=north west][inner sep=0.75pt]    {$L_{0} \in \mathcal{N}^{2}$};
\draw (127,26.4) node [anchor=north west][inner sep=0.75pt]    {$\lambda _{c,-2}$};
\draw (492,24.4) node [anchor=north west][inner sep=0.75pt]    {$\lambda _{c,+2}$};
\end{tikzpicture}
\caption{Index relation for $L\rightarrow L_0\in\mathcal{N}^2$}
\label{fig: index relation-3} 
\end{figure} 
\end{enumerate}
\end{prop}
Here we omit its proof and one can find it in the Appendix \ref{sec: proof of corollary index-lables}. Moreover, for the definitions of $\sigma^{\pm}(j)$ appearing in Corollary \ref{coro: uniform est for bi-family}, we also refer to the proof in the Appendix \ref{sec: proof of corollary index-lables}.\\

When $L$ approaches $L_0\in\mathcal{N}$, for different stationary KdV operators $\A$ and $\B$, we observe different behaviors. For $\A$, recall the eigenmodes $(\zeta,\F_{\zeta})$ satisfying
\begin{equation*}
\left\{
\begin{array}{l}
     \F_{\zeta}'''(x)+\F_{\zeta}'(x)+\zeta \F_{\zeta}(x)=0,x\in(0,L),  \\
     \F_{\zeta}(0)=\F_{\zeta}(L)= \F_{\zeta}'(L)=0.
\end{array}
\right.
\end{equation*}
Due to Proposition \ref{prop: asymptotic expansion for A0}, the asymptotic behavior for $(\zeta,\F_{\zeta})$  is consistent across all types of $(k,l)$, 
\begin{itemize}
    \item For any $(k,l)$ solving \eqref{eq:klIC}, thanks to Proposition \ref{prop: asymptotic expansion for A0},
    \[
    \zeta=\ii\frac{(2k+l)(k-l)(2l+k)}{3\sqrt{3}(k^2+k l+l^2)^{\frac{3}{2}}}+\bigO((L-L_0)^2).
    \]
    Furthermore, for boundary derivatives, we also have $|\F_{\zeta}'(0)|=\bigO(|L-L_0|)$.
\end{itemize}
Whereas for $\B$, the situation becomes more complex. Recall the elliptic eigenmode $(\ii\lambda_{j},\E_j)_{j\in\Lambda_E}$ for $\B$
\begin{equation*}
\left\{
\begin{array}{l}
     \E_j'''+\E_j'+\ii\lambda_j\E_j=0,  \\
     \E_j(0)=\E_j(L)=\E_j'(0)-\E_j'(L)=0.
\end{array}
\right.   
\end{equation*}
Let $(k_m,l_m)$ solve \eqref{eq:klIC} and $\lambda_{c,m}=\frac{(2k_m+l_m)(k_m-l_m)(2l_m+k_m)}{3\sqrt{3}(k_m^2+k_m l_m+l_m^2)^{\frac{3}{2}}}$. $\G_{c,m}$ and $\Tilde{\G}_{c,m}$ are defined in \eqref{eq: exact formula for critical eigenfunctions} and \eqref{eq: defi of tilde-G}. Due to Proposition \ref{prop: Index set M-E}, we characterize elliptic eigenmodes based on different types of $(k_m,l_m)$.
\begin{enumerate}
    \item If $(k_m,l_m)\in\mathcal{S}_3$,  $\lambda_m=\lambda_{c,m}+\bigO((L-L_0)^2)$ and for boundary derivatives, $|\E_m'(L)|=|\E_m'(0)|=\bigO(|L-L_0|)$.
    \item If $(k_m,l_m)\in\mathcal{S}_2$, $\lambda_{2m-1}=\lambda_{c,m}+\bigO(|L-L_0|)$ and $\lambda_{2m}=\lambda_{c,m}+\bigO(|L-L_0|)$. Furthermore, for boundary derivatives, $|\E'_{2m-1}(L)|=\bigO(1)$ and $|\E'_{2m}(L)|=\bigO(1)$. 
    \item If $(k_m,l_m)\in\mathcal{S}_1$, in particular, we obtain two eigenvalues close to $0$, i.e., $\ii\lambda_{\pm1}=\pm\frac{\left|L-L_0\right|}{\pi \sqrt{3}k_m}$.  Furthermore, for boundary derivatives, we also have $|\E_1'(L)|=\bigO(1)$ and $|\E_{-1}'(L)|=\bigO(1)$. 
\end{enumerate}

In summary, we obtain the following two tables to show the asymptotic behaviors of eigenvalues and eigenfunctions of $\B$ and $\A$ near critical lengths.
\begin{table}[]
    \centering
\begin{tabular}{|c|c|c|c|}
\hline
  &$\mathcal{S}_1$ & $\mathcal{S}_2$ & $\mathcal{S}_3$ \\
\hline
Eigenvalues &$\bigO(|L-L_0|)$ & $\bigO(|L-L_0|)$ & $\bigO((L-L_0)^2)$\\
\hline
Boundary derivatives &$|\E'(L)|=\bigO(1)$& $|\E'(L)|=\bigO(1)$ &  $|\E'(L)|=\bigO(|L-L_0|)$ \\
\hline
\end{tabular}
    \caption{Different asymptotic behaviors for eigenmodes of $\B$}
    \label{tab: asym-B}
\end{table}
\begin{table}[]
    \centering
\begin{tabular}{|c|c|c|c|}
\hline
 $\A$ &$\mathcal{S}_1$  & $\mathcal{S}_2$  &$\mathcal{S}_3$   \\
\hline
Eigenvalues &$\bigO((L-L_0)^2)$ & $\bigO((L-L_0)^2)$ & $\bigO((L-L_0)^2)$ \\
\hline
Boundary derivative &$|\F_{\zeta}'(0)|=\bigO(|L-L_0|)$ & $|\F_{\zeta}'(0)|=\bigO(|L-L_0|)$& $|\F_{\zeta}'(0)|=\bigO(|L-L_0|)$ \\
\hline
\end{tabular}
    \caption{Consistent asymptotic behaviors for eigenmodes of $\A$}
    \label{tab: asym-A}
\end{table}

\subsubsection{Perturbation of the unreachable space}\label{sec: quasi-invariant subsapce}
Now we turn to comparing the eigenvectors in the elliptic subspace and the unreachable subspace. Depending on different types of $(k_m,l_m)$ at the critical length, we have different eigenfunction approximations due to Proposition \ref{prop: Low-frequency behaviors: Singular limits}. We give the following two examples:
\begin{exa}[$L_0=2\pi\sqrt{\frac{7}{3}}$]
We begin with a simple case. For $L_0=2\pi\sqrt{\frac{7}{3}}$, it is in $\mathcal{N}^2$ and only one pair solve \eqref{eq:klIC}, i.e. $k=2,l=1$. At $2\pi\sqrt{\frac{7}{3}}$, we have two unreachable eigenmodes 
\begin{gather*}
(\ii\lambda_{c,1},\G_1)=(\ii\frac{20}{21\sqrt{21}},e^{\frac{5 \ii x}{\sqrt{21}}}-3 e^{-\frac{\ii x}{\sqrt{21}}} + 2 e^{-\frac{4 \ii x}{\sqrt{21}}})\\
(\ii\lambda_{c,-1},\G_{-1})=(-\ii\frac{20}{21\sqrt{21}},e^{-\frac{5 \ii x}{\sqrt{21}}}-3 e^{\frac{\ii x}{\sqrt{21}}} + 2 e^{\frac{4 \ii x}{\sqrt{21}}})
\end{gather*}
 For $L$ near $2\pi\sqrt{\frac{7}{3}}$, there are two perturbed elliptic eigenmodes: $(\ii\lambda_1,\E_1)$ and $(\ii\lambda_{-1},\E_{-1})$ satisfying
\begin{gather*}
\lambda_{1}=\lambda_{c,1}+\bigO(|L-2\pi\sqrt{\frac{7}{3}}|^2),\;\lambda_{-1}=-\lambda_{c,-1}+\bigO(|L-2\pi\sqrt{\frac{7}{3}}|^2),\\
 \E_{1}(x)=\frac{\alpha_{1}}{2}\G_1+\bigO(|L-2\pi\sqrt{\frac{7}{3}}|) ,\;
    \E_{-1}(x)=\frac{\alpha_{1}}{2}\G_{-1} +\bigO(|L-2\pi\sqrt{\frac{7}{3}}|).
\end{gather*} 
In this case, the unreachable space is $M(L_0)=\textrm{Span}\{\Re\G_1,\Im\G_{1}\}$, and the elliptic subspace for $L$ near $2\pi\sqrt{\frac{7}{3}}$ is $U_E(L)=\textrm{Span}\{\Re\E_1,\Im\E_{1}\}\approx M(L_0)+\bigO(|L-2\pi\sqrt{\frac{7}{3}}|)$. 
\end{exa}
\begin{exa}[$L_0=2\pi$]
We observe a different situation. For $L_0=2\pi$, it is in $\mathcal{N}^1$ and the unique unreachable eigenmode is $(0,\frac{1-\cos{x}}{\sqrt{3\pi}})$. However, there are two perturbed elliptic eigenmodes: $(\ii\lambda_1,\E_1)$ and $(\ii\lambda_{-1},\E_{-1})$ satisfying
\begin{gather*}
\lambda_{1}=\frac{1}{\sqrt{3}\pi}|L-2\pi|+\bigO(|L-2\pi|^2),\;\lambda_{-1}=-\frac{1}{\sqrt{3}\pi}|L-2\pi|+\bigO(|L-2\pi|^2),\\
 \E_{1}(x)=\frac{\sqrt{3}(1-\cos{x})-3\ii\sin{x}}{3\sqrt{2\pi}}+\bigO(|L-2\pi|) ,\;
    \E_{-1}(x)=\frac{\sqrt{3}(1-\cos{x})+3\ii\sin{x}}{3\sqrt{2\pi}}+\bigO(|L-2\pi|).
\end{gather*} 
In this case, the eigenfunctions $\E_{\pm1}$ involve not only the Type 1 eigenfunction $1-\cos{x}$ but also the Type 2 eigenfunction
$\ii\sin{x}$ (see Section \ref{sec: Eigenvalues and eigenfunctions at the critical length}). For $L$ near $2\pi$, we have
\[
U_E(L)=\textrm{Span}\{\E_1,\E_{-1}\}=\textrm{Span}\{\E_1+\E_{-1},\E_{1}-\E_{-1}\}\approx \textrm{Span}\{1-\cos{x},\ii\sin{x}\}+\bigO(|L-2\pi|).
\]
On the other hand, the unreachable space $M(L_0)=\textrm{Span}\{1-\cos{x}\}$. Thus, in this simple case, $U_E(L)$ contains twice as many directions as $M(L_0)$, which implies that $U_E(L)$ is not a perturbation of $M(L_0)$.

In order to get a perturbed subspace of $M(L_0)$, we need to delete the irrelevant directions in $U_E(L)$. In this case, we can easily observe that $\E_1+\E_{-1}=\frac{2\sqrt{3}(1-\cos{x})}{3\sqrt{2\pi}}+\bigO(|L-2\pi|)$ and $\E_1-\E_{-1}=\frac{2\ii\sin{x}}{\sqrt{2\pi}}+\bigO(|L-2\pi|)$. Thus, we drop the direction $\E_1-\E_{-1}$ and consider a subspace generated by $\E_1+\E_{-1}$, i.e., $\textrm{Span}\{\E_1+\E_{-1}\}\approx M(L_0)+\bigO(|L-2\pi|)$.
\end{exa}
Through these two examples, we summarize 
\begin{enumerate}
    \item For $L_0\in \mathcal{N}^2$, the elliptic subspace $U_E(L)$ can be seen as a perturbation of $M(L_0)$;
    \item For $L_0\in \mathcal{N}^1\cup \mathcal{N}^3$, due to the involvement of Type 2 eigenfunctions, we need to define the quasi-invariant subspace $M_{\B}(L)$ (see below) to denote the perturbation of $M(L_0)$.
\end{enumerate}
As we all know, the unreachable subspace is generated only by Type 1 eigenfunctions $\{\G_{m}\}_{m\in\Lambda_E}$. Therefore, for each pair $(k_m,l_m)\in \mathcal{S}_1\cup\mathcal{S}_2$, we define the atom subspace as follows:
\begin{equation}\label{eq: defi-atom-space}
V_{k_m,l_m}:=\left\{
\begin{array}{cc}
   \rm{Span}\{\E_{1}+\E_{-1}\},  &(k_m,l_m)\in  \mathcal{S}_1 ,\\
    \textrm{Span}\{a^+_m\E_{2m-1}+b^+_m\E_{2m},\overline{a^+_m}\E_{-2m+1}+\overline{b^+_m}\E_{-2m}\}, & (k_m,l_m)\in  \mathcal{S}_2,\\
    \textrm{Span}\{\E_m,\E_{-m}\},&(k_m,l_m)\in  \mathcal{S}_3.
\end{array}
\right.
\end{equation}
Here the coefficients $a^+_m,b^+_m$ are well-chosen such that $a^+_m\E_{2m-1}+b^+_m\E_{2m}\sim\G_{m}$ and $\overline{a^+_m}\E_{-2m+1}+\overline{b^+_m}\E_{-2m}\sim\G_{-m}$. We point out that for each $(k_m,l_m)$, the atom space $V_{k_m,l_m}$ is not an eigenspace. Moreover, for $(k_m,l_m)\in \mathcal{S}_1\cup\mathcal{S}_2$, $V_{k_m,l_m}$ is not even an invariant subspace for $\B$.

Near different critical lengths, we define the quasi-invariant subspaces for the operator $\B$
\begin{equation}\label{eq: defi-quasi-space-B}
M_{\B}(L):=\oplus_{m\in\Lambda_E(L_0)} V_{k_m,l_m},\text{ for $L$ near $L_0$}.
\end{equation}
In general, $M_{\B}(L)$ is not an invariant subspace for $\B$, but as $L\rightarrow L_0$, its limit is the unreachable subspace, which is invariant for $\B$. That is a reason we call it \textit{quasi-invariant}.

In addition, similarly, we can also define the quasi-invariant subspaces for the operator $\A$
\begin{equation}\label{eq: defi-quasi-space-A}
M_{\A}(L):=\oplus_{m\in\Lambda_E(L_0)}\textrm{Span}\{\F_{\zeta_m}: \zeta_m=\ii\lambda_{c,m}(k_m,l_m)+\bigO((L-L_0)^2)\}.
\end{equation}
We shall see that these quasi-invariant subspaces play an important role when we derive the sharp decay rates in the next section.

\section{Part II: Sharp stability analysis}\label{sec: Part II: Sharp stability analysis}
In this section, we refine our transition-stabilization method to prove Theorem \ref{thm: main theorem linear version}. At first, we present our revised transition-stabilization package in $[0,T_0]$ as follows:\\

\begin{tikzpicture}[
    auto,
    block/.style={
        rectangle, draw, 
        text width=9em, text centered, rounded corners, minimum height=4em
    },
    arrow/.style={
        -{Latex[width=2mm,length=3mm]}, thick
    },
    Arrow/.style={
        {Latex[width=2mm,length=3mm]}-{Latex[width=2mm,length=3mm]}, thick
    }
]

\node (initial)[block, text width=6em] {$y_0\in H_{\A}(L)$ };
\node (kato) [block, right= 3cm of initial] {$y(\frac{T_0}{2})\in D(\A^2)$};
\node (z) [block, above=of kato] {$z^0\in H_{\B}(L)\cap D(\B^2)$};
\node (h)[block, below=of kato]{$\sum_{j=1}^{N_0+2}c_j h_{\mu_j}$};
\node (0)[block, right= 3.5cm of z,text width=6em]{$z^{T_0}\in H_{\B}(L)$};
\node (size)[block,right=3cm of h]{$\sum_{j=1}^{N_0+2}c_j e^{-\mu_j\frac{T_0}{2}}h_{\mu_j}$};
\node (final)[block, right= 3cm of kato]{$y(T_0)\in H_{\A}(L)$ 
};

\draw [arrow] (initial) -- (kato) node[midway, above] {Smoothing effects} ;
\draw [arrow] (kato) -- (z) node[midway, left] {\textcolor{red}{$\mathcal{T}_c$}};
\draw [arrow] (kato) -- (h) node[midway, left] {\textcolor{red}{$\mathcal{T}_c$}};
\draw [arrow] (z)--(0) node[midway, above] {Constructive } node[midway,below] {\textcolor{red}{exact}-control};
\draw [arrow] (h)--(size) node[midway, above] {Free flow} ;
\draw[arrow](0)--(final) node[midway, right] {\textcolor{red}{$\mathcal{T}_{\varrho}$}};
\draw[arrow](size)--(final) node[midway, right] {\textcolor{red}{$\mathcal{T}_{\varrho}$}};
\end{tikzpicture}\\
Details are provided in Section \ref{sec: details-stage-2}.

We begin by introducing the definitions of our transition maps in Section \ref{sec: Projections and state space decomposition}. Following this, we present three model cases to illustrate our refined transition-stabilization method.
\begin{enumerate}
    \item In section \ref{sec: Around Type I unreachable pair}, we specifically focus on the case involving a Type I unreachable pair\footnote{This is defined in Definition \ref{defi: types of unreachable pairs}.}. Choosing $2\pi$ as a model, the typical feature of this case is the presence of $0$ as an eigenvalue at the critical length, which makes this case unique in two distinct aspects.\\
   As discussed in Section \ref{sec: limiting analysis on different types}, there are two perturbed eigenvalues $\ii\lambda_{\pm1}$, close to $0$. Technically, during integration by parts, we may encounter the factor $\frac{1}{\ii\lambda_{\pm1}}$, which diverges to $\infty$. This indicates that we must exercise extra caution in the proof compared to other cases. \\
    Furthermore, the perturbed elliptic eigenfunctions, $\E_{\pm1}$ exhibit additional symmetry not present in other cases, which makes the construction of the bi-orthogonal family distinct from other situations.
    \item In section \ref{sec: Around Type II unreachable pair}, we focus on the model case $2\pi\sqrt{7}$ involving a Type II unreachable pair. Here, the dimension of the elliptic subspace $U_E(L)$ is twice that of the unreachable space $M(L_0)$. This distinction suggests a different method for constructing the bi-orthogonal family compared to the approach used in Section \ref{sec: A transition-stabilization method}.
    \item In section \ref{sec: Around Type III unreachable pair}, a model case $2\pi\sqrt{\frac{7}{3}}$ is considered. In this case, the quasi-invariant subspace $M_{\B}(L)$ coincides with $U_E(L)$. We employ the same bi-orthogonal family as in Section \ref{sec: A transition-stabilization method} with a particular focus on its dependence on $L$. 
\end{enumerate}

\subsection{Transition projections and state space decomposition}\label{sec: Projections and state space decomposition}
 As we presented in Section \ref{sec: duality arguments and HUM}, we have a decomposition for $L^2(0,L)$ depending on a projection operator $\Pi$, as defined in \eqref{defi: finite-co-dimensional projector}. 
We use the definition of quasi-invariant subspaces $M_{\B}(L)$ in \eqref{eq: defi-quasi-space-B} and $M_{\A}(L)$ in \eqref{eq: defi-quasi-space-A}. We have the following two decompositions for $L^2(0,L)$.
\begin{equation}
L^2(0,L)=M_{\A}(L)\oplus H_{\A}(L) \textrm{ and }
L^2(0,L)=M_{\B}(L)\oplus H_{\B}(L),
\end{equation}
defined for $\A$ and $\B$ respectively. Roughly speaking, we shall prove the exponential stability for \eqref{eq: linear KdV-stability-intro} in subspace $H_{\A}(L)$ with a uniform decay rate independent of $L$, while in $M_{\A}(L)$, with a sharp decay rate $\sim|L-L_0|^2$. By applying our transition-stabilization method, we need to establish a link between the original KdV system and the intermediate system using the modulated functions $h_{\mu}$. Therefore, we introduce the following two transition projections $\mathcal{T}_c$ and $\mathcal{T}_{\varrho}$, which specify the relation between $D(\A^2)\cap H_{\A}(L)$ and $D(\B^2)\cap H_{\B}(L)$. 
\begin{prop}\label{prop: transition projection-c}
Let  $\{h_{\mu_j}\}_{1\leq j\leq N_0+2}$ be modulated functions defined in Section \ref{sec: preliminary}. The transition projector $\mathcal{T}_c$ is defined by
\begin{gather*}
\mathcal{T}_c: D(\A^2)\cap H_{\A}(L)\rightarrow \C^{N_{0}+2},\\
z\mapsto (c_j)_{1\leq j\leq N_0+2},
\end{gather*}
such that $z-\sum_j c_jh_{\mu_j}\in D(\B^2)\cap H_{\B}(L)$. Moreover, the coefficients $(c_j)_{1\leq j\leq N_0+2}$ are uniformly bounded w.r.t $L$.
\end{prop}
Before the introduction of the second transition map, we recall the quasi-invariant subspace $M_{\B}(L)$ (see \eqref{eq: defi-atom-space} and \eqref{eq: defi-quasi-space-B}). For different situations, we choose different well-prepared directions in $H_{\B}(L)$. More precisely, for $L_0\in\mathcal{N}$ and $L$ near $L_0$,
\begin{enumerate}
    \item If $L_0\in \mathcal{N}^2$, $N_0$ is even and we choose $N_0$ eigenfunctions of $\B$, i.e. $\{E_m\}_{1\leq |m|\leq \frac{N_0}{2}}=\{\E_j,\E_{-j}\}_{1\leq j\leq \frac{N_0}{2}}$.
    \item If $L_0\in \mathcal{N}^3$ and $N_0$ is even, then we set $\{E_m\}_{1\leq |m|\leq \frac{N_0}{2}}=\{a^-_n\E_{2n-1}+b^-_n\E_{2n},-\overline{a^-_n}\E_{1-2n}-\overline{b^-_n}\E_{-2n}\}_{1\leq n\leq \frac{N_0}{2}}$.
    \item If $L_0\in \mathcal{N}^3$ and $N_0$ is odd, then we set $E_0=\E_{1}-\E_{-1}$ and $\{E_m\}_{1\leq |m|\leq \frac{N_0-1}{2}}=\{a^-_n\E_{2n}+b^-_n\E_{2n+1},-\overline{a^-_n}\E_{-2n}-\overline{b^-_n}\E_{-2n-1}\}_{1\leq n\leq \frac{N_0-1}{2}}$.
    \item If $L_0\in \mathcal{N}^3$, $N_0=1$, and $E_0=\E_{1}-\E_{-1}$.
\end{enumerate}
Here the coefficients $a^-_n,b^-_n$ are well-chosen such that $a^-_n\E_{2n-1}+b^-_n\E_{2n}\sim\widetilde{\G_n}$, etc.
\begin{prop}
Let  $\{E_{m}\}_{|m|\leq \frac{N_0}{2}}\subset H_{\B}(L)$ be linearly independent in $H_{\B}(L)$ (see Lemma \ref{lem: construction of v-2pi}, \ref{lem: construction of v-2pi-sqrt-7} and \ref{lem: construction of v-7/3pi} for explicit constructions of $E_j$ in different situations). Let  $\{h_{\mu_j}\}_{1\leq j\leq N_0+2}$ be the modulated functions defined in Proposition \ref{prop: transition projection-c}. The transition projector $\mathcal{T}_{\varrho}$ is defined by
\begin{gather*}
\mathcal{T}_{\varrho}: D(\B^2)\cap H_{\B}(L)\rightarrow \C^{N_0},\\
z=\sum_{|j|\leq \frac{N_0}{2}}\varrho_j E_j\mapsto (\varrho_j)_{|j|\leq \frac{N_0}{2}},
\end{gather*}
such that such that $z+\sum_j c_je^{-\mu_j T}h_{\mu_j}\in D(\A^2)\cap H_{\A}(L)$. Moreover, the coefficients $(\varrho_j)_{|j|\leq \frac{N_0}{2}}$ are uniformly bounded w.r.t $L$.
\end{prop}
\subsection{Around Type I unreachable pair $(k, l)$}\label{sec: Around Type I unreachable pair}
In this section, we present the features near $\mathcal{N}^1$ critical lengths through an example modal case $L_0=2\pi$. The main result in this part is the following:
\begin{thm}\label{thm: main result in modal case-2pi}
Let $T>0$ and $I=[2\pi- 1,2\pi+ 1]$.  For every $L\in I\setminus\{2\pi\}$, and $\forall y^0\in H_{\A}(L)$, there exists a constant $C=C(T)=\Tilde{\mathcal{K}}e^{\frac{\mathcal{K}}{\sqrt{T}}}$, which is independent of $L$, such that the following quantitative observability inequality 
\begin{equation}\label{eq: quantitative-ob-main-2pi}
   \|S(T)y^0\|^2_{L^2(0,L)} \leq  C^2\int_0^T|\p_x y(t,0)|^2dt
\end{equation}
holds for any solution $y$ to \eqref{eq: linear KdV-stability-intro}.
\end{thm}
\subsubsection{Quasi-invariant subspaces and transition maps}
We begin by recalling some basic properties. 
\begin{prop}
For $L_0=2\pi$, there is a unique eigenmode $(\lambda_c,\G_c)=(0,\frac{1-\cos{x}}{\sqrt{3}\pi})$ such that $\G_c'''(x)+\G_c'(x)=0$ with $\G_c(0)=\G_c(2\pi)=\G_c'(0)=\G_c'(2\pi)=0$.
\end{prop}
For $L_0=2\pi$, the Condition (C) becomes 
\begin{assu}\label{ass: Compact intervals including 2pi}
Suppose that $0<\delta<\pi$ is sufficiently small. Let $I=[2\pi-\delta,2\pi+\delta]$ be a small compact interval such that $I\cap \mathcal{N}=\{2\pi\}$.   
\end{assu}
\begin{lem}\label{lem: real eigenvalue close to 0}
Suppose that Assumption \ref{ass: Compact intervals including 2pi} holds. Suppose that $\zeta_0\in\R$ is the first eigenvalue close to $0$, then $\zeta_0$ and its associated real eigenfunction $\F_{\zeta_0}$ satisfy that $\zeta_0\sim(L-2\pi)^2,\;
    |\F_{\zeta_0}'(0)|\sim|L-2\pi|$.
\end{lem}
\begin{proof}
We can just apply Proposition \ref{prop: asymptotic expansion for A0} to this particular case $L_0=2\pi$. However, there is a slight difference. Here we can obtain the perturbed eigenvalue $\zeta_0\in\R$ instead of roughly saying that $\zeta_0\in\C$.
\end{proof}
Let us denote this real eigenvalue $\zeta_0$ and the associated eigenfunction $\F_{\zeta_0}$. By Proposition \ref{prop: Index set M-E}, we know that $\lambda_{1}$ and $\lambda_{-1}$ are two perturbed elliptic eigenvalues of $\B$ in the interval $(0,L)$ with $L\in [2\pi-\delta,2\pi+\delta]\setminus \{2\pi\}$. Moreover, we know that 
\begin{equation*}
\lambda_{1}=\frac{1}{\sqrt{3}\pi}|L-2\pi|+\bigO(|L-2\pi|^2),\;\lambda_{-1}=-\frac{1}{\sqrt{3}\pi}|L-2\pi|+\bigO(|L-2\pi|^2).
\end{equation*} 
We define $E_1$ and $E_{-1}$ by 
\begin{equation}\label{rem: defi-of E+ and E-}
E_1=\E_{1}+\E_{-1}=2\Re{\E_{1}},\;  E_{-1}=\E_{1}-\E_{-1}= 2\ii\Im{\E_{1}}. 
\end{equation}
\begin{prop}\label{coro: expansion of E+ and E-}
Suppose that Assumption \ref{ass: Compact intervals including 2pi} holds. Then, the following statements hold: 
\begin{enumerate}
    \item $E_1(x)$ and $E_{-1}(x)$ have the following asymptotic expansions near $2\pi$ 
    \begin{align*}
    E_1(x)=\sqrt{\frac{2}{3\pi}}(1-\cos{x})+\bigO(|L-2\pi|),\;
    E_{-1}(x)=-\ii\sqrt{\frac{2}{\pi}} \sin{x}+\bigO(|L-2\pi|).
    \end{align*}
    \item $E_1'(L)=\epsilon_0
(L-2\pi)+\bigO((L-2\pi)^2)$ and $E_{-1}'(L)=-\ii\sqrt{\frac{2}{\pi}}+\bigO(|L-2\pi|)$, where\\
$\epsilon_0=\frac{3 \ii \sqrt{2} - (8 + 13 \ii) \sqrt{6} - (12 - 128 \ii) \sqrt{2} \pi + 64 \sqrt{6} \pi^2}{192 \pi^{5/2}}\neq0$.
\end{enumerate}
\end{prop}
\begin{proof}
Here we simply apply Proposition \ref{prop: Low-frequency behaviors: Singular limits} in our particular case $L_0=2\pi$.
\end{proof}
Using the notations defined in Section \ref{sec: Projections and state space decomposition}, we have 
\begin{gather*}
    M_{\B}=Span\{E_1\},\;M_{\A}=Span\{F_{\zeta_0}\},\;M(L_0)=Span\{1-\cos{x}\}.
\end{gather*}
For $s\geq0$, we define the following Hilbert subspaces 
\begin{equation}\label{eq: defi for space h0 and h}
    H_{\A}^s:=\{u\in D(\A^s):\poscalr{u}{\F_{\zeta_0}}_{(0,L)}=0\},\;    H_{\B}^s:=\{u\in D(\B^s):\poscals{u}{E_1}_{L^2(0,L)}=0\}.
\end{equation}
\textbf{Compensate bi-orthogonal family}
We perform a similar procedure as we prove Theorem \ref{thm: main result with L-L0}. First, we recall the bi-orthogonal family to $e^{-\ii t\lambda_j}$ defined in Proposition \ref{prop: biorthogonal family}. Then, recalling Proposition \ref{prop: Index set M-E}, for the modal case $L_0=2\pi$, there exist two eigenvalues $\lambda_{\pm 1}$ such that $\lim_{L\rightarrow 2\pi}\lambda_{1}=\lim_{L\rightarrow 2\pi}\lambda_{-1}=0$ and $\lambda_{-1}=-\lambda_{1}$.
\begin{lem}\label{lem: bi-orthogonal family-2pi}
There exists a family of functions $\{\vartheta_j\}_{j\in\Z\backslash\{\pm 1\}}$ such that 
\begin{enumerate}
    \item For $j\neq 0,\pm 1$, $\vartheta_j=\phi_j$, where $\phi_j$ is defined in Proposition \ref{prop: biorthogonal family}.
    \item $\supp{\vartheta_j}\subset [-\frac{T}{2},\frac{T}{2}],\forall j\in\Z\backslash\{\pm 1\}$.\label{eq: compact support of theta_j-1}
    \item $\int_{-\frac{T}{2}}^{\frac{T}{2}}\vartheta_j(s)e^{-\ii \lambda_ks}ds=\delta_{jk},\forall j\in\Z\backslash\{0,\pm 1\}$ and $\forall k\in\Z\backslash\{0\}$. Moreover, 
    \begin{equation}
         \int_{-\frac{T}{2}}^{\frac{T}{2}}\vartheta_0(s)e^{-\ii \lambda_k s}ds=0,\forall k\neq 0, \pm 1,\;         \int_{-\frac{T}{2}}^{\frac{T}{2}}\vartheta_0(s)e^{-\ii s\lambda_{1}}ds=\int_{-\frac{T}{2}}^{\frac{T}{2}}\vartheta_0(s)e^{-\ii s\lambda_{-1}}ds=1.
    \end{equation}
    \label{eq: biorthorgonal-2pi}
    \item For $N\in \N$, there exists a constant $K=K(N)$ such that $\|\phi^{(m)}_j\|_{L^{\infty}(\R)}\leq Ce^{\frac{K}{\sqrt{T}}}|\lambda_j|^m,\forall j\in\Z\backslash\{\pm 1\}, \forall m\in\{0,1,\cdots,N\}$. Here the constant $C$ appearing in the inequality might depend on $N$ but not on $T$, $L$ and $j$.\label{eq: bound for derivative of theta-1}
\end{enumerate}
\end{lem}
\begin{proof}
We put the proof in Appendix \ref{sec: app-compensate-bi-2pi}.    
\end{proof}

\noindent\textbf{Transition maps near Type I critical lengths}
As we already know, if $y$ is a solution to \eqref{eq: linear KdV-stability-intro}, 
using the smoothing effects, for $y^0\in H_{\A}$, we obtain that $y(\frac{T}{2})\in H_{\A}^2$. Then we have the following lemma to compensate for the boundary condition of the KdV operator.
\begin{lem}\label{lem: defi of coeff-1-2-3}
Let $\mu_1$, $\mu_2$, and $\mu_3$ be three distinct real positive constants, for any real function $z\in  H_{\A}^2\subset H_{\A}$, there exist real constants $c_1$, $c_2$ and $c_3$ such that $z- c_1 h_{\mu_1}- c_2 h_{\mu_2}-c_3 h_{\mu_3}\in H_{\B}^2$. 
Here $h_{\mu_j}(1\leq j\leq3)$ are modulated functions.
\end{lem}
\begin{proof}
For any $\mu_1$, $\mu_2$, and $\mu_3$,  we are able to construct the real modulated functions $h_{\mu_1}$, $h_{\mu_2}$, and $h_{\mu_3}$.
Let $f:= z- c_1 h_{\mu_1}- c_2 h_{\mu_2}-c_3 h_{\mu_3}$. Then it is easy to check that $f(0)=f(L)=0$. For the boundary derivative, we ask that $ c_1+ c_2+c_3=-\p_xz(0)$ and $c_1\mu_1+ c_2\mu_2+c_3\mu_3=-\p_x(\mathcal{P}z)(0)$. 
    Next, for the projection on the direction $E_1$, we have $\poscals{z- c_1 h_{\mu_1}- c_2h_{\mu_2}-c_3 h_{\mu_3}}{E_1}_{L^2(0,L)}= 0$,
    which is equivalent to 
    \begin{gather*}
        c_1  \poscals{h_{\mu_1}}{E_1}_{L^2(0,L)}+ c_2 \poscals{h_{\mu_2}}{E_1}_{L^2(0,L)}+c_3\poscals{h_{\mu_3}}{E_1}_{L^2(0,L)}=  \poscals{z}{E_1}_{L^2(0,L)}
    \end{gather*}
    For the term $\poscals{h_{\mu_j}}{E_1}_{L^2(0,L)}$, we obtain 
\begin{align*}
\poscals{h_{\mu_j}}{\E_{\pm 1}}_{L^2(0,L)}=\int_0^Lh_{\mu_j}(x)\overline{\E_{\pm 1}(x)}dx
=\int_0^Lh_{\mu_j}(x)\overline{\E_{\pm 1}(x)}dx
=-\frac{\overline{\E'_{\pm 1}(L)}}{\ii\lambda_{\pm 1}}-\frac{\mu_j}{\ii\lambda_{\pm 1}}\int_0^Lh_{\mu_j}(x)\overline{\E_{\pm 1}(x)}dx.
\end{align*}
Thus, we obtain $\poscals{h_{\mu_j}}{\E_{1}}_{L^2(0,L)}=\frac{\overline{\E'_{1}(L)}}{\mu_j+\ii\lambda_{1}},\; \poscals{h_{\mu_j}}{\E_{-1}}_{L^2(0,L)}=\frac{\overline{\E'_{-1}(L)}}{\mu_j+\ii\lambda_{-1}}$,
which implies that 
\begin{align*}
\poscals{h_{\mu_j}}{E_1}_{L^2(0,L)}=\frac{\overline{\E'_{1}(L)}}{\mu_j+\ii\lambda_{1}}+\frac{\overline{\E'_{-1}(L)}}{\mu_j+\ii\lambda_{-1}}=\frac{\mu_j\overline{E'_1(L)}-\ii\lambda_{1}\overline{E'_{-1}(L)}}{\mu_j^2+\lambda_{1}^2},\;
\poscals{h_{\mu_j}}{E_{-1}}_{L^2(0,L)}
=\frac{\mu_j\overline{E'_{-1}(L)}-\ii\lambda_{1}\overline{E'_1(L)}}{\mu_j^2+\lambda_{1}^2}.
\end{align*}
Recalling the definitions in Remark \ref{rem: defi-of E+ and E-},  $E_1=\E_{1}+\E_{-1}=2\Re{\E_{1}}$ is a real function. However, $E_{-1}=\E_{1}-\E_{-1}=2\ii\Im{\E_{1}}$ is a purely imaginary function. Thus, the coefficient 
\begin{equation*}
\poscals{h_{\mu_j}}{E_1}_{L^2(0,L)}=\frac{\mu_j\overline{E'_1(L)}-\ii\lambda_{1}\overline{E'_{-1}(L)}}{\mu_j^2+\lambda_{1}^2}=\frac{\mu_jE'_1(L)+\ii\lambda_{1}E'_{-1}(L)}{\mu_j^2+\lambda_{1}^2} \in \R.   
\end{equation*}
Since $z$ is also a real function, we obtain $\poscals{z}{E_1}_{L^2(0,L)}\in\R$. 
Let 
\begin{equation*}
    M(c_1,c_2,c_3)=\left(\begin{array}{ccc}
     1&1&1  \\
     \mu_1&\mu_2&\mu_3\\
     \frac{\mu_1\overline{E'_1(L)}-\ii\lambda_{1}\overline{E'_{-1}(L)}}{\mu_1^2+\lambda_{1}^2}&\frac{\mu_2\overline{E'_1(L)}-\ii\lambda_{1}\overline{E'_{-1}(L)}}{\mu_2^2+\lambda_{1}^2}&\frac{\mu_3\overline{E'_1(L)}-\ii\lambda_{1}\overline{E'_{-1}(L)}}{\mu_3^2+\lambda_{1}^2}
\end{array}
\right).
\end{equation*}
Therefore, we obtain a linear system for the coefficients $(c_1,c_2,c_3)$, we obtain
\begin{equation*}
\left(\begin{array}{ccc}
     1&1&1  \\
     \mu_1&\mu_2&\mu_3\\
     \frac{\mu_1\overline{E'_1(L)}-\ii\lambda_{1}\overline{E'_{-1}(L)}}{\mu_1^2+\lambda_{1}^2}&\frac{\mu_2\overline{E'_1(L)}-\ii\lambda_{1}\overline{E'_{-1}(L)}}{\mu_2^2+\lambda_{1}^2}&\frac{\mu_3\overline{E'_1(L)}-\ii\lambda_{1}\overline{E'_{-1}(L)}}{\mu_3^2+\lambda_{1}^2}
\end{array}
\right)
\left(\begin{array}{c}
     c_1  \\
     c_2\\
     c_3
\end{array}\right)=
\left(\begin{array}{c}
     -\p_xz(0)  \\
     -\p_x(\mathcal{P}z)(0)\\
     -\poscals{z}{E_1}_{L^2(0,L)}
\end{array}\right).
\end{equation*}
It is easy to compute that 
\begin{multline*}
\det{M^{-1}}\\=\frac{ (\lambda_{1}^2 + \mu_1^2) (\lambda_{1}^2 + \mu_2^2) (\lambda_{1}^2 + \mu_3^2)}{(\mu_1 - \mu_2) (\mu_1 - \mu_3) (\mu_2 - \mu_3) \left(\ii E'_{-1}(L) \lambda_{1} (\lambda_{1}^2 - \mu_2 \mu_3 - \mu_1 (\mu_2 + \mu_3)) + E'_1(L) (-\mu_1 \mu_2 \mu_3 + \lambda_{1}^2 (\mu_1 + \mu_2 + \mu_3))\right)}.
\end{multline*}
Choosing three distinct positive real numbers $(\mu_1,\mu_2,\mu_3)$, the real matrix $M$ is invertible and we find the real coefficients $(c_1,c_2,c_3)$ that compensate the boundary conditions.
\end{proof}
\begin{lem}\label{lem: construction of v-2pi}
For any initial state $z^0\in H_{\B}^2$ and any final state $z^T\in H_{\B}^2$ (as defined in \eqref{eq: defi for space h0 and h}), there exists a function $v\in C^1(0,T)$ such that the solution $z$ to the controlled KdV system \eqref{eq: KdV boundary derivative difference} with $L\in[2\pi-\delta,2\pi+\delta]$, 
satisfies that $z(T)=z^T\in H_{\B}^2$. In particular, we can choose that $z(T,x)=\varrho(T)E_{-1}(x)$. Furthermore, $\varrho$ such that 
\begin{equation*}
     z(T)+ c_1 h_{\mu_1} e^{-\mu_1 \frac{T}{2}}+ c_2 h_{\mu_2} e^{-\mu_2 \frac{T}{2}}+c_3 h_{\mu_3} e^{-\mu_3 \frac{T}{2}}\in H_{\A}.
\end{equation*}
\end{lem}
\begin{proof}
Without loss of generality, we can require that $z^0_-:=\poscals{z^0}{E_{-1}}_{L^2(0,L)}\in \ii \R$ and $z^T_-:=\poscals{z^T}{E_{-1}}_{L^2(0,L)}\in \ii \R$. In fact, if we obtain the exact controllability for $z^0_-\in\ii\R$, it is standard to obtain the exact controllability in the general case by considering the real part and imaginary part of $z^0$. Hence, in the following proof, we only consider the case that $z^0_-:=\poscals{z^0}{E_{-1}}_{L^2(0,L)}\in \ii \R$ and $z^T_-:=\poscals{z^T}{E_{-1}}_{L^2(0,L)}\in \ii \R$.  Suppose that $c_1$, $c_2$ and $c_3$ are fixed. As we presented in Section \ref{sec: construction of the control}, we aim to construct the control function $v\in C^1(0,T)$ for the  KdV system \eqref{eq: KdV boundary derivative difference} 
with the constraints $v(0)=v(T)=0$. Thanks to the family $\{\vartheta_j\}_{j\in\Z\backslash\{\pm 1\}}$, we construct $v$ as follows
\begin{equation}\label{eq: defi of control v-vartheta-2pi}
    v(t)=\sum_{k\neq\pm 1}v_k\vartheta_k(t-\frac{T}{2}).
\end{equation}
Here $h_k$ remains the same as in Section \ref{sec: construction of the control} and 
\begin{align*}
z^0(x)=z^0_-E_{-1}(x)+\sum_{k\in\Z\setminus\{0,\pm 1\}}z^0_k\E_k(x),\;
z(T,x)=z^T_-E_{-1}(x)+\sum_{k\in\Z\setminus\{0,\pm 1\}}z^T_k\E_k(x).
\end{align*}
Then the solution $z$ to the system \eqref{eq: KdV-z-v-2pi} has the following expansion 
\begin{equation}\label{eq: construction-z-expansion-2pi}
\begin{aligned}
z(t,x)&=e^{\ii \lambda_{1}t}z^0_-\E_{1}(x)-e^{\ii \lambda_{-1}t}z^0_-\E_{-1}(x)+\sum_{k\in\Z\setminus\{0,\pm 1\}}z^0_ke^{\ii\lambda_k t}\E_k(x)\\
&-\sum_{k\neq0} \sum_{j\neq\pm 1}\ii\lambda_kh_k\int_0^te^{\ii(t-s)\lambda_k}v_j\vartheta_j(s-\frac{T}{2})ds\E_k(x).
\end{aligned}
\end{equation}
In particular, at $t=T$, 
\begin{align*}
z(T,x)&=e^{\ii \lambda_{1}T}z^0_-\E_{1}(x)-e^{\ii \lambda_{-1}T}z^0_-\E_{-1}(x)+\sum_{k\in\Z\setminus\{0,\pm 1\}}z^0_ke^{\ii\lambda_k T}\E_k(x)\\
&-\sum_{k\neq0} \sum_{j\neq\pm 1}\ii\lambda_kh_k\int_0^Te^{\ii(T-s)\lambda_k}v_j\vartheta_j(s-\frac{T}{2})ds\E_k(x)\\
&=z^T_-E_{-1}(x)+\sum_{k\in\Z\setminus\{0,\pm 1\}}z^T_k\E_k(x).
\end{align*}
Using the property \ref{eq: biorthorgonal-2pi} in Lemma \ref{lem: bi-orthogonal family-2pi}, as a consequence, we obtain the following equations
\begin{gather*}
e^{\ii \lambda_{1}T}z^0_- -\ii\lambda_{1}h_{1}e^{\ii\frac{T\lambda_{1}}{2}}v_0=z^T_-,\;
-e^{\ii \lambda_{-1}T}z^0_--\ii\lambda_{-1}h_{-1}e^{\ii\frac{T\lambda_{-1}}{2}}v_0=-z^T_-,\\
z^0_ke^{\ii\lambda_k T}-\ii\lambda_kh_ke^{\ii\frac{\lambda_k T}{2}}v_k=z^T_k,k\neq 0,\pm 1.
\end{gather*}
For the last equation, we obtain  $v_k=\frac{z^0_ke^{\ii\frac{\lambda_k T}{2}}-z^T_ke^{-\ii\frac{\lambda_k T}{2}}}{\ii\lambda_kh_k},k\neq 0,\pm 1$. 
By Lemma \ref{lem: size of h_j}, we know that $-\ii\lambda_{1}h_{1}=\overline{\E'_{1}(L)}$ and $-\ii\lambda_{-1}h_{-1}=\overline{\E'_{-1}(L)}$. Hence, the first two equations are equivalent to 
\begin{align*}
e^{\ii \lambda_{1}T}z^0_- +\overline{\E'_{1}(L)}e^{\ii\frac{T\lambda_{1}}{2}}v_0=z^T_-,\;
-e^{\ii \lambda_{-1}T}z^0_- +\overline{\E'_{-1}(L)}e^{\ii\frac{T\lambda_{-1}}{2}}v_0=-z^T_-.
\end{align*}
Using the fact that $\lambda_{1}=-\lambda_{-1}$ and $\overline{\E_{1}}=\E_{-1}$, we obtain the following equations
\begin{align*}
e^{\ii \lambda_{1}T}z^0_- +\overline{\E'_{1}(L)}e^{\ii\frac{T\lambda_{1}}{2}}v_0=z^T_-,\;
-e^{-\ii \lambda_{1}T}z^0_- +\E'_{1}(L)e^{-\ii\frac{T\lambda_{1}}{2}}v_0=-z^T_-.
\end{align*}
Due to the conditions that $z^0_-\in\ii\R$ and $z^T_-\in\ii\R$, we deduce that 
\begin{equation*}
    v_0=\frac{z^T_-e^{-\ii\frac{\lambda_{1} T}{2}}-z^0_-e^{\ii\frac{\lambda_{1} T}{2}}}{\overline{\E'_{1}(L)}}=\frac{z^0_-e^{\ii\frac{\lambda_{1} T}{2}}-z^T_-e^{-\ii\frac{\lambda_{1} T}{2}}}{\ii\lambda_{1}h_{1}}\in\R,
\end{equation*}
In particular, we can choose that $z(T,x)=\varrho E_{-1}(x)\in \h^2$. Indeed, $z^T_-=\varrho\in\ii\R$ and $z^T_j=0$ for $j\neq 0,\pm 1$. To achieve this final target, we construct a special control function $v$ as follows:
\begin{equation}\label{eq: E_-construction of v-2pi}
v(t):=\frac{z^0_-e^{\ii\frac{\lambda_{1} T}{2}}-\varrho(T)e^{-\ii\frac{\lambda_{1} T}{2}}}{\ii\lambda_{1}h_{1}}\vartheta_0(t-\frac{T}{2})+\sum_{k\neq0,\pm 1}\frac{z^0_ke^{\ii\frac{\lambda_k T}{2}}}{\ii\lambda_kh_k}\vartheta_k(t-\frac{T}{2})
\end{equation}
For this specific final target $z(T)=\varrho E_{-1}$, we aim to prove that $ z(T)+ c_1 h_{\mu_1} e^{-\mu_1 T/2}+ c_2 h_{\mu_2} e^{-\mu_2 T/2}+c_3 h_{\mu_3} e^{-\mu_3 T/2}\in  H_{\A}$, which is equivalent to 
\begin{equation}
    \poscalr{z(T)+ c_1 h_{\mu_1} e^{-\mu_1 T/2}+ c_2 h_{\mu_2} e^{-\mu_2 T/2}+c_3 h_{\mu_3} e^{-\mu_3 T/2}}{\F_{\zeta_0}}_{(0,L)}= 0.
\end{equation}
By integration by parts, we obtain that
\begin{align*}
 \int_0^Lh_{\mu_j}(x)\F_{\zeta_0}(L-x)dx=\frac{\F_{\zeta_0}'(0)h_{\mu_j}'(L)}{ \zeta_0+\mu_j},\;
 \int_0^L\E_{\pm1}(x)\F_{\zeta_0}(L-x)dx=\frac{\F_{\zeta_0}'(0)\E_{\pm1}'(L)}{ \zeta_0+\ii\lambda_{\pm1}}.
\end{align*}
Therefore, we know that
\begin{equation}\label{eq: defi of varrho}
\varrho(T)=-\frac{\zeta_0^2+\lambda_{1}^2}{\zeta_0\F_{\zeta_0}'(0)E_{-1}'(L)-\ii\lambda_{1}\F_{\zeta_0}'(0)E_1'(L) }\sum_{j=1}^3c_j e^{-\mu_j T/2}\frac{\F_{\zeta_0}'(0)h_{\mu_j}'(L)}{ \zeta_0+\mu_j}.
\end{equation}
As we observed in Remark \ref{rem: defi-of E+ and E-}, $E_{-1}'(L)\in\ii\R$ and $E_1'(L)\in\R$. By Lemma \ref{lem: defi of coeff-1-2-3}, the term 
\begin{equation*}
c_1 e^{-\mu_1 T/2}\frac{\F_{\zeta_0}'(0)h_{\mu_1}'(L)}{ \zeta_0+\mu_1}+c_2 e^{-\mu_2 T/2}\frac{\F_{\zeta_0}'(0)h_{\mu_2}'(L)}{ \zeta_0+\mu_2}+c_3 e^{-\mu_3 T/2}\frac{\F_{\zeta_0}'(0)h_{\mu_3}'(L)}{ \zeta_0+\mu_3}\in\R,
\end{equation*}
which implies that $\varrho\in\ii\R$.
\end{proof}
\begin{rem}
Here we notice that $\varrho\in\ii\R$. This yields that the final target $z^T$ is always a real-valued function. This is a particular case for Remark \ref{rem: rho} in $L_0=2\pi$. In the current setup, we have $E_{-1}=2\ii\Im{\E_1}\sim2\ii\sin{x}$.
\end{rem}
\begin{lem}\label{lem: est for varrho-2pi}
Let $T\in (0, 2)$, and three distinct positive parameters be $\mu_1$, $\mu_2$ and $\mu_3$ with $\mu_j>|\zeta_0|+1$, $j=1,2,3$. We fix three real constants $c_1$, $c_2$, and $c_3$. Suppose that Assumption \ref{ass: Compact intervals including 2pi} holds.  For every $L\in I\setminus\{2\pi\}$, $\varrho(T)$, as we defined in Lemma \ref{lem: construction of v-2pi} is uniformly bounded by a effectively computable constant $C_{\varrho}=C_{\varrho}(\mu_1,\mu_2,\mu_3,c_1,c_2,c_3)$, i.e. $|\varrho|\leq C_{\varrho}$.
\end{lem}
\begin{proof}
As we defined in \eqref{eq: defi of varrho} in Lemma \ref{lem: construction of v-2pi}, 
\begin{equation*}
\varrho(T)=-\frac{\zeta_0^2+\lambda_{1}^2}{\zeta_0\F_{\zeta_0}'(0)E_{-1}'(L)-\ii\lambda_{1}\F_{\zeta_0}'(0)E_1'(L) }\sum_{j=1}^3c_j e^{-\mu_j T/2}\frac{\F_{\zeta_0}'(0)h_{\mu_j}'(L)}{ \zeta_0+\mu_j}
\end{equation*}
For $\frac{\zeta_0^2+\lambda_{1}^2}{\zeta_0\F_{\zeta_0}'(0)E_{-1}'(L)-\ii\lambda_{1}\F_{\zeta_0}'(0)E_1'(L) }$, due to the estimates in Proposition \ref{prop: Asymp in L-low} and Lemma \ref{lem: real eigenvalue close to 0}, 
\begin{align*}
|\zeta_0|\sim |L-2\pi|^2,\;|\lambda_{1}|\sim |L-2\pi|,\; |\F_{\zeta_0}'(0)|\sim |L-2\pi|.
\end{align*}
Applying Proposition \ref{prop: Low-frequency behaviors: Singular limits}, combining with \eqref{rem: defi-of E+ and E-} and Proposition \ref{coro: expansion of E+ and E-}, we derive that
\begin{align*}
|\zeta_0E_{-1}'(L)-\ii\lambda_{1}E_1'(L)|&=\left|\left(-\frac{1}{3\pi}(L-2\pi)^2+\bigO((L-2\pi)^3)\right)\left(-\ii\sqrt\frac{2}{\pi}+\bigO(|L-2\pi|)\right)\right.\\
&\left.-\ii\left(\frac{1}{\sqrt{3}\pi}|L-2\pi|+\bigO(|L-2\pi|^2)\right)\left(\epsilon_0(L-2\pi)+\bigO((L-2\pi)^2)\right)\right|\\
&=\left|\ii\frac{1}{\sqrt{3}\pi}(\frac{2}{3}-\epsilon_0)(L-2\pi)^2+\bigO(|L-2\pi|^3)\right|\\
&=\frac{1}{\sqrt{3}\pi}|\frac{2}{3}-\epsilon_0|(L-2\pi)^2+\bigO(|L-2\pi|^3).
\end{align*}
Note that $|\frac{2}{3}-\epsilon_0|>0$. Therefore, $|\zeta_0\F_{\zeta_0}'(0)E_{-1}'(L)-\ii\lambda_{1}\F_{\zeta_0}'(0)E_1'(L)|\sim |\F_{\zeta_0}'(0)||L-2\pi|^2$, which deduces that 
\begin{equation*}
|\zeta_0\F_{\zeta_0}'(0)E_{-1}'(L)-\ii\lambda_{1}\F_{\zeta_0}'(0)E_1'(L)|\sim |\F_{\zeta_0}'(0)||L-2\pi|^2.
\end{equation*}
Thus, 
\begin{equation*}
\left|\frac{\zeta_0^2+\lambda_{1}^2}{\zeta_0\F_{\zeta_0}'(0)E_{-1}'(L)-\ii\lambda_{1}\F_{\zeta_0}'(0)E_1'(L) }\right|\lesssim \frac{1}{|\F_{\zeta_0}'(0)|}.
\end{equation*}
For $c_j e^{-\mu_j T/2}\frac{\F_{\zeta_0}'(0)h_{\mu_j}'(L)}{ \zeta_0+\mu_j}$, $j=1,2,3$, we know that $|c_j e^{-\mu_j T/2}\frac{\F_{\zeta_0}'(0)h_{\mu_j}'(L)}{ \zeta_0+\mu_j}|\leq |c_j||\F_{\zeta_0}'(0)h_{\mu_j}'(L)|$.
Using Proposition \ref{prop: h_mu-est-uniform}, we know that $|h_{\mu_j}'(L)|\lesssim 1$. Therefore, we obtain that $\varrho$ is uniformly bounded by $C_{\varrho}=C_0(|c_1|+|c_2|+|c_3|)$, where $C_0>1$ is a constant which is independent of $L$ and $T$, i.e. $ |\varrho|\leq C_{\varrho}:= C_0(|c_1|+|c_2|+|c_3|)$.
\end{proof}
\subsubsection{Revised transition-stabilization method}
\textbf{A priori estimates for the intermediate system}
\begin{lem}\label{lem: est-v-2pi}
Let $v\in C^1(0,T)$ be the control function that we construct in the formula \eqref{eq: E_-construction of v-2pi} in Lemma \ref{lem: construction of v-2pi}. We have the following estimate for $v$
\begin{equation}
    \|v\|_{L^{\infty}(0,T)}+\|v'\|_{L^{\infty}(0,T)}\leq Ce^{\frac{K}{\sqrt{T}}}\|z^0\|_{H^3(0,L)}.
\end{equation}
\end{lem}
\begin{proof}
We only prove the estimate for $\|v\|_{L^{\infty}(0,T)}$. The estimate for the derivative follows similarly. 
\begin{align*}
\|v\|_{L^{\infty}(0,T)}&=\left|\frac{z^0_-e^{\ii\frac{\lambda_{1} T}{2}}-\varrho(T)e^{-\ii\frac{\lambda_{1} T}{2}}}{\ii\lambda_{1}h_{1}}\vartheta_0(t-\frac{T}{2})+\sum_{k\neq0,\pm 1}\frac{z^0_ke^{\ii\frac{\lambda_k T}{2}}}{\ii\lambda_kh_k}\vartheta_k(t-\frac{T}{2})\right|\\
&\leq \left|\frac{z^0_-e^{\ii\frac{\lambda_{1} T}{2}}-\varrho(T)e^{-\ii\frac{\lambda_{1} T}{2}}}{\ii\lambda_{1}h_{1}}\vartheta_0(t-\frac{T}{2})\right|+\sum_{k\neq0,\pm 1}\frac{|z^0_k|}{|\lambda_kh_k|}\|\vartheta_k\|_{L^{\infty(\R)}}\\
&\leq \frac{|z^0_-|+|\varrho(T)|}{|\lambda_{1}h_{1}|}\|\vartheta_0\|_{L^{\infty}(\R)}+\sum_{k\neq0,\pm 1}\frac{|z^0_k|}{|\lambda_kh_k|}\|\vartheta_k\|_{L^{\infty(\R)}}\\
&\leq \frac{|z^0_-|+|\varrho(T)|}{|\lambda_{1}h_{1}|}\|\vartheta_0\|_{L^{\infty}(\R)}+\sum_{k\neq 0,\pm 1,|k|\leq J}\frac{|z_k^0|}{|\E_k'(L)|}\|\vartheta_k\|_{L^{\infty}(\R)}+\sum_{|k|>J}\frac{|z_k^0|}{|\E_k'(L)|}\|\vartheta_k\|_{L^{\infty}(\R)}
\end{align*}
By Lemma \ref{lem: bi-orthogonal family-2pi}, we know that for $k\neq\pm 1$, $\|\vartheta_k\|_{L^{\infty}(\R)}\leq Ce^{\frac{K}{\sqrt{T}}}$. Thanks to Proposition \ref{prop: Low-frequency behaviors: uniform estimates}, we know that $|\E_j'(L)|>\gamma$ for $j\neq0,\pm 1$ and $|j|\leq J$. Applying Proposition \ref{prop: Low-frequency behaviors: Singular limits}, we know that $|\E_{1}'(L)|>\gamma$. For high-frequencies, by Proposition \ref{prop: High-frequency behaviors: uniform estimates}, $|\E_j'(L)|>\gamma|j|$.  Combining all these estimates, we conclude that
\begin{equation*}
    \|v\|_{L^{\infty}(0,T)}\leq Ce^{\frac{K}{\sqrt{T}}}\|z^0\|_{H^3(0,L)}
\end{equation*}
\end{proof}
Based on the lemma above, we have the following estimates on the boundary derivative.
\begin{lem}\label{lem: control-toy-2pi}
Suppose that Assumption \ref{ass: Compact intervals including 2pi} holds. Suppose that $z^0\in H_{\B}^2$ and $\poscals{z^0}{E_{-1}}\in\ii\R$. There exists a unique solution $z\in C([0,T];L^2(0,L))$ to the equation \eqref{eq: KdV boundary derivative difference}
such that $z(T,x)=\varrho E_{-1}(x)$ and $\varrho\in\ii\R$. Then for every $L\in I\setminus\{L_0\}$, there are two constants $K_1$ and $K_2$, independent of $L$, such that 
$\p_x z(\cdot,L)\in L^2(0,T)$ and 
\begin{equation}\label{eq: L^2-estimate of derivative of z-2pi}
\|\p_x z(\cdot,L)\|_{L^{\infty}(0,T)}\leq K_1e^{\frac{2K}{\sqrt{T}}}\|z^0\|_{H^{6}(0,L)}.
\end{equation}
and for any $t\in(0,T]$, we have the following estimate
\begin{equation}\label{eq: L^2-estimate of z(t)-2pi}
\|z(t,\cdot)\|_{L^2(0,L)}\leq K_2 e^{\frac{2K}{\sqrt{T}}}\|z^0\|_{H^{3}(0,L)}
\end{equation}
\end{lem}
\begin{proof}
Recalling that the expansion for the solution $z$ in \eqref{eq: construction-z-expansion-2pi} and using the property \ref{eq: biorthorgonal-2pi} in Lemma \ref{lem: bi-orthogonal family-2pi}, we know that
\begin{equation}\label{eq: expansion-z-2pi}
\begin{aligned}
z(t,x)&=\left(e^{\ii \lambda_{1}t}z^0_--\sum_{k\neq0,\pm 1}\frac{z^0_k\lambda_{1}h_{1}e^{\ii\frac{\lambda_k T}{2}}}{\lambda_kh_k}\int_0^te^{\ii(t-s)\lambda_{1}}\vartheta_k(s-\frac{T}{2})ds\right)\E_{1}(x)\\
&-\left(e^{\ii \lambda_{-1}t}z^0_-+\sum_{k\neq0,\pm 1}\frac{z^0_k\lambda_{-1}h_{-1}e^{\ii\frac{\lambda_k T}{2}}}{\lambda_kh_k}\int_0^te^{\ii(t-s)\lambda_{-1}}\vartheta_k(s-\frac{T}{2})ds\right)\E_{-1}(x)\\
&+\sum_{j\neq0,\pm 1}\sum_{k\neq0,\pm 1}\frac{z^0_k\lambda_jh_je^{\ii\frac{\lambda_k T}{2}}}{\lambda_kh_k}\int_t^Te^{\ii(t-s)\lambda_j}\vartheta_k(s-\frac{T}{2})ds\E_j(x)\\
&-\sum_{j\neq0}\frac{\lambda_jh_j(z^0_-e^{\ii\frac{\lambda_{1} T}{2}}-\varrho(T)e^{-\ii\frac{\lambda_{1} T}{2}})}{\lambda_{1}h_{1}}\int_0^te^{\ii(t-s)\lambda_j}\vartheta_0(s-\frac{T}{2})ds\E_j(x).
\end{aligned}
\end{equation}
Thus, we write $\p_x z(t,L)=I_1+I_2+I_3+I_4$, with
\begin{align*}
I_1&=e^{\ii \lambda_{1}t}z^0_--\sum_{k\neq0,\pm 1}\frac{z^0_k\lambda_{1}h_{1}e^{\ii\frac{\lambda_k T}{2}}}{\lambda_kh_k}\int_0^te^{\ii(t-s)\lambda_{1}}\vartheta_k(s-\frac{T}{2})ds\E'_{1}(L)\\
I_2&=-e^{\ii \lambda_{-1}t}z^0_--\sum_{k\neq0,\pm 1}\frac{z^0_k\lambda_{-1}h_{-1}e^{\ii\frac{\lambda_k T}{2}}}{\lambda_kh_k}\int_0^te^{\ii(t-s)\lambda_{-1}}\vartheta_k(s-\frac{T}{2})ds\E'_{-1}(L)\\
I_3&=\sum_{j\neq0,\pm 1}\sum_{k\neq0,\pm 1}\frac{z^0_k\lambda_jh_je^{\ii\frac{\lambda_k T}{2}}}{\lambda_kh_k}\int_t^Te^{\ii(t-s)\lambda_j}\vartheta_k(s-\frac{T}{2})ds\E'_j(L)\\
I_4&=-\sum_{j\neq0}\frac{\lambda_jh_j(z^0_-e^{\ii\frac{\lambda_{1} T}{2}}-\varrho(T)e^{-\ii\frac{\lambda_{1} T}{2}})}{\lambda_{1}h_{1}}\int_0^te^{\ii(t-s)\lambda_j}\vartheta_0(s-\frac{T}{2})ds\E'_j(L)
\end{align*}
Then, we estimate term by term. For $I_1$,
\begin{align*}
|I_1|
&\leq |z^0_-|+\sum_{k\neq0,\pm 1}\frac{T|z^0_k||\lambda_{1}h_{1}|}{|\lambda_kh_k|}\|\vartheta_k\|_{L^{\infty}(\R)}|\E'_{1}(L)|\\
&\leq |z^0_-|+\sum_{k\neq0,\pm 1,|k|\leq J}\frac{T|z^0_k||\E'_{1}(L)|^2}{|\E'_k(L)|}\|\vartheta_k\|_{L^{\infty}(\R)}+\sum_{|k|>J}\frac{T|z^0_k||\E'_{1}(L)|^2}{|\E'_k(L)|}\|\vartheta_k\|_{L^{\infty}(\R)}.
\end{align*}
For high-frequencies, by Proposition \ref{prop: High-frequency behaviors: uniform estimates}, we know that $|\E_k'(L)|>\gamma|k|$ for $|k|>J$. For low-frequencies, by Proposition \ref{prop: Low-frequency behaviors: uniform estimates}, we know that $|\E_k'(L)|>\gamma$, for $|k|\leq J$ and $k\neq 0,\pm 1$. Thanks to Proposition \ref{lem: bi-orthogonal family-2pi}, we know that $ \|\vartheta_k\|_{L^{\infty}(\R)}\leq Ce^{\frac{K}{\sqrt{T}}},k\neq 0,\pm 1$. 
Combining all these estimates, we obtain 
\begin{align*}
|I_1|
\leq Ce^{\frac{K}{\sqrt{T}}}\left(|z^0_-|+\sum_{k\neq 0,\pm 1,|k|\leq J}\frac{|z_k^0|}{ \gamma}|\E'_{1}(L)|^2+\sum_{|k|>J}\frac{|z_k^0|}{\gamma |k|}|\E'_{1}(L)|^2\right)
\leq Ce^{\frac{K}{\sqrt{T}}}\|z^0\|_{L^2(0,L)}.
\end{align*}
The same procedure holds for $I_2$. Thus, $|I_2|\leq Ce^{\frac{K}{\sqrt{T}}}\|z^0\|_{L^2(0,L)}$. For $I_3$, integrating by parts twice, we write $I_3=I_3^1+I_3^2+I_3^3$ with
\begin{align*}
I_3^1&=-\sum_{j,k\in\Z\backslash\{0,\pm 1\}}e^{\ii \frac{T}{2}\lambda_k}\frac{h_j z_k^0}{ h_k\lambda_k\lambda_j}\vartheta'_k(t-\frac{T}{2})\E'_j(L),\\
I_3^2&=-\sum_{j,k\in\Z\backslash\{0,\pm 1\}}e^{\ii \frac{T}{2}\lambda_k}\frac{h_j z_k^0}{ h_k\lambda_k\lambda_j}\int_t^Te^{\ii(t-s)\lambda_j}\vartheta''_k(s-\frac{T}{2})ds\E'_j(L),\\
I_3^3&=\sum_{j,k\in\Z\backslash\{0,\pm 1\}}e^{\ii \frac{T}{2}\lambda_k}\frac{h_j z_k^0}{\ii h_k\lambda_k}\vartheta_k(t-\frac{T}{2})\E'_j(L).
\end{align*}
As usual, we compute one by one. We begin with the term $I_3^3$, it is easy to see that $|I_3^3|\leq \|v\|_{L^{\infty}(\R)}|h'(L)|+\frac{|z^0_-|+|\varrho(T)|}{|\lambda_{1}h_{1}|}\|\vartheta_0\|_{L^{\infty}(\R)}$.
By Lemma \ref{lem: est-v-2pi} and Lemma \ref{lem: bi-orthogonal family-2pi}, we know that $|I_3^3|\leq Ce^{\frac{K}{\sqrt{T}}}\|z^0\|_{L^2(0,L)}$. 
For the term $I_3^1$, we split the sum into high-frequency and low-frequency parts 
\begin{align*}
|I_3^1|
&\leq\sum_{j,k\in\Z\backslash\{0,\pm 1\},|k|\leq J}\frac{|h_j| |z_k^0|}{ |h_k\lambda_k||\lambda_j|}\|\vartheta'_k\|_{L^(\infty)(\R)}|\E'_j(L)|+\sum_{j,k\in\Z\backslash\{0,\pm 1\},|k|> J}\frac{|h_j| |z_k^0|}{ |h_k\lambda_k||\lambda_j|}\|\vartheta'_k\|_{L^(\infty)(\R)}|\E'_j(L)|.
\end{align*}
As what we present in the proof of estimating $I_1$, we use once again Proposition \ref{prop: High-frequency behaviors: uniform estimates}, Proposition \ref{prop: Low-frequency behaviors: uniform estimates}, and Proposition \ref{lem: bi-orthogonal family-2pi} to derive the following estimates. 
\begin{align*}
|I_3^1|
&\leq Ce^{\frac{K}{\sqrt{T}}}\left(\sum_{j,k\in\Z\backslash\{0,\pm 1\},|k|\leq J}\frac{|\E'_j(L)|^2|\lambda_k| |z_k^0|}{ \gamma|\lambda_j|^2}+\sum_{j,k\in\Z\backslash\{0,\pm 1\},|k|> J}\frac{|\E'_j(L)|^2|\lambda_k| |z_k^0|}{ \gamma|k||\lambda_j|^2}\right)\\
&\leq Ce^{\frac{K}{\sqrt{T}}}\left(\sum_{j,k\in\Z\backslash\{0,\pm 1\},|j|\leq J,|k|\leq J}\frac{|\E'_j(L)|^2|\lambda_k| |z_k^0|}{ \gamma|\lambda_j|^2}+\sum_{j,k\in\Z\backslash\{0,\pm 1\},|j|> J,|k|\leq J}\frac{|\E'_j(L)|^2|\lambda_k| |z_k^0|}{ \gamma|\lambda_j|^2}\right.\\
&\left.+\sum_{j,k\in\Z\backslash\{0,\pm 1\},|j|\leq J,|k|> J}\frac{|\E'_j(L)|^2|\lambda_k| |z_k^0|}{ \gamma|k||\lambda_j|^2}+\sum_{j,k\in\Z\backslash\{0,\pm 1\},|j|> J,|k|> J}\frac{|\E'_j(L)|^2|\lambda_k| |z_k^0|}{ \gamma|k||\lambda_j|^2}\right).
\end{align*}
Using Proposition \ref{prop: Asymp in L} and Proposition \ref{prop: High-frequency behaviors: uniform estimates}, we know that $\frac{|\E'_j(L)|^2}{|\lambda_j|^2}\sim \frac{1}{|j|^4}$ for $|j|>J$. We take the term $\sum_{j,k\in\Z\backslash\{0,\pm 1\},|j|> J,|k|> J}\frac{|\E'_j(L)|^2|\lambda_k| |z_k^0|}{ \gamma|k||\lambda_j|^2}$ for example. Other terms can be treated similarly. 
\begin{align*}
\sum_{j,k\in\Z\backslash\{0,\pm 1\},|j|> J,|k|> J}\frac{|\E'_j(L)|^2|\lambda_k| |z_k^0|}{ \gamma|k||\lambda_j|^2}&=\sum_{|j|> J}\frac{|\E'_j(L)|^2}{ |\lambda_j|^2}\sum_{|k|> J}\frac{|\lambda_k| |z_k^0|}{ \gamma|k|}\\
&\leq \|z^0\|_{H^3(0,L)}\sum_{|j|> J}\frac{|\E'_j(L)|^2}{ |\lambda_j|^2}\left(\sum_{|k|> J}\frac{1}{ \gamma^2|k|^2}\right)^{\frac{1}{2}}\\
&\leq C\|z^0\|_{H^3(0,L)}.
\end{align*}
Then the sums are bounded by $|I_3^1|\leq Ce^{\frac{K}{\sqrt{T}}}\|z^0\|_{H^3(0,L)}$.
The term $I_3^2$ follows the same procedure and we could find $|I_3^2|\leq Ce^{\frac{K}{\sqrt{T}}}\|z^0\|_{H^6(0,L)}$. 
Therefore, we obtain the estimate for $\|\p_x z(\cdot,L)\|_{L^{\infty}(0,T)}$ as follows,  \begin{align*}
\|\p_x z(\cdot,L)\|_{L^{\infty}(0,T)}&\leq K_1e^{\frac{2K}{\sqrt{T}}}\|z^0\|_{H^{6}(0,L)}.   
\end{align*}
Same procedure could give the estimate for $\|z(t,\cdot)\|_{L^2(0,L)}\leq K_2 e^{\frac{2K}{\sqrt{T}}}\|z^0\|_{H^{3}(0,L)}$.
\end{proof}
\textbf{Iteration Schemes with uniform constants}
\begin{prop}\label{prop: estimates on fixed interval-iteration preparation-2pi}
Let $T\in (0, 2)$, and three distinct positive parameters $\mu_1=\mu$, $\mu_2=2\mu$ and $\mu_3=3\mu$ with $\mu>0$. Suppose that Assumption \ref{ass: Compact intervals including 2pi} holds.  For every $L\in I\setminus\{2\pi\}$, and every real initial state $y^0\in H_{\A}$, there exists a function $u\in L^2(0, T)$ satisfying $u= u_1+ u_2+u_3+u_4$ in $(0, T)$ with $u_1(t)=u_2(t)=u_3(t)=u_4(t)=0, \forall t\in (0, T/2)$, and
\begin{equation*}
 \|u_1\|_{L^{\infty}(0, T)}\leq \mathcal{K}e^{\frac{2\sqrt{2}K}{\sqrt{T}}}\frac{\mu^4}{T^3}\|y^0\|_{L^2(0,L)},\;  
    \|u_{j+1}\|_{L^{\infty}(0,T)}\leq\mathcal{K}\frac{e^{-\frac{\mu^{\frac{1}{3}}}{4}L}}{T^3}\|y^0\|_{L^2(0,L)},j=1,2,3,   
\end{equation*}
such that the unique solution $y(t)=
\left\{
\begin{array}{ll}
    S(t) y^0, & t\in (0, T/2), \\
     y_1(t)+ y_2(t)+y_3(t)+y_4(t),& t\in (T/2, T),
\end{array}
\right.
$ to \eqref{eq: linearized KdV system-control-intro},  where $y_j(t)$ solves the equation \eqref{eq: defi-y_j-2pi}, 
satisfies 
\begin{gather*}
\|y_1(t, \cdot)\|_{L^2(0, L)}\leq \mathcal{K}e^{\frac{2\sqrt{2}K}{\sqrt{T}}}\frac{\mu^{4}}{T^3}\|y^0\|_{L^2(0,L)},  y_1(T, x)= \varrho E_{-1}(x), \text{ where }  |\varrho|\leq \mathcal{K}\frac{e^{-\frac{T}{4}\mu}}{T^3}\|y^0\|_{L^2(0,L)},\\
\|y_{j+1}(t, \cdot)\|_{L^2(0, L)}\leq  \mathcal{K}\frac{\mu^{\frac{1}{2}}e^{-\mu_j(t-\frac{T}{2})}}{T^3}\|y^0\|_{L^2(0,L)},j=1,2,3,\forall t\in (T/2, T).
\end{gather*}
All constants appearing in this proposition are independent of $L$ and $T$.
\end{prop}
\begin{proof}
We first split our time interval $[0,T]$ into two parts $[0,\frac{T}{2}]$ and $[\frac{T}{2},T]$. In the first part, we only use the free KdV flow $y(t)=S(t)y^0$. By smoothing effects, we know that $y(\frac{T}{2},\cdot)\in H_{\A}^2\subset H^6(0,L)$ is a real function and we have the following estimate
\begin{equation}\label{eq: smoothing estimate-toy-2pi}
\|y(\frac{T}{2},\cdot)\|_{H^6(0,L)}+|\p_x y(\frac{T}{2},0)|+|\p_x\mathcal{P} y(\frac{T}{2},0)|\leq \frac{C_1}{T^3}\|y^0\|_{L^2(0,L)}.    
\end{equation}
Then we consider the second time interval $[\frac{T}{2},T]$. We notice that our initial datum $y(\frac{T}{2},\cdot)\in H_{\A}^2$. Thanks to Lemma \ref{lem: defi of coeff-1-2-3}, we are able to define $z^0_{\frac{T}{2}}(x)=y(\frac{T}{2},x)- c_1 h_{\mu_1}(x)- c_2 h_{\mu_2}(x)-c_3 h_{\mu_3}(x)\in H_{\B}^2$,
where $c_1$, $c_2$, and $c_3$ solve the equations
\begin{equation*}
M\left(\begin{array}{c}
    c_1  \\
    c_2\\
    c_3
\end{array}\right):=\left(\begin{array}{ccc}
     1&1&1  \\
     \mu&2\mu&3\mu\\
     \frac{\mu\overline{E'_1(L)}-\ii\lambda_{1}\overline{E'_{-1}(L)}}{\mu^2+\lambda_{1}^2}&\frac{2\mu\overline{E'_1(L)}-\ii\lambda_{1}\overline{E'_{-1}(L)}}{4\mu^2+\lambda_{1}^2}&\frac{3\mu\overline{E'_1(L)}-\ii\lambda_{1}\overline{E'_{-1}(L)}}{9\mu^2+\lambda_{1}^2}
\end{array}
\right)
\left(\begin{array}{c}
     c_1  \\
     c_2\\
     c_3
\end{array}\right)=
\left(\begin{array}{c}
     -\p_xy(\frac{T}{2},0)  \\
     -\p_x(\mathcal{P}y)(\frac{T}{2},0)\\
     -\poscals{y(\frac{T}{2},\cdot)}{E_1}_{L^2(0,L)}
\end{array}\right)
\end{equation*}
Without loss of generality, as usual, we treat the case $\poscals{z^0_{\frac{T}{2}}}{E_{-1}}_{L^2(0,L)}\in \ii\R$. We write the solutions as 
\begin{align*}
c_1&=F_1\left(-\p_xy(\frac{T}{2},0)(30 \ii E'_1(L) \mu^{3} -E'_{-1}(L) \lambda_{1} (\lambda_{1}^{2} + 19 \mu^{2})) -\ii \p_x(\mathcal{P}y)(\frac{T}{2},0)  (E'_1(L) (\lambda_{1}^{2} - 6 \mu^{2}) - 5 \ii E'_{-1}(L) \lambda_{1} \mu)\right.\\
&\left.+ i \poscals{y(\frac{T}{2},\cdot)}{E_1}_{L^2(0,L)} (\lambda_{1}^{2} + 4 \mu^{2}) (\lambda_{1}^{2} + 9 \mu^{2})\right),\\
c_2&=F_2\left(-\p_xy(\frac{T}{2},0)(12 \ii E'_1(L) \mu^{3} -E'_{-1}(L) \lambda_{1} (\lambda_{1}^{2} + 13 \mu^{2})) +\p_x(\mathcal{P}y)(\frac{T}{2},0)(- \ii E'_1(L) (\lambda_{1}^{2} - 3 \mu^{2}) - 4 E'_{-1}(L) \lambda_{1} \mu)\right.\\
&\left. \ii\poscals{y(\frac{T}{2},\cdot)}{E_1}_{L^2(0,L)} (\lambda_{1}^{2} + \mu^{2}) (\lambda_{1}^{2} + 9 \mu^{2})\right),\\
c_3&=F_3\left(\p_xy(\frac{T}{2},0) (6 \ii E'_1(L) \mu^{3} -E'_{-1}(L) \lambda_{1} (\lambda_{1}^{2} + 7 \mu^{2}))+\ii \p_x(\mathcal{P}y)(\frac{T}{2},0)
  (E'_1(L) (\lambda_{1}^{2} - 2 \mu^{2}) - 3 \ii E'_{-1}(L) \lambda_{1} \mu)\right.\\
&\left.-\ii\poscals{y(\frac{T}{2},\cdot)}{E_1}_{L^2(0,L)} (\lambda_{1}^{2} + \mu^{2}) (\lambda_{1}^{2} + 4 \mu^{2}) \right).
\end{align*}
where $F_1=\frac{\lambda_{1}^{2} + \mu^{2}}{2 \mu^{2} (\ii E'_1(L) (6 \lambda_{1}^{2} \mu - 6 \mu^{3}) -E'_{-1}(L)\lambda_{1} (\lambda_{1}^{2} - 11 \mu^{2}))}$, $F_2=\frac{\lambda_{1}^{2} + 4 \mu^{2}}{\mu^{2} (- \ii E'_1(L) (6 \lambda_{1}^{2} \mu - 6 \mu^{3}) -E'_{-1}(L) \lambda_{1} (- \lambda_{1}^{2} + 11 \mu^{2}))}$, and $F_3=\frac{(\lambda_{1}^{2} + 9 \mu^{2}) }{2 \mu^{2} (- \ii E'_1(L) (6 \lambda_{1}^{2} \mu - 6 \mu^{3}) -E'_{-1}(L) \lambda_{1} (- \lambda_{1}^{2} + 11 \mu^{2}))}$.
Due to the estimates in Proposition \ref{prop: Asymp in L-low} and Proposition \ref{prop: Low-frequency behaviors: Singular limits}, we know that $|E'_1(L)|\sim |L-2\pi|,\;|\lambda_{1}|\sim |L-2\pi|$. 
Thus, 
\begin{equation}\label{eq: est-F_j-2pi}
|F_j|\sim \frac{1}{\mu^3|L-2\pi|},j=1,2,3.
\end{equation}
Moreover, we can verify the following estimates
\begin{align*}
|-\p_xy(\frac{T}{2},0)(30 \ii E'_1(L) \mu^{3} -E'_{-1}(L) \lambda_{1} (\lambda_{1}^{2} + 19 \mu^{2}))|&\lesssim |L-2\pi|\mu^3 |\p_xy(\frac{T}{2},0)|,\\
|\p_x(\mathcal{P}y)(\frac{T}{2},0)  (E'_1(L) (\lambda_{1}^{2} - 6 \mu^{2}) - 5 \ii E'_{-1}(L) \lambda_{1} \mu)|&\lesssim |L-2\pi|\mu^2|\p_x(\mathcal{P}y)(\frac{T}{2},0)|.
\end{align*}
We need to be careful with the term $\poscals{y(\frac{T}{2},\cdot)}{E_1}_{L^2(0,L)}$. Since $\E_{\pm1}$ are eigenfunctions of $\B$ in $(0,L)$, we know that $\B(\E_{\pm1})=\ii\lambda_{\pm1}\E_{\pm1}$. Let $f_{\pm1}(x)=\E_{\pm1}(L-x)$ for $x\in(0,L)$. It is easy to see that
\begin{equation*}
\B(f_{\pm1})=-\ii\lambda_{\pm1}f_{\pm1}.
\end{equation*}
Then $f_{-1}$ is the eigenfunction of $\B$ associated with $\ii\lambda_{1}$. Since every eigenvalue of $\B$ is simple, we know that $f_{-1}(x)=\E_{-1}(L-x)\in\mathrm{Span}\E_1$. Considering that $\Im f'_{-1}(0)=-\Im\E_{-1}'(L)=-\Im\E_{-1}'(0)=\Im\E'_{1}(0)\neq0$. We know that $f_{-1}(x)=\E_{-1}(L-x)=\E_1(x)$. Similarly, $\E_{1}(L-x)=\E_{-1}(x)$. This implies that $E_1(L-x)=E_(x)$. Due to Lemma \ref{lem: real eigenvalue close to 0}, $\F_{\zeta_0}(L-x)=-2r_1(1-\cos{(L-x)})+\bigO((L-2\pi))$. Combining with Proposition \ref{coro: expansion of E+ and E-}, we deduce that
\begin{equation*}
E_1(x)=\widetilde{r}_1\F_{\zeta_0}(L-x)+\bigO((L-2\pi)).   
\end{equation*}
Then, we look at the term $\poscals{y(\frac{T}{2},\cdot)}{E_1}_{L^2(0,L)}$. 
\begin{align*}
|\poscals{y(\frac{T}{2},\cdot)}{E_1}_{L^2(0,L)}|
&\leq\widetilde{r}_1|\int_0^Ly(\frac{T}{2},x)\F_{\zeta_0}(L-x)dx|+\|y(\frac{T}{2},\cdot)\|_{L^1(0,L)}\bigO(|L-2\pi|)\\
&\leq \widetilde{r}_1|\poscalr{y(\frac{T}{2},\cdot)}{\F_{\zeta_0}}_{(0,L)}|+\|y(\frac{T}{2},\cdot)\|_{L^1(0,L)}\bigO(|L-2\pi|).
\end{align*}
Since $y(\frac{T}{2},\cdot)\in\h^2_0$, $\poscalr{y(\frac{T}{2},\cdot)}{\F_{\zeta_0}}_{(0,L)}=0$ by the definition of the space $H_{\A}^2$ in \eqref{eq: defi for space h0 and h}. Due to H\"older inequality, we know that $\|y(\frac{T}{2},\cdot)\|_{L^1(0,L)}\leq L\|y(\frac{T}{2},\cdot)\|_{L^{\infty}(0,L)}\leq 3\pi\|y(\frac{T}{2},\cdot)\|_{L^{\infty}(0,L)}$. 
Hence,
\begin{equation*}
|\poscals{y(\frac{T}{2},\cdot)}{E_1}_{L^2(0,L)} (\lambda_{1}^{2} + 4 \mu^{2}) (\lambda_{1}^{2} + 9 \mu^{2})|\lesssim |L-2\pi|\mu^4\|y(\frac{T}{2},\cdot)\|_{L^{\infty}(0,L)}.   
\end{equation*}
Using the estimate \eqref{eq: smoothing estimate-toy-2pi}, there exists a constant $\Tilde{C}_1$, which is independent of $T$ and $L$ such that
\begin{equation*}
|c_1|\leq |F_1||L-2\pi|\mu^4 \frac{\Tilde{C}_1}{T^3}\|y^0\|_{L^2(0,L)}.   
\end{equation*}
Therefore, combining with the estimates of $F_1$ in \eqref{eq: est-F_j-2pi}, we deduce that $|c_1|\leq \frac{\Tilde{C}_1\mu}{T^3}\|y^0\|_{L^2(0,L)}$.
We perform same estimates for $c_2$ and $c_3$ with constants $\Tilde{C}_2$ and $\Tilde{C}_3$. Let $\Tilde{C}:=\Tilde{C}_1+\Tilde{C}_2+\Tilde{C}_3+1$. We obtain 
\begin{equation}\label{eq: est-c_j}
|c_j|\leq \frac{\Tilde{C}\mu}{T^3}\|y^0\|_{L^2(0,L)}, j=1,2,3   
\end{equation}
By the definition of $\varrho$ in \eqref{eq: defi of varrho}, we know that $\varrho\in\ii\R$ and $|\varrho|\leq\frac{C_3e^{-\frac{T}{4}\mu}}{T^3}\|y^0\|_{L^2(0,L)}$.
This implies that $z(T,x)=\varrho E_{-1}(z)$ is a real function. Using Lemma \ref{lem: control-toy-2pi}, we are able to construct a continuous function $v$ such that 
\begin{equation}
\left\{
\begin{array}{lll}
    \p_tz+\p_x^3z+\p_xz=0 & \text{ in }(\frac{T}{2},T)\times(0,L), \\
     z(t,0)=z(t,L)=0&  \text{ in }(\frac{T}{2},T),\\
     \p_xz(t,L)-\p_xz(t,0)=v(t)&\text{ in }(\frac{T}{2},T),\\
     z(\frac{T}{2},x)=z^0_{\frac{T}{2}}(x)&\text{ in }(0,L),
\end{array}
\right.
\end{equation}
such that $z(T,x)=\varrho E_{-1}(x)$ and
\begin{align*}
\|\p_x z(\cdot,L)\|_{L^{\infty}(\frac{T}{2},T)}\leq  K_3e^{\frac{2\sqrt{2}K}{\sqrt{T}}}\|z^0_{\frac{T}{2}}\|_{H^{6}},\;
\|z(t,\cdot)\|_{L^2(\frac{T}{2},T)}\leq  K_2e^{\frac{2\sqrt{2}K}{\sqrt{T}}}\|z^0_{\frac{T}{2}}\|_{H^{3}}\leq K_2e^{\frac{2\sqrt{2}K}{\sqrt{T}}}\|z^0_{\frac{T}{2}}\|_{H^{6}} .
\end{align*}
By Proposition \ref{prop: h_mu-est-uniform}, there exists a constant $C_h$ such that 
\begin{align*}
\|h_{\mu_j}\|_{L^2(0,L)}\leq C_h|\mu_j|^{-\frac{1}{2}},\;
\|h_{\mu_j}\|_{H^6(0,L)}\leq C_h|\mu_j|^{\frac{5}{2}},j=1,2,3.
\end{align*}
Using the estimates \eqref{eq: smoothing estimate-toy-2pi} and \eqref{eq: est-c_j},
\begin{align*}
\|z^0_{\frac{T}{2}}\|_{H^{6}}
\leq \frac{C_1}{T^3}\|y^0\|_{L^2(0,L)}+3\frac{\Tilde{C}\mu}{T^3}C_h\mu^{\frac{5}{2}}\|y^0\|_{L^2(0,L)}
\leq \frac{C_1+3\Tilde{C}C_h\mu^{4}}{T^3}\|y^0\|_{L^2(0,L)}.
\end{align*}
Therefore, we obtain
\begin{align*}
\|\p_x z(\cdot,L)\|_{L^{\infty}(\frac{T}{2},T)}\leq K_3e^{\frac{2\sqrt{2}K}{\sqrt{T}}}\frac{C_1+3\Tilde{C}C_h\mu^{4}}{T^3}\|y^0\|_{L^2(0,L)}, \\
\|z(t,\cdot)\|_{L^{2}(0,L)}\leq K_2e^{\frac{2\sqrt{2}K}{\sqrt{T}}}\frac{C_1+3\Tilde{C}C_h\mu^{4}}{T^3}\|y^0\|_{L^2(0,L)}.
\end{align*}
For the other part, for $j=1,2,3$, we define a function $z_{\mu_j}(t,x)$ by 
$$z_{\mu_j}(t,x)=e^{-\mu_j(t-\frac{T}{2})}c_jh_{\mu_j}(x).$$ 
Then $z_{\mu_j}$ satisfies the equation
\begin{equation}
\left\{
\begin{array}{lll}
    \p_tz_{\mu_j}+\p_x^3z_{\mu_j}+\p_xz_{\mu_j}=-\mu_j z_{\mu_j}+\mu_j z_{\mu_j}=0 & \text{ in }(\frac{T}{2},T)\times(0,L), \\
     z_{\mu_j}(t,0)=z_{\mu_j}(t,L)=0&  \text{ in }(\frac{T}{2},T),\\
     \p_xz_{\mu_j}(t,L)=e^{-\mu_j(t-\frac{T}{2})}c_jh'_{\mu_j}(L)&\text{ in }(\frac{T}{2},T),\\
     z_{\mu_j}(\frac{T}{2},x)=c_jh_{\mu_j}(x)&\text{ in }(0,L),
\end{array}
\right.
\end{equation}
Then it is easy to see 
$$\|z_{\mu_j}(t,\cdot)\|_{L^2(0,L)}\leq \|e^{-\mu_j(t-\frac{T}{2})}c_jh_{\mu_j}(\cdot)\|_{L^2(0,L)}
\leq C_he^{-\mu_j(t-\frac{T}{2})}\frac{\Tilde{C}\mu^{\frac{1}{2}}}{T^3}\|y^0\|_{L^2(0,L)}.$$
Using the estimate \eqref{eq: est-h_mu'(L)}, the control cost 
\begin{align*}
\|\p_xz_{\mu_j}(\cdot,L)\|_{L^{\infty}(\frac{T}{2},T)}
\leq\frac{\Tilde{C}\mu}{T^3}\|y^0\|_{L^2(0,L)}|h'_{\mu_j}(L)|
\leq \frac{C_4e^{-\frac{\mu^{\frac{1}{3}}}{4}L}}{T^3}\|y^0\|_{L^2(0,L)}.
\end{align*}
We set 
\begin{equation}
\begin{array}{ll}
  w_1(t)=\left\{
  \begin{array}{ll}
     0  &t\in[0,\frac{T}{2}), \\
     \p_x z(t,L)  &t\in[\frac{T}{2},T],
  \end{array}
\right.   &w_{j+1}(t)=\left\{
  \begin{array}{ll}
     0  &t\in[0,\frac{T}{2}), \\
     \p_x z_{\mu_j}(t,L)  &t\in[\frac{T}{2},T], 
  \end{array}
\right. j=1,2,3.
\end{array}
\end{equation}
We consider the solution $y_j$ to 
\begin{equation}\label{eq: defi-y_j-2pi}
\left\{
\begin{array}{lll}
    \p_ty_j+\p_x^3y_j+\p_xy_j=0 & \text{ in }(\frac{T}{2},T)\times(0,L), \\
     y_j(t,0)=y_j(t,L)=0&  \text{ in }(\frac{T}{2},T),\\
     \p_x y_j(t,L)= w_j(t)&\text{ in }(\frac{T}{2},T),\\
     y_j(\frac{T}{2},x)=y_j^0(x)&\text{ in }(0,L),
\end{array}
\right.    
\end{equation}
where $y_1^0(x)=z^0_{\frac{T}{2}}(x)$, $y_{j+1}^0(x)=c_jh_{\mu_j}(x)$, $j=1,2,3$. We have the following properties:
\begin{enumerate}
    \item $y(\frac{T}{2},x)=y_1^0(x)+y_2^0(x)+y_3^0(x)+y_4^0(x)$;
    \item By the uniqueness, we know that $y_1(t,x)=z(t,x)$, $y_{j+1}(t,x)=z_{\mu_j}(t,x)$, $j=1,2,3$, in $(\frac{T}{2},T)\times(0,L)$.
\end{enumerate}
Now let us consider the solutions in the time interval $[0,T]$. Define
\begin{equation}
  Y(t,x)=\left\{
  \begin{array}{ll}
     y(t,x)  &(t,x)\in[0,\frac{T}{2}]\times(0,L), \\
     y_1(t,x)+y_2(t,x)+y_3(t,x)+y_4(t,x)  &(t,x)\in[\frac{T}{2},T]\times(0,L).
  \end{array}
\right. 
\end{equation}
Then $Y$ solves the equation \eqref{eq: linearized KdV system-control-intro}
where $u(t)=u_1(t)+u_2(t)+u_3(t)+u_4(t)$. Indeed, $Y\in C([0,T],L^2(0,L))$ and in particular, $Y$ is continuous at the time $t=\frac{T}{2}$. For $t\in[0,\frac{T}{2}]$, by energy estimates, $\|Y(t,\cdot)\|_{L^2(0,L)}\leq \|y^0\|_{L^2(0,L)}$.
We define a new constant $\mathcal{K}$, independent of $L$,
\begin{equation*}
    \mathcal{K}:=(K_2+K_3)(C_1+6\Tilde{C}C_h)+2\Tilde{C}C_h+2C_4+2C_3+1.
\end{equation*}
 Since $|\lambda_{1}|=c_0|L-2\pi|+\bigO((L-2\pi)^2)$, choosing $\delta$ sufficiently small and $\mu>\max\{2c_0\pi,1\}$, for $t\in[\frac{T}{2},T]$, we collect the estimates above,
\begin{align}
\|y_1(t,\cdot)\|_{L^2(0,L)}&\leq K_2e^{\frac{2\sqrt{2}K}{\sqrt{T}}}\frac{C_1+3\Tilde{C}C_h\mu^{4}}{T^3}\|y^0\|_{L^2(0,L)}\leq \mathcal{K}e^{\frac{2\sqrt{2}K}{\sqrt{T}}}\frac{\mu^{4}}{T^3}\|y^0\|_{L^2(0,L)},\label{eq: L^2-y_1-single-2pi}\\
\|y_{j+1}(t,\cdot)\|_{L^2(0,L)}&\leq C_he^{-\mu_j(t-\frac{T}{2})}\frac{\Tilde{C}\mu^{\frac{1}{2}}}{T^3}\|y^0\|_{L^2(0,L)}\leq \mathcal{K}\frac{\mu^{\frac{1}{2}}e^{-\mu_j(t-\frac{T}{2})}}{T^3}\|y^0\|_{L^2(0,L)},j=1,2,3.\label{eq: L^2-y_j+1-single-2pi}
\end{align}
As for the control cost
\begin{align}
\|w_1\|_{L^{\infty}(0,T)}&\leq K_3e^{\frac{2\sqrt{2}K}{\sqrt{T}}}\frac{C_1+3\Tilde{C}C_h\mu^{4}}{T^3}\|y^0\|_{L^2(0,L)}\leq \mathcal{K}e^{\frac{2\sqrt{2}K}{\sqrt{T}}}\frac{\mu^4}{T^3}\|y^0\|_{L^2(0,L)} ,\label{eq: cost estimate-w-1-2pi}\\
\|w_{j+1}\|_{L^{\infty}(0,T)}&\leq \frac{C_4e^{-\frac{\mu^{\frac{1}{3}}}{4}L}}{T^3}\|y^0\|_{L^2(0,L)}\leq\mathcal{K}\frac{e^{-\frac{\mu^{\frac{1}{3}}}{4}L}}{T^3}\|y^0\|_{L^2(0,L)} ,j=1,2,3.
\end{align}
\end{proof}
\begin{prop}[Iteration schemes]\label{prop: iteration schemes-2pi}
Let $T>0$. Suppose that Assumption \ref{ass: Compact intervals including 2pi} holds.  For every $L\in I\setminus\{2\pi\}$, and $\forall y^0\in H_{\A}$. There exists a function $u(t) \in L^2(0,T)$ such that the solution $y$ to the system \eqref{eq: linearized KdV system-control-intro} 
satisfies $\lim_{t\rightarrow T^-}\|y(t, \cdot)\|_{L^2(0, L)}= 0$
    and there exists a constant $\mathcal{K}$ such that
\begin{equation}\label{eq: est-control-L-infty-2pi}
 \|u\|_{L^{\infty}(0,T)}\leq \Tilde{\mathcal{K}}e^{\frac{\mathcal{K}}{\sqrt{T}}}\|y^{0}\|_{L^2(0,L)}.   
\end{equation}
\end{prop}
\begin{proof}
For every $L\in I\setminus \{2\pi\}$, without loss of generality, we set $T\in(0,1)$. Suppose that $T=2^{-n_0}$. Let $T_n=2^{-n_0}(1-2^{-n})$ and $\mathcal{I}_n=[T_{n-1},T_n)$, $n\in\N$. Let us also take a constant $Q>0$ that is independent of $T$ and $L$. And $Q$ will be fixed later on. Now our concerned time interval $[0,T)$ has a partition as $\bigcup_{n=1}^{\infty}\mathcal{I}_n$.
We fix our choice of $\mu_{1,n}= Q 2^{\frac{3}{2}(n_0+ n)}, n= 1,2,...$ and $\mu_{2,n}=2\mu_{1,n}$ and $\mu_{3,n}=3\mu_{1,n}$. On each time interval $\mathcal{I}_n$, we construct the control function $u^n(t)\in L^2(\mathcal{I}_n)$ the unique solution $y_n$ of the Cauchy problem 
\begin{equation}
\left\{
\begin{array}{lll}
    \p_ty_n+\p_x^3y_n+\p_xy_n=0 & \text{ in }\mathcal{I}_n\times(0,L), \\
     y_n(t,0)=y_n(t,L)=0&  \text{ in }\mathcal{I}_n,\\
     \p_x y_n(t,L)= u^n(t)&\text{ in }\mathcal{I}_n,\\
     y_n(T_{n-1},x)=y^{n-1}(x)&\text{ in }(0,L),
\end{array}
\right.
\end{equation}
satisfies 
\begin{align*}
\|y(t,\cdot)\|_{L^2(0,L)}&\leq \|y^{n-1}\|_{L^(0,L)},t\in(T_{n-1},\frac{T_{n-1}+T_n}{2}],\\
\|y_n(t,\cdot)\|_{L^2(0,L)}&\leq \mathcal{K}\left(e^{2\sqrt{2}K2^{\frac{n_0+n}{2}}}Q^{4}2^{6(n_0+n)}+ Q^{\frac{1}{2}}2^{\frac{3}{4}(n_0+n)}\right)2^{3(n_0+n)}\|y^{n-1}\|_{L^2(0,L)},t\in (\frac{T_{n-1}+T_n}{2},T_{n})\\
\|u^n\|_{L^{\infty}(\mathcal{I}_n)}&\leq \mathcal{K}\left(e^{2\sqrt{2}K2^{\frac{n_0+n}{2}}}Q^{4}2^{6(n_0+n)}+e^{-\frac{Q^{\frac{1}{3}}2^{\frac{n_0+n}{2}}}{4}L}\right)2^{3(n_0+n)}\|y^{n-1}\|_{L^2(0,L)},\\
\|y^n\|_{L^2(0,L)}&\leq 2\mathcal{K}e^{-Q2^{\frac{n_0+n}{2}-2}}2^{3(n_0+n)} \|y^{n-1}\|_{L^2(0,L)}
\end{align*}
With the help of the good choice of $Q$, we simplify the estimates for $n>1$,
\begin{align}
\|y_n(t,\cdot)\|_{L^2(0,L)}&\leq  e^{-\frac{Q2^{\frac{n_0}{2}}(2^{\frac{n-1}{2}}-1)}{4(2-\sqrt{2})}}\|y^{0}\|_{L^2(0,L)},t\in(T_{n-1},\frac{T_{n-1}+T_n}{2}],\\
\|y_n(t,\cdot)\|_{L^2(0,L)}&\leq (Q^{4}+Q^{\frac{1}{2}})e^{-\frac{Q2^{\frac{n_0}{2}}(2^{\frac{n-1}{2}}-1)}{4(2-\sqrt{2})}} \|y^{0}\|_{L^2(0,L)},t\in (\frac{T_{n-1}+T_n}{2},T_{n})\\
\|u^n\|_{L^{\infty}(\mathcal{I}_n)}&\leq 2e^{-\frac{Q2^{\frac{n_0}{2}}(2^{\frac{n-1}{2}}-1)}{4(2-\sqrt{2})}} \|y^{0}\|_{L^2(0,L)}\label{eq: final estimate for w_n-2pi}. 
\end{align}
Therefore, we know that for any $L\in I\setminus\{L_0\}$, $\lim_{t\rightarrow T^{-}}\|y(t,\cdot)\|_{L^2(0,L)}=0$.
Moreover, combining with the estimates for $n=1$,
there exists a constant $\Tilde{\mathcal{K}}$ such that $\|\p_xy_1(\cdot,L)\|_{L^{\infty}(\mathcal{I}_1)}\leq \Tilde{\mathcal{K}}e^{\frac{\mathcal{K}}{\sqrt{T}}}\|y^{0}\|_{L^2(0,L)} $,
and $\|u\|_{L^{\infty}(0,T)}\leq \Tilde{\mathcal{K}}e^{\frac{\mathcal{K}}{\sqrt{T}}}\|y^{0}\|_{L^2(0,L)}$.
\end{proof}
\subsubsection{Proof of main results}
In this section, we present our proof of Theorem \ref{thm: main theorem linear version} in the model case $L_0=2\pi$.

\begin{proof}[Proof of Theorem \ref{thm: main result in modal case-2pi}]
In the following, we focus on the situation $L\in [2\pi-\delta,2\pi+\delta]$ with some explicit and sufficiently small $\delta$. Indeed, for $L\in [2\pi-1, 2\pi+1]\setminus (2\pi-\delta,2\pi+\delta)$, we get a uniform exponential decay rate thanks to Theorem \ref{thm: main result with L-L0}. Using once again the HUM on $H_{\A}$, it suffices to analyze the controlled KdV system
\begin{equation}
\left\{
\begin{array}{lll}
    \p_t\Tilde{y}+\p_x^3\Tilde{y}+\p_x\Tilde{y}=0 & \text{ in }(0,T)\times(0,L), \\
     \Tilde{y}(t,0)=\Tilde{y}(t,L)=0&  \text{ in }(0,T),\\
     \p_x\Tilde{y}(t,L)=u(t)&\text{ in }(0,T),\\
     \Tilde{y}(0,x)=\Tilde{y}^0(x)&\text{ in }(0,L).
\end{array}
\right.
\end{equation} 
The quantitative observability \eqref{eq: quantitative-ob-main-2pi} is equivalent to the following estimate $\|u\|_{L^2(0,T)}\leq C\|\Tilde{y}^0\|_{L^2(0,L)}$. 
Using Proposition \ref{prop: iteration schemes-2pi} and the estimate \eqref{eq: est-control-L-infty-2pi}, we obtain 
\begin{equation*}
 \|u\|_{L^2(0,T)}\leq T \|u\|_{L^{\infty}(0,T)}\leq \Tilde{\mathcal{K}}e^{\frac{\mathcal{K}}{\sqrt{T}}}\|\Tilde{y}^0\|_{L^2(0,L)}.
\end{equation*}
\end{proof}
To complete the full picture of the stability result, we continue to analyze the behaviors of the solution in the subspace $\C\F_{\zeta_0}$. As we defined in Lemma \ref{lem: real eigenvalue close to 0}, we know that
\begin{align*}
    \zeta_0=-\frac{1}{3\pi}(L-2\pi)^2+\bigO(|L-2\pi|^3),\;
    |\F_{\zeta_0}'(0)|=3r_1|L-2\pi|+\bigO(|L-2\pi|^2),
\end{align*}
where $r_1$ is a normalized constant such that $\|\F_{\zeta_0}\|_{L^2(0,L)}=1$, and we know that $|r_1|$ is uniformly bounded as we presented in Proposition \ref{prop: Asymp in L}.
\begin{prop}
Let $T>0$. Let $I=[2\pi- 1,2\pi+ 1]$.  For every $L\in I\setminus\{2\pi\}$, and $y^0=\F_{\zeta_0}$, there exists a constant $C=C(T,L)$ such that the following quantitative observability inequality 
\begin{equation}\label{eq: quantitative-ob-main-2pi-zeta0}
   \|S(T)\F_{\zeta_0}\|^2_{L^2(0,L)} \leq  C^2\int_0^T|\p_x y(t,0)|^2dt
\end{equation}
holds for any solution $y$ to \eqref{eq: linear KdV-stability-intro} with initial datum $y^0=\F_{\zeta_0}$.
\end{prop}
\begin{proof}
We present a direct proof of this proposition. Since $\F_{\zeta_0}$ is the eigenfunction associated with the eigenvalue $\zeta_0$. We know that the solution $y$ to the KdV system is in the form $y(t,x)=e^{\zeta_0t}\F_{\zeta_0}(x)$. Hence, $|\p_xy(t,0)|=e^{\zeta_0t}|\F'_{\zeta_0}(0)|$ and 
\begin{equation*}
\int_0^T|\p_x y(t,0)|^2dt=|\F'_{\zeta_0}(0)|^2\int_0^Te^{2\zeta_0t}dt=-\frac{|\F'_{\zeta_0}(0)|^2}{2\zeta_0}(1-e^{2\zeta_0T})=18r_1T(L-2\pi)^2+\bigO(|L-2\pi|^3).    
\end{equation*}
There exists a constant $C^2\gtrsim \frac{e^{2\zeta_0 T}}{T(L-2\pi)^2}$ such that the observability \eqref{eq: quantitative-ob-main-2pi-zeta0} holds.
\end{proof}

\subsection{Around Type II unreachable pair $(k, l)$}\label{sec: Around Type II unreachable pair}
\subsubsection{Quasi-invariant subspaces and transition projectors}
In this section, we aim to prove a similar result as we presented in Section \ref{sec: Around Type I unreachable pair}. We assume that $L_0=2\pi\sqrt{7}$ in this sequel. 
\begin{prop}
For $L_0=2\pi\sqrt{7}$, there two eigenvalues $\lambda_{c,\pm1}=\pm\frac{6\sqrt{7}}{49}$ with the associated eigenfunctions $\G_{c,1}(x)= \frac{1}{\sqrt{84\pi\sqrt{7}}}\left(-e^{\ii x \frac{3\sqrt{7}}{7}}-4e^{-\ii x \frac{2\sqrt{7}}{7}}+5e^{-\ii x \frac{\sqrt{7}}{7}}\right)$ and $\G_{c,-1}=\overline{\G}_{c,1}$ such that 
\begin{equation*}
\left\{
\begin{array}{l}
     \G_{c,\pm1}'''(x)+\G_{c,\pm1}'(x)\pm\ii\frac{6\sqrt{7}}{49}\G_{c,\pm1}=0,x\in(0,2\pi\sqrt{7}),  \\
     \G_{c,\pm1}(0)=\G_{c,\pm1}(2\pi\sqrt{7})=0,\\
     \G_{c,\pm1}'(0)=\G_{c,\pm1}'(2\pi\sqrt{7})=0.
\end{array}
\right.
\end{equation*}
\end{prop}
This is just a particular case of the system \eqref{eq: eigenvalue problem with critical length} at $L_0=2\pi\sqrt{7}$. Moreover, in this case, we find Type 2 eigenfunctions $\widetilde{\G}_{c,\pm1}$ defined by
\begin{equation*}
\widetilde{\G}_{c,1}(x):=\frac{1}{\sqrt{28\sqrt{7}\pi}}\left(3e^{\ii x \frac{3\sqrt{7}}{7}}-2e^{-\ii x \frac{2\sqrt{7}}{7}}-e^{-\ii x \frac{\sqrt{7}}{7}}\right).
\end{equation*}
We notice that $\poscals{\widetilde{\G}_{c,1}}{\G_{c,1}}=0$ and $\widetilde{\G}'_{c,1}(0)=\widetilde{\G}'_{c,1}(2\pi\sqrt{7})=\frac{\ii}{\sqrt{\sqrt{7}\pi}}\neq0$.
\begin{lem}\label{lem: eigenvalue close to 7}
Suppose that $0<\delta<\pi\sqrt{7}$ is sufficiently small. Let $I=[2\pi\sqrt{7}-\delta,2\pi\sqrt{7}+\delta]$ be a small compact interval such that $I\cap \mathcal{N}=\{2\pi\sqrt{7}\}$. Then there are two eigenvalues $\zeta_{\pm}$ and their associated real eigenfunctions $\F_{\zeta_{\pm}}$ satisfying that
\begin{align*}
|\zeta_{\pm}-\ii\lambda_{c,\pm1}|\sim(L-2\pi\sqrt{7})^2,\;|\F_{\zeta_{\pm}}'(0)|\sim|L-2\pi\sqrt{7}|. 
\end{align*}
\end{lem}
As presented before, this is a particular case $L_0=2\pi\sqrt{7}$ of Proposition \ref{prop: asymptotic expansion for A0}. By Proposition \ref{prop: Index set M-E}, we know that for each $\lambda_{c,\pm}$, there are two eigenvalues of $\B$ in $(0,L)$ approaching the same limit. We use the following notations:
\begin{align*}
    \lambda_{1}&:=\lambda_{c,1}-\frac{9}{49\pi}(L-2\pi\sqrt{7})-\frac{\sqrt{21}}{21\pi}|L-2\pi\sqrt{7}|+\bigO((L-2\pi\sqrt{7})^2);\\
    \lambda_{2}&:=\lambda_{c,1}-\frac{9}{49\pi}(L-2\pi\sqrt{7})+\frac{\sqrt{21}}{21\pi}|L-2\pi\sqrt{7}|+\bigO((L-2\pi\sqrt{7})^2);\\
     \lambda_{-1}&:=\lambda_{c,-1}+\frac{9}{49\pi}(L-2\pi\sqrt{7})+\frac{\sqrt{21}}{21\pi}|L-2\pi\sqrt{7}|+\bigO((L-2\pi\sqrt{7})^2),\\
     \lambda_{-2}&:=\lambda_{c,-1}+\frac{9}{49\pi}(L-2\pi\sqrt{7})-\frac{\sqrt{21}}{21\pi}|L-2\pi\sqrt{7}|+\bigO((L-2\pi\sqrt{7})^2).
\end{align*}
In particular, we denote the corresponding eigenfunctions by $\E_j$, $j=\pm1,\pm2$. 
\begin{coro}\label{coro: expansion of E^g and E^b}
Suppose that $0<\delta<\pi\sqrt{7}$ is sufficiently small. Let $I=[2\pi\sqrt{7}-\delta,2\pi\sqrt{7}+\delta]$ be a small compact interval such that $I\cap \mathcal{N}=\{2\pi\sqrt{7}\}$. Let $\E_{1,+}= \frac{\sqrt{98+18\sqrt{21}}}{14}\E_1+\frac{\sqrt{98-18\sqrt{21}}}{14}\E_2$ and $\E_{1,-}=\frac{\sqrt{98- 18\sqrt{21}}}{14}\E_1-\frac{\sqrt{98+ 18\sqrt{21}}}{14}\E_2$.
Then, the following statements hold: 
\begin{enumerate}
    \item $\E_{1}$ and $\E_2$ have the asymptotic expansions near $2\pi\sqrt{7}$:
    \begin{gather*}
\E_1= \frac{\sqrt{98 + 18\sqrt{21}}}{14}\G_{c,1}+ \frac{\sqrt{98 - 18\sqrt{21}}}{14}\widetilde{\G}_{c,1}+\bigO(|L-2\pi\sqrt{7}|),\\
\E_2= \frac{\sqrt{98 - 18\sqrt{21}}}{14}\G_{c,1} -\frac{\sqrt{98 + 18\sqrt{21}}}{14}\widetilde{\G}_{c,1}+\bigO(|L-2\pi\sqrt{7}|).
    \end{gather*}
    \item $\E_{1,+}$ and $\E_{1,-}(x)$ have the following asymptotic expansions near $2\pi\sqrt{7}$ 
    \begin{align*}
    \E_{1,+}(x)=\G_{c,1}+\bigO(L-2 \pi\sqrt{7}),\;
    \E_{1,-}(x)=\Tilde{\G}_{c,1}+\bigO(L-2 \pi\sqrt{7}).
    \end{align*}
    \item $(\E_{1,+})'(L)=\bigO(|L-2\pi\sqrt{7}|)$ and $(\E_{1,-})'(L)=\frac{i}{\sqrt{\pi\sqrt{7}}}+\bigO(L-2\pi\sqrt{7})$.
\end{enumerate}
\end{coro}
\begin{proof}
Here we simply apply Proposition \ref{prop: Low-frequency behaviors: Singular limits} in our particular case $L_0=2\pi\sqrt{7}$. We point out that $\poscals{\E_{1,+}}{\E_{1,-}}_{L^2(0,L)}=0$.
\end{proof}Similar results hold for $\E_{-1,+}$ and $\E_{-1,-}$. Here we omit the details and only take $\E_{1,+}$ and $\E_{1,-}$ as an example. 
For $s\geq0$, we define the following Hilbert subspaces 
\begin{equation}\label{eq: defi for space h0 and h-2pi-7}
   H_{\A}^s:=\{u\in D(\A^s):\poscalr{u}{\F_{\zeta_+}}_{(0,L)}=\poscalr{u}{\F_{\zeta_-}}_{(0,L)}=0\},\;    H_{\B}^s:=\{u\in D(\B^s):\poscals{u}{\E_{\pm1,+}}_{L^2(0,L)}=0\}.
\end{equation}
\textbf{Compensate bi-orthogonal family}
We perform a similar procedure as we prove Theorem \ref{thm: main result with L-L0}. First, we recall the bi-orthogonal family to $e^{-\ii t\lambda_j}$ defined in Proposition \ref{prop: biorthogonal family}. Then, recalling Proposition \ref{prop: Index set M-E}, for the modal case $L_0=2\pi\sqrt{7}$, there exist two eigenvalues $\lambda_{\pm 1}$ such that $\lim_{L\rightarrow 2\pi\sqrt{7}}\lambda_{1}=\lim_{L\rightarrow 2\pi\sqrt{7}}\lambda_{2}=\lambda_{c,1}$ and $\lim_{L\rightarrow 2\pi\sqrt{7}}\lambda_{-1}=\lim_{L\rightarrow 2\pi\sqrt{7}}\lambda_{-2}=\lambda_{c,-1}$.

\begin{lem}\label{lem: bi-orthogonal family-2pi-7}
Let $\mathcal{J}=\{\pm\}\cup(\Z\setminus\{0,\pm1,\pm2\})$ be a countable index set. There exists a family of functions $\{\vartheta_j\}_{j\in\mathcal{J}}$ such that 
\begin{enumerate}
    \item For $j\neq 0,\pm 1,\pm2$, $\vartheta_j=\phi_j$, where $\phi_j$ is defined in Proposition \ref{prop: biorthogonal family}.
    \item $\supp{\vartheta_j}\subset [-\frac{T}{2},\frac{T}{2}],\forall j\in\Z\backslash\{\pm 1,\pm2\}$.\label{eq: compact support of theta_j-2}
    \item $\int_{-\frac{T}{2}}^{\frac{T}{2}}\vartheta_j(s)e^{-\ii \lambda_ks}ds=\delta_{jk},\forall j\in\Z\backslash\{0,\pm 1,\pm2\}$ and $\forall k\in\Z\backslash\{0\}$. Moreover, 
    \begin{gather*}
    \int_{-\frac{T}{2}}^{\frac{T}{2}}\vartheta_{\pm}(s)e^{-\ii \lambda_k s}ds=0,\forall k\neq 0, \pm 1,\pm2\\         \int_{-\frac{T}{2}}^{\frac{T}{2}}\vartheta_{+}(s)e^{-\ii s\lambda_{1}}ds=1,\;\int_{-\frac{T}{2}}^{\frac{T}{2}}\vartheta_+(s)e^{-\ii s\lambda_{2}}ds=\prod_{k\neq1,2}\left(\frac{\lambda_k-\lambda_{2}}{\lambda_k-\lambda_{1}}\right)\\
     \int_{-\frac{T}{2}}^{\frac{T}{2}}\vartheta_{-}(s)e^{-\ii s\lambda_{-1}}ds=1,\;\int_{-\frac{T}{2}}^{\frac{T}{2}}\vartheta_-(s)e^{-\ii s\lambda_{-2}}ds=\prod_{k\neq1,2}\left(\frac{\lambda_k-\lambda_{2}}{\lambda_k-\lambda_{1}}\right).
    \end{gather*}
    \label{eq: biorthorgonal-2pi-7}
    \item For $N\in \N$, there exists a constant $K=K(N)$ such that $\|\phi^{(m)}_j\|_{L^{\infty}(\R)}\leq Ce^{\frac{K}{\sqrt{T}}}|\lambda_j|^m,\forall j\in\Z\backslash\{\pm 1\}, \forall m\in\{0,1,\cdots,N\}$. Here the constant $C$ appearing in the inequality might depend on $N$ but not on $T$, $L$ and $j$.\label{eq: bound for derivative of theta-2}
\end{enumerate}
\end{lem}
\begin{proof}
We put the proof in Appendix \ref{sec: app-bi-7pi}.
\end{proof}
In this section, we always denote $a(L):=\prod_{k\neq1,2}\left(\frac{\lambda_k-\lambda_{2}}{\lambda_k-\lambda_{1}}\right)$. Then, we consider the KdV system
\begin{equation}\label{eq: Linear KdV-2pi}
\left\{
\begin{array}{lll}
    \p_ty+\p_x^3y+\p_xy=0 & \text{ in }(0,T/2)\times(0,L), \\
     y(t,0)=y(t,L)=\p_xy(t,L)=0&  \text{ in }(0,T/2),\\
     y(0,x)=y^0(x)&\text{ in }(0,L).
\end{array}
\right.
\end{equation} 
Using the smoothing effects, for $y^0\in \h^0_0$, we obtain that $y(\frac{T}{2})\in H_{\A}^2$. Then we have the following lemma to compensate for the boundary condition of the KdV operator.
\begin{lem}\label{lem: defi of coeff-1-2-3-4-pi-7}
Let $\mu_1$, $\mu_2$, $\mu_3$, and $\mu_4$ be four distinct real positive constants, for any function $z\in  H_{\A}^2\subset H_{\A}$, there exist constants $c_1$, $c_2$, $c_3$, and $c_4$ such that 
\begin{equation}
    z- c_1 h_{\mu_1}- c_2 h_{\mu_2}-c_3 h_{\mu_3}-c_4 h_{\mu_4}\in \h^2.
\end{equation}
Here $h_{\mu_j}(1\leq j\leq4)$ are modulated functions defined in Section \ref{sec: preliminary}.
\end{lem}
\begin{proof}
For any $\mu_1$, $\mu_2$, $\mu_3$, and $\mu_4$,  we are able to construct the real functions $h_{\mu_1}$, $h_{\mu_2}$, $h_{\mu_3}$, and $h_{\mu_4}$ as 
\begin{equation}
\left\{
\begin{aligned}
     &h_{\mu_j}'''+h_{\mu_j}'=\mu_j h_{\mu_j}(x),  \\
     &h_{\mu_j}(0)=h_{\mu_j}(L)=0,\\
     &h_{\mu_j}'(L)-h_{\mu_j}'(0)=1,
\end{aligned}
\right.
\end{equation}
Let $f:= z- c_1 h_{\mu_1}- c_2 h_{\mu_2}-c_3 h_{\mu_3}-c_4 h_{\mu_4}$. Then it is easy to check that $f(0)=f(L)=0$. For the boundary derivative, we ask that
\begin{align*}
    c_1+ c_2+c_3+c_4&=-\p_xz(0),\\
    c_1\mu_1+ c_2\mu_2+c_3\mu_3+c_4\mu_4&=-\p_x(\mathcal{P}z)(0)
\end{align*}
    Next, for the projection on the direction $\E^g_{+}$, we have 
    \begin{gather}
        \poscals{z- c_1 h_{\mu_1}- c_2h_{\mu_2}-c_3 h_{\mu_3}-c_4 h_{\mu_4}}{\E_{1,+}}_{L^2(0,L)}= 0,
    \end{gather}
    which is equivalent to 
    \begin{gather}
        c_1  \poscals{h_{\mu_1}}{\E_{1,+}}_{L^2(0,L)}+ c_2 \poscals{h_{\mu_2}}{\E_{1,+}}_{L^2(0,L)}+c_3\poscals{h_{\mu_3}}{\E_{1,+}}_{L^2(0,L)}+c_4  \poscals{h_{\mu_4}}{\E_{1,+}}_{L^2(0,L)}=  \poscals{z}{\E_{1,+}}_{L^2(0,L)}
    \end{gather}
   For the term $\poscals{h_{\mu_j}}{\E_1}_{L^2(0,L)}$, we obtain $\poscals{h_{\mu_j}}{\E_{1}}_{L^2(0,L)}=\frac{\overline{\E_{1}'(L)}}{\mu_j+\ii\lambda_{1}}$. We deduce that
\begin{align}
\poscals{h_{\mu_j}}{\E_{1,+}}_{L^2(0,L)}&=\frac{\sqrt{98+18\sqrt{21}}}{14}\frac{\overline{\E_{1}'(L)}}{\mu_j+\ii\lambda_{1}}+\frac{\sqrt{98- 18\sqrt{21}}}{14}\frac{\overline{\E_{2}'(L)}}{\mu_j+\ii\lambda_{2}},\\
\poscals{h_{\mu_j}}{\E_{-1,+}}_{L^2(0,L)}&=\frac{\sqrt{98+ 18\sqrt{21}}}{14}\frac{\overline{\E_{-1}'(L)}}{\mu_j+\ii\lambda_{-1}}+\frac{\sqrt{98-18\sqrt{21}}}{14}\frac{\overline{\E_{-2}'(L)}}{\mu_j+\ii\lambda_{-2}}.
\end{align}
Therefore, we obtain a linear system of the coefficients
\begin{equation}
\left\{
\begin{array}{rl}
     c_1+ c_2+c_3+c_4&=-\p_xz(0)  \\
     c_1\mu_1+ c_2\mu_2+c_3\mu_3+c_4\mu_4&=-\p_x(\mathcal{P}z)(0),\\
     \sum_{j=1}^4c_j\left( \frac{\sqrt{98+18\sqrt{21}}}{14}\frac{\overline{\E_{1}'(L)}}{\mu_j+\ii\lambda_{1}}+\frac{\sqrt{98- 18\sqrt{21}}}{14}\frac{\overline{\E_{2}'(L)}}{\mu_j+\ii\lambda_{2}}\right)&= -\poscals{z}{\E_{1,+}}_{L^2(0,L)},\\
     \sum_{j=1}^4c_j \left(\frac{\sqrt{98+ 18\sqrt{21}}}{14}\frac{\overline{\E_{-1}'(L)}}{\mu_j+\ii\lambda_{-1}}+\frac{\sqrt{98-18\sqrt{21}}}{14}\frac{\overline{\E_{-2}'(L)}}{\mu_j+\ii\lambda_{-2}}\right)&= -\poscals{z}{\E_{-1,+}}_{L^2(0,L)}.
\end{array}
\right.
\end{equation}
Choosing the distinct positive real numbers $(\mu_1,\mu_2,\mu_3,\mu_4)$, then we find the real coefficients $(c_1,c_2,c_3,c_4)$ that compensate the boundary conditions.
\end{proof}

\begin{lem}\label{lem: construction of v-2pi-sqrt-7}
For any initial state $z^0\in H_{\B}^2$ and any final state $z^T\in H_{\B}^2$ (as defined in \eqref{eq: defi for space h0 and h}), there exists a function $v\in C^1(0,T)$ such that the solution $z$ to the controlled KdV system 
\begin{equation}\label{eq: KdV-z-v-2pi}
\left\{
\begin{array}{lll}
    \p_tz+\p_x^3z+\p_xz=0 & \text{ in }(0,T)\times(0,L), \\
     z(t,0)=z(t,L)=0&  \text{ in }(0.T),\\
     \p_xz(t,L-\p_xz(t,0))=v(t)&\text{ in }(0,T),\\
     z(0,x)=z^0\in H_B^2&\text{ in }(0,L),
\end{array}
\right.
\end{equation}
satisfies that $z(T)=z^T\in \h^2$. In particular, we can choose that $z(T,x)=\varrho_1(T)\E_{1,-}(x)+\varrho_{-1}(T)\E_{-1,-}(x)$. Furthermore, $\varrho_{\pm1}$ satisfies that $\varrho_1=\overline{\varrho_{-1}}$ and 
\begin{equation*}
     z(T)+ \sum_{j=1}^4c_j h_{\mu_j} e^{-\mu_j \frac{T}{2}}\in H_{\A}.
\end{equation*}
\end{lem}
\begin{proof}
Suppose that $c_1$, $c_2$, $c_3$, and $c_4$ are fixed. As we presented in Section \ref{sec: construction of the control}, we aim to construct the control function $v\in C^1(0,T)$ for 
\begin{equation*}
\left\{
\begin{array}{lll}
    \p_tz+\p_x^3z+\p_xz=0 & \text{ in }(0,T)\times(0,L), \\
     z(t,0)=z(t,L)=0&  \text{ in }(0,T),\\
     \p_xz(t,L)-\p_xz(t,0)=v(t)&\text{ in }(0,T),\\
     z(0,x)=z^0(x)&\text{ in }(0,L),
\end{array}
\right.
\end{equation*}
with the constraints $v(0)=v(T)=0$. Thanks to the family $\{\vartheta_j\}_{j\in\mathcal{J}}$, we construct $v$ as follows
\begin{equation}\label{eq: defi of control v-vartheta-2pi7}
    v(t)=v_-\vartheta_-(t-\frac{T}{2})+v_+\vartheta_+(t-\frac{T}{2})+\sum_{k\neq\pm 1,\pm2}v_k\vartheta_k(t-\frac{T}{2}), v_+=\overline{v_-}.
\end{equation}
Here $h_k$ remains the same as in Section \ref{sec: construction of the control} and we focus on the real solutions ,i.e.,  $z^0_{1,-}=\overline{z^0_{-1,-}}$ and $z^T_{1,-}=\overline{z^T_{-1,-}}$ and we have the expansions 
\begin{gather*}
z^0(x)=z^0_{1,-}\E_{1,-}(x)+z^0_{-1,-}\E_{-1,-}(x)+\sum_{k\in\Z\setminus\{0,\pm 1,\pm2\}}z^0_k\E_k(x),\\
z(T,x)=z^T_{1,-}\E_{1,-}(x)+z^T_{-1,-}\E_{-1,-}(x)+\sum_{k\in\Z\setminus\{0,\pm 1,\pm2\}}z^T_k\E_k(x).
\end{gather*}
Then the solution $z$ to the system \eqref{eq: KdV-z-v-2pi} has the following expansion 
\begin{align*}
z(t,x)&=\frac{\sqrt{98 - 18\sqrt{21}}}{14}e^{\ii \lambda_{1}t}z^0_{1,-}\E_{1}(x)-\frac{\sqrt{98+18\sqrt{21}}}{14}e^{\ii \lambda_{2}t}z^0_{-1,-}\E_{2}(x)+\frac{\sqrt{98 - 18\sqrt{21}}}{14}e^{\ii \lambda_{-1}t}z^0_{-1,-}\E_{-1}(x)\notag\\
&-\frac{\sqrt{98+18\sqrt{21}}}{14}e^{\ii \lambda_{-2}t}z^0_{-1,-}\E_{-2}(x)+\sum_{k\in\Z\setminus\{0,\pm 1\}}z^0_ke^{\ii\lambda_k t}\E_k(x)\notag\\
&-\sum_{k\neq0} \sum_{j\neq\pm 1}\ii\lambda_kh_k\int_0^te^{\ii(t-s)\lambda_k}v_j\vartheta_j(s-\frac{T}{2})ds\E_k(x).\label{eq: construction-z-expansion-2pi7}
\end{align*}
In particular, at $t=T$, 
\begin{align*}
z(T,x)&=\frac{\sqrt{98-18\sqrt{21}}}{14}e^{\ii \lambda_{1}T}z^0_{1,-}\E_{1}(x)-\frac{\sqrt{98+18\sqrt{21}}}{14}e^{\ii \lambda_{2}T}z^0_{-1,-}\E_{2}(x)+\frac{\sqrt{98-18\sqrt{21}}}{14}e^{\ii \lambda_{-1}T}z^0_{-1,-}\E_{-1}(x)\\
&-\frac{\sqrt{98+18\sqrt{21}}}{14}e^{\ii \lambda_{-2}T}z^0_{-1,-}\E_{-2}(x)+\sum_{k\in\Z\setminus\{0,\pm 1,\pm2\}}z^0_ke^{\ii\lambda_k T}\E_k(x)\\
&-\sum_{k\neq0} \sum_{j\neq\pm 1}\ii\lambda_kh_k\int_0^Te^{\ii(T-s)\lambda_k}v_j\vartheta_j(s-\frac{T}{2})ds\E_k(x).
\end{align*}
As a consequence, we use the property \ref{eq: biorthorgonal-2pi-7} in Lemma \ref{lem: bi-orthogonal family-2pi-7}. Thus,
\begin{align*}
\frac{\sqrt{98-18\sqrt{21}}}{14}e^{\ii \lambda_{1}T}z^0_{1,-} -\ii\lambda_{1}h_{1}e^{\ii\frac{T\lambda_{1}}{2}}v_+&=\frac{\sqrt{98-18\sqrt{21}}}{14}z^T_{1,-},\\
-\frac{\sqrt{98+18\sqrt{21}}}{14}e^{\ii \lambda_{2}T}z^0_{1,-}-\ii\lambda_{2}h_{2}e^{\ii\frac{T\lambda_{2}}{2}}v_+a(L)&=-\frac{\sqrt{98+ 18\sqrt{21}}}{14}z^T_{1,-},\\
\frac{\sqrt{98-18\sqrt{21}}}{14}e^{\ii \lambda_{-1}T}z^0_{-1,-} -\ii\lambda_{-1}h_{-1}e^{\ii\frac{T\lambda_{-1}}{2}}v_-&=\frac{\sqrt{98-18\sqrt{21}}}{14}z^T_{-1,-},\\
-\frac{\sqrt{98+18\sqrt{21}}}{14}e^{\ii \lambda_{-2}T}z^0_{-1,-}-\ii\lambda_{-2}h_{-2}e^{\ii\frac{T\lambda_{-2}}{2}}v_-a(L)&=-\frac{\sqrt{98+18\sqrt{21}}}{14}z^T_{-1,-},\\
z^0_ke^{\ii\lambda_k T}-\ii\lambda_kh_ke^{\ii\frac{\lambda_k T}{2}}v_k&=z^T_k,k\neq 0,\pm 1,\pm2.
\end{align*}
We notice that if we take the conjugate of the first equation, we obtain the third one. At the same time, if we take the conjugate of the second equation, we obtain the fourth one. This implies that there are two independent equations above and we can find $\Re(v_+)$ and $\Im(v_+)$ accordingly.
Therefore ,we obtain
\begin{gather*}
    v_k=\frac{z^0_ke^{\ii\frac{\lambda_k T}{2}}-z^T_ke^{-\ii\frac{\lambda_k T}{2}}}{\ii\lambda_kh_k},k\neq 0,\pm 1,\pm2,\\
    v_+=\frac{\sqrt{98-18\sqrt{21}}}{14}\frac{z^0_{1,-}e^{\ii\frac{\lambda_1 T}{2}}-z^T_{1,-}e^{-\ii\frac{\lambda_1 T}{2}}}{\ii\lambda_1h_1},\;\;
    v_-=\frac{\sqrt{98-18\sqrt{21}}}{14}\frac{z^0_{-1,-}e^{\ii\frac{\lambda_{-1} T}{2}}-z^T_{-1,-}e^{-\ii\frac{\lambda_{-1} T}{2}}}{\ii\lambda_{-1}h_{-1}}
\end{gather*}
In particular, we can choose that $z(T,x)=\varrho_{1}\E_{1,-}(x)+\varrho_{-1}\E_{-1,-}(x)\in \h^2$. Indeed, $z^T_{\pm,-}=\varrho_{\pm1}$ and $z^T_j=0$ for $j\neq 0,\pm 1,\pm2$. To achieve this final target, we construct a special control function $v$ as follows:
\begin{multline*}
v(t):=\frac{\sqrt{98-18\sqrt{21}}}{14}(\frac{z^0_{1,-}e^{\ii\frac{\lambda_1 T}{2}}-\varrho_{1}e^{-\ii\frac{\lambda_1 T}{2}}}{\ii\lambda_1h_1}\vartheta_+(t-\frac{T}{2})+\frac{z^0_{-1,-}e^{\ii\frac{\lambda_{-1} T}{2}}-\varrho_{-1}e^{-\ii\frac{\lambda_{-1} T}{2}}}{\ii\lambda_{-1}h_{-1}}\vartheta_-(t-\frac{T}{2}))\\
+\sum_{k\neq0,\pm 1}\frac{z^0_ke^{\ii\frac{\lambda_k T}{2}}}{\ii\lambda_kh_k}\vartheta_k(t-\frac{T}{2}).
\end{multline*}
For this specific final target $z(T)=\varrho_{1}\E_{1,-}+\varrho_{-1}\E_{-1,-}$, we aim to prove that $ z(T)+ \sum_{j=1}^4c_j h_{\mu_j} e^{-\mu_j T/2}\in  H_{\A}$, which is equivalent to 
\begin{equation}
    \poscalr{ z(T)+ \sum_{j=1}^4c_j h_{\mu_j} e^{-\mu_j T/2}}{\F_{\zeta_{\pm}}}_{(0,L)}= 0.
\end{equation}
By integration by parts, we obtain that
\begin{align*}
 \int_0^Lh_{\mu_j}(x)\overline{\F_{\zeta_{\pm}}(L-x)}dx&=\frac{\overline{\F_{\zeta_{\pm}}'(0)}h_{\mu_j}'(L)}{ \zeta_{\pm}+\mu_j},\\
 \int_0^L\E_{\pm1}(x)\overline{\F_{\zeta_{\pm}}(L-x)}dx&=\frac{\overline{\F_{\zeta_{\pm}}'(0)}\E_{\pm1}'(L)}{ \zeta_{\pm}+\ii\lambda_{\pm1}},\\
 \int_0^L\E_{\pm2}(x)\overline{\F_{\zeta_{\pm}}(L-x)}dx&=\frac{\overline{\F_{\zeta_{\pm}}'(0)}\E_{\pm2}'(L)}{ \zeta_{\pm}+\ii\lambda_{\pm2}}.
\end{align*}
Therefore, we know that $\varrho_1(T)$ and $\varrho_{-1}(T)$ satisfy the following linear equations
\begin{equation*}
M
\begin{pmatrix}
\varrho_1 \\
\varrho_{-1}
\end{pmatrix} =
-\begin{pmatrix}
\sum_{j=1}^4 c_j e^{-\frac{1}{2}\mu_jT} \frac{h_{\mu_j}'(L)}{\mu_j+\zeta_+} \\
\sum_{j=1}^4 c_j e^{-\frac{1}{2}\mu_jT} \frac{h_{\mu_j}'(L)}{\mu_j+\overline{\zeta_+}}
\end{pmatrix},
\end{equation*}
where we define $M$ as
\begin{equation*}
\begin{pmatrix}
\frac{\sqrt{98-18\sqrt{21}}}{14} \frac{\mathcal{E}_1'(L)}{\mathrm{i}\lambda_1+\zeta_+}- \frac{\sqrt{98+18\sqrt{21}}}{14}\frac{\mathcal{E}_2'(L)}{\mathrm{i}\lambda_2+\zeta_+} & \frac{\sqrt{98-18\sqrt{21}}}{14}\frac{\overline{\mathcal{E}_1'(L)}}{-\mathrm{i}\lambda_1+\zeta_+} -\frac{\sqrt{98+18\sqrt{21}}}{14}\frac{\overline{\mathcal{E}_2'(L)}}{-\mathrm{i}\lambda_2+\zeta_+} \\
\frac{\sqrt{98-18\sqrt{21}}}{14} \frac{\mathcal{E}_1'(L)}{\mathrm{i}\lambda_1+\overline{\zeta_+}} -\frac{\sqrt{98+18\sqrt{21}}}{14} \frac{\mathcal{E}_2'(L)}{\mathrm{i}\lambda_2+\overline{\zeta_+}} & \frac{\sqrt{98-18\sqrt{21}}}{14} \frac{\overline{\mathcal{E}_1'(L)}}{-\mathrm{i}\lambda_1+\overline{\zeta_+}} -\frac{\sqrt{98+18\sqrt{21}}}{14} \frac{\overline{\mathcal{E}_2'(L)}}{-\mathrm{i}\lambda_2+\overline{\zeta_+}}
\end{pmatrix}
\end{equation*}
So it is easy to check that $\overline{\varrho_1}=\varrho_{-1}$.
\end{proof}
\begin{rem}
Here $\varrho_1$ and $\varrho_{-1}$ are chosen such that $z^T$ is real-valued, which is a particular case for Remark \ref{rem: rho} in $L_0=2\pi\sqrt{7}$. Here the directions $E_1\sim \Re{\widetilde{\G_1}}$ and $E_{-1}\sim 2\ii\Im{\widetilde{\G_1}}$.
\end{rem}
\subsubsection{Revised transition-stabilization method}
\textbf{Iteration schemes}
\begin{prop}\label{prop: estimates on fixed interval-iteration preparation-2pi7}
Let $T\in (0, 2)$, and four distinct positive parameters $\mu_j=\mu$, $j=1,2,3,4$ with $\mu>0$. Suppose that $0<\delta<\pi$ is sufficiently small. Let $I=[2\pi\sqrt{7}-\delta,2\pi\sqrt{7}+\delta]$ be a small compact interval such that $I\cap \mathcal{N}=\{2\pi\sqrt{7}\}$.  For every $L\in I\setminus\{2\pi\sqrt{7}\}$, and every real initial state $y^0\in H_{\A}$, there exists a function $w\in L^2(0, T)$ satisfying 
\begin{itemize}
    \item $w= w_1+ w_2+w_3+w_4+w_5$ in $(0, T)$;
    \item $w_1(t)=w_2(t)=w_3(t)=w_4(t)=w_5(t)=0, \forall t\in (0, T/2)$;
    \item $\|w_1\|_{L^{\infty}(0, T)}\leq \mathcal{K}e^{\frac{2\sqrt{2}K}{\sqrt{T}}}\frac{\mu^N}{T^3}\|y^0\|_{L^2(0,L)}$; 
    \item $\|w_{j+1}\|_{L^{\infty}(0,T)}\leq\mathcal{K}\frac{e^{-\frac{\mu^{\frac{1}{3}}}{4}L}}{T^3}\|y^0\|_{L^2(0,L)},j=1,2,3$.
\end{itemize}
such that the unique solution $y$ of the Cauchy problem 
\begin{equation}\label{eq: controlled KdV-single interval-2pi}
\left\{
\begin{array}{lll}
    \p_ty+\p_x^3y+\p_xy=0 & \text{ in }(0,T)\times(0,L), \\
     y(t,0)=y(t,L)=0&  \text{ in }(0,T),\\
     \p_x y(t,L)= w(t)&\text{ in }(0,T),\\
     y(0,x)=y^0(x)&\text{ in }(0,L),
\end{array}
\right.
\end{equation}
satisfies 
\begin{itemize}
    \item $y= S(t) y^0, \forall t\in (0, T/2)$. In this period we have $\|y(t, \cdot)\|_{L^2(0, L)}\leq \|y^0\|_{L^2(0, L)}$, $\forall t\in (0, T/2)$; 
    \item $y= y_1+ y_2+y_3+y_4+y_5$ in $(T/2, T)$;
    \item $\|y_1(t, \cdot)\|_{L^2(0, L)}\leq \mathcal{K}e^{\frac{2\sqrt{2}K}{\sqrt{T}}}\frac{\mu^{N}}{T^3}\|y^0\|_{L^2(0,L)};$ 
    \item $y_1(T, x)= \varrho_1\E_{1,-}(x)+\varrho_{-1}\E_{-1,-}(x)$, where $|\varrho_{\pm1}|$ uniformly bounded; 
    \item $\|y_{j+1}(t, \cdot)\|_{L^2(0, L)}\leq  \mathcal{K}\frac{\mu^{\frac{1}{2}}e^{-\mu_j(t-\frac{T}{2})}}{T^3}\|y^0\|_{L^2(0,L)},j=1,2,3,4$, $\forall t\in (T/2, T)$.
\end{itemize}
All constants appearing in this proposition are independent of $L$ and $T$.
\end{prop}
\begin{prop}[Iteration schemes]\label{prop: iteration schemes-2pi-7}
Let $T>0$. Suppose that $0<\delta<\pi$ is sufficiently small. Let $I=[2\pi\sqrt{7}-\delta,2\pi\sqrt{7}+\delta]$ be a small compact interval such that $I\cap \mathcal{N}=\{2\pi\sqrt{7}\}$.  For every $L\in I\setminus\{2\pi\sqrt{7}\}$, and $\forall y^0\in H_{\A}$. There exists a function $u(t) \in L^2(0,T)$ such that the solution $y$ to the system 
\begin{equation}\label{eq: controlled-Kdv with control u-2pi}
\left\{
\begin{array}{lll}
    \p_ty+\p_x^3y+\p_xy=0 & \text{ in }(0,T)\times(0,L), \\
     y(t,0)=y(t,L)=0&  \text{ in }(0,T),\\
     \p_x y(t,L)= u(t)&\text{ in }(0,T),\\
     y(0,x)=y^0(x)&\text{ in }(0,L),
\end{array}
\right.    
\end{equation}    
satisfies 
    \begin{equation}
        \lim_{t\rightarrow T^-}\|y(t, \cdot)\|_{L^2(0, L)}= 0.
    \end{equation}
    and there exists a constant $\mathcal{K}$ such that
\begin{equation}\label{eq: est-control-L-infty-2pi7}
 \|u\|_{L^{\infty}(0,T)}\leq \Tilde{\mathcal{K}}e^{\frac{\mathcal{K}}{\sqrt{T}}}\|y^{0}\|_{L^2(0,L)}.   
\end{equation}
\end{prop}

\subsection{Around Type III unreachable pair $(k, l)$}\label{sec: Around Type III unreachable pair}
\subsubsection{Quasi-invariant subspaces and transition projectors}
In this section, we look at a different case for $L_0=2\pi\sqrt{\frac{7}{3}}$.
\begin{prop}
For $L_0=2\pi\sqrt{\frac{7}{3}}$, there are two eigenvalues $\lambda_{c,1}=\ii\frac{20}{21\sqrt{21}}$ and $\lambda_{c,-1}=-\ii\frac{20}{21\sqrt{21}}$. We denote the associated eigenfunctions by $\G_{c,\pm1}$.
\end{prop}
\begin{proof}
We only consider the eigenfunction $\G_{c,1}$ associated with the eigenvalue $\lambda_{c,1}=\ii\frac{20}{21\sqrt{21}}$. The other case with $\lambda_{c,-1}=-\ii\frac{20}{21\sqrt{21}}$ can be dealt with in the same way. We start with 
\begin{equation*}
\left\{
\begin{array}{l}
     \G_{c,+}'''(x)+\G_{c,+}'(x)+\ii\frac{20}{21\sqrt{21}}\G_{c,+}(x)=0,x\in(0,2\pi\sqrt{\frac{7}{3}}),  \\
     \G_{c,+}(0)=\G_{c,+}(2\pi\sqrt{\frac{7}{3}})=0,\\
     \G_{c,+}'(2\pi\sqrt{\frac{7}{3}})=0.
\end{array}
\right.
\end{equation*}
It is easy to find that the solution is in the form:
\begin{equation*}
   \G_{c,+}(x)=r_1 e^{-\frac{4 \ii x}{\sqrt{21}}} +r_2 e^{-\frac{\ii x}{\sqrt{21}}} +r_3 e^{\frac{5 \ii x}{\sqrt{21}}}.
\end{equation*}
Using the boundary conditions, 
\begin{equation*}
\left\{\begin{array}{l}
      r_1+r_2+r_3=0,\\
     -\frac{4\ii}{\sqrt{21}}r_1 e^{-\frac{8\pi\ii}{3}} -\frac{\ii x}{\sqrt{21}}r_2 e^{-\frac{2\pi\ii}{3}} +\frac{5\ii x}{\sqrt{21}}r_3 e^{\frac{10\pi \ii}{3}}=0.
\end{array}
\right.
\end{equation*}
Hence, we know that $\G_{c,+1}(x)=\sqrt{\frac{3}{28\pi\sqrt{21}}}\left(2e^{-\frac{4 \ii x}{\sqrt{21}}} -3 e^{-\frac{\ii x}{\sqrt{21}}} + e^{\frac{5 \ii x}{\sqrt{21}}}\right)$ is normalized eigenfunction. This gives us the unique solution associated with the eigenvalue $\lambda_{c,1}$.
\end{proof}
\begin{lem}\label{lem: pertubation eigenvalues-7/3}
Suppose that $0<\delta<\pi$ is sufficiently small. Let $I=[2\pi\sqrt{\frac{7}{3}}-\delta,2\pi\sqrt{\frac{7}{3}}+\delta]$ be a small compact interval such that $I\cap \mathcal{N}=\{2\pi\sqrt{\frac{7}{3}}\}$. Suppose that $\zeta_{\pm}\in\C$ are the two eigenvalues close to $\lambda_{c,\pm1}$, then $\zeta_{\pm}$ and its associated eigenfunction $\E_{\zeta_{\pm}}$ satisfy that
\begin{align*}
    \zeta_{\pm}&=\ii\lambda_{c,\pm1}+\bigO( |L-2\pi|^2),\\
    |\F_{\zeta_{\pm}}'(0)|&\sim|L-2\pi|.   
\end{align*}
\end{lem}
\begin{proof}
This lemma is just a particular case of our general result Proposition \ref{prop: asymptotic expansion for A0}. The proof remains the same. We just point out that, in this particular case, we have the following asymptotic expansion for $\zeta_-$ and $\E_{\zeta_-}$ for instance:
\begin{align}
\zeta_+&=\ii\frac{20}{21\sqrt{21}}-\frac{81 \sqrt{21}}{4802\pi}(L-2\pi\sqrt{\frac{7}{3}})^2+\bigO((L-2\pi\sqrt{\frac{7}{3}})^3),\\
\F_{\zeta_+}(x)&=\G_{c,1}(x) +\bigO(|L-2\pi\sqrt{\frac{7}{3}}|),
\end{align}
Moreover,
\begin{equation}
|\E'_{\zeta_-}(0)|\sim |L-2\pi\sqrt{\frac{7}{3}}|.
\end{equation}
Similar results hold for $\zeta_+$ and $\E_{\zeta_+}$.
\end{proof}
By Proposition \ref{prop: Index set M-E}, there exist two eigenvalues $\lambda_{\pm1}$ of $\B$ in the interval $(0,L)$ with $L\in [2\pi\sqrt{\frac{7}{3}}-\delta,2\pi\sqrt{\frac{7}{3}}+\delta]\setminus \{2\pi\sqrt{\frac{7}{3}}\}$, as defined in Lemma \ref{lem: pertubation eigenvalues-7/3}. Moreover, we know that 
\begin{equation*}
\lambda_{1}=\lambda_{c,1}-\frac{9}{392 \sqrt{7} \pi}
(L-2\pi\sqrt{\frac{7}{3}})^2+\bigO((L-2\pi)^3),\;\lambda_{-1}=\lambda_{c,-1}+\frac{9}{392 \sqrt{7} \pi}(L-2\pi\sqrt{\frac{7}{3}})^2+\bigO((L-2\pi)^3).
\end{equation*} 
\begin{coro}
Suppose that $0<\delta<\pi$ is sufficiently small. Let $I=[2\pi\sqrt{\frac{7}{3}}-\delta,2\pi\sqrt{\frac{7}{3}}+\delta]$ be a small compact interval such that $I\cap \mathcal{N}=\{2\pi\sqrt{\frac{7}{3}}\}$. Then, the following statements hold: 
\begin{enumerate}
    \item $\E_{\pm1}(x)$ have the following asymptotic expansions near $2\pi\sqrt{\frac{7}{3}}$:
    \begin{align*}
    \E_{1}(x)=\G_{c,1}+\bigO(|L-2\pi\sqrt{\frac{7}{3}}|),\;
    \E_{-1}(x)=\G_{c,-1}+\bigO(|L-2\pi\sqrt{\frac{7}{3}}|).
    \end{align*}
    \item $\E'_{\pm1}(L)= \sqrt{\frac{3}{7\pi\sqrt{21}}} \left( \frac{9}{14} +\pm\frac{25}{576\pi}\ii \right)(L-2\pi\sqrt{\frac{7}{3}})+\bigO((L-2\pi\sqrt{\frac{7}{3}})^2)$.
\end{enumerate}
\end{coro}
\begin{proof}
Here we simply apply Proposition \ref{prop: Low-frequency behaviors: Singular limits} in our particular case $L_0=2\pi\sqrt{\frac{7}{3}}$.
\end{proof}
For $s\geq0$, we define the following Hilbert subspaces 
\begin{equation}\label{eq: defi for space K0 and K}
\begin{aligned}
H_{\A}^s&:=\{u\in D(\A^s):\poscalr{u}{\F_{\zeta_+}}_{(0,L)}=\poscalr{u}{\F_{\zeta_-}}_{(0,L)}=0\},\\    H_{\B}^s&:=\{u\in D(\B^s):\poscals{u}{\E_{1}}_{L^2(0,L)}=\poscals{u}{\E_{-1}}_{L^2(0,L)}=0\}.
\end{aligned}
\end{equation}
We aim to prove the uniform quantitative observability
\begin{thm}
Suppose that $0<\delta<\pi$ is sufficiently small. Let $I=[2\pi\sqrt{\frac{7}{3}}-\delta,2\pi\sqrt{\frac{7}{3}}+\delta]$ be a small compact interval such that $I\cap \mathcal{N}=\{2\pi\sqrt{\frac{7}{3}}\}$. For every $L\in I\setminus\{2\pi\sqrt{\frac{7}{3}}\}$, and $\forall y^0\in H_{\A}$, there exists a constant $C$ such that the following quantitative observability inequality 
\begin{equation}\label{eq: quantitative-ob-7/3}
   \|S(T)y^0\|^2_{L^2(0,L)} \leq  C^2\int_0^T|\p_x y(t,0)|^2dt
\end{equation}
holds for any solution $y$ to the KdV system
\begin{equation}
\left\{
\begin{array}{lll}
    \p_ty+\p_x^3y+\p_xy=0 & \text{ in }(0,T)\times(0,L), \\
     y(t,0)=y(t,L)=0&  \text{ in }(0,T),\\
     \p_xy(t,L)=0&\text{ in }(0,T),\\
     y(0,x)=y^0(x)&\text{ in }(0,L).
\end{array}
\right.
\end{equation} 
Here $S(t)$ is the semi-group generated by $\A$, and we point out that $C$ is independent of $L$.
\end{thm}
As we already performed in the case $L_0=2\pi$, we first recall the bi-orthogonal family to $e^{-\ii t\lambda_j}$ defined in Proposition \ref{prop: biorthogonal family}. In the current case $L_0=2\pi\sqrt{\frac{7}{3}}$, by Remark \ref{rem: uniform sequence-eigenvalues}, we know that $\{\lambda_j(L)\}_{j\in\Z,j\neq0}$ is a uniform regular $3-$sequence. Then we are able to utilize the bi-orthogonal family defined in Proposition \ref{prop: biorthogonal family} to construct our control function $v$. To apply the same method, we need the following lemma to compensate for boundary conditions between different KdV operators.
\begin{lem}\label{lem: defi of coeff-7/3}
Let $\mu_j$, $j=1,2,3,4$ be four distinct real positive constants, for any function $z\in  H_{\A}^2\subset H_{\A}$, there exist  constants $c_1$, $c_2$, $c_3$ and $c_4$ such that 
\begin{equation}
    z- c_1 h_{\mu_1}- c_2 h_{\mu_2}-c_3 h_{\mu_3}-c_4 h_{\mu_4}\in H_{\B}^2.
\end{equation}
Here $h_{\mu_j}(1\leq j\leq4)$ are transition functions defined in Subsection \ref{sec: preliminary}.
\end{lem}
\begin{proof}
For any $\mu_1$, $\mu_2$, $\mu_3$ and $\mu_4$,  we are able to construct the transition functions $h_{\mu_1}$, $h_{\mu_2}$, $h_{\mu_3}$, and $h_{\mu_4}$ as 
\begin{equation}
\left\{
\begin{aligned}
     &h_{\mu_j}'''+h_{\mu_j}'=\mu_j h_{\mu_j}(x),  \\
     &h_{\mu_j}(0)=h_{\mu_j}(L)=0,\\
     &h_{\mu_j}'(L)-h_{\mu_j}'(0)=1,
\end{aligned}
\right.
\end{equation}
Let $f:= z- c_1 h_{\mu_1}- c_2 h_{\mu_2}-c_3 h_{\mu_3}-c_4 h_{\mu_4}$. Then it is easy to check that $f(0)=f(L)=0$. For the boundary derivative, we ask that
\begin{align*}
    c_1+ c_2+c_3+c_4&=-\p_xz(0),\\
    c_1\mu_1+ c_2\mu_2+c_3\mu_3+c_4\mu_4&=-\p_x(\mathcal{P}z)(0)
\end{align*}
    Next, for the projection on the direction $\E_{1}$ and $\E_{-1}$, we have 
    \begin{align*}
        \poscals{z- c_1 h_{\mu_1}- c_2h_{\mu_2}-c_3 h_{\mu_3}}{\E_{1}}_{L^2(0,L)}&= 0,\\
        \poscals{z- c_1 h_{\mu_1}- c_2h_{\mu_2}-c_3 h_{\mu_3}}{\E_{-1}}_{L^2(0,L)}&= 0.
    \end{align*}
    which is equivalent to 
    \begin{equation*}
        \sum_{j=1}^4c_j\poscals{h_{\mu_j}}{\E_{1}}_{L^2(0,L)}=\poscals{z}{\E_{1}}_{L^2(0,L)},\;
        \sum_{j=1}^4c_j\poscals{h_{\mu_j}}{\E_{-1}}_{L^2(0,L)}=\poscals{z}{\E_{-1}}_{L^2(0,L)}.
    \end{equation*}
Thus, we obtain
\begin{equation*}
\poscals{h_{\mu_j}}{\E_{1}}_{L^2(0,L)}=\frac{\overline{\E'_{1}(L)}}{\mu_j+\ii\lambda_{1}},\; \poscals{h_{\mu_j}}{\E_{-1}}_{L^2(0,L)}=\frac{\overline{\E'_{-1}(L)}}{\mu_j+\ii\lambda_{-1}}.
\end{equation*}
Therefore, we obtain a linear system of the real coefficients
\begin{equation}
\left\{
\begin{array}{rl}
      c_1+ c_2+c_3+c_4&=-\p_xz(0),\\
    c_1\mu_1+ c_2\mu_2+c_3\mu_3+c_4\mu_4&=-\p_x(\mathcal{P}z)(0),\\
     \sum_{j=1}^4c_j\frac{\overline{\E'_{1}(L)}}{\mu_j+\ii\lambda_{1}}&=\poscals{z}{\E_{1}}_{L^2(0,L)},\\
        \sum_{j=1}^4c_j\frac{\overline{\E'_{-1}(L)}}{\mu_j+\ii\lambda_{-1}}&=\poscals{z}{\E_{-1}}_{L^2(0,L)}.
\end{array}
\right.
\end{equation}
Let 
\begin{equation}
    M(c_1,c_2,c_3,c_4)=\left(\begin{array}{cccc}
     1&1&1 &1 \\
     \mu_1&\mu_2&\mu_3&\mu_4\\
     \frac{\overline{\E'_{1}(L)}}{\mu_1+\ii\lambda_{1}}&\frac{\overline{\E'_{1}(L)}}{\mu_2+\ii\lambda_{1}}&\frac{\overline{\E'_{1}(L)}}{\mu_3+\ii\lambda_{1}}&\frac{\overline{\E'_{1}(L)}}{\mu_4+\ii\lambda_{1}}\\
     \frac{\overline{\E'_{-1}(L)}}{\mu_1+\ii\lambda_{-1}}&\frac{\overline{\E'_{-1}(L)}}{\mu_2+\ii\lambda_{-1}}&\frac{\overline{\E'_{-1}(L)}}{\mu_3+\ii\lambda_{-1}}&\frac{\overline{\E'_{-1}(L)}}{\mu_4+\ii\lambda_{-1}}
\end{array}
\right).
\end{equation}
Rewriting the system above into the matrix form, for the coefficients $(c_1,c_2,c_3)$, we obtain
\begin{equation}
M
\left(\begin{array}{c}
     c_1  \\
     c_2\\
     c_3\\
     c_4
\end{array}\right)=
\left(\begin{array}{c}
     -\p_xz(0)  \\
     -\p_x(\mathcal{P}z)(0)\\
     -\poscals{z}{\E_{j_0}}_{L^2(0,L)}\\
     -\poscals{z}{\E_{-j_0}}_{L^2(0,L)}
\end{array}\right).
\end{equation}
Using the fact that $\lambda_{1}=-\lambda_{-1}$, it is easy to compute that 
\begin{equation*}
|\det{M^{-1}}|=\frac{\prod_{j=1}^4(\lambda_{1}^2+\mu_j^2)}{2 |\E'_{1}(L)|^2 \lambda_{1} \prod_{1\leq m<n\leq4}(\mu_m-\mu_n)}.
\end{equation*} 
Choosing three distinct positive real numbers $(\mu_1,\mu_2,\mu_3,\mu_4)$, the matrix $M$ is invertible and we find the coefficients $(c_1,c_2,c_3,c_4)$ that compensate the boundary conditions.
\end{proof}

\begin{lem}\label{lem: construction of v-7/3pi}
For any initial state $z^0\in H_{\B}^2$ and any final state $z^T\in H_{\B}^2$ (as defined in \eqref{eq: defi for space K0 and K}), there exists a function $v\in C^1(0,T)$ such that the solution $z$ to the controlled KdV system 
\begin{equation}\label{eq: KdV-z-v-7/3pi}
\left\{
\begin{array}{lll}
    \p_tz+\p_x^3z+\p_xz=0 & \text{ in }(0,T)\times(0,L), \\
     z(t,0)=z(t,L)=0&  \text{ in }(0.T),\\
     \p_xz(t,L-\p_xz(t,0))=v(t)&\text{ in }(0,T),\\
     z(0,x)=z^0\in H_{\B}^2&\text{ in }(0,L),
\end{array}
\right.
\end{equation}
satisfies that $z(T)=z^T\in H_{\B}^2$. In particular, we can choose that $z(T,x)=\varrho_+(T)\E_s(x)+\varrho_-(T)\E_{-s}(x)$, for $s>1$ and $s\in \N^*$. Furthermore, $\varrho$ such that 
\begin{equation*}
     z(T)+ c_1 h_{\mu_1} e^{-\mu_1 \frac{T}{2}}+ c_2 h_{\mu_2} e^{-\mu_2 \frac{T}{2}}+c_3 h_{\mu_3} e^{-\mu_3 \frac{T}{2}}+c_4 h_{\mu_4} e^{-\mu_4 \frac{T}{2}}\in\K_0.
\end{equation*}
\end{lem}
\begin{proof}
Suppose that $c_1$, $c_2$, $c_3$ and $c_4$ are fixed. As we presented in Section \ref{sec: construction of the control}, we aim to construct the control function $v\in C^1(0,T)$ for 
\begin{equation*}
\left\{
\begin{array}{lll}
    \p_tz+\p_x^3z+\p_xz=0 & \text{ in }(0,T)\times(0,L), \\
     z(t,0)=z(t,L)=0&  \text{ in }(0,T),\\
     \p_xz(t,L)-\p_xz(t,0)=v(t)&\text{ in }(0,T),\\
     z(0,x)=z^0(x)&\text{ in }(0,L),
\end{array}
\right.
\end{equation*}
with the constraints $v(0)=v(T)=0$. Thanks to the family $\{\phi_j\}_{j\in\Z\backslash\{0,\pm 1\}}$, we construct $v$ as follows
\begin{equation}\label{eq: defi of control v-vartheta-2pi7-3}
    v(t)=\sum_{k\neq0,\pm 1}v_k\phi_k(t-\frac{T}{2}).
\end{equation}
Here $h_k$ remains the same as in Section \ref{sec: construction of the control} and 
\begin{equation*}
z^0(x)=\sum_{k\in\Z\setminus\{0,\pm 1\}}z^0_k\E_k(x),\;z(T,x)=\sum_{k\in\Z\setminus\{0,\pm 1\}}z^T_k\E_k(x).
\end{equation*}
Then the solution $z$ to the system \eqref{eq: KdV-z-v-7/3pi} has the following expansion 
\begin{equation*}
z(t,x)=\sum_{k\in\Z\setminus\{0,\pm 1\}}z^0_ke^{\ii\lambda_k t}\E_k(x)-\sum_{k\neq0} \sum_{j\neq\pm 1}\ii\lambda_kh_k\int_0^te^{\ii(t-s)\lambda_k}v_j\phi_j(s-\frac{T}{2})ds\E_k(x).
\end{equation*}
In particular, at $t=T$, 
\begin{align*}
z(T,x)&=\sum_{k\in\Z\setminus\{0,\pm 1\}}z^0_ke^{\ii\lambda_k T}\E_k(x)-\sum_{k\neq0} \sum_{j\neq\pm 1}\ii\lambda_kh_k\int_0^Te^{\ii(T-s)\lambda_k}v_j\phi_j(s-\frac{T}{2})ds\E_k(x)\\
&=\sum_{k\in\Z\setminus\{0,\pm 1\}}z^T_k\E_k(x).
\end{align*}
As a consequence, we obtain the following equations
\begin{align*}
-\sum_{j\neq\pm 1}\ii\lambda_{1}h_{1}\int_0^Te^{\ii\lambda_{1}(T-s)}v_j\phi_j(s-\frac{T}{2})ds&=0,\\
-\sum_{j\neq\pm 1}\ii\lambda_{-1}h_{-1}\int_0^Te^{\ii\lambda_{-1}(T-s)}v_j\phi_j(s-\frac{T}{2})ds&=0,\\
z^0_ke^{\ii\lambda_k T}-\sum_{j\neq\pm 1}\ii\lambda_kh_k\int_0^Te^{\ii(T-s)\lambda_k}v_j\phi_j(s-\frac{T}{2})ds&=z^T_k,k\neq 0,\pm 1.
\end{align*}
The first two equations are direct consequences of the fact that $\{\phi_j\}_{j\neq0,\pm 1}$ is a bi-orthogonal family to $e^{-\ii t\lambda_{\pm 1}}$. Furthermore, for the last equation, we obtain 
\begin{equation*}
v_k=\frac{z^0_ke^{\ii\frac{\lambda_k T}{2}}-z^T_ke^{-\ii\frac{\lambda_k T}{2}}}{\ii\lambda_kh_k},k\neq 0,\pm 1.
\end{equation*}
In particular, we can choose that $z(T,x)=\varrho_+\E_s(x)+\varrho_-\E_{-s}(x)\in H_{\B}^2$.  To achieve this final target, we construct a special control function $v$ as follows:
\begin{equation*}
v(t):=\frac{z^0_se^{\ii\frac{\lambda_s T}{2}}-\varrho_+(T)e^{-\ii\frac{\lambda_s T}{2}}}{\ii\lambda_sh_s}\phi_s(t-\frac{T}{2})+\frac{z^0_{-s}e^{\ii\frac{\lambda_{-s} T}{2}}-\varrho_-(T)e^{-\ii\frac{\lambda_{-s} T}{2}}}{\ii\lambda_{-s}h_{-s}}\phi_{-s}(t-\frac{T}{2})+\sum_{|k|\neq0,1,s}\frac{z^0_ke^{\ii\frac{\lambda_k T}{2}}}{\ii\lambda_kh_k}\phi_k(t-\frac{T}{2})
\end{equation*}
For this specific final target $z(T)=\varrho_+\E_s+\varrho_-\E_{-s}$, we aim to prove that $ z(T)+ c_1 h_{\mu_1} e^{-\mu_1 T/2}+ c_2 h_{\mu_2} e^{-\mu_2 T/2}+c_3 h_{\mu_3} e^{-\mu_3 T/2}+c_4 h_{\mu_4} e^{-\mu_4 T/2}\in  H_{\A}$, which is equivalent to 
\begin{equation*}
\begin{aligned}
\poscalr{z(T)+ c_1 h_{\mu_1} e^{-\mu_1 T/2}+ c_2 h_{\mu_2} e^{-\mu_2 T/2}+c_3 h_{\mu_3} e^{-\mu_3 T/2}+c_4 h_{\mu_4} e^{-\mu_4 T/2}}{\F_{\zeta_+}}_{(0,L)}&= 0,\\
\poscalr{z(T)+ c_1 h_{\mu_1} e^{-\mu_1 T/2}+ c_2 h_{\mu_2} e^{-\mu_2 T/2}+c_3 h_{\mu_3} e^{-\mu_3 T/2}+c_4 h_{\mu_4} e^{-\mu_4 T/2}}{\F_{\zeta_-}}_{(0,L)}&= 0.
\end{aligned}    
\end{equation*}
By integration by parts, we obtain that
\begin{align*}
\int_0^Lh_{\mu_j}(x)\overline{\F_{\zeta_{\pm}}(L-x)}dx&=\frac{\overline{\F_{\zeta_{\pm}}'(0)}h_{\mu_j}'(L)}{ \zeta_{\pm}+\mu_j},j=1,2,3,4,\\
 \int_0^L\E_{\pm s}(x)\overline{\F_{\zeta_{\pm}}(L-x)}dx&=\frac{\overline{\F_{\zeta_{\pm}}'(0)}\E_{\pm s}'(L)}{ \zeta_{\pm}+\ii\lambda_{\pm s}}.
\end{align*}
Therefore, we know that
\begin{equation*}
\begin{aligned}
\varrho_+(T)&=-\frac{(\zeta_+^2+\lambda_s^2)(\zeta_-^2+\lambda_s^2)}{2\ii\lambda_s\E'_s(L)}\sum_{j=1}^4c_j e^{-\mu_j\frac{T}{2}}\frac{(\mu_j+\ii\lambda_s)h_{\mu_j}'(L)}{(\zeta_+-\ii\lambda_s)(\zeta_++\mu_j)(\zeta_-+\mu_j)(\zeta_--\ii\lambda_s)},\\
\varrho_-(T)&=\frac{(\zeta_+^2+\lambda_s^2)(\zeta_-^2+\lambda_s^2)}{2\ii\lambda_s\E'_{-s}(L)}\sum_{j=1}^4c_j e^{-\mu_j\frac{T}{2}}\frac{(\mu_j-\ii\lambda_s)h_{\mu_j}'(L)}{(\zeta_++\ii\lambda_s)(\zeta_++\mu_j)(\zeta_-+\mu_j)(\zeta_-+\ii\lambda_s)}.
\end{aligned}
\end{equation*}
\end{proof}
\begin{rem}
Here $\varrho_{\pm}$ are chosen such that $z^T$ is real-valued. Here we consider a particular situation where $L_0=2\pi\sqrt{\frac{7}{3}}$ and the directions $E_{1}=\Re\E_{1}$ and $E_{-1}=2\ii\Im\E_{1}$.
\end{rem}
\subsubsection{Revised transition-stabilization method}
\textbf{A priori estimates for the intermediate system}
\begin{lem}\label{lem: est for varrho-2pi-7/3}
Let $T\in (0, 2)$, and four distinct positive parameters be $\mu_j$, with $\mu_j>|\zeta_+|+|\zeta_-|+1$, $j=1,2,3,4$. We fix four constants $c_j$, $j=1,2,3,4$. Suppose that $0<\delta<\pi$ is sufficiently small. Let $I=[2\pi\sqrt{\frac{7}{3}}-\delta,2\pi\sqrt{\frac{7}{3}}+\delta]$ be a small compact interval such that $I\cap \mathcal{N}=\{2\pi\sqrt{\frac{7}{3}}\}$.  For every $L\in I\setminus\{2\pi\sqrt{\frac{7}{3}}\}$, $\varrho_{\pm}(T)$, as we defined in Lemma \ref{lem: construction of v-7/3pi} is uniformly bounded by a effectively computable constant $C_{\varrho}=C_{\varrho}(\mu_1,\mu_2,\mu_3,c_1,c_2,c_3)$, i.e. $|\varrho_{\pm}|\leq C_{\varrho}$
\end{lem}
\textbf{Iteration Schemes with uniform constants}
\begin{prop}\label{prop: estimates on fixed interval-iteration preparation-2pi-7/3}
Let $T\in (0, 2)$, and four distinct positive parameters $\mu_j=j\mu$, $j=1,2,3,4$ with $\mu>0$. Suppose that $0<\delta<\pi$ is sufficiently small. Let $I=[2\pi\sqrt{\frac{7}{3}}-\delta,2\pi\sqrt{\frac{7}{3}}+\delta]$ be a small compact interval such that $I\cap \mathcal{N}=\{2\pi\sqrt{\frac{7}{3}}\}$.  For every $L\in I\setminus\{2\pi\sqrt{\frac{7}{3}}\}$, and every real initial state $y^0\in H_{\A}$, there exists a function $w\in L^2(0, T)$ satisfying 
\begin{itemize}
    \item $w= w_1+ w_2+w_3+w_4+w_5$ in $(0, T)$;
    \item $w_1(t)=w_2(t)=w_3(t)=w_4(t)=w_5(t)=0, \forall t\in (0, T/2)$;
    \item $\|w_1\|_{L^{\infty}(0, T)}\leq \mathcal{K}e^{\frac{2\sqrt{2}K}{\sqrt{T}}}\frac{\mu^N}{T^3}\|y^0\|_{L^2(0,L)}$; 
    \item $\|w_{j+1}\|_{L^{\infty}(0,T)}\leq\mathcal{K}\frac{e^{-\frac{\mu^{\frac{1}{3}}}{4}L}}{T^3}\|y^0\|_{L^2(0,L)},j=1,2,3$.
\end{itemize}
such that the unique solution $y$ of the Cauchy problem 
\begin{equation}\label{eq: controlled KdV-single interval-2pi-7-3}
\left\{
\begin{array}{lll}
    \p_ty+\p_x^3y+\p_xy=0 & \text{ in }(0,T)\times(0,L), \\
     y(t,0)=y(t,L)=0&  \text{ in }(0,T),\\
     \p_x y(t,L)= w(t)&\text{ in }(0,T),\\
     y(0,x)=y^0(x)&\text{ in }(0,L),
\end{array}
\right.
\end{equation}
satisfies 
\begin{itemize}
    \item $y= S(t) y^0, \forall t\in (0, T/2)$. In this period we have $\|y(t, \cdot)\|_{L^2(0, L)}\leq \|y^0\|_{L^2(0, L)}$, $\forall t\in (0, T/2)$; 
    \item $y= y_1+ y_2+y_3+y_4+y_5$ in $(T/2, T)$;
    \item $\|y_1(t, \cdot)\|_{L^2(0, L)}\leq \mathcal{K}e^{\frac{2\sqrt{2}K}{\sqrt{T}}}\frac{\mu^{N}}{T^3}\|y^0\|_{L^2(0,L)};$ 
    \item $y_1(T, x)= \varrho_+\E_s(s)+\varrho_-\E_{-s}(x) $, where $\varrho_+=\overline{\varrho_-}$; 
    \item $\|y_{j+1}(t, \cdot)\|_{L^2(0, L)}\leq  \mathcal{K}\frac{\mu^{\frac{1}{2}}e^{-\mu_j(t-\frac{T}{2})}}{T^3}\|y^0\|_{L^2(0,L)},j=1,2,3,4$, $\forall t\in (T/2, T)$.
\end{itemize}
All constants appearing in this proposition are independent of $L$ and $T$.
\end{prop}

\begin{prop}[Iteration schemes]\label{prop: iteration schemes-2pi-7/3}
Let $T>0$. Suppose that $0<\delta<\pi$ is sufficiently small. Let $I=[2\pi\sqrt{\frac{7}{3}}-\delta,2\pi\sqrt{\frac{7}{3}}+\delta]$ be a small compact interval such that $I\cap \mathcal{N}=\{2\pi\sqrt{\frac{7}{3}}\}$.  For every $L\in I\setminus\{2\pi\sqrt{\frac{7}{3}}\}$, and $\forall y^0\in H_{\A}$. There exists a function $u(t) \in L^2(0,T)$ such that the solution $y$ to the system 
\begin{equation}\label{eq: controlled-Kdv with control u-2pi7-3}
\left\{
\begin{array}{lll}
    \p_ty+\p_x^3y+\p_xy=0 & \text{ in }(0,T)\times(0,L), \\
     y(t,0)=y(t,L)=0&  \text{ in }(0,T),\\
     \p_x y(t,L)= u(t)&\text{ in }(0,T),\\
     y(0,x)=y^0(x)&\text{ in }(0,L),
\end{array}
\right.    
\end{equation}    
satisfies 
    \begin{equation}
        \lim_{t\rightarrow T^-}\|y(t, \cdot)\|_{L^2(0, L)}= 0.
    \end{equation}
    and there exists a constant $\mathcal{K}$ such that
\begin{equation}\label{eq: est-control-L-infty-2pi-7/3}
 \|u\|_{L^{\infty}(0,T)}\leq \Tilde{\mathcal{K}}e^{\frac{\mathcal{K}}{\sqrt{T}}}\|y^{0}\|_{L^2(0,L)}.   
\end{equation}
\end{prop}

\subsection{Sketch proof of Theorem \ref{thm: main theorem linear version}: general case}\label{sec: general case-proof}
Recall the classification in Section \ref{sec: Sharp spectral analysis and classification of critical lengths}. Our proof of Theorem \ref{thm: main theorem linear version} is divided into the following three distinguished cases.
\begin{enumerate}
    \item Case: $L_0\in\mathcal{N}^1$. \\
In this situation, we are exactly in the same case as Subsection \ref{sec: Around Type I unreachable pair}. There is a unique pair $(k,l)$ satisfying that $k=l$. We just repeat the procedure presented in Subsection \ref{sec: Around Type I unreachable pair}, and for the directions $\{E_m\}$ generating the second transition map, we have
    \begin{align*}
    E_0\sim2\ii\sin{k x}.
    \end{align*}
    \item Case: $L_0\in\mathcal{N}^3$.\\
    In this situation, we have all pairs $(k,l)$ satisfying that $k\equiv l\mod 3$ and we may have a special pair $k=l$. In Subsection \ref{sec: Around Type I unreachable pair}, we present how we adjust the construction of the bi-orthogonal family for $k=l$. Moreover, in Subsection \ref{sec: Around Type II unreachable pair}, we show how we compensate for the unreachable pair of Type II. Since we have a finite number of pairs to deal with, we can construct a suitable bi-orthogonal family $\{\vartheta_j\}_{j\in\mathcal{J}}$ with all the uniform estimates in Subsection \ref{sec: Around Type I unreachable pair} and Subsection \ref{sec: Around Type II unreachable pair} holding true. There are two possibilities:
    \begin{itemize}
        \item If the dimension $N_0$ is odd, this means that we include a special pair $k=l$ in this case. For the directions $\{E_m\}$ generating the second transition map, we have
    \begin{align*}
    E_0\sim2\ii\sin{k x}, &\text{ as the }2\pi \text{ case;}\\
    E_m\sim2\Re{\widetilde{\G}_m},E_{-m}\sim 2\ii\Im{\widetilde{\G}_m}, &\text{ similar to the }2\pi\sqrt{7} \text{ case.}
    \end{align*}
    Then, $z^T=\sum_{|j|\leq\frac{N_0-1}{2}}\varrho_mE_m$ is real-valued.  Then we repeat the revised transition-stabilization method based on this well-prepared bi-orthogonal family $\{\vartheta_j\}_{j\in\mathcal{J}}$ and we are able to conclude in this situation.
    \item If $N_0$ is even, then in this case, we only have Type II unreachable pairs. Hence, we just repeat the procedure in Subsection \ref{sec: Around Type II unreachable pair}. For the directions $\{E_m\}$ generating the second transition map, we only have
    \begin{align*}
    E_m\sim2\Re{\widetilde{\G}_m},E_{-m}\sim 2\ii\Im{\widetilde{\G}_m}, &\text{ similar to the }2\pi\sqrt{7} \text{ case.}
    \end{align*}
    Then, $z^T=\sum_{0<|j|\leq\frac{N_0}{2}}\varrho_mE_m$ is real-valued.
    \end{itemize}
    \item Case: $L_0\in\mathcal{N}^2$.\\
    In this case, we only have Type III unreachable pairs. We just repeat the procedure in Subsection \ref{sec: Around Type III unreachable pair}.
    
    In this situation, we do not have the Type 2 eigenfunctions. For the directions $\{E_m\}$ generating the second transition map, they are just real parts and imaginary parts of hyperbolic eigenfunctions.
    \begin{align*}
    E_m\sim2\Re{\E_m},E_{-m}\sim 2\ii\Im{\E_m}, &\text{ similar to the }2\pi\sqrt{\frac{7}{3}} \text{ case}, |j|\geq \frac{N_0}{2}.
    \end{align*}
    Then, $z^T=\sum_{\frac{N_0}{2}+1\leq|j|\leq N_0}\varrho_mE_m$ is real-valued.
\end{enumerate}

\section{Nonlinear case}
\subsection{Existence and smoothness of the invariant manifold}
This section is devoted to showing the existence and smoothness of the invariant manifold for the nonlinear KdV system \eqref{eq: KdV system-stability-intro} by following the ideas given in
\cite{Minh-Wu2004,Chu-Coron-Shang}. The first step is to show
that the nonlinear perturbation has a small global Lipschitz constant. To that end, we modify the nonlinear part of the original system \eqref{eq: KdV system-stability-intro} by using some smooth cut-off mapping, and consider the following equation
\begin{equation}
\left\{
\begin{array}
[c]{l}%
\p_t y+\p_x^3y+\p_x y+\Phi_{\varepsilon}(\left\Vert y\right\Vert _{L^{2}%
(0,L)})y\p_xy=0,\\
y(t,0)=y(t,L)=\p_x y(t,L)=0,\\
y(0,x)=y_{0}(x)\in L^{2}(0,L).
\end{array}
\right.  \label{new system}%
\end{equation}
Here $\varepsilon>0$ is small enough,  and
$\Phi_{\varepsilon}:\left[
0,+\infty\right)  \rightarrow\left[  0,1\right]  $ is defined by%
\[
\Phi_{\varepsilon}\left(  x\right)  =\Phi\left(
\frac{x}{\varepsilon}\right) ,\,\forall x\in\lbrack0,+\infty),
\]
where $\Phi\in C^{\infty}\left(  \left[  0,+\infty\right)  ;\left[
0,1\right]  \right)  $ satisfies
\[
\Phi(x)=\left\{
\begin{array}
[c]{l}%
1,\text{ when }x\in [  0,\displaystyle\frac{1}{2}]  ,\\
\\
0,\text{ when }x\in\left[  1,+\infty\right)  ,
\end{array}
\right.
\]
and
\[
\Phi^{\prime}\leq0.
\]
It can be readily checked that
\begin{align}
&\Phi_{\varepsilon}(x)=1,\text{ when }x\in [0,\displaystyle\frac{1}{2}] ,
\nonumber\\
&\Phi_{\varepsilon}(x)=0,\text{ when }x\in\left[  \varepsilon,+\infty\right)  .
\label{Phi=0}%
\end{align}
Moreover, there exists some constant $C>0$ such that
\begin{equation}
0\leq-\Phi_{\varepsilon}^{\prime}(x)\leq\frac{C}{\varepsilon},\quad\forall
x\in\left[  0,+\infty\right)  . \label{derivative of Phi}%
\end{equation}
In (\ref{derivative of Phi}) and in the following, $C$ denotes various
positive constants, which may vary from line to line, but do not depend on
$\varepsilon\in\left(  0,1\right]  $ and $y_{0}\in L^{2}(0,L).$

For the well-posedness of \eqref{new system}, we prove the following proposition on the global (in positive
time) existence and uniqueness of the solution to system (\ref{new system}).

\begin{prop}
\label{Glocal well-posedness}For every $y_{0}\in L^{2}\left(
0,L\right) $, there exists a unique mild solution $
y\!\in\!C([0,+\infty);\!L^{2}( 0,L) )\cap L_{loc}^{2}\left(
[0,+\infty );H_{0}^{1}\left(  0,L\right)  \right) $ of \eqref{new
system}. Moreover, there exists $C>0$ such that for every
$\varepsilon>0,$ for every $y_{0}\in L^{2}\left(  0,L\right)  $ and
for every $T>0,$ the unique solution of (\ref{new system}) satisfies
\begin{equation} \left\Vert y\right\Vert _{L^{2}\left(
0,T;H_{0}^{1}\left(  0,L\right) \right)
}^{2}\leq\frac{8T+2L}{3}\left\Vert y_{0}\right\Vert _{L^{2}\left(
0,L\right)  }^{2}+CT\left\Vert y_{0}\right\Vert
_{L^{2}\left(
0,L\right)  }^{4}. \label{Prop-local well-posedness-inequality}%
\end{equation}
\end{prop}

\subsubsection{Properties of the semigroup generated by
(\ref{new system})}
Let
\(
\mathcal{S}(t):L^{2}\left(  0,L\right)  \rightarrow L^{2}\left(  0,L\right)
\)
be the semigroup on $L^{2}\left(  0,L\right)  $ defined by
\(
S_N(t)(y_{0}):=y\left(  t,x\right) 
\), where $y\left(  t,x\right)  $ is the unique solution of (\ref{new
system}) with respect to the initial value $y_{0}\in L^{2}\left(
0,L\right) $. Let $T>0$. Then, for every $t\in\left[ 0,T\right]  $,
$S_N(t)=S(t)+R(t)$, or equivalently, $y(t,x)=w\left(  t,x\right)  +\alpha\left(  t,x\right)$ where, as above, for every $y_{0}\in L^{2}\left(  0,L\right)  $, $w\left(
t,\cdot\right)  :=S\left(  t\right)  y_{0}$ is the unique solution of
\eqref{eq: linear KdV-stability-intro}.
and $\alpha\left(  t,\cdot\right)  :=R(t)y_{0}$ is the unique solution of
\[
\left\{
\begin{array}
[c]{l}%
\p_t\alpha+\p_x^3\alpha+\p_x\alpha+\Phi_{\varepsilon}\left(  \left\Vert
w+\alpha\right\Vert _{L^{2}\left(  0,L\right)  }\right)  \left(  \p_x w%
\alpha+\p_x\alpha w+w\p_x w +\alpha\p_x\alpha\right)  =0,\\
\alpha\left(  t,0\right)  =\alpha\left(  t,L\right)=\p_x\alpha\left(
t,L\right)  =0,\\
\alpha\left(  0,x\right)  =0.
\end{array}
\right.
\]
Let $M:=\oplus_{|j|\leq N_0}\textrm{Span}\{\F_{\zeta_j}: \zeta_j=\ii\lambda_{c,j}(k,l)+\bigO((L-L_0)^2)\}$
where $\F_{\zeta_j}$ is an eigenfunction of $\A$ corresponding to the eigenvalue $\zeta_j$. Then we can do the following
decomposition of $L^{2}\left(  0,L\right):=H\oplus M $, where $H:=\{f\in D(\A):\poscalr{f}{\F_{\zeta_j}}_{0,L}=0\}$. 
Let $P_M$ and $P_{H}$ be the canonical projections on $M$ and $H$, respectively, with $P_H+P_M=\mathrm{Id}$. It is clear that $S(t)$ leaves $M$ and $H$ invariant and
$S(t)$ commutes with $P_H$ and $P_M$. Denote by $S_{M}(t):M\rightarrow
M$ and $S_{H}(t):H\rightarrow H$ the restriction of
$S(t)$ on $M$ and $H$ respectively. Then, by Theorem \ref{thm: main theorem linear version}, there exists $C_0>0$ such that $\left\Vert S_{H}(t)\right\Vert
\leq e^{-C_0 t}, \forall t\geq0.$

\subsubsection{ Global Lipschitzianity of the map $R\left(  t\right)  :L^{2}\left(
0,L\right)  \rightarrow
L^{2}\left(  0,L\right) $}

The aim of this part is to prove and estimate the global
Lipschitzianity of the map $R\left(  t\right)  :L^{2}\left(
0,L\right)  \rightarrow
L^{2}\left(  0,L\right)  $. To that end, we consider%
\[
\left\{
\begin{array}
[c]{l}%
\p_t\alpha+\p_x^3\alpha+\p_x\alpha+\Phi_{\varepsilon}\left(  \left\Vert
w+\alpha\right\Vert _{L^{2}\left(  0,L\right)  }\right)  \left(  \p_x w%
\alpha+\p_x\alpha w+w\p_x w +\alpha\p_x\alpha\right)  =0,\\
\alpha\left(  t,0\right)  =\alpha\left(  t,L\right)=\p_x\alpha\left(
t,L\right)  =0,\\
\alpha\left(  0,x\right)  =0.
\end{array}
\right.
\]
and%
\[
\left\{
\begin{array}
[c]{l}%
\p_t\Tilde{\alpha}+\p_x^3\Tilde{\alpha}+\p_x\Tilde{\alpha}+\Phi_{\varepsilon}\left(  \left\Vert
\Tilde{w}+\Tilde{\alpha}\right\Vert _{L^{2}\left(  0,L\right)  }\right)  \left(  \p_x \Tilde{w}%
\Tilde{\alpha}+\p_x\Tilde{\alpha} \Tilde{w}+\Tilde{w}\p_x \Tilde{w} +\Tilde{\alpha}\p_x\Tilde{\alpha}\right)  =0,\\
\Tilde{\alpha}\left(  t,0\right)  =\Tilde{\alpha}\left(  t,L\right)=\p_x\Tilde{\alpha}\left(
t,L\right)  =0,\\
\Tilde{\alpha}\left(  0,x\right)  =0.
\end{array}
\right.
\]
where $w$ is the solution of
\[
\left\{
\begin{array}
[c]{l}%
\p_tw+\p_x^3w+\p_xw=0,\\
w\left(  t,0\right)  =w\left(  t,L\right)=\p_xw\left(
t,L\right)  =0,\\
w\left(  0,x\right)  =0.
\end{array}
\right.
\]
and $\Tilde{w}$ is the solution of
\[
\left\{
\begin{array}
[c]{l}%
\p_t\Tilde{w}+\p_x^3\Tilde{w}+\p_x\Tilde{w}=0,\\
\Tilde{w}\left(  t,0\right)  =\Tilde{w}\left(  t,L\right)=\p_x\Tilde{w}\left(
t,L\right)  =0,\\
\Tilde{w}\left(  0,x\right)  =0.
\end{array}
\right.
\]
Set
\begin{align*}
\Delta & :=\alpha-\Tilde{\alpha},\quad
y:=\alpha+w,\quad\Tilde{y}:=\Tilde
{w}+\Tilde{\alpha},\\
\Phi_{1}  &  :=\Phi_{\varepsilon}\left(  \left\Vert y\right\Vert
_{L^{2}\left(  0,L\right)  }\right)
,\quad\Phi_{2}:=\Phi_{\varepsilon}\left( \left\Vert
\Tilde{y}\right\Vert _{L^{2}\left(  0,L\right)  }\right)  .
\end{align*}
Then we obtain%
\begin{equation}
\left\{
\begin{array}
[c]{l}%
\p_t\Delta+\p_x^3\Delta+\p_x\Delta=-\Phi_{1}y\p_xy+\Phi_{2}\Tilde
{y}\p_x\Tilde{y}=\Phi_{1}\left[  -\left(  \alpha+w\right) \p_x  \Delta
-\left(  \p_x\Tilde{\alpha}+\p_x w\right) \Delta-\Tilde{\alpha
}\p_x\left(  w-\Tilde{w}\right)  \right. \\
 \quad \quad\quad\quad\quad\quad\quad \quad\,\left. -\p_x\Tilde{\alpha}\left( w-\Tilde{w}\right)
-w\p_xw+\Tilde{w}\p_x\Tilde{w}\right] -\left(  \Phi_{1}-\Phi
_{2}\right)  \left(  \Tilde{\alpha}\p_x\Tilde{w}+
\Tilde{w}\p_x\Tilde{\alpha}+\Tilde{w}\p_x\Tilde{w}+\Tilde{\alpha}\p_x
\Tilde{\alpha}\right)  ,\\
\Delta\left(  t,0\right)  =\Delta\left(  t,L\right) =\p_x\Delta\left(
t,L\right)  =0,\\
\Delta\left(  0,x\right)  =0.
\end{array}
\right.  \label{lambda1}%
\end{equation}
Moreover, by the definition of $\Phi_{1},\,\Phi_{2}$ and
(\ref{Phi=0}), we get
\begin{equation}
\Phi_{1}=\Phi_{2}=0,\text{ \ \ \ }\forall\left\Vert y\right\Vert
_{L^{2}\left(  0,L\right)  }\geq\varepsilon,\, \forall \left\Vert \Tilde{y}\right\Vert _{L^{2}\left(
0,L\right)  }\geq\varepsilon. \label{Phi1=Phi2=0}%
\end{equation}

Now we are in a position to prove the following proposition on the global
Lipschitzianity of the map $R(t)$. With our notation, we have
\(R(t)y_{0}-R(t)\Tilde{y}_{0}=\alpha\left(  t,\cdot\right)  -\Tilde
{\alpha}\left(  t,\cdot\right)  =\Delta\).

\begin{prop}
\label{global Lipschitz}Let $T>0$. There exists $\varepsilon_{0}\in(0,1]$ and
$\tilde C :(0,\varepsilon_{0}]\rightarrow(0,+\infty)$ such that
\begin{gather}
\label{estRliptendvers0}\left\Vert \Delta\right\Vert _{L^{2}\left(
0,L\right)  }\leq\tilde C\left(  \varepsilon\right)  \left\Vert y_{0}%
-\Tilde{y}_{0}\right\Vert _{L^{2}\left(  0,L\right)
},\quad \forall\,\Tilde{y}_{0},y_{0}\in L^{2}\left(  0,L\right)
,\forall t\in\left[  0,T\right]  ,
\forall\varepsilon\in(0,\varepsilon_{0}] ,\\
\tilde C \left(  \varepsilon\right)  \rightarrow0\text{ as }\varepsilon
\rightarrow0^{+}.
\end{gather}

\end{prop}

\subsubsection{Smoothness of the semigroup}

\begin{lem}
\label{lemma0} Let $\varepsilon>0$ and $T>0$ be given. Then the nonlinear map
$S_N(t)$ defined by the unique solution of (\ref{new system}) is of class
$C^{3}$ from $L^{2}\left(  0,L\right)  $ to $C\left(  \left[  0,T\right]
;L^{2}\left(  0,L\right)  \right)  $. Moreover, its derivative $S_N^{(1)}$ at
$y_{0}\in L^{2}\left(  0,L\right)  $ is given by
\begin{equation}
S_N^{(1)}(y_{0})(h):=\mathcal{D}^{(1)}(y)(h),\, \forall
h\in L^{2}(0,L),
\end{equation}
where $\mathcal{D}^{(1)}(y)(h)$ is defined by the following system \eqref{e1}
with $y=S_N(y_{0})$.
\begin{equation}
\label{e1}%
\begin{cases}
\p_t\Delta+\p_x^3\Delta+\p_x\Delta+\Phi^{\prime}_{\varepsilon}(\|y\|_{L^{2}%
(0,L)})\displaystyle\frac{\int_{0}^{L} y\Delta\, dx}{\|y\|_{L^{2}(0,L)}}y\p_x y
+\Phi_{\varepsilon}(\|y\|_{L^{2}(0,L)})(y\p_x\Delta+\Delta \p_x y)=0,\\
\Delta(t,0)=\Delta(t,L)=\p_x\Delta(t,L)=0,\\
\Delta(0,x)=h(x),
\end{cases}
\end{equation}
\end{lem}

\begin{proof}
We refer to \cite{Zhang1995} and \cite[Theorem 5.4]{Bona-Sun-Zhang} for a
detailed argument in related circumstances.
\end{proof}

\subsubsection{Invariant manifolds}
Combining \cite[Remark 2.3]{Minh-Wu2004}, and Proposition \ref{global Lipschitz}, we are in
a position to apply \cite[Theorem 2.19]{Minh-Wu2004} and \cite[Theorem
2.28]{Minh-Wu2004}. This gives, if $\varepsilon>0$ is small enough which will
be always assumed from now on, the existence of an invariant manifold
for (\ref{new system}) which is of class $C^{3}$. More precisely, there exists a map
$g:M\rightarrow M^{\bot}$ of class $C^{3}$ satisfying $g(0)=0$ and
$g^{\prime}(0)=0$, such that, if
\[
G:=\left\{  x_{1}+g\left(  x_{1}\right)  :x_{1}\in M\right\}  ,
\]
then, for every $y_{0}\in G$ and for every $t\in[0,+\infty)$, $\mathcal{S}%
(t)y_{0}\in G$. Moreover, Theorem \ref{thm: nonlinear stability result} holds, if (and only if),
\begin{gather}
\label{dyncenter}\mathcal{S}(t)y_{0}\rightarrow0 \text{ as } t \rightarrow
+\infty, \, \forall y_{0}\in G \text{ such that } \|y_{0}\|_{L^{2}(0,L)}
\text{ is small enough.}%
\end{gather}
(For this last statement, see (2.42) in \cite{Minh-Wu2004}.) We
prove \eqref{dyncenter} in the next section.

\subsection{Dynamic on the invariant manifold: proof of Theorem \ref{thm: nonlinear stability result}}\label{sec: Dynamic on the center manifold}
In this section, we prove \eqref{dyncenter}, which concludes the
proof of Theorem \ref{thm: nonlinear stability result}.

\begin{proof}[Proof of Theorem \ref{thm: nonlinear stability result}]
Let $y_{0}\in \mathcal{M}$. Let, for $t\in[0,+\infty)$, $y(t)(x):=y(t,x):=(\mathcal{S}%
(t)y_{0})(x)$. We write
\begin{equation}
y(t,x)=\sum_{j}p_j(t)\F_{\zeta_j}(x)+y^{\star}(t,x)\text{,} \label{decomposition-y}%
\end{equation}
where $y^{\star}(t,x)\in
M^{\bot}$. Let $P(t)=\left(\begin{array}{c}
     p_1(t)  \\
     \vdots\\
     p_{N_0}(t)
\end{array}\right)$, $\F(x)=\left(\begin{array}{c}
     \F_{\zeta_{1}}(x)  \\
     \vdots\\
    \F_{\zeta_{N_0}}(x)
\end{array}\right)$ and the Gramian matrix with respect to the duality $\poscalr{\cdot}{\cdot}_{L^2(0,L)}$ be $D(L):=\left(\poscalr{\F_{\zeta_j}}{\F_{\zeta_k}}_{L^2(0,L)}\right)_{jk}$. Thus, the coefficients $P(t)$ is defined by the matrix equation $D(L)P(t)=\left(\begin{array}{c}
     \poscalr{y}{\F_{\zeta_1}}_{L^2(0,L)}  \\
     \vdots\\
     \poscalr{y}{\F_{\zeta_{N_0}}}_{L^2(0,L)}
\end{array}\right)$. We notice that $D(L)=\mathrm{Id}+\bigO(|L-L_0|)$. Therefore, we obtain $p_j(t)=\poscalr{y}{\F_{\zeta_j}}_{L^2(0,L)}\left(1+\bigO(|L-L_0|)\right)$. 
\begin{align}
\frac{dp_j(t)}{dt}  &  =\int_{0}^{L}\p_t y(t,x)  \overline{\F_{\zeta_j}(L-x)}dx\nonumber\\
&=\int_{0}^{L}\left( -\p_x^3y(t,x)-\p_xy(t,x)- y(t,x) \p_x y(t,x)\right)  \overline{\F_{\zeta_j}(L-x)}
(x)dx+\bigO(|L-L_0|)\nonumber\\
&  =-\int_{0}^{L}y(t,x)  \overline{(\p_x^3+\p_x)\F_{\zeta_j}(L-x)}
(x)dx-\int_{0}^{L} y \p_x y \overline{\F_{\zeta_j}(L-x)}
(x)dx+\bigO(|L-L_0|)\nonumber\\
&  =-\overline{\zeta_j}p_j(t)-\frac{1}{2}\int_{0}^{L} y^2 \overline{\p_x \F_{\zeta_j}(L-x)}dx+\bigO(|L-L_0|^2)\nonumber\\
&=-\overline{\zeta_j}p_j(t)+\bigO(|L-L_0|^2). \label{1}
\end{align}
This concludes the proof of \eqref{dyncenter} and the proof of Theorem
\ref{thm: nonlinear stability result}.
\end{proof}

\appendix

\section{Spectral analysis of $\B$ and $\A$}\label{sec: appendix-spectrum-A-B}
In this section, we demonstrate the results in Section \ref{sec: spectral analysis of A}.
\subsection{Asymptotic analysis on the operator $\B$}
\begin{proof}[Proof of Proposition \ref{prop: type-1-2}]
Let $\G$ satisfy the equation 
\begin{equation}
\left\{
\begin{array}{ll}
     \G'''+\G'+\ii\lambda_{c}\G=0,  & \text{ in }(0,L_0),\\
     \G(0)=\G(L_0)=0,& \\
     \G'(0)=\G'(L_0),& 
\end{array}
\right.       
\end{equation}
with $\lambda_c=\frac{(2k+l)(k-l)(2l+k)}{3\sqrt{3}(k^2+k l+l^2)^{\frac{3}{2}}}$ and $L_0=2\pi\sqrt{\frac{k^2+k l+l^2}{3}}$. Then the general form of eigenfunctions is as follows
\begin{equation*}
\G(x):= e^{\ii x\frac{\sqrt{3} (2 k + l) }{3\sqrt{k^2 + k l + l^2}}}  +C_1 e^{-\ii x\frac{\sqrt{3} (k + 2 l) }{3 \sqrt{k^2 + k l + l^2}}}  +C_2 e^{\ii x\frac{\sqrt{3} (-k + l)}{ \sqrt{k^2 + k l + l^2}}}.
\end{equation*}
Using the boundary conditions, we obtain
\begin{equation*}
\left\{
\begin{array}{l}
     1+C_1+C_2=0,  \\
   -\frac{\ii((-2+C_1+C_2)k++(-1+2C_1-C_2) l)}{\sqrt{3}\sqrt{k^2 + k l + l^2}}=\frac{\ii e^{-\ii\frac{2}{3}(k+2l)\pi}((-2+C_1+C_2)k++(-1+2C_1-C_2) l)}{\sqrt{3}\sqrt{k^2 + k l + l^2}} .
\end{array}
\right.
\end{equation*}
There are two situations,
\begin{enumerate}
    \item If $e^{-\ii\frac{2}{3}(k+2l)\pi}\neq1\Leftrightarrow (k-l)\not\equiv0\mod{3}$, we know that $C_1=\frac{l}{k}$ and $C_2=-\frac{l+k}{k}$. In this case, we have
    \begin{equation*}
    \G(x)= e^{\ii x\frac{\sqrt{3} (2 k + l) }{3\sqrt{k^2 + k l + l^2}}}  +\frac{l}{k} e^{-\ii x\frac{\sqrt{3} (k + 2 l) }{3 \sqrt{k^2 + k l + l^2}}}  -\frac{l+k}{k} e^{\ii x\frac{\sqrt{3} (-k + l)}{ \sqrt{k^2 + k l + l^2}}},    
    \end{equation*}
    with $\G'(0)=\G'(L_0)=0$.
    \item If $e^{-\ii\frac{2}{3}(k+2l)\pi}=1\Leftrightarrow (k-l)\equiv0\mod{3}$, we know $C_2=-(1+C_1)$ and the second equation is trivial. In fact, we obtain two linearly independent solutions $\G$ as before and $\Tilde{\G}$ in the following form
    \begin{equation*}
    \Tilde{\G}(x)= e^{\ii x\frac{\sqrt{3} (2 k + l) }{3\sqrt{k^2 + k l + l^2}}}  - e^{-\ii x\frac{\sqrt{3} (k + 2 l) }{3 \sqrt{k^2 + k l + l^2}}},     
    \end{equation*}
    with $\Tilde{\G}'(0)=\Tilde{\G}'(L_0)=\ii\frac{\sqrt{3}(k+l)}{\sqrt{k^2+k l+l^2}}\neq0$.
\end{enumerate}
\end{proof}

\begin{rem}\label{rem: pm-sign for eigenvalues}
One can easily deduce from \eqref{eq: t-L defined eq-1} and \eqref{eq: t-L defined eq-2} that $\tau_{-j}=-\tau_j$ and $\lambda_{-j}=-\lambda_j$, $\forall j\in \Z\backslash\{0\}$. 
\end{rem}

\subsection{Modulated functions}\label{sec: modulated functions-appendix}
\begin{proof}[Proof of Lemma \ref{lem: existence of modulated functions}]
For the algebraic equation $\xi^3+\xi-\mu=0$, we know that there exists a real root denoted by $\omega$, with $\omega>0$, which indicates that $\mu=\omega^3+\omega$. 
Moreover, we obtain the three roots for $\xi^3+\xi+\mu=0$, which are given by
\begin{equation}\label{eq: 3 roots}
\xi_1=\omega,\quad \xi_2=-\frac{\omega}{2}+\frac{\sqrt{4+3\omega^2}}{2}\ii, \text{ and}\quad \xi_3=  -\frac{\omega}{2}-\frac{\sqrt{4+3\omega^2}}{2}\ii.  
\end{equation} 
We construct the solution to \eqref{eq: static KdV with nonvanishing boundary condition} in the following form:
\begin{equation}\label{eq: formal solution h}
h_{\mu}(x)=C_1 e^{\omega x}+C_2 e^{-\frac{\omega}{2}x}\cos{\frac{\sqrt{4+3\omega^2}}{2} x}+C_3 e^{-\frac{\omega}{2}x}\sin{\frac{\sqrt{4+3\omega^2}}{2} x}.
\end{equation}
In addition, we obtain the first derivative of $h_{\omega}$,
\begin{align*}
h'_{\mu}(x)=C_1\omega e^{\omega x}+\frac{-C_2\omega+C_3\sqrt{4+3\omega^2}}{2}e^{-\frac{\omega}{2}x}\cos{\frac{\sqrt{4+3\omega^2}}{2} x}-\frac{C_2\sqrt{4+3\omega^2}+C_3\omega}{2}e^{-\frac{\omega}{2}x}\sin{\frac{\sqrt{4+3\omega^2}}{2} x}.
\end{align*} 
For the boundary condition, we plug \eqref{eq: formal solution h} into the system \eqref{eq: static KdV with nonvanishing boundary condition} and obtain 
{\footnotesize
\begin{equation}\label{eq: equations for coeff}
\left\{
\begin{aligned}
&C_1+C_2=0,\\
&C_1 e^{\omega L}+C_2 e^{-\frac{\omega}{2}L}\cos{\frac{\sqrt{4+3\omega^2}}{2} L}+C_3 e^{-\frac{\omega}{2}L}\sin{\frac{\sqrt{4+3\omega^2}}{2} L}=0,\\
&C_1\omega e^{\omega L}+\frac{-C_2\omega+C_3\sqrt{4+3\omega^2}}{2}e^{-\frac{\omega}{2}L}\cos{\frac{\sqrt{4+3\omega^2}}{2} L}-\frac{C_2\sqrt{4+3\omega^2}+C_3\omega}{2}e^{-\frac{\omega}{2}L}\sin{\frac{\sqrt{4+3\omega^2}}{2} L}\\
-&C_1\omega -\frac{-C_2\omega+C_3\sqrt{4+3\omega^2}}{2}=1.
\end{aligned}
\right.
\end{equation}
}
Therefore, we solve the linear equations for $(C_1,C_2,C_3)$ and obtain the solution $h_{\mu}$ to the equation \eqref{eq: static KdV with nonvanishing boundary condition}, i.e.
\begin{equation}
h_{\mu}(x)=\frac{e^{\frac{\omega}{2}(2x-L)}\sin{\frac{\sqrt{4+3\omega^2}}{2} L}-e^{-\frac{\omega}{2}(L+x)}\sin{\frac{\sqrt{4+3\omega^2}}{2} (L-x)}-e^{\frac{\omega}{2}(2L-x)}\sin{\frac{\sqrt{4+3\omega^2}}{2} x}}{\sqrt{4+3\omega^2}\cosh{\omega L}+3\omega\sinh{\frac{\omega}{2}L} \sin{\frac{\sqrt{4+3\omega^2}}{2} L}-\sqrt{4+3\omega^2}\cosh{\frac{\omega}{2}L}\cos{\frac{\sqrt{4+3\omega^2}}{2} L}}
\end{equation}  
\end{proof}
\begin{proof}[Proof of Proposition \ref{prop: h_mu-est-uniform}]
First, we concentrate on the $\|h_{\omega}\|_{L^{\infty}(0,L)}$. As $\omega\rightarrow+\infty$, for the denominator terms,
\begin{equation*}
\cosh{\omega L}\sim e^{\omega L},\quad \sinh{\frac{\omega}{2}L} \sim\cosh{\frac{\omega}{2}L}\sim e^{\frac{\omega}{2}L}.
\end{equation*}
And for the numerator terms, since $x\in(0,L)$, we know that $-(L+x)<-L<0$, which implies that $|e^{-\frac{\omega}{2}(x+L)}\sin{\frac{\sqrt{4+3\omega^2}}{2} (L-x)}|\leq e^{-\frac{\omega}{2}L}$. Moreover, $e^{\frac{\omega}{2}(2x-L)}<e^{\frac{\omega}{2}(2L-x)}$, $\forall x\in(0,L)$. Therefore, we deduce that as $\omega\rightarrow+\infty$,
\begin{equation*}
|h_{\mu}(x)|\sim \frac{e^{\frac{\omega}{2}(2L-x)}}{\omega e^{\omega L}}\sim\frac{e^{-\frac{\omega}{2}x}}{\omega}\sim \frac{1}{\omega} ,\forall x\in(0,L).
\end{equation*}
Hence, we obtain $\|h_{\mu}\|_{L^{\infty}(0,L)}\sim  \frac{1}{\omega}$ as $\omega\rightarrow+\infty$. 
For the $L^2-$norm, 
\begin{align*}
\|h_{\mu}\|_{L^2(0,L)}^2&=\int_0^L  |h_{\mu}(x)|^2dx
                           \sim \int_0^L  \frac{e^{-\omega x}}{\omega^2}dx
                           \sim \frac{1}{\omega^3}(1-e^{-\omega L})\sim \frac{1}{\omega^3}.
\end{align*}
Hence, we know that as $\omega\rightarrow+\infty$, $\|h_{\mu}\|_{L^2(0,L)}\sim \omega^{-\frac{3}{2}}\sim \mu^{-\frac{1}{2}}$. 
Now we look at the first derivative of $h_{\mu}$. By simple computation, we know its expression is as follows:
\begin{align*}
h'_{\mu}(x)&=\frac{\omega e^{\frac{\omega}{2}(2x-L)}\sin{\frac{\sqrt{4+3\omega^2}}{2} L}+\frac{\omega}{2}e^{-\frac{\omega}{2}(L+x)}\sin{\frac{\sqrt{4+3\omega^2}}{2} (L-x)}+\frac{\sqrt{4+3\omega^2}}{2}e^{-\frac{\omega}{2}(L+x)}\cos{\frac{\sqrt{4+3\omega^2}}{2} (L-x)}}{\sqrt{4+3\omega^2}\cosh{\omega L}+3\omega\sinh{\frac{\omega}{2}L} \sin{\frac{\sqrt{4+3\omega^2}}{2} L}-\sqrt{4+3\omega^2}\cosh{\frac{\omega}{2}L}\cos{\frac{\sqrt{4+3\omega^2}}{2} L}}\\
+&\frac{\frac{\omega}{2}e^{\frac{\omega}{2}(2L-x)}\sin{\frac{\sqrt{4+3\omega^2}}{2} x}-\frac{\sqrt{4+3\omega^2}}{2}e^{\frac{\omega}{2}(2L-x)}\cos{\frac{\sqrt{4+3\omega^2}}{2} x}}{\sqrt{4+3\omega^2}\cosh{\omega L}+3\omega\sinh{\frac{\omega}{2}L} \sin{\frac{\sqrt{4+3\omega^2}}{2} L}-\sqrt{4+3\omega^2}\cosh{\frac{\omega}{2}L}\cos{\frac{\sqrt{4+3\omega^2}}{2} L}}.
\end{align*}
As $\omega\rightarrow+\infty$, 
\begin{align*}
|h'_{\mu}(x)|\sim\frac{\omega e^{\frac{\omega}{2}(2x-L)}+\omega e^{-\frac{\omega}{2}(L+x)}+\omega e^{\frac{\omega}{2}(L+x)}+\omega e^{\frac{\omega}{2}(2L-x)}+\omega e^{\frac{\omega}{2}(2L-x)}}{\omega e^{\omega L}}
\sim \frac{\omega e^{\frac{\omega}{2}(2L-x)}}{\omega e^{\omega L}}
\sim e^{-\frac{\omega}{2}x}.
\end{align*}
Thus, we obtain $\|h'_{\mu}\|_{L^{\infty}(0,L)}\sim  1$ as $\omega\rightarrow+\infty$. 
And similarly, for $L^2-$norm, as $\omega\rightarrow+\infty$,
\begin{align*}
\|h'_{\mu}\|_{L^2(0,L)}^2&=\int_0^L  |h'_{\mu}(x)|^2dx                           \sim \int_0^L  e^{-\omega x}dx
                           \sim \frac{1}{\omega}(1-e^{-\omega L})\sim \frac{1}{\omega}.
\end{align*}
Consequently, we obtain $\|h'_{\mu}\|_{L^2(0,L)}\sim \omega^{-\frac{1}{2}}\sim \mu^{-\frac{1}{6}}$ as $\omega\rightarrow+\infty$. 
At the endpoint $x=0$ and $x=L$,
\begin{equation}
\begin{aligned}
h_{\mu}'(0)&=\frac{\frac{3}{2}\omega e^{-\frac{\omega}{2}L}\sin{\frac{\sqrt{4+3\omega^2}}{2} L}+\frac{\sqrt{4+3\omega^2}}{2}e^{-\frac{\omega}{2}L}\cos{\frac{\sqrt{4+3\omega^2}}{2} L}-\frac{\sqrt{4+3\omega^2}}{2}e^{\omega L}}{\sqrt{4+3\omega^2}\cosh{\omega L}+3\omega\sinh{\frac{\omega}{2}L} \sin{\frac{\sqrt{4+3\omega^2}}{2} L}-\sqrt{4+3\omega^2}\cosh{\frac{\omega}{2}L}\cos{\frac{\sqrt{4+3\omega^2}}{2} L}}\\
&\sim 1, \text{ as }\omega \rightarrow+\infty,
\end{aligned}
\end{equation}
and
\begin{equation}\label{eq: est-h_mu'(L)}
\begin{aligned}
h_{\mu}'(L)&=\frac{\frac{3}{2}\omega e^{\frac{\omega}{2}L}\sin{\frac{\sqrt{4+3\omega^2}}{2} L}+\frac{\sqrt{4+3\omega^2}}{2}e^{-\omega L}-\frac{\sqrt{4+3\omega^2}}{2}e^{\frac{\omega}{2}L}\cos{\frac{\sqrt{4+3\omega^2}}{2} L}}{\sqrt{4+3\omega^2}\cosh{\omega L}+3\omega\sinh{\frac{\omega}{2}L} \sin{\frac{\sqrt{4+3\omega^2}}{2} L}-\sqrt{4+3\omega^2}\cosh{\frac{\omega}{2}L}\cos{\frac{\sqrt{4+3\omega^2}}{2} L}}\\
&\sim e^{-\frac{\omega}{2}L}, \text{ as }\omega \rightarrow+\infty.
\end{aligned}
\end{equation}
In addition, we know that
\begin{equation}
\lim_{\omega\rightarrow+\infty}h_{\mu}'(0)=1,\quad \lim_{\omega\rightarrow+\infty}h_{\mu}'(L)=0.
\end{equation}
\end{proof}
\begin{lem}
For $L\notin\mathcal{N}$, there exists a unique solution to the equation \eqref{eq: defi of h}. Moreover, the unique solution is 
\begin{equation}
    h(x)=-\frac{e^{\ii L}+1}{2\ii (e^{\ii L}-1)}+\frac{e^{\ii x}}{2\ii (e^{\ii L}-1)}+\frac{e^{\ii(L-x)}}{2\ii (e^{\ii L}-1)}.
\end{equation}
\end{lem}
\begin{proof}
The proof is straightforward. It is easy to verify that $h(x)=-\frac{e^{\ii L}+1}{2\ii (e^{\ii L}-1)}+\frac{e^{\ii x}}{2\ii (e^{\ii L}-1)}+\frac{e^{\ii(L-x)}}{2\ii (e^{\ii L}-1)}$ solves the equation \eqref{eq: defi of h}. Indeed,
\begin{align*}
h'''(x)&=\frac{-e^{\ii x}}{2 (e^{\ii L}-1)}+\frac{e^{\ii(L-x)}}{2(e^{\ii L}-1)},\\
h'(x)&=\frac{e^{\ii x}}{2 (e^{\ii L}-1)}-\frac{e^{\ii(L-x)}}{2(e^{\ii L}-1)}.
\end{align*}
Hence, $h'''+h'=0$. As for boundary conditions, 
\begin{align*}
h(0)&=-\frac{e^{\ii L}+1}{2\ii (e^{\ii L}-1)}+\frac{1}{2\ii (e^{\ii L}-1)}+\frac{e^{\ii L}}{2\ii (e^{\ii L}-1)}=0,\\
h(L)&=-\frac{e^{\ii L}+1}{2\ii (e^{\ii L}-1)}+\frac{e^{\ii L}}{2\ii (e^{\ii L}-1)}+\frac{1}{2\ii (e^{\ii L}-1)}=0,\\
h'(L)&=\frac{e^{\ii L}}{2 (e^{\ii L}-1)}-\frac{1}{2(e^{\ii L}-1)}=\frac{1}{2},\\
h'(0)&=\frac{1}{2(e^{\ii L}-1)}-\frac{e^{\ii L}}{2 (e^{\ii L}-1)}=-\frac{1}{2}.
\end{align*}
Hence, $h(0)=h(L)=0$, and $h'(L)-h'(0)=1$. In particular, we point out that $h(L-x)=h(x)$, $\forall x\in [0,L]$. For the uniqueness, assume that $h_1$ and $h_2$ are both solutions to the equation \eqref{eq: defi of h}, define $\Tilde{h}=h_1-h_2$. Then $\Tilde{h}$ is the eigenfunction of the operator $\A$ corresponding to the eigenvalue $0$. We know that $0$ is not an eigenvalue of $\A$, which implies that $\Tilde{h}\equiv0$. Therefore, we obtain a unique solution to the equation \eqref{eq: defi of h},
\begin{equation*}
h(x)=-\frac{e^{\ii L}+1}{2\ii (e^{\ii L}-1)}+\frac{e^{\ii x}}{2\ii (e^{\ii L}-1)}+\frac{e^{\ii(L-x)}}{2\ii (e^{\ii L}-1)}.    
\end{equation*}
\end{proof}
Thanks to the expression of $h$, we have the following corollary.
\begin{coro}\label{coro: smooth h}
$h\in C^{\infty}(0,L)$. $h^{(2m)}(L)=h^{(2m)}(0)=0$, and $h^{(2m+1)}(L)-h^{(2m+1)}(0)=(-1)^m$, $\forall m\in\N$.
\end{coro}
\subsection{Basic propertis of the operator $\A$}\label{sec: proof-basic-properties-A}
\begin{proof}[Proof of Lemma \ref{lem: basic prop of eigen of bad op}]
For the first statement, using simple integration by parts, 
    \begin{multline*}
        \lambda \int_0^L \F(x)\overline{\F(x)} dx=   \int_0^L \A\F(x) \overline{\F(x)} dx=- \int_0^L \F(x) \A \overline{\F(x)} dx- |\F'(0)|^2\\
       =       - \overline{\lambda} \int_0^L \F(x)  \overline{\F(x)} dx-|\F'(0)|^2,
    \end{multline*}
thus $2\Re{\lambda}\|\F\|^2_{L^2(0,L)}= (\lambda+ \overline{\lambda} )\int_0^L \F(x) \overline{\F(x)} dx= -|\F'(0)|^2$. Therefore, using that $\F$ is normalized, we know that 
\begin{equation}
\Re{\lambda}=-\frac{1}{2}|\F'(0)|^2\leq 0.   
\end{equation}
If we further require that $L\notin \mathcal{N}$, we know that $\F'(0)\neq 0$, thus $\Re{\lambda}<0$. This means that $\Re \lambda_k\neq 0$. The second statement follows by taking conjugation to the system \eqref{eq: eigenvalue problem with Neumann}. We prove the third statement by the argument of contradiction. Suppose that $\F_1$ and $\F_2$ are both eigenfunctions of the same eigenvalue $\lambda$. Then we define a function $u$ by 
\begin{equation*}
    u(x)=\F_1(x)-\frac{\F_1'(0)}{\F_2'(0)}\F_2(x),\;x\in[0,L].
\end{equation*}
Then $u$ is also an eigenfunction associated with the eigenvalue $\lambda$ with $u(0)=u(L)=u'(0)=u'(L)=0$, which is a contradiction to the fact that $L\notin\mathcal{N}$. Thus, all eigenvalues are simple for $L\notin\mathcal{N}$. Now we turn to the last statement. We know that $\lambda\in \R$ and $\lambda<0$. Then, there is a unique $\tau\in \R$ such that $-\lambda=2\tau(4\tau^2+1)$ and $\tau>0$. Therefore, the three roots of
\begin{equation*}
    (\ii\xi)^3+\ii\xi+ \lambda=0,
\end{equation*}
read as 
\begin{equation*}
    \xi_1=\ii\tau+\sqrt{3\tau^2+1},\xi_2=\ii\tau-\sqrt{3\tau^2+1},\xi_3=-2\ii\tau.
\end{equation*}
Hence, we obtain the eigenfunctions in the following form:
\begin{equation}
\F(x)=r_1 e^{(-\tau+\ii\sqrt{1+3\tau^2})x}+r_2 e^{(-\tau-\ii\sqrt{1+3\tau^2})x}+r_3e^{2\tau x}.
\end{equation}
In addition, we are able to compute the derivative of $\E$ in the following form:
\begin{equation}
\F'(x)= r_1(-\tau+\ii\sqrt{1+3\tau^2})e^{(-\tau+\ii\sqrt{1+3\tau^2})x}- r_2(\tau+\ii\sqrt{1+3\tau^2}) e^{(-\tau-\ii\sqrt{1+3\tau^2})x}+2\tau r_3e^{2\tau x}.
\end{equation}
Combining with the boundary conditions, the coefficients $r_1$, $r_2$, and $r_3$ satisfy the following equations:
\begin{equation*}
\left\{
\begin{array}{l}
     \F(0)=r_1+r_2+r_3=0,  \\
     \F(L)=r_1 e^{(-\tau+\ii\sqrt{1+3\tau^2})L}+r_2 e^{(-\tau-\ii\sqrt{1+3\tau^2})L}+r_3e^{2\tau L}=0,\\
     \F'(L)=r_1(-\tau+\ii\sqrt{1+3\tau^2})e^{(-\tau+\ii\sqrt{1+3\tau^2})L}- r_2(\tau+\ii\sqrt{1+3\tau^2}) e^{(-\tau-\ii\sqrt{1+3\tau^2})L}+2\tau r_3e^{2\tau L}=0.
\end{array}
\right.
\end{equation*}
From the first and second equations, we obtain that
\begin{equation*}
     r_2=-r_1\frac{e^{(-\tau+\ii\sqrt{1+3\tau^2})L}-e^{2\tau L}}{e^{-(\tau+\ii\sqrt{1+3\tau^2})L}-e^{2\tau L}},\quad  r_3=-r_1\frac{e^{-(\tau+\ii\sqrt{1+3\tau^2})L}-e^{(-\tau+\ii\sqrt{1+3\tau^2})L}}{e^{-(\tau+\ii\sqrt{1+3\tau^2})L}-e^{2\tau L}}.
\end{equation*}
Thus, after simplifying the third equation, we obtain
\begin{equation}
    -e^{-3L\tau} \sqrt{1 + 3\tau^2} + \sqrt{1 + 3\tau^2} \cos\left( L \sqrt{1 + 3\tau^2} \right) - 3t \sin\left( L \sqrt{1 + 3\tau^2} \right)=0.
\end{equation}
It is equivalent to 
\begin{equation*}
\cos{(L \sqrt{1 + 3\tau^2}+\theta(\tau))}=e^{-3L\tau}\frac{\sqrt{1 + 3\tau^2}}{\sqrt{1 + 12\tau^2}}\in (0,1),   
\end{equation*}
where $\tan{\theta(\tau)}=\frac{3\tau}{\sqrt{1 + 3\tau^2}}>0$. There exists a unique solution such that in each interval $L \sqrt{1 + 3\tau^2}+\theta(\tau)\in(-\frac{\pi}{2}+2k\pi,\frac{\pi}{2}+2k\pi)$. While $0<L<L\sqrt{1 + 3\tau^2}+\theta(\tau)$, thus, we know that $k\geq0$.
\end{proof}

\subsection{Computation details}
\subsubsection{Proof of Lemma \ref{lem: critical eigenvalues}}\label{sec: Proof of Lemma-lem: critical eigenvalues}
\begin{proof}[Proof of Lemma \ref{lem: critical eigenvalues}]
Let us define a function $g(t)=2t(4t^2-1)$. It is easy to check that the following properties hold for $g$.
\begin{itemize}
    \item $|g(t)|<\frac{2\sqrt{3}}{9}$ is equivalent to $|t|<\frac{\sqrt{3}}{3}$, i.e. $3t^2<1$.
    \item For any $s_0\in (-\frac{2\sqrt{3}}{9},\frac{2\sqrt{3}}{9})$, $g(t)=s_0$ has three different roots.
\end{itemize}
We first claim that $\lambda_{c}(k_j,l_j)<\frac{2\sqrt{3}}{9}$. In fact, this is deduced by 
\begin{equation*}
\frac{(2k+l)(k-l)(2l+k)}{3\sqrt{3}(k^2+k l+l^2)^{\frac{3}{2}}}<\frac{2\sqrt{3}}{9} ,k\geq l.   
\end{equation*}
Let $x=\frac{l}{k}\in (0,1]$. It suffices to prove that 
\begin{equation}
    (2+x)^2(1-x)^2(2x+1)^2<4(1+x+x^2)^3, x\in (0,1].
\end{equation}
Indeed, $4(1+x+x^2)^3-(2+x)^2(1-x)^2(2x+1)^2=27 x^2 + 54 x^3 + 27 x^4>0$ for $x\in (0,1]$. Now since $|\lambda_c(k_j,l_j)|<\frac{2\sqrt{3}}{9}$, we know that $g(\tau_c)=\lambda_{c}(k_j,l_j)$ has three different roots and $3|\tau_c|^2<1$.
\end{proof}

\subsubsection{Remainders of the proof of Proposition \ref{prop: Asymp in L-low}}\label{sec: Remainders of the proof of Proposition-prop: Asymp in L-low}
Following the analysis of the high frequencies, we continue to study the asymptotic behaviors at low frequencies. The most technical part of the proof of Proposition \ref{prop: Asymp in L-low} is to verify the condition of the implicit function theorem, which involves a detailed and sophisticated computation of many derivatives of the function $F(t,L)$ defined in \eqref{eq: defi of function F(t,L)}. Here we complete the details of the proof of Proposition \ref{prop: Asymp in L-low}.
\begin{proof}[Remainders of the proof of Proposition \ref{prop: Asymp in L-low}]
Let recall the definition of $F(t,L)$,
\begin{equation*}
F(t,L)=2\sqrt{1-3t^2}\cos{(2t L)}-(\sqrt{1-3t^2}+3t)\cos{((\sqrt{1-3t^2}-t)L)}+(3t-\sqrt{1-3t^2})\cos{((\sqrt{1-3t^2}+t)L)}.
\end{equation*}
Then the first derivatives are the following:
\begin{align*}
\p_t F(t,L)&=-\frac{6t}{\sqrt{1-3t^2}}\cos{2t L}-4L\sqrt{1-3t^2}\sin{2tL}-\frac{3\sqrt{1-3t^2}-3t}{\sqrt{1-3t^2}}\cos{(\sqrt{1-3t^2}-t)L}\\
-&L\frac{(\sqrt{1-3t^2}+3t)^2}{\sqrt{1-3t^2}}\sin{((\sqrt{1-3t^2}-t)L)}+\frac{3\sqrt{1-3t^2}+3t}{\sqrt{1-3t^2}}\cos{((\sqrt{1-3t^2}+t)L)}\\
+&L\frac{(\sqrt{1-3t^2}-3t)^2}{\sqrt{1-3t^2}}\sin{((\sqrt{1-3t^2}+t)L)},\\
\p_L F(t,L)&=-4 t \sqrt{1 - 3 t^2}
   \sin{2 L t} + (1-6t^2 +2t\sqrt{1 - 3 t^2}) \sin{(
   (-t + \sqrt{1-3t^2})L)}\\
   -& (6t^2-1 +2t \sqrt{1-3t^2}) \sin{((t + \sqrt{1-3t^2})L)}.
\end{align*}   
Then we plug $(t,L)=(\frac{\pi}{L_0}\frac{2k+l}{3}, L_0)$ in the expressions above and obtain
\begin{align*}
(\p_t F)(\frac{\pi}{L_0}\frac{2k+l}{3}, L_0)&=-\frac{4k+2l}{l}\cos{\frac{2\pi(2k+l)}{3}}-4l\pi\sin{\frac{2\pi(2k+l)}{3}}-\frac{2(l-k)}{l}\cos{\frac{2\pi(l-k)}{3}}\\
-&\frac{4\pi(k+l)^2}{l}\sin{\frac{2\pi(l-k)}{3}}+\frac{2k+4l}{l}\cos{\frac{2\pi(k+2l)}{3}}+\frac{4\pi k^2}{l}\sin{\frac{2\pi(k+2l)}{3}},\\
\p_L F(\frac{\pi}{L_0}\frac{2k+l}{3}, L_0)&=-4 \frac{\pi}{L_0}\frac{2k+l}{3} \frac{l\pi}{L_0}
   \sin{\frac{2\pi(2k+l)}{3}} + (-3(\frac{\pi}{L_0}\frac{2k+l}{3})^2+(\frac{l\pi}{L_0})^2 +2\frac{\pi}{L_0}\frac{2k+l}{3}\frac{l\pi}{L_0}) \sin{
   \frac{2\pi(l-k)}{3}}\\
   -& (3(\frac{\pi}{L_0}\frac{2k+l}{3})^2-(\frac{l\pi}{L_0})^2 +2\frac{\pi}{L_0}\frac{2k+l}{3} \frac{l\pi}{L_0}) \sin{\frac{2\pi(k+2l)}{3}}\\
   &=-\frac{\pi^2}{L_0^2}\frac{4l(2k+l)}{3} 
   \sin{\frac{2\pi(2k+l)}{3}} - \frac{\pi^2}{L_0^2}\frac{(2k+l)^2-3l^2-2l(2k+l)}{3}  \sin{
   \frac{2\pi(l-k)}{3}}\\
   -& \frac{\pi^2}{L_0^2}\frac{(2k+l)^2-3l^2+2l(2k+l)}{3} \sin{\frac{2\pi(k+2l)}{3}}\\
   &=-\frac{4\pi^2}{3L_0^2}\left(l(2k+l)
   \sin{\frac{2\pi(2k+l)}{3}} + (k^2-l^2)\sin{
   \frac{2\pi(l-k)}{3}}
   + k(k+2l)\sin{\frac{2\pi(k+2l)}{3}}\right).
\end{align*}
For the second derivatives, we obtain
\begin{align*}
     \p_L^2 &F(t,L)=-8t^2 \sqrt{1 - 3 t^2}
   \cos{2 L t} + (-t + \sqrt{1-3t^2})^2(3t+\sqrt{1 - 3 t^2}) \cos{(
   (-t + \sqrt{1-3t^2})L)}\\
   -& (t + \sqrt{1-3t^2})^2(3t-\sqrt{1 - 3 t^2}) \cos{((t + \sqrt{1-3t^2})L)},\\
   \p_{tL}&F(t,L)=-8tL \sqrt{1 - 3 t^2}\cos{2 L t}+(\frac{12t^2}{\sqrt{1-3t^2}}-4\sqrt{1-3t^2})\sin{2 L t}\\
   -&\frac{L(\sqrt{1-3t^2}-t)(3t + \sqrt{1-3t^2})^2}{\sqrt{1-3t^2}}\cos{((\sqrt{1-3t^2}-t)L)}-\frac{2(-1+6t^2 +6t\sqrt{1-3t^2})}{\sqrt{1-3t^2}} \sin{((\sqrt{1-3t^2}-t)L)}\\
   +&\frac{L(\sqrt{1-3t^2}+t)(3t-\sqrt{1-3t^2})^2}{\sqrt{1-3t^2}}\cos{((\sqrt{1-3t^2}+t)L)}-\frac{2(1-6t^2 +6t\sqrt{1-3t^2})}{\sqrt{1-3t^2}} \sin{((\sqrt{1-3t^2}+t)L)}.
\end{align*}
\begin{enumerate}
    \item In the case $k\equiv l\mod{3}$, we know that 
    $F(\frac{\pi}{L_0}\frac{2k+l}{3}, L_0)=\p_t F(\frac{\pi}{L_0}\frac{2k+l}{3}, L_0)=\p_L F(\frac{\pi}{L_0}\frac{2k+l}{3}, L_0)=0$. Indeed,
    \begin{align*}
(\p_t F)(\frac{\pi}{L_0}\frac{2k+l}{3}, L_0)&=-\frac{4k+2l}{l}-\frac{2(l-k)}{l}+\frac{2k+4l}{l}=0.\\
\p_L F(\frac{\pi}{L_0}\frac{2k+l}{3}, L_0)&=-\frac{4\pi^2}{3L_0^2}\left(l(2k+l)
   \sin{\frac{2\pi(2k+l)}{3}} + (k^2-l^2)\sin{
   \frac{2\pi(l-k)}{3}}
   + k(k+2l)\sin{\frac{2\pi(k+2l)}{3}}\right)\\
   &=0.
\end{align*}
    Thus, we look at the second derivatives
    \begin{align*}
    \p^2_L F(\frac{\pi}{L_0}\frac{2k+l}{3}, L_0)&=-8(\frac{(2k+l)\pi}{3L_0})^2 \frac{l\pi}{L_0}
   \cos{2 \frac{(2k+l)\pi}{3}}\\
   +& (-\frac{(2k+l)\pi}{3L_0} + \frac{l\pi}{L_0})^2(\frac{(2k+l)\pi}{L_0}+\frac{l\pi}{L_0}) \cos{(
   -\frac{(2k+l)\pi}{3} + l\pi)}\\
   -& (\frac{(2k+l)\pi}{3L_0} + \frac{l\pi}{L_0})^2(\frac{(2k+l)\pi}{L_0}-\frac{l\pi}{L_0}) \cos{(\frac{(2k+l)\pi}{3} + l\pi)}\\
   &=-8\frac{l(2k+l)^2\pi^3}{9L_0^3}
   \cos{\frac{(4k+2l)\pi}{3}} + \frac{8(l-k)^2(k+l)\pi^3}{9L_0^3} \cos{(
   \frac{2(l-k)\pi}{3})}\\
   -&\frac{8k(k+2l)^2\pi^3}{9L_0^3}) \cos{(\frac{2(k+2l)\pi}{3})}\\
   &=\frac{8\pi^3}{L_0^3}(-l(2k+l)^2+ (k-l)^2(k+l) - k(k+2l)^2 )\\
   &=-\frac{8\pi^3kl(k+l)}{L_0^3},
   \end{align*}
\begin{align*}
   \p_{tL}F(\frac{\pi}{L_0}\frac{2k+l}{3}, L_0)&=-8\frac{(2k+l)\pi}{3} \frac{l\pi}{L_0}\cos{2 \frac{(2k+l)\pi}{3}}+(\frac{12(\frac{\pi}{L_0}\frac{2k+l}{3})^2}{\frac{l\pi}{L_0}}-4\frac{l\pi}{L_0})\sin{2 \frac{(2k+l)\pi}{3}}\\
   -&\frac{(l\pi-\frac{(2k+l)\pi}{3})(\frac{(2k+l)\pi}{L_0} + \frac{l\pi}{L_0})^2}{\frac{l\pi}{L_0}}\cos{(l\pi-\frac{(2k+l)\pi}{3})}\\
   -&\frac{2(-1+6(\frac{\pi}{L_0}\frac{2k+l}{3})^2 +6\frac{\pi}{L_0}\frac{2k+l}{3}\frac{l\pi}{L_0})}{\frac{l\pi}{L_0}} \sin{(l\pi-\frac{(2k+l)\pi}{3})}\\
   +&\frac{(l\pi+\frac{(2k+l)\pi}{3})(\frac{(2k+l)\pi}{L_0}-\frac{l\pi}{L_0})^2}{\frac{l\pi}{L_0}}\cos{(l\pi+\frac{(2k+l)\pi}{3})}\\
   -&\frac{2(1-6(\frac{\pi}{L_0}\frac{2k+l}{3})^2 +6\frac{\pi}{L_0}\frac{2k+l}{3}\frac{l\pi}{L_0})}{\frac{l\pi}{L_0}} \sin{((l\pi+\frac{(2k+l)\pi}{3}))}\\
   &=-8\frac{l(2k+l)\pi^2}{3L_0} \cos{ \frac{(4k+2l)\pi}{3}}+\frac{4l\pi}{L_0}(\frac{(2k+l)^2}{3l^2 }-1)\sin{ \frac{(4k+2l)\pi}{3}}\\
   -&\frac{8\pi^2(l-k)(k+l)^2}{3l L_0}\cos{\frac{2(l-k)\pi}{3}}
   +\frac{8\pi^2k^2(k+2l)}{3l L_0}\cos{\frac{(2k+4l)\pi}{3}}\\
   -&2\frac{2\pi(2k+l)^2+6l(2k+l)\pi-4\pi(k^2+kl+l^2)}{3l L_0} \sin{\frac{2(l-k)\pi}{3}}\\
   -&2\frac{-2\pi(2k+l)^2+6l(2k+l)\pi+4\pi(k^2+kl+l^2)}{3l L_0} \sin{\frac{(2k+4l)\pi}{3}}\\
   &=-\frac{8\pi^2}{3l L_0}\left(l^2(2k+l)+(l-k)(k+l)^2-k^2(k+2l)\right)\\
   &=\frac{8\pi^2(k-l)(k+2l)(2k+l)}{3lL_0}.
    \end{align*}
    Thus, we know that the Hessian $\nabla^2 F$ at the point $(\frac{\pi}{L_0}\frac{2k+l}{3}, L_0)$ is 
    \begin{equation*}
   \nabla^2 F(\frac{\pi}{L_0}\frac{2k+l}{3}, L_0)=\left(
   \begin{array}{cc}
        24\pi\frac{(k^2+kl)L_0}{l}&\frac{8\pi^2(k-l)(k+2l)(2k+l)}{3lL_0}  \\
        \frac{8\pi^2(k-l)(k+2l)(2k+l)}{3lL_0}& -\frac{8\pi^3kl(k+l)}{L_0^3}
   \end{array}
   \right).
    \end{equation*}
    \item In the case $k\not\equiv l\mod{3}$, we need to check that $\p^2_L F$ is not vanishing at the point $(\frac{\pi}{L_0}\frac{2k+l}{3}, L_0)$. We take $k=3m+1+l$ for example, where $m\in\Z$ with $m\geq0$.  By the computation above, we know that in this situation, 
    \begin{align*}
    \p^2_L F(\frac{\pi}{L_0}\frac{2k+l}{3}, L_0)&=-8\frac{l(2k+l)^2\pi^3}{9L_0^3}
   \cos{\frac{(4k+2l)\pi}{3}} + \frac{8(l-k)^2(k+l)\pi^3}{9L_0^3} \cos{(
   \frac{2(l-k)\pi}{3})}\\
   -&\frac{8k(k+2l)^2\pi^3}{9L_0^3}) \cos{(\frac{2(k+2l)\pi}{3})}\\
   &=-8\frac{l(2k+l)^2\pi^3}{9L_0^3}
   \cos{\frac{(6l+12m+4)\pi}{3}} + \frac{8(l-k)^2(k+l)\pi^3}{9L_0^3} \cos{(
   -\frac{2(3m+1)\pi}{3})}\\
   -&\frac{8k(k+2l)^2\pi^3}{9L_0^3}) \cos{(\frac{2(3m+3l+1)\pi}{3})}\\
   &=8\frac{l(2k+l)^2\pi^3}{9L_0^3}
   \cos{\frac{\pi}{3}} - \frac{8(l-k)^2(k+l)\pi^3}{9L_0^3} \cos{
   \frac{\pi}{3}}+\frac{8k(k+2l)^2\pi^3}{9L_0^3}) \cos{\frac{\pi}{3}}\\
   &=8\frac{l(2k+l)^2\pi^3}{9L_0^3}\cos{\frac{\pi}{3}}\neq0.
   \end{align*}
\end{enumerate}
\end{proof}
\subsubsection{Proof of Proposition \ref{prop: Index set M-E}}\label{sec: proof of corollary index-lables}
\begin{proof}[Proof of Proposition \ref{prop: Index set M-E}]
We consider case by case. Indeed, for the detailed classification information, we refer to Section \ref{sec: Sharp spectral analysis and classification of critical lengths}.
\begin{enumerate}
     \item \textit{Case 1: $k=l$.} In this case, the eigenvalues at the critical length $L_0$ include $0$ as an eigenvalue. Moreover, $N_0$ is odd and all eigenvalues $\{\ii\lambda_{c,j}\}$ are labeled with $j\in\{\frac{1-N_0}{2},\cdots,-1,0,1,\cdots, \frac{N_0-1}{2}\}$. By Proposition \ref{prop: Asymp in L}, we know that for every $\lambda_{c,j}$, there exist two different $\lambda_{\sigma^+(j)}$ and $\lambda_{\sigma^-(j)}$ such that
    \begin{equation*}
        \lambda_{\sigma^{\pm}(j)}(L)=\lambda_{c,j}-\frac{(k_j-l_j)(k_j+2l_j)(2k_j+l_j)}{2\pi (k_j^2+k_jl_j+l_j^2)^2}(L-L_0)\pm\frac{\left|L-L_0\right|}{\pi \sqrt{k_j^2+k_jl_j+l_j^2}}+\bigO((L-L_0)^2).
    \end{equation*}
    Hence, $ \lim_{L\rightarrow L_0}\lambda_{\sigma^+(j)}(L)=\lim_{L\rightarrow L_0}\lambda_{\sigma^-(j)}(L)=\lambda_{c,j}$. In particular, when $j=0$, $\lambda_{\sigma^{\pm}(0)}(L)=\pm\frac{2\left|L-L_0\right|}{3L_0}+\bigO((L-L_0)^2)$. We define $\sigma^+$ by $\sigma^+(j)=2j+1$, for $0\leq j\leq \frac{N_0-1}{2}$ and $\sigma^+(j)=2j$, $\frac{1-N_0}{2}\leq j\leq-1$. We define $\sigma^-$ by $\sigma^-(j)=-\sigma^+(-j)$.
    Then, $\mathcal{M}_E=\{-N_0,\cdots,-1,1,\cdots,N_0\}$ and $\#\mathcal{M}_E=2N_0<\infty$.

    \item \textit{Case 2: $k\equiv l\mod{3}$ but $k\neq l$.} 
     In this case, $0$ is NOT an eigenvalue at the critical length $L_0$. Thus, $N_0$ is even and all eigenvalues $\{\ii\lambda_{c,j}\}$ are listed with $0<|j|\leq \frac{N_0}{2}$. By Proposition \ref{prop: Asymp in L}, we know that for every $\lambda_{c,j}$, there exist two different $\lambda_{\sigma^+(j)}$ and $\lambda_{\sigma^-(j)}$ such that $\lambda_{\sigma^{\pm}(j)}(L)=\lambda_{c,j}-\frac{(k_j-l_j)(k_j+2l_j)(2k_j+l_j)}{2\pi (k_j^2+k_jl_j+l_j^2)^2}(L-L_0)\pm\frac{\left|L-L_0\right|}{\pi \sqrt{k_j^2+k_jl_j+l_j^2}}+\bigO((L-L_0)^2)$. 
    Hence, $ \lim_{L\rightarrow L_0}\lambda_{\sigma^+(j)}(L)=\lim_{L\rightarrow L_0}\lambda_{\sigma^-(j)}(L)=\lambda_{c,j}$. We define $\sigma^+$ by $\sigma^+(j)=2j$, for $1\leq j\leq \frac{N_0}{2}$ and $\sigma^+(j)=2j+1$, $\frac{-N_0}{2}\leq j\leq-1$. We define $\sigma^-$ by $\sigma^-(j)=-\sigma^+(-j)$. Then, $\mathcal{M}_E=\{-N_0,\cdots,-1,1.\cdots,N_0\}$ and $\#\mathcal{M}_E=2N_0<\infty$. 
    \item \textit{Case 3: $k_1\not\equiv l_1\mod{3}$.} 
    In this case, $0$ is NOT an eigenvalue at the critical length $L_0$. Thus, $N_0$ is even and all eigenvalues $\{\ii\lambda_{c,j}\}$ are listed with $0<|j|\leq \frac{N_0}{2}$. By Proposition \ref{prop: Asymp in L}, we know that for every $\lambda_{c,j}$, there exists a unique $\lambda_{j}$ such that $\lambda_j(L)=\lambda_{c,j}-\frac{l(k+l)(2k+l)^2}{27kL_0^5}(L-L_0)^2+\bigO((L-L_0)^3)$.    Hence, $ \lim_{L\rightarrow L_0}\lambda_{j}(L)=\lambda_{c,j}$. Thus, $\mathcal{M}_E=\{-\frac{N_0}{2},\cdots,-1,1.\cdots,\frac{N_0}{2}\}$ and  $\#\mathcal{M}_E=N_0<\infty$.
\end{enumerate}
\end{proof}

\subsubsection{Remainders of the proof of Proposition \ref{prop: Low-frequency behaviors: uniform estimates}}\label{sec: Remainders of the proof-uniform-low}
\begin{proof}[Proof of Proposition \ref{prop: Low-frequency behaviors: uniform estimates}]
Before we begin our proof, we first notice that for any $L\in I\setminus\{L_0\}$, $\frac{L_0}{2}< L< 2L_0$ since that $0<\delta< \frac{L_0}{2}$. 
We argue by contradiction. Suppose that there exists $j_0\notin \mathcal{M}_E$ and $|\lambda_{j_0}|< K$ such that $|\E'_{j_0}(0)|=|\E'_{j_0}(L)|\leq\gamma$. For simplicity, we denote by $\Lambda=\lambda_{j_0}$ in this proof. We first extend the function $\E_{j_0}$ trivially past the endpoints of the interval $[0,L]$, we obtain a function 
$f(x)=\E_{j_0}(x)$, for $x\in[0,L]$, and $f(x)=0$, for $x\notin[0,L]$. We know that $f\in H^{\frac{3}{2}^-}(\R)$ and satisfies the equation for all $x\in\R$
\begin{equation}
    f'''+f'+\ii\Lambda f=2\E_{j_0}''(0)\delta_0-2\E_{j_0}''(L)\delta_L+\E_{j_0}'(0)\delta'_0-\E_{j_0}'(L)\delta'_L.
\end{equation}
Then, the extended function $f$, via Fourier transform, further satisfies the equation:
\begin{equation}
\Hat{f}(\xi)\cdot((\ii\xi)^3+\ii \xi+\ii\Lambda)=2\alpha-2\beta e^{-\ii\xi L}+\ii\theta\xi(1-e^{-\ii\xi L}),
\end{equation}
where $\alpha=\E_{j_0}''(0),\beta=\E_{j_0}''(L),\theta=\E_{j_0}'(0)=\E_{j_0}'(L)$. In addition, we know that $|\theta|\leq \gamma$. By Palay-Wiener theorem, it is easy to see that $\Hat{f}$ is a holomorphic function when we extend $\xi$ to complex values, as $f$ is compactly supported in $\R$. Therefore, away from the zeros of the polynomial $(\ii\xi)^3+\ii \xi+\ii\Lambda$, we have the following expression
\begin{equation}
    \Hat{f}(\xi)=\ii\frac{2\alpha-2\beta e^{-\ii\xi L}+\ii\theta\xi(1-e^{-\ii\xi L})}{\xi^3-\xi-\Lambda}.
\end{equation}
We claim that $(\alpha,\beta)\neq(0,0)$. In fact, suppose that $\alpha=\beta=0$, then $\ii\frac{\ii\theta\xi(1-e^{-\ii\xi L})}{\xi^3-\xi-\Lambda}\in L^2(\R)$. This implies that the quotient is an entire function. However, we know that $|\Lambda|< K$, all the roots of this polynomial $\xi^3-\xi-\Lambda$ lie in a disc of radius $R=R(K):=(1+\frac{3}{2}K)^{\frac{1}{3}}>1$ in the complex plane centered at the origin:
\begin{equation*}
|\xi|^3=|\xi+\Lambda|<K+|\xi|<K+\frac{1}{3}|\xi|^3+\frac{2}{3}.    
\end{equation*}
Let $\eta\in D_{2R}$ and $\Gamma_{3R}=\p D_{3R}$, by Cauchy's integral formula, we obtain
\begin{align*}
    |\Hat{f}(\eta)|&=|\ii\frac{\ii\theta\eta(1-e^{-\ii\eta L})}{\eta^3-\eta-\Lambda}|=|\frac{1}{2\pi\ii}\int_{\Gamma_{3R}}\ii\frac{\ii\theta\zeta(1-e^{-\ii\zeta L})}{(\zeta^3-\zeta-\Lambda)(\zeta-\eta)}d\zeta|.
\end{align*}
On the one hand, for both the numerator and the denominator, we have the following estimates:
\begin{align*}
  |\theta\zeta(1-e^{-\ii\zeta L})|&\leq (1+e^{3R L})|\zeta|\gamma\leq 3 R(1+e^{3R L})\gamma ,\forall\zeta\in \p D_{3R}\\
  |(\zeta^3-\zeta-\Lambda)(\zeta-\eta)|&\geq R(26K+\frac{52}{3})=\frac{52}{3}R^4>17R^4,\forall\zeta\in \p D_{3R},\eta\in D_{2R}.
\end{align*}
Hence, we know that $|\Hat{f}(\eta)|\leq |\frac{3 R(1+e^{3R L})\gamma}{17R^4}|\cdot 3R\leq \frac{9(1+e^{3R L})\gamma}{17R^2}$. On the other hand, for $\xi\in (D_{2R})^c$, we have the following estimates
\begin{align*}
    |\xi^3-\xi-\Lambda|\geq \frac{2}{3}|\xi|^3-\frac{1}{3}|\xi|\geq \frac{16}{3}R^3-\frac{2}{3}R>4R^3,\;
    \frac{|\xi|}{|\xi^3-\xi-\Lambda|}\leq \frac{|\xi|}{\frac{2}{3}|\xi|^3-\frac{1}{3}|\xi|}<\frac{2}{|\xi|^2}
\end{align*}
thus $|\Hat{f}(\xi)|\leq |\ii\frac{\ii\theta\xi(1-e^{-\ii\xi L})}{\xi^3-\xi-\Lambda}|<\frac{2\gamma(1+e^{3R L})}{|\xi|^2}$. As a consequence,
\begin{align*}
    \int_{\R}|\Hat{f}(\xi)|^2d\xi&=\int_{|\xi|>2R}|\Hat{f}(\xi)|^2d\xi+\int_{|\xi|\leq 2R}|\Hat{f}(\xi)|^2d\xi\\
    &\leq \int_{|\xi|>2R}(\frac{2\gamma(1+e^{3R L})}{|\xi|^2})^2d\xi+\int_{|\xi|\leq 2R} \frac{81(1+e^{3R L})^2\gamma^2}{17^2R^4}d\xi\\
    &\leq \left(\frac{(1+e^{3R L})^2}{3R^3}+4R\cdot\frac{81(1+e^{3R L})^2}{17^2R^4}\right)\gamma^2
    \leq \frac{2(1+e^{3R L})^2\gamma^2}{3R^3}.
\end{align*}
Hence, provided $\gamma< \frac{\sqrt{3}R^{\frac{3}{2}}}{\sqrt{2}(1+e^{6R L_0})}$ small enough, we know that $\int_{\R}|\Hat{f}(\xi)|^2d\xi<1$, which contradicts to our assumption that $\|f\|_{L^2}=1$. Moreover, by a simple variation of the proceeding argument, there exists a constant $\gamma_*=\gamma_*(K)$ such that $|\alpha|+|\beta|\geq \gamma_*$, which is forced by the normalized condition of $f$. Indeed, $\forall\eta\in D_{2R}$, we obtain the following estimate
\begin{align*}
  |\Hat{f}(\eta)|
    =|\frac{1}{2\pi}\int_{\Gamma_{3R}}\frac{2\alpha-2\beta e^{-\ii\zeta L}+\ii\theta\zeta(1-e^{-\ii\zeta L})}{(\zeta^3-\zeta-\Lambda)(\zeta-\eta)}d\zeta|
    \leq (\frac{9}{17R^2}+\frac{9R e^{3R L}}{17R^3})\gamma+\frac{3}{2R^3}(|\alpha|+e^{3L R}|\beta|).
\end{align*}
And for $\xi\in (D_{2R})^c\cap\R$, we have $ |\Hat{f}(\xi)|=\frac{|2\alpha-2\beta e^{-\ii\xi L}+\ii\theta\xi(1-e^{-\ii\xi L})|}{|\xi^3-\xi-\Lambda|}<\frac{4\gamma}{|\xi|^2}+\frac{6(|\alpha|+|\beta|)}{2|\xi|^3-|\xi|}<\frac{4\gamma}{|\xi|^2}+\frac{6(|\alpha|+|\beta|)}{|\xi|^3}$. 
Therefore, 
\begin{align*}
    \int_{\R}|\Hat{f}(\xi)|^2d\xi&=\int_{|\xi|>2R}|\Hat{f}(\xi)|^2d\xi+\int_{|\xi|\leq 2R}|\Hat{f}(\xi)|^2d\xi\\
     &\leq \int_{|\xi|>2R}\left(\frac{32\gamma^2}{|\xi|^4}+\frac{72(|\alpha|+|\beta|)^2}{|\xi|^6}\right)d\xi+\int_{|\xi|\leq 2R}\left(\frac{162(1+e^{3R L})^2\gamma^2}{289R^4}+\frac{9}{2R^6}(|\alpha|+e^{3L R}|\beta|)^2\right)d\xi\\
     &\leq \frac{3(3+2e^{6R L})\gamma^2}{R^3}+\frac{1+18e^{6L R}}{R^5}(|\alpha|+|\beta|)^2.
\end{align*}
Let $\gamma_*=\frac{R^{\frac{5}{2}}}{(1+18e^{12 L_0 R})^{\frac{1}{2}}}$. Thus, provided that $\gamma<\frac{(2\pi-1)R^{\frac{3}{2}}}{\sqrt{3(3+2e^{12R L_0})}}$, 
\begin{equation*}
\frac{1+18e^{12L_0 R}}{R^5}(|\alpha|+|\beta|)^2\geq \frac{1+18e^{6L R}}{R^5}(|\alpha|+|\beta|)^2
\geq\int_{\R}|\Hat{f}(\xi)|^2d\xi-\frac{3(3+2e^{6R L})\gamma^2}{R^3} >2\pi-(2\pi-1)=1.
\end{equation*}
 This implies that $|\alpha|+|\beta|>\gamma_*$. We define another constant $\varepsilon_*=\varepsilon_*(K)=\frac{1}{2}e^{-\frac{2}{L_0}\sqrt{\frac{3}{4}(1+3/2K)^{\frac{2}{3}}-1}}$.
In addition, we assume that $\varepsilon_*|\beta|\leq |\alpha| \leq \varepsilon_*^{-1}|\beta|$ and $\min\{|\alpha|,|\beta|\}\geq \frac{\varepsilon_*}{\varepsilon_*+1}\gamma_*$. Otherwise, either 
\begin{equation*}
    |\alpha|<\varepsilon_*|\beta|, \text{ with }|\beta|\geq \frac{1}{1+\varepsilon_*}\gamma_*,\;
\text{or }|\alpha| >\varepsilon_*^{-1}|\beta|, \text{ with }|\alpha|\geq \frac{1}{1+\varepsilon_*}\gamma_*.
\end{equation*}
If we are in the first case, the zeros of the numerator that lie in $D_R$ satisfy that $2\alpha-2\beta e^{-\ii\xi L}+\ii\theta\xi(1-e^{-\ii\xi L})=0$. 
This implies that $2|\beta|e^{L\Im{\xi}}\leq 2|\alpha|+\gamma R(1+e^{\Im{\xi} L})$.
Thus, we obtain 
\begin{equation*}
e^{L\Im{\xi}}\leq \frac{2|\alpha|+\gamma R}{2|\beta|-\gamma R}\leq \frac{2\varepsilon_*|\beta|+\gamma R}{2|\beta|-\gamma R}\leq\frac{2\varepsilon_*\gamma_*+(1+\varepsilon_*)\gamma R}{2\gamma_*-(1+\varepsilon_*)\gamma R}.   
\end{equation*}
While for the latter case, we know that $2|\alpha| \leq 2|\beta|e^{L\Im{\xi}}+\gamma R(1+e^{\Im{\xi} L})$. This implies that
\begin{equation*}
e^{L\Im{\xi}}\geq \frac{2|\alpha|-\gamma R}{2|\beta|+\gamma R}\geq \frac{2|\alpha|-\gamma R}{2\varepsilon_*|\alpha|+\gamma R}\geq \frac{2\gamma_*-(1+\varepsilon_*)\gamma R}{2\varepsilon_*\gamma_*+(1+\varepsilon_*)\gamma R} .   
\end{equation*}
Provided that $\gamma$ satisfies that $(1+\varepsilon_*)\gamma R<\frac{\gamma_*}{1+\varepsilon_*}$, we get the estimates 
\begin{equation}\label{eq: geq-Im-xi}
    |\Im{\xi}|\geq \frac{2}{L_0}\ln{\frac{2\gamma_*-(1+\varepsilon_*)\gamma R}{2\varepsilon_*\gamma_*+(1+\varepsilon_*)\gamma R}}\geq \frac{2}{L_0}\ln{\frac{1+2\varepsilon_*}{1+2\varepsilon_*+2\varepsilon_*^2}}.
\end{equation}
On the other hand, we turn to the zeros of the denominator, which all lie in $D_R$, $\xi^3-\xi-\Lambda=0$. Suppose that $\xi_0$ is a real root of $\xi^3-\xi-\Lambda=0$, and $|\Lambda|>\frac{2\sqrt{3}}{9}$, we know that all three roots are as follows
\begin{equation*}
    \xi_1=\xi_0,\xi_2=-\frac{1}{2}\xi_0+\ii\sqrt{\frac{3}{4}\xi_0^2-1},\xi_3=-\frac{1}{2}\xi_0-\ii\sqrt{\frac{3}{4}\xi_0^2-1}.
\end{equation*}
Thus, $|\Im{\xi}|=\sqrt{\frac{3}{4}\xi_0^2-1}\leq \sqrt{\frac{3}{4}(1+\frac{3}{2}K)^{\frac{2}{3}}-1}=\frac{2}{L_0}\ln{\frac{1}{2\varepsilon_*}}$. This is a contradiction to the estimates \eqref{eq: geq-Im-xi}. Thanks to the fact that $\varepsilon_*|\beta|\leq |\alpha|\leq \varepsilon_*^{-1}|\beta|$, all the roots of $2\alpha-2\beta e^{-\ii L\xi}$ are of the form:
\begin{equation*}
    \mu_0+\frac{2\pi n}{L},\text{ with }|\Re{\mu_0}|\leq \frac{\pi}{L},\Im{\mu_0}\leq \frac{1}{L}\ln{\frac{1}{\varepsilon_*}}.
\end{equation*}
Furthermore, if $\mu$ is such a zero of $2\alpha-2\beta e^{-\ii L\xi}$ that is in $D_R$, we pick a circle $\Gamma_r(\mu)$ in the complex plane centered at $\mu$ with radius $r\leq\min\{\frac{\pi}{L},R\}$. We claim that under certain conditions, which will be chosen later, there is only one solution of $2\alpha-2\beta e^{-\ii\xi L}+\ii\theta\xi(1-e^{-\ii\xi L})$ that lies in the domain $[-\frac{\pi}{8L}+\Re{\mu},\frac{\pi}{L}+\Re{\mu})\times \R$, in particular this root lies inside $\Gamma_r(\mu)$. \par
For any $\xi\in\Gamma_r(\mu)$, we have
\begin{align*}
|2\alpha-2\beta e^{-\ii L\xi}|^2&=|(2\alpha-2\beta e^{-\ii L\xi})-(2\alpha-2\beta e^{-\ii L\mu})|^2\\
&=|2\beta e^{-\ii L\mu}(e^{-\ii L\xi_r}-1)|^2\\
&=4|\alpha|^2\left((e^{Lrb}\cos{ra L}-1)^2+(e^{Lrb}\sin{raL})^2\right),
\end{align*}
where $\xi_r=\xi-\mu=r(a+\ii b)$ with $a^2+b^2=1$. 
\begin{enumerate}
    \item If $|a|\geq\frac{1}{8}$, we know that $\frac{Lr}{8}\leq |aLr|\leq \frac{\pi}{8}$. Thus, we obtain
$(e^{Lrb}\sin{raL})^2\geq (e^{-Lr}\frac{|raL|}{2})^2\geq (\frac{1}{3}\cdot\frac{Lr}{16})^2\geq (\frac{Lr}{48})^2$.
    \item If $|a|<\frac{1}{8}$, it is easy to see that $|b|\geq\frac{7}{8}$. If $b\leq-\frac{7}{8}$, then $(e^{Lrb}\cos{ra L}-1)^2\geq (1-e^{-\frac{7Lr}{8}})^2\geq (\frac{Lr}{48})^2$. Else $b\geq \frac{7}{8}$, $(e^{Lrb}\cos{ra L}-1)^2\geq  (\frac{Lr}{2})^2$
\end{enumerate}
Therefore, we know that $|2\alpha-2\beta e^{-\ii L\xi}|\geq 2|\alpha|\frac{Lr}{48}\geq \frac{\varepsilon_*\gamma_*Lr}{24(1+\varepsilon_*)},\forall\xi\in\Gamma_r(\mu)$. Provided that $\gamma R(1+e^{2L_0R})<\frac{\varepsilon_*\gamma_*L_0r}{96(1+\varepsilon_*)}$, which implies that $\gamma R(1+e^{LR})\leq \frac{\varepsilon_*\gamma_*Lr}{48(1+\varepsilon_*)},\forall L\in(L_0-\delta,L_0+\delta)$, we obtain
\begin{equation}\label{eq: no root at the circle}
|2\alpha-2\beta e^{-\ii\xi L}+\ii\theta\xi(1-e^{-\ii\xi L})|\geq    \frac{\varepsilon_*\gamma_*Lr}{48(1+\varepsilon_*)},  \forall\xi\in\Gamma_r(\mu).  
\end{equation}
Moreover, under the condition above, we claim that there is no root in $[-\frac{\pi}{8L}+\Re{\mu},\frac{\pi}{L}+\Re{\mu})\times \R\setminus D_r(\mu)$. In fact, for any $\xi\in D_R\cap [-\frac{\pi}{8L}+\Re{\mu},\frac{\pi}{L}+\Re{\mu})$, we have the following two cases.
\begin{enumerate}
    \item If $|\Im(\xi-\mu)|\geq \frac{r}{2}$, then 
$|2\alpha-2\beta e^{-\ii L\xi}|=2|\alpha||e^{-\ii L(\xi-\mu)}-1|\geq \frac{\varepsilon_*\gamma_*Lr}{24(1+\varepsilon_*)}$.
    \item If $|\Re(\xi-\mu)|\geq \frac{r}{2}$ and $|\Im(\xi-\mu)|< \frac{r}{2}$, then we know that 
    \begin{equation*}
     |2\alpha-2\beta e^{-\ii L\xi}|=2|\alpha||e^{ L\Im(\xi-\mu)}\sin{(\Re(\xi-\mu)L)}|\geq \frac{\varepsilon_*\gamma_*Lr}{24(1+\varepsilon_*)}.
    \end{equation*}
\end{enumerate}
Next, we aim to show that, shrinking the upper bound of $\gamma$ if necessary, there is exactly one root inside $\Gamma_r(\mu)$. As shown in \eqref{eq: no root at the circle}, there is no solution on $\Gamma_r(\mu)$. Therefore, the number of solutions (counting multiplicity) inside $\Gamma_r(\mu)$ is given by
\begin{equation*}
    \frac{1}{2\pi\ii}\int_{\Gamma_r(\mu)}\frac{\left(2\alpha-2\beta e^{-\ii\xi L}+\ii\theta\xi(1-e^{-\ii\xi L})\right)'}{2\alpha-2\beta e^{-\ii\xi L}+\ii\theta\xi(1-e^{-\ii\xi L})}d\xi=\frac{1}{2\pi\ii}\int_{\Gamma_r(\mu)}\frac{2\ii L \beta e^{-\ii\xi L}+\ii\theta(1-e^{-\ii\xi L})-\theta\xi L e^{-\ii\xi L}}{2\alpha-2\beta e^{-\ii\xi L}+\ii\theta\xi(1-e^{-\ii\xi L})}d\xi.
\end{equation*}
Since $\mu$ is the only solution for $2\alpha-2\beta e^{-\ii L\xi}=0$, we know that
\begin{equation*}
    1=\frac{1}{2\pi\ii}\int_{\Gamma_r(\mu)}\frac{\left(2\alpha-2\beta e^{-\ii\xi L}\right)'}{2\alpha-2\beta e^{-\ii\xi L}}d\xi=\frac{1}{2\pi\ii}\int_{\Gamma_r(\mu)}\frac{2\ii L \beta e^{-\ii\xi L}}{2\alpha-2\beta e^{-\ii\xi L}}d\xi.
\end{equation*}
Define an integer $N_r$ by 
\begin{equation*}
   N_r:= \left|\frac{1}{2\pi\ii}\int_{\Gamma_r(\mu)}\frac{2\ii L \beta e^{-\ii\xi L}+\ii\theta(1-e^{-\ii\xi L})-\theta\xi L e^{-\ii\xi L}}{2\alpha-2\beta e^{-\ii\xi L}+\ii\theta\xi(1-e^{-\ii\xi L})}d\xi-\frac{1}{2\pi\ii}\int_{\Gamma_r(\mu)}\frac{2\ii L \beta e^{-\ii\xi L}}{2\alpha-2\beta e^{-\ii\xi L}}d\xi\right|.
\end{equation*}
It suffices to find a sufficient condition such that $N_r<1$.
\begin{align*}
N_r&=\left|\frac{1}{2\pi\ii}\int_{\Gamma_r(\mu)}\frac{2 L \beta \theta\xi e^{-\ii\xi L}(1-e^{-\ii\xi L})}{\left(2\alpha-2\beta e^{-\ii\xi L}+\ii\theta\xi(1-e^{-\ii\xi L})\right)\left(2\alpha-2\beta e^{-\ii\xi L}\right)}d\xi+\frac{1}{2\pi\ii}\int_{\Gamma_r(\mu)}\frac{\ii\theta(1-e^{-\ii\xi L})-\theta\xi L e^{-\ii\xi L}}{2\alpha-2\beta e^{-\ii\xi L}}d\xi\right|\\
&\leq \frac{1}{2\pi}\int_{\Gamma_r(\mu)}\left|\frac{2 L \beta \theta\xi e^{-\ii\xi L}(1-e^{-\ii\xi L})}{\left(2\alpha-2\beta e^{-\ii\xi L}+\ii\theta\xi(1-e^{-\ii\xi L})\right)\left(2\alpha-2\beta e^{-\ii\xi L}\right)}\right|d\xi+\frac{1}{2\pi}\int_{\Gamma_r(\mu)}\left|\frac{\ii\theta(1-e^{-\ii\xi L})-\theta\xi L e^{-\ii\xi L}}{2\alpha-2\beta e^{-\ii\xi L}}\right|d\xi\\
&\leq r\left(\frac{2 L|\beta| \gamma R e^{RL}(1+e^{R L})}{\frac{\varepsilon_*\gamma_*Lr}{48(1+\varepsilon_*)}\cdot\frac{|\alpha|Lr}{24}}+\frac{\gamma(1+e^{R L})+\gamma R L e^{R L}}{\frac{\varepsilon_*\gamma_*L}{24(1+\varepsilon_*)}}\right)\\
&\leq 48(1+\varepsilon_*)\gamma\left(\frac{96 R e^{2RL_0}(1+e^{2R L_0})}{r\varepsilon_*^2\gamma_*L_0}+\frac{1+e^{2R L_0}+ 2R L_0 e^{2R L_0}}{\varepsilon_*\gamma_*L_0}\right).
\end{align*}
Thus, it suffices to require that $\gamma$ satisfies $48(1+\varepsilon_*)\gamma\left(\frac{96 R e^{2RL_0}(1+e^{2R L_0})}{r\varepsilon_*^2\gamma_*L_0}+\frac{1+e^{2R L_0}+ 2R L_0 e^{2R L_0}}{\varepsilon_*\gamma_*L_0}\right)<1$.
Hence, we conclude that all zeros of the numerator $2\alpha-2\beta e^{-\ii\xi L}+\ii\theta\xi(1-e^{-\ii\xi L})$ in $D_R$ are of the form $\mu_0+\frac{2k\pi}{L}+\bigO(r),k\in\Z$, where $|\mu_0|\leq \frac{2\pi}{L}$. Because all zeros of the denominator should also be the zeros of the numerator, assuming that $\xi_0^3-\xi_0=\Lambda$, we know that the roots of the denominator are of the form
\begin{equation*}
     \xi_1=\xi_0,
    \xi_2=\xi_0+\frac{2k\pi}{L}+2\bigO(r),
    \xi_3=\xi_0+\frac{2(k+l)\pi}{L}+2\bigO(r),
\end{equation*}
where $k,l$ are positive integers. 
Since $\xi_1+\xi_2+\xi_3=0,\xi_1\xi_2+\xi_2\xi_3+\xi_3\xi_1=-1,\xi_1\xi_2\xi_3=-\Lambda$, therefore, we obtain $3\xi_0+(2k+l)\frac{2\pi}{L}+4\bigO(r)=0$.
Therefore, we know that $|\xi_0+(2k+l)\frac{2\pi}{3L}|\leq \frac{4}{3}r$. Since $j\notin \mathcal{M}_E$, we know that $\min_{k,l}|\xi_0+(2k+l)\frac{2\pi}{3L}|>0$.
In particular, if we require that $\frac{4}{3}r< \min_{k,l}|\xi_0+(2k+l)\frac{2\pi}{3L}|$. Then, we obtain a contradiction, which implies that there exists a constant $\gamma=\gamma(K)>0$ such that for any eigenfunction $\E_j$ associated to $\lambda_j$ satisfying $|\lambda_j|\leq K$ and $j\notin\mathcal{M}_E$, we obtain $|\E_j'(0)|= |\E_j'(L)|\geq \gamma$.
\end{proof}
\subsubsection{Rotation structure of Type 1 and Type 2 eigenfunctions}\label{sec: rotation structure of eigenmodes}
In this part, we include some computation details in Remark \ref{rem: rotation of eigenmodes}. Recall that
\begin{equation*}
\begin{pmatrix}
\E_j^+\\
\E_j^-
\end{pmatrix}=
\begin{pmatrix}
C^+_1&C^+_2\\
C^-_1&C^-_2
\end{pmatrix}
\begin{pmatrix}
\G_j\\
G_j
\end{pmatrix}+\bigO(|L-L_0|),
\end{equation*}
\begin{gather*}
C_{1}^{\pm}:=-\frac{-2\pi^2(k^2 + 4kl + l^2) \pm \sqrt{3}\pi L_0(k-l)}{\sqrt{3} L_0 \sqrt{6L_0^2 \pm 2\sqrt{3}\pi L_0(k-l)}},\;\;C^{\pm}_{2}:=- \frac{\pi(k+l)(2\pi(k-l) \pm \sqrt{3}L_0)}{L_0 \sqrt{6L_0^2 \pm 2\sqrt{3}\pi L_0(k-l)}}.
\end{gather*}
First, we show that the coefficient matrix $C_{Rot}:=\begin{pmatrix}
C^+_1&C^+_2\\
C^-_1&C^-_2
\end{pmatrix}\in O(2)$. 
\begin{lem}
$C_{Rot}$ is an orthogonal matrix.
\end{lem}
\begin{proof}
At first glance, we could expand $L_0=2\pi\sqrt{\frac{k^2+kl+l^2}{3}}$, then check that $(C^+_1)^2+(C^+_2)^2=(C^-_1)^2+(C^-_2)^2=1$ and $C^+_1C^-_1+C^-_2C^-_2=0$. Let us first check the orthogonality. 
\begin{align*}
C^+_1C^-_1&=\frac{-2\pi^2(k^2 + 4kl + l^2) + \sqrt{3}\pi L_0(k-l)}{\sqrt{3} L_0 \sqrt{6L_0^2 +2\sqrt{3}\pi L_0(k-l)}}\frac{-2\pi^2(k^2 + 4kl + l^2) -\sqrt{3}\pi L_0(k-l)}{\sqrt{3} L_0 \sqrt{6L_0^2- 2\sqrt{3}\pi L_0(k-l)}}\\
&=\frac{4\pi^4(k^2 + 4k l + l^2)^2 - 3\pi^2 L_0^2(k-l)^2}{3 L_0^2 \sqrt{36L_0^4 -12\pi^2 L_0^2(k-l)^2}},\\
C^+_2C^-_2&=\frac{\pi(k+l)(2\pi(k-l) + \sqrt{3}L_0)}{L_0 \sqrt{6L_0^2+ 2\sqrt{3}\pi L_0(k-l)}}\frac{\pi(k+l)(2\pi(k-l)-\sqrt{3}L_0)}{L_0 \sqrt{6L_0^2-2\sqrt{3}\pi L_0(k-l)}}\\
&=\frac{\pi^2(k+l)^2(4\pi^2(k-l)^2- 3L_0^2)}{L_0^2 \sqrt{36L_0^4-12 \pi^2 L_0^2(k-l)^2}}.
\end{align*}
It suffices to prove that $4\pi^4(k^2 + 4k l + l^2)^2 - 3\pi^2 L_0^2(k-l)^2+3\pi^2(k+l)^2(4\pi^2(k-l)^2- 3L_0^2)=0$. We directly compute this term, using that $3L_0^2=4\pi^2(k^2 + k l + l^2)$:
\begin{align*}
4\pi^4(k^2 + 4k l + l^2)^2 - 3\pi^2 L_0^2(k-l)^2+3\pi^2(k+l)^2(4\pi^2(k-l)^2- 3L_0^2)\\
=4\pi^4\left((k^2 + 4k l + l^2)^2+3(k^2-l^2)^2\right)-3\pi^2L_0^2((k-l)^2+3(k+l)^2)\\
=16\pi^4(k^2+l^2+kl)^2-12\pi^2L_0^2(k^2+l^2+kl)\\
=9L_0^4-3L_0^2\times3L_0^2=0.
\end{align*}
Then, we check the norm of the first row. Let 
\begin{align*}
D=\sqrt{3} L_0 \sqrt{6L_0^2 + 2\sqrt{3}\pi L_0(k-l)},\\
N_1=-2\pi^2(k^2 + 4kl + l^2)+ \sqrt{3}\pi L_0(k-l),\\
N_2=\sqrt{3}\pi(k+l)(2\pi(k-l) + \sqrt{3}L_0).
\end{align*}
It suffices to check that $N_1^2+N_2^2=D^2$. For the term involving $\pi^4$, we have
\begin{equation*}
4\pi^4 \left((k^2+4kl+l^2)^2 + 3(k^2-l^2)^2 \right)=9L_0^4.
\end{equation*}
For $\pi^2$ term: $3\pi^2 L_0^2 \left( (k-l)^2 + 3(k+l)^2 \right)= 12\pi^2 L_0^2(k^2+kl+l^2) = 9L_0^4$. For the mixing term involving $\pi^3$, we have
\begin{align*}
-4\sqrt{3}\pi^3L_0(k-l)(k^2 + 4kl + l^2)+ 12\sqrt{3}\pi^3L_0(k-l)(k+l)^2\\
=-4\sqrt{3}\pi^3L_0(k-l)\left(k^2 + 4kl + l^2-3(k+l)^2\right)\\
=8\sqrt{3}\pi^3L_0(k-l)(k^2+kl+l^2)\\
=6\sqrt{3}\pi L_0^3(k-l).
\end{align*}
Therefore, we obtain $N_1^2+N_2^2=18L_0^4+6\sqrt{3}\pi L_0^3(k-l)=D^2$. Similarly, $(C^-_1)^2+(C^-_2)^2=1$.
\end{proof}
Then, we have the following result.
\begin{prop}
$C_{Rot}=\begin{pmatrix}
-\cos\left(\frac{3\theta}{2}\right) & -\sin\left(\frac{3\theta}{2}\right) \\
\sin\left(\frac{3\theta}{2}\right) & -\cos\left(\frac{3\theta}{2}\right)
\end{pmatrix}$ is a rotation matrix where $\cos{\theta}=\frac{\pi(k-l)}{\sqrt{3}L_0}=\poscals{\G_j}{\widetilde{\G}_j}_{L^2(0,L_0)}$. 
\end{prop}
\begin{proof}
Using this $\theta$, we simplify the notation of $C^+_1$. By $3L_0^2=4\pi^2(k^2 + k l + l^2)$, we know that
\begin{equation*}
2\pi^2(k^2+4kl+l^2) = 4\pi^2\left(\frac{3L_0^2}{4\pi^2}\right) - 2\pi^2\left(\frac{3L_0^2}{\pi^2}\cos^2\theta\right) = 3L_0^2(1 - 2\cos^2\theta) = -3L_0^2\cos{2\theta}
\end{equation*}
For the denominator, we have
\begin{equation*}
\sqrt{6L_0^2+ 2\sqrt{3}\pi L_0(k-l)} = \sqrt{6L_0^2(1 +\cos\theta)} = \sqrt{12}L_0 
\cos\left(\frac{\theta}{2}\right).
\end{equation*}
So we deduce that 
\begin{equation*}
C_1^+ = -\frac{3L_0^2\cos(2\theta)+ 3L_0^2\cos\theta}{\sqrt{3}L_0 \times \sqrt{12}L_0 \times \cos(\frac{\theta}{2})} = -\frac{\cos(2\theta) + \cos\theta}{2 \cos(\frac{\theta}{2})}=-\frac{2\cos(\frac{3\theta}{2})\cos(\frac{\theta}{2})}{2\cos(\frac{\theta}{2})} = -\cos\left(\frac{3\theta}{2}\right)
\end{equation*}
All three other terms can be treated similarly. Then we conclude
$C_{Rot}=\begin{pmatrix}
-\cos\left(\frac{3\theta}{2}\right) & -\sin\left(\frac{3\theta}{2}\right) \\
\sin\left(\frac{3\theta}{2}\right) & -\cos\left(\frac{3\theta}{2}\right)
\end{pmatrix}$ is a rotation matrix where $\cos{\theta}=\frac{\pi(k-l)}{\sqrt{3}L_0}=\poscals{\G_j}{\widetilde{\G}_j}_{L^2(0,L_0)}$. 
\end{proof}

\subsubsection{Growth bounds for $C^0-$semigroup}
In order to consider the spectral localization properties of the operator $\A$, we give the definition of growth bound and essential growth bound of the infinitesimal generator of a linear $C^0-$semigroup. We refer to \cite[Definition 4.15]{Webb-book} or \cite[Definition 2.1]{Chu-Coron-Shang} for this definition.
\begin{defi}\label{defi: growth bounds}
Let $K:D(K)\subset X\rightarrow X$ be the infinitesimal generator of a linear $C^0-$semigroup $\{S_K(t)\}_{t\geq0}$ on a Banach space $X$. We define $\omega_0(K)\in[-\infty,+\infty)$ the growth bound of $K$ by
\begin{equation*}
\omega_0(K):=\lim_{t\rightarrow+\infty}\frac{\ln{\|S_K(t)\|_{\mathcal{L}(X)}}}{t}.
\end{equation*}
The essential growth bound of $\omega_{ess}(K)\in[-\infty,+\infty)$ of $K$ is defined by 
\begin{equation*}
\omega_{ess}(K):=\lim_{t\rightarrow+\infty}\frac{\ln{\|S_K(t)\|_{ess}}}{t},
\end{equation*}
where $\|S_K(t)\|_{ess}$ is the essential norm of $S_K(t)$ defined by 
\begin{equation*}
\|S_K(t)\|_{ess}:=\kappa(S_K(t)B_X(0,1)), 
\end{equation*}
where $B_X(0,1)$ is the unit ball in $X$ and, for each bounded set $B$, 
\begin{equation*}
\kappa(B):=\inf\{\varepsilon>0: B \text{ can be covered by a finite number of balls of radius }\leq \varepsilon \}
\end{equation*}
is the Kuratovsky measure of non-compactness.
\end{defi}
The following result is proved by Webb \cite[Proposition 4.11]{Webb-book} and by Engel and Nagel \cite[Corollary 2.11]{Engel-Nagel-book}.
\begin{thm}\label{thm: finite eigenvalues in the strip}
Let $K:D(K)\subset X\rightarrow X$ be the infinitesimal generator of a linear $C^0-$semigroup $\{S_K(t)\}_{t\geq0}$ on a Banach space $X$. Then $\omega_0(K)=\max(\omega_{ess}(K),\max_{\lambda\in\sigma(K)\setminus\sigma_{ess}(K)}\Re{\lambda})$. Furthermore, if $\omega_{ess}(K)<\omega_0(K)$, then for each $\gamma\in(\omega_{ess}(K),\omega_0(K)] $, $\{\lambda\in\sigma(K):\Re{\lambda}\geq\gamma\}\subset\sigma_p(K)$ is non-empty, finite, and contains only poles of the resolvent of $K$.
\end{thm}
\subsubsection{Remainders of the proof of Proposition \ref{prop: asymptotic expansion for A0}}\label{sec: proof of expansion of A0}
Now we are in a position to prove the asymptotic behaviors for eigenmodes of $\A$ near the eigenmodes $(\ii\lambda_c(L_0),\E_c)$ with $L$ close to $L_0$. As we presented in Proposition \ref{prop: asymptotic expansion for A0}, the crucial part is to verify the conditions of the implicit function theorem. Here we complete the details of the proof.
\begin{proof}[Remainders of the proof of Proposition \ref{prop: asymptotic expansion for A0}]
We need to check the invertibility of the Jacobian matrix $J_{G,\tau}$. First, It is easy to notice that at the point $(t_r,t_i,L)=(\frac{\pi}{L_0}\frac{2k+l}{3},0,L_0)$,
\begin{align*}
\sqrt{1-3(t_r+\ii t_i)^2}|_{t_r=\frac{\pi}{L_0}\frac{2k+l}{3},t_i=0}&=\frac{\pi l}{L_0},\\
(t_r+\ii t_i)L|_{t_r=\frac{\pi}{L_0}\frac{2k+l}{3},t_i=0,L=L_0}&=\frac{\pi(2k+l)}{3},\\
\sqrt{1-3(t_r+\ii t_i)^2}L|_{t_r=\frac{5}{2\sqrt{21}},t_i=0,L=L_0}&=\pi l.
\end{align*}
Therefore, we deduce that 
\begin{equation*}
\begin{aligned}
G_r(\frac{\pi}{L_0}\frac{2k+l}{3},0,L_0)&=\Re{\left(-\frac{\pi l}{L_0}\cos{3\cdot \frac{\pi(2k+l)}{3}}+\frac{\pi l}{L_0}\cos{\pi l}-\ii\frac{\pi l}{L_0}\sin{3\cdot \frac{\pi(2k+l)}{3}}+3\ii \frac{\pi}{L_0}\frac{2k+l}{3}\sin{\pi l} \right)}\\
&=-\frac{\pi l}{L_0}(-1)^{2k+l}+\frac{\pi l}{L_0}(-1)^l\\
&=0.\\
G_i(\frac{\pi}{L_0}\frac{2k+l}{3},0,L_0)&=\Im{\left(-\frac{\pi l}{L_0}\cos{3\cdot \frac{\pi(2k+l)}{3}}+\frac{\pi l}{L_0}\cos{\pi l}-\ii\frac{\pi l}{L_0}\sin{3\cdot \frac{\pi(2k+l)}{3}}+3\ii \frac{\pi}{L_0}\frac{2k+l}{3}\sin{\pi l} \right)}\\
&=0.
\end{aligned}
\end{equation*}
By the definition, we know that
{\footnotesize
\begin{align*}
\frac{\p G_r}{\p L}&=\Re{\left(3\tau\sqrt{1-3\tau^2} \sin{3L\tau} - (1-3\tau^2)\sin{L\sqrt{1-3\tau^2}}- 3\ii\tau\sqrt{1-3\tau^2}\cos{3L\tau} + 3\ii \tau\sqrt{1-3\tau^2}\cos{L \sqrt{1-3\tau^2}}\right)},\\
\frac{\p G_i}{\p L}&=\Im{\left(3\tau\sqrt{1-3\tau^2} \sin{3L\tau} - (1-3\tau^2)\sin{L\sqrt{1-3\tau^2}}- 3\ii\tau\sqrt{1-3\tau^2}\cos{3L\tau} + 3\ii \tau\sqrt{1-3\tau^2}\cos{L \sqrt{1-3\tau^2}}\right)}.
\end{align*}
}
Then, we have
{\footnotesize
\begin{equation*}
\begin{aligned}
\frac{\p G_r}{\p L}(\frac{\pi}{L_0}\frac{2k+l}{3},0,L_0)&=\Re\left(\frac{\pi^2(2k+l) l}{L_0^2} \sin{(2k+l)\pi} - (\frac{\pi l}{L_0})^2\sin{l\pi}- \ii\frac{\pi^2(2k+l) l}{L_0^2}\cos{(2k+l)\pi}\right.\\
&\left.+ \ii \frac{\pi^2(2k+l) l}{L_0^2}\cos{l\pi}\right)\\
&=0.\\
\frac{\p G_i}{\p L}(\frac{\pi}{L_0}\frac{2k+l}{3},0,L_0)&=\Im\left(\frac{\pi^2(2k+l) l}{L_0^2} \sin{(2k+l)\pi} - (\frac{\pi l}{L_0})^2\sin{l\pi}- \ii\frac{\pi^2(2k+l) l}{L_0^2}\cos{(2k+l)\pi}\right.\\
&\left. + \ii \frac{\pi^2(2k+l) l}{L_0^2}\cos{l\pi}\right)\\
&=- \frac{\pi^2(2k+l) l}{L_0^2}(-1)^{2k+l} +  \frac{\pi^2(2k+l) l}{L_0^2}(-1)^l\\
&=0.
\end{aligned}
\end{equation*}
}
In addition, we compute the second derivative of $G$ with respect to $L$,
{\footnotesize
\begin{align*}
\frac{\p^2 G_r}{\p L^2}&=\Re\left(9 \tau^2 \sqrt{1 - 3 \tau^2} \cos{3 L \tau} - (1 - 3 \tau^2)^{\frac{3}{2}} \cos{L \sqrt{1 - 3 \tau^2}} + 3 \ii \tau(3 \tau \sqrt{1 - 3 \tau^2} \sin(3 L \tau)\right.\\
&\left. + (-1 + 3 \tau^2) \sin{L \sqrt{1 - 3 \tau^2}})\right),\\
\frac{\p^2 G_i}{\p L^2}&=\Im\left(9 \tau^2 \sqrt{1 - 3 \tau^2} \cos{3 L \tau} - (1 - 3 \tau^2)^{\frac{3}{2}} \cos{L \sqrt{1 - 3 \tau^2}} + 3 \ii \tau(3 \tau \sqrt{1 - 3 \tau^2} \sin(3 L \tau)\right.\\
&\left. + (-1 + 3 \tau^2) \sin{L \sqrt{1 - 3 \tau^2}})\right).
\end{align*}}
Thus, {\footnotesize
\begin{align*}
\frac{\p^2 G_r}{\p L^2}(\frac{\pi}{L_0}\frac{2k+l}{3},0,L_0)&=\Re\left( \frac{\pi^3(2k+l)^2 l}{L_0^3} \cos{(2k+l)\pi} - \frac{\pi^3 l^3}{L_0^3} \cos{l\pi} + \ii \frac{\pi^3(2k+l)^2 l^3}{L_0^3} \sin((2k+l)\pi)\right.\\
&\left. +\ii \frac{\pi^3(2k+l) l^3}{L_0^3} \sin{l\pi}\right)\\
&= \frac{\pi^3(2k+l)^2 l}{L_0^3}(-1)^{2k+l} - \frac{\pi^3 l^3}{L_0^3} (-1)^l\\
&= \frac{4\pi^3kl(k+l) }{L_0^3}(-1)^{l},\\
\frac{\p^2 G_i}{\p L^2}((\frac{\pi}{L_0}\frac{2k+l}{3},0,L_0)&=\Im\left( \frac{\pi^3(2k+l)^2 l}{L_0^3} \cos{(2k+l)\pi} - \frac{\pi^3 l^3}{L_0^3} \cos{l\pi}
+ \ii \frac{\pi^3(2k+l)^2 l^3}{L_0^3} \sin((2k+l)\pi)\right.\\
&\left.+\ii \frac{\pi^3(2k+l) l^3}{L_0^3} \sin{l\pi}\right)\\
&=0.
\end{align*}
}
To apply the implicit function theorem, we also need to calculate the Jacobian matrix of $G$ with respect to $t_r$ and $t_i$. Then, we need to compute $\frac{\p G_r}{\p t_r}$, $\frac{\p G_r}{\p t_i}$, $\frac{\p G_i}{\p t_r}$, and $\frac{\p G_i}{\p t_i}$. For the real part, we have the following expressions:
\begin{align*}
\frac{\p G_r}{\p t_r}&=\Re{\left(\frac{3\left( \tau\cos{3 L \tau} - (\tau + 3\ii L\tau^2)\cos{L\sqrt{1-3\tau^2}} +\ii \tau\sin{3L\tau}\right)}{\sqrt{1-3\tau^2}}\right.}\\
&{\left.-3\ii L\sqrt{1-3\tau^2}\cos{3L\tau}+3L\sqrt{1-3\tau^2}\sin{3L\tau}+3(\ii+ L\tau)\sin{L\sqrt{1-3\tau^2}}\right)},\\
\frac{\p G_r}{\p t_i}&=\Re{\left(\frac{3\ii\left( \tau\cos{3 L \tau} - (\tau + 3\ii L\tau^2)\cos{L\sqrt{1-3\tau^2}} +\ii \tau\sin{3L\tau}\right)}{\sqrt{1-3\tau^2}}\right.}\\
&{\left.+3L\sqrt{1-3\tau^2}\cos{3L\tau}+3\ii L\sqrt{1-3\tau^2}\sin{3L\tau}+3\ii(\ii+ L\tau)\sin{L\sqrt{1-3\tau^2}}\right)}.
\end{align*}
Similarly, for the imaginary part, the derivatives are as follows:
\begin{align*}
\frac{\p G_i}{\p t_r}&=\Im{\left(\frac{3\left( \tau\cos{3 L \tau} - (\tau + 3\ii L\tau^2)\cos{L\sqrt{1-3\tau^2}} +\ii \tau\sin{3L\tau}\right)}{\sqrt{1-3\tau^2}}\right.}\\
&{\left.-3\ii L\sqrt{1-3\tau^2}\cos{3L\tau}+3L\sqrt{1-3\tau^2}\sin{3L\tau}+3(\ii+ L\tau)\sin{L\sqrt{1-3\tau^2}}\right)}\\
\frac{\p G_i}{\p t_i}&=\Im{\left(\frac{3\ii\left( \tau\cos{3 L \tau} - (\tau + 3\ii L\tau^2)\cos{L\sqrt{1-3\tau^2}} +\ii \tau\sin{3L\tau}\right)}{\sqrt{1-3\tau^2}}\right.}\\
&{\left.+3L\sqrt{1-3\tau^2}\cos{3L\tau}+3\ii L\sqrt{1-3\tau^2}\sin{3L\tau}+3\ii(\ii+ L\tau)\sin{L\sqrt{1-3\tau^2}}\right)}.
\end{align*}
Then, we plug $t_r=\frac{\pi}{L_0}\frac{2k+l}{3},t_i=0,L=L_0$ in the formulas above and we obtain
{\footnotesize
\begin{align*}
\frac{\p G_r}{\p t_r}(\frac{\pi}{L_0}\frac{2k+l}{3},0,L_0)&=\Re{\left(\frac{3\left(\frac{\pi}{L_0}\frac{2k+l}{3}\cos{(2k+l)\pi} - (\frac{\pi}{L_0}\frac{2k+l}{3} + 3\ii \frac{\pi(2k+l)}{3}\frac{\pi}{L_0}\frac{2k+l}{3})\cos{l\pi} +\ii \frac{\pi}{L_0}\frac{2k+l}{3}\sin{(2k+l)\pi}\right)}{\frac{\pi l}{L_0}}\right.}\\
&{\left.-3\ii\pi l\cos{(2k+l)\pi}+3\pi l\sin{(2k+l)\pi}+3(\ii+ \frac{\pi(2k+l)}{3})\sin{l\pi}\right)}\\
&=\Re{\left(\frac{\ii\frac{\pi^2(2k+l)^2}{L_0}(-1)^{l+1}}{\frac{l\pi}{L_0}}+3\ii l\pi(-1)^{l+1}\right)}\\
&=0,
\end{align*}
\begin{align*}
\frac{\p G_i}{\p t_r}(\frac{\pi}{L_0}\frac{2k+l}{3},0,L_0)&=\Im{\left(\frac{3\left(\frac{\pi}{L_0}\frac{2k+l}{3}\cos{(2k+l)\pi} - (\frac{\pi}{L_0}\frac{2k+l}{3} + 3\ii \frac{\pi(2k+l)}{3}\frac{\pi}{L_0}\frac{2k+l}{3})\cos{l\pi} +\ii \frac{\pi}{L_0}\frac{2k+l}{3}\sin{(2k+l)\pi}\right)}{\frac{\pi l}{L_0}}\right.}\\
&{\left.-3\ii\pi l\cos{(2k+l)\pi}+3\pi l\sin{(2k+l)\pi}+3(\ii+ \frac{\pi(2k+l)}{3})\sin{l\pi}\right)}\\
&=\Im{\left(\frac{\ii\frac{\pi^2(2k+l)^2}{L_0}(-1)^{l+1}}{\frac{l\pi}{L_0}}+3\ii l\pi(-1)^{l+1}\right)}\\
&=4\pi\frac{k^2+kl+l^2}{l},
\end{align*}
\begin{align*}
\frac{\p G_r}{\p t_i}(\frac{\pi}{L_0}\frac{2k+l}{3},0,L_0)&=\Re{\left(\ii\frac{3\left(\frac{\pi}{L_0}\frac{2k+l}{3}\cos{(2k+l)\pi} - (\frac{\pi}{L_0}\frac{2k+l}{3} + 3\ii \frac{\pi(2k+l)}{3}\frac{\pi}{L_0}\frac{2k+l}{3})\cos{l\pi} +\ii \frac{\pi}{L_0}\frac{2k+l}{3}\sin{(2k+l)\pi}\right)}{\frac{\pi l}{L_0}}\right.}\\
&{\left.+3\pi l\cos{(2k+l)\pi}+3\ii\pi l\sin{(2k+l)\pi}+3\ii(\ii+ \frac{\pi(2k+l)}{3})\sin{l\pi}\right)}\\
&=-4\pi\frac{k^2+kl+l^2}{l},\\
\frac{\p G_i}{\p t_i}(\frac{\pi}{L_0}\frac{2k+l}{3},0,L_0)&=\Im{\left(\ii\frac{3\left(\frac{\pi}{L_0}\frac{2k+l}{3}\cos{(2k+l)\pi} - (\frac{\pi}{L_0}\frac{2k+l}{3} + 3\ii \frac{\pi(2k+l)}{3}\frac{\pi}{L_0}\frac{2k+l}{3})\cos{l\pi} +\ii \frac{\pi}{L_0}\frac{2k+l}{3}\sin{(2k+l)\pi}\right)}{\frac{\pi l}{L_0}}\right.}\\
&{\left.+3\pi l\cos{(2k+l)\pi}+3\ii\pi l\sin{(2k+l)\pi}+3\ii(\ii+ \frac{\pi(2k+l)}{3})\sin{l\pi}\right)}\\
&=0.
\end{align*}
}
In summary, the Jacobian matrix of $G$ with respect to $(\tau_r,\tau_i)$ at the point $(\frac{\pi}{L_0}\frac{2k+l}{3},0,L_0)$ is 
\begin{equation*}
J_{G,\tau}(\frac{\pi}{L_0}\frac{2k+l}{3},0,L_0)=\left(
\begin{array}{cc}
    \frac{\p G_r}{\p t_r} &\frac{\p G_r}{\p t_i}  \\
     \frac{\p G_i}{\p t_r}& \frac{\p G_i}{\p t_i}
\end{array}
\right)|_{t_r=\frac{\pi}{L_0}\frac{2k+l}{3},t_i=0,L=L_0}=\left(
\begin{array}{cc}
    0&-4\pi\frac{k^2+kl+l^2}{l}  \\
    4\pi\frac{k^2+kl+l^2}{l}& 0
\end{array}
\right).
\end{equation*} 
The derivatives with respect to $L$ are
\begin{equation*}
\begin{aligned}
\left(
\begin{array}{c}
     \frac{\p G_r}{\p L}  \\
     \frac{\p G_i}{\p L}
\end{array}
\right)|_{t_r=\frac{\pi}{L_0}\frac{2k+l}{3},t_i=0,L=L_0}=\left(
\begin{array}{c}
    0\\
    0
\end{array}
\right),\quad \left(
\begin{array}{c}
     \frac{\p^2 G_r}{\p L^2}  \\
     \frac{\p^2 G_i}{\p L^2}
\end{array}
\right)|_{t_r=\frac{\pi}{L_0}\frac{2k+l}{3},t_i=0,L=L_0}=\left(
\begin{array}{c}
    \frac{4\pi^3kl(k+l) }{L_0^3}(-1)^{l}\\
    0
\end{array}
\right).
\end{aligned}
\end{equation*}
Then we obtain the following expansions :
\begin{equation*}
t_r(L)=\frac{\pi}{L_0}\frac{2k+l}{3}+\bigO((L-L_0)^3),\;t_i(L)=(-1)^{l+1}\frac{\pi^2kl^2(k+l)}{2(k^2+kl+l^2)}(L-L_0)^2+\bigO((L-L_0)^3).
\end{equation*}
We plug $\tau=\frac{\pi}{L_0}\frac{2k+l}{3}+(-1)^{l+1}\frac{\pi^2kl^2(k+l)}{2(k^2+kl+l^2)}(L-L_0)^2+\bigO((L-L_0)^3)$ in the formula \eqref{eq: general form for eigenfunction-bad op}, then we derive the following expansion for the eigenfunction
\begin{equation}
\F_{\zeta}(x)=r_1\G_{c}(x) +  r_1 \Tilde{f}(x)(L-L_0)+\bigO((L-L_0)^2),
\end{equation}
where $\G_{c}(x)=\frac{1}{k}e^{-\frac{\ii (2k + l) x}{\sqrt{3}\sqrt{k^2 + kl + l^2}}}\left(l -(k+l) e^{\frac{\ii \sqrt{3} k x}{\sqrt{k^2 + kl + l^2}}} + k e^{\frac{\ii \sqrt{3} (k + l) x}{\sqrt{k^2 + kl + l^2}}}\right)$ is the eigenfunction associated with $\lambda_{c}=-\ii\frac{(2k+l)(k-l)(2l+k)}{3\sqrt{3}(k^2+k l+l^2)^{\frac{3}{2}}}$ at the critical length $L_0$ and
\begin{align*}
\Tilde{f}(x)&=-\frac{\sqrt{3} e^{-\frac{\ii (2k + l) x}{\sqrt{3}\sqrt{k^2 + kl + l^2}}}(k + l)}{2 (k^2 + kl + l^2)^{\frac{5}{2}}} \left(-3 (-1)^l k l (k + l) (2k + l) + e^{\frac{\ii \sqrt{3} k x}{\sqrt{k^2 + kl + l^2}}} (3 (-1)^l k (k - l) l (k + l) - \ii (k^2 + kl + l^2)^2)\right.\\
&\left.+ e^{\frac{\ii \sqrt{3} (k + l) x}{\sqrt{k^2 + kl + l^2}}} (3 (-1)^l k l (k + l) (k + 2l) + \ii (k^2 + kl + l^2)^2)\right).
\end{align*}
Now we complete the proof of Proposition \ref{prop: asymptotic expansion for A0}.
\end{proof}

\section{Bi-orthogonal family}\label{sec: appendix bi-orthogonal family}
\subsection{Preparations}
Before we present our concrete construction, we introduce some notations and definitions.
\begin{defi}
Let $\mathcal{J}$ be a countable index set. We say that a sequence  of real numbers $\{\mu_j(L)\}_{j\in\mathcal{J}}$, which depends on a parameter $L$,  is \textbf{uniformly regular} if the following gap condition holds uniformly in $L$, i.e.
\begin{equation}\label{eq: uniform gap condition}
    \mathbf{g}(\{\mu_j(L)\}_{j\in\mathcal{J}}):=\inf_L\inf_{m\neq n}|\mu_m(L)-\mu_n(L)|>0.
\end{equation} 
We can also define uniform regular $3-$sequences if we require that the following asymptotic behaviors of $\mu_j$ are uniformly in $L$:
 the sequence $\{\mu_j\}_{j\in\Z\backslash\{0\}}$ is a regular increasing sequence: $\cdots< \mu_{-2}< \mu_{-1}<0<\mu_{1}< \mu_{2}<\cdots$. Moreover, there exists $C>0$ such that $\mu_{\pm j}=\pm C|j|^{3}+\bigO_{j\rightarrow\infty}(|j|^{\alpha-1})$, for $j\in\N^*$.
\end{defi}
\begin{rem}\label{rem: uniform sequence-eigenvalues}
By Proposition \ref{prop: Asymp in L} and Proposition \ref{prop: Asymp in L-low}, we know that there are two different cases for the localization of the eigenvalues of the operator $\B$.
\begin{enumerate}
    \item If $k\not\equiv l\mod 3$, the sequence  $\{\lambda_j(L)\}_{j\in\Z\backslash\{0\}}$ is a uniform regular $3-$sequence.
    \item If $k\equiv l\mod 3$, for $j\in\Lambda_E$, there are $2N_0$ pairs of eigenvalues $(\lambda_{\sigma^+(j)},\lambda_{\sigma^-(j)})$, $j\in\{-N_0,\cdots,-1,1,\cdots,N_0\}$ such that $\lim_{L\rightarrow L_0}|\lambda_{\sigma^+(j)}-\lambda_{\sigma^-(j)}|=0$. Hence, Therefore, the sequence  $\{\lambda_j(L)\}_{j\notin\Lambda_E}$ is a uniform regular $3-$sequence.
\end{enumerate}
\end{rem}
Based on the definition above, we introduce the following two special functions (see more details in \cite{Lissy-2014,Tucsnak-Weiss-2009}). 
\begin{lem}{\cite[Lemma 2.4]{Lissy-2014}}\label{lem: defi of phi}
For each $L\notin\mathcal{N}$, let $\Psi_j$ be defined as follows:
\begin{equation}\label{eq: defi of function Psi}
    \Psi_j(z):=\prod_{k\neq j}\left(1-\frac{z}{\lambda_k-\lambda_j}\right).
\end{equation}
Let $I$ satisfy the condition {\bf (C)}. Then for every $L\in I\setminus\{L_0\}$, 
\begin{enumerate}
    \item if $k\not\equiv l\mod 3$, the sequence  $\{\lambda_j(L)\}_{j\in\Z\backslash\{0\}}$ is a uniform regular $3-$sequence. Moreover, 
    \begin{equation}\label{eq: estimate of Psi}
    |\Psi_j(z)|\leq K_1e^{K_2|z|^{\frac{1}{3}}}P(|z|),
\end{equation}
where $K_1$ is independent of $z$, $l$, and $L$, $K_2=\frac{2^{\frac{5}{2}}L_0^{\frac{3}{2}}}{\sqrt{6\pi}}$, and $P$ is a polynomial in $|z|$. 
\item if $k\equiv l\mod 3$, for $j\in\Lambda_E$, there are $2N_0$ pairs of eigenvalues $(\lambda_{\sigma^+(j)},\lambda_{\sigma^-(j)})$, $j\in\{-N_0,\cdots,N_0\}$ such that $\lim_{L\rightarrow L_0}|\lambda_{\sigma^+(j)}-\lambda_{\sigma^-(j)}|=0$. Therefore, the sequence  $\{\lambda_j(L)\}_{j\notin\Tilde{\Lambda_E}}$ is a uniform regular $3-$sequence with the following uniform estimates 
    \begin{equation}\label{eq: estimate of Psi-2}
|\Psi_j(z)|\leq K_1e^{K_2|z|^{\frac{1}{3}}}P(|z|),j\notin\Lambda_E,\; 
      |\Psi_{\sigma^{\pm}(j)}(z)|\leq \frac{1}{|\lambda_{\sigma^+(j)}-\lambda_{\sigma^-(j)}|}K_1e^{K_2|z|^{\frac{1}{3}}}P(|z|).    
\end{equation}
where $K_1$ is independent of $z$, $l$, and $L$, $K_2=\frac{2^{\frac{5}{2}}L_0^{\frac{3}{2}}}{\sqrt{6\pi}}$, and $P$ is a polynomial in $|z|$. 
\end{enumerate}
\end{lem}
\begin{lem}{\cite[Lemma 2.3]{Lissy-2014}}\label{lem: defi of sigma-beta}
We define a function $\sigma_{\beta}(x)=e^{-\frac{\beta}{1-x^2}}$ for $|x|<1$ and $\sigma_{\beta}(x)=0$ for $|x|\geq1$. 
Let $\Sigma_{\beta}(z):=\frac{1}{\|\sigma_{\beta}\|_{L^1(-1,1)}}\int_{-1}^1\sigma_{\beta}(x)e^{-\ii\nu_{\delta}(\beta) xz}dx$, with a small parameter $\delta>0$ and $\nu_{\delta}(\beta)=\frac{64K_2^3(1+\delta)^3}{81\beta^2}$. Then 
\begin{equation}\label{eq: growth of Sigma}
    |\Sigma_{\beta}(z)|\leq e^{\nu_{\delta}(\beta)|\Im z|}.
\end{equation}
In addition, for $x\in\R$, there exists a constant $K_3$, independent of $\beta$, such that
\begin{equation}\label{eq: estimate of Sigma}
    |\Sigma_{\beta}(x)|\leq K_3\sqrt{\beta+1}e^{\frac{3\beta}{4}-K_2(1+\frac{\delta}{2})|x|^{\frac{1}{3}}}.
\end{equation}
\end{lem}
We omit the proof of these two lemmas, one can see \cite{Lissy-2014} for more details. As a consequence of these two Lemmas, we are able to construct the bi-orthogonal family to $e^{-\ii\lambda_js}$.
\subsection{Construction of bi-orthogonal family}
\begin{proof}[Proof of Proposition \ref{prop: biorthogonal family}]
Let $\psi_j(z):=\Psi_j(z-\lambda_j)\Sigma_{\beta}(-z+\lambda_j)$. Clearly, since 
$$
\Sigma_{\beta}(0)=\frac{1}{\|\sigma_{\beta}\|_{L^1(-1,1)}}\int_{-1}^1\sigma_{\beta}(x)dx=1,\text{ and }\Psi_j(0)=1,
$$
we have $\psi(\lambda_j)=1$. For $k\neq j$, $\Psi_j(\lambda_k-\lambda_j)=0$ by the definition \eqref{eq: defi of function Psi}. We conclude that $\psi_j(\lambda_k)=\delta_{jk}$. Set $\nu_{\delta}(\beta)=\frac{T(1-\delta)}{2}$, which implies that $\beta=\frac{8\sqrt{2}K_2^{\frac{3}{2}}(1+\delta)^{\frac{3}{2}}}{9T^{\frac{1}{2}}(1-\delta)^{\frac{1}{2}}}$ and $\nu_{\delta}\rightarrow\frac{T}{2}$ as $\delta\rightarrow0$. By the estimates \eqref{eq: estimate of Psi} and \eqref{eq: estimate of Sigma}, we obtain $|\psi_j(x)|\leq \frac{C(N)K_1K_3\sqrt{\beta+1}e^{\frac{3\beta}{4}}}{1+|x-\lambda_j|^{2N}}$.
Now we fix $N>2$ and set $K>\max\{C(N)K_1K_3,\frac{16\sqrt{2}K_2^{\frac{3}{2}}}{9}\}$, and $\delta$ as close to $0$ as needed. Then we obtain
\begin{equation}\label{eq: L^2-es-psi-K}
|\psi_j(x)|\leq \frac{K e^{\frac{K}{\sqrt{T}}}}{1+|x-\lambda_j|^{2N}}.
\end{equation}
This implies that $\psi_j\in L^2(\R)$. By the estimates \eqref{eq: estimate of Psi} and \eqref{eq: growth of Sigma}, we also obtain $|\psi_j(z)|\lesssim e^{\frac{T|z|}{2}}$. Hence, using the Paley-Wiener Theorem, $\psi_j$ is the Fourier transform of a function $\phi_j\in L^2(\R)$ with compact support in $[-\frac{T}{2},\frac{T}{2}]$, which proves the first statement. The second statement comes from $\delta_{jk}=\psi_j(\lambda_k)=\int_{-\frac{T}{2}}^{\frac{T}{2}}e^{-\ii s\lambda_k}\phi_j(s)ds$. 
For the last statement, for $0\leq m\leq N$, since $\psi_j$ is the Fourier transform of $\phi_j$, we know that $\|\phi^{(m)}_j\|_{L^{\infty}(\R)}\leq \|x^{m}\psi_j(x)\|_{L^1(\R)}\leq \|\frac{K e^{\frac{K}{\sqrt{T}}}x^m}{1+|x-\lambda_j|^{2N}}\|_{L^1(\R)}$.
By changing of variable, we obtain
\begin{equation}\label{eq: m-th derivative of phi_j}
    \|\phi^{(m)}_j\|_{L^{\infty}(\R)}\leq C e^{\frac{K}{\sqrt{T}}}|\lambda_j|^m,
\end{equation}
where $C$ is a universal constant independent of $T$ and $j$.  
\end{proof}
\subsection{Compensate bi-orthogonal family for $2\pi$}\label{sec: app-compensate-bi-2pi}
\begin{proof}[Proof of Lemma \ref{lem: bi-orthogonal family-2pi}]
First, as we presented in Remark \ref{rem: uniform sequence-eigenvalues}, we notice that
\begin{align}
    &\inf_L\inf_{j\neq k,k\neq \pm 1}|\lambda_j(L)-\lambda_k(L)|>0,\label{eq: uniform-gap-condition-2pi-1} \\
    &\inf_L\inf_{j\neq \pm 1}|\lambda_j(L)-\lambda_{1}(L)|>0,\label{eq: uniform-gap-condition-2pi-2} \\
    &\inf_L\inf_{j\neq \pm 1}|\lambda_j(L)-\lambda_{-1}(L)|>0,\label{eq: uniform-gap-condition-2pi-3}.
\end{align}
Recalling the definitions of $\psi_j$ in Proposition \ref{prop: biorthogonal family}, for $j\neq 0,\pm 1$,
\begin{equation*}
\psi_j(z):=\Psi_j(z-\lambda_j)\Sigma_{\beta}(-z+\lambda_j)=\prod_{k\neq j}\left(1-\frac{z-\lambda_j}{\lambda_k-\lambda_j}\right)\Sigma_{\beta}(-z+\lambda_j) 
\end{equation*}
The uniform condition \eqref{eq: uniform-gap-condition-2pi-1} ensures that $\psi_j\in L^2(\R)$ and $|\psi_j(z)|\lesssim e^{\frac{T|z|}{2}}$, same as we presented in the proof of Proposition \ref{prop: biorthogonal family} based on the estimates  \eqref{eq: estimate of Psi} and \eqref{eq: growth of Sigma}. Thus, by Paley-Wiener's Theorem, we obtain the same function $\phi_j$, whose Fourier transform is $\psi_j$, for $j\neq 0,\pm 1$.  
Let $\vartheta_j=\phi_j$, for $j\neq 0,\pm 1$. Then, 
\begin{align*}
\supp{\vartheta_j}\subset [-\frac{T}{2},\frac{T}{2}],\;\|\vartheta^{(m)}_j\|_{L^{\infty}(\R)}\leq Ce^{\frac{K}{\sqrt{T}}}|\lambda_j|^m,\forall j\neq0, \pm 1,\;
\int_{-\frac{T}{2}}^{\frac{T}{2}}\vartheta_j(s)e^{-\ii \lambda_ks}ds=\delta_{jk},\forall j,k\in\Z\backslash\{0,\pm 1\},
\end{align*}
holds in the same way as $\phi_j$ in Proposition \ref{prop: biorthogonal family}. Then, we focus on defining the function $\vartheta_0$. We begin by defining the function
\begin{equation}\label{eq: defi-psi-0-bi-orthogonal family-2pi7}
\psi_0(z):=C_1\prod_{k\neq\pm 1}\left(1-\frac{z-\lambda_{1}}{\lambda_k-\lambda_{1}}\right)\Sigma_{\beta}(-z+\lambda_{1})+C_2\prod_{k\neq\pm 1}\left(1-\frac{z-\lambda_{-1}}{\lambda_k-\lambda_{-1}}\right)\Sigma_{\beta}(-z+\lambda_{-1}),
\end{equation}
where $C_1$ and $C_2$ are two coefficients to be chosen later. Using the uniform gap conditions \eqref{eq: uniform-gap-condition-2pi-2} and \eqref{eq: uniform-gap-condition-2pi-3}, by Lemma \ref{lem: defi of phi}, we know that $z\in\C$, (similarly for $\lambda_{-1}$)
\begin{equation}\label{eq: prod-growth-2pi}
|\prod_{k\neq\pm 1}\left(1-\frac{z-\lambda_{1}}{\lambda_k-\lambda_{1}}\right)|\leq K_1e^{K_2|z-\lambda_{1}|^{\frac{1}{3}}}P(|z-\lambda_{1}|),
\end{equation}
where $K_1$ and $K_2$ are the same as in Lemma \ref{lem: defi of phi} and $P$ is a polynomial. Thus, in particular, for $x\in \R$,
\begin{align*}
|\psi_0(x)|\leq |C_1|K_1e^{K_2|x-\lambda_{1}|^{\frac{1}{3}}}P(|x-\lambda_{1}|)|\Sigma_{\beta}(-z+\lambda_{1})|+|C_2|K_1e^{K_2|x-\lambda_{-1}|^{\frac{1}{3}}}P(|x-\lambda_{-1}|)|\Sigma_{\beta}(-z+\lambda_{-1})|.
\end{align*}
Using the estimates \eqref{eq: estimate of Sigma}, set $\nu_{\delta}(\beta)=\frac{T(1-\delta)}{2}$ and we obtain, with the same $K$ and $N$ as in the estimate \eqref{eq: L^2-es-psi-K}, 
\begin{equation}\label{eq: growth-psi_0-2pi}
|\psi_0(x)|\leq \frac{K|C_1| e^{\frac{K}{\sqrt{T}}}}{1+|x-\lambda_{1}|^{2N}}+\frac{K|C_2| e^{\frac{K}{\sqrt{T}}}}{1+|x-\lambda_{-1}|^{2N}}.    
\end{equation}
Moreover, by the estimates \eqref{eq: prod-growth-2pi} and \eqref{eq: growth of Sigma}, we deduce that $|\psi_0(z)|\lesssim (|C_1|+|C_2|)e^{\frac{T|z|}{2}}$. 
Assume that $|C_1|$ and $|C_2|$ are uniformly bounded with respect to $L$. Then, applying Paley-Wiener's Theorem, we obtain a function $\vartheta_0\in L^2(\R)$, whose Fourier transform is $\psi_0$ and $\supp{\vartheta_0}\subset[-\frac{T}{2},\frac{T}{2}]$. We shall verify the uniform bound assumption at the end of the proof for well-chosen coefficients $C_1$ and $C_2$. As a consequence of our choice of $\vartheta_0$, we know that $\int_{-\frac{T}{2}}^{\frac{T}{2}}e^{-\ii s\lambda_k}\vartheta_0(s)ds=\psi_0(\lambda_k)$.
For $k\neq \pm 1$, it is easy to verify that $\psi_0(\lambda_k)=0$. 
For $k=\pm 1$, using the fact that $\lambda_{1}=-\lambda_{-1}$,
\begin{align*}
\psi_0(\lambda_{1})&=C_1\prod_{l\neq\pm 1}\left(1-\frac{\lambda_{1}-\lambda_{1}}{\lambda_l-\lambda_{1}}\right)\Sigma_{\beta}(0)+C_2\prod_{l\neq \pm 1}\left(1-\frac{\lambda_{1}-\lambda_{-1}}{\lambda_l+\lambda_{1}}\right)\Sigma_{\beta}(-\lambda_{1}+\lambda_{-1})\\
&=C_1+C_2\prod_{l\neq \pm 1}\left(1-\frac{2\lambda_{1}}{\lambda_l+\lambda_{1}}\right)\Sigma_{\beta}(-2\lambda_{1}),\\
\psi_0(\lambda_{-1})
&=C_1\prod_{l\neq\pm 1}\left(1+\frac{2\lambda_{1}}{\lambda_l-\lambda_{1}}\right)\Sigma_{\beta}(2\lambda_{1})+C_2.
\end{align*}
For the infinite product, we use the fact that $\lambda_{-l}=-\lambda_l$ as we observed in Proposition \ref{prop: eigenvalues of A},
\begin{align*}
\prod_{l\neq \pm 1}\left(1-\frac{2\lambda_{1}}{\lambda_l+\lambda_{1}}\right)&=\prod_{l\neq 1,l>0}\left(1-\frac{2\lambda_{1}}{\lambda_l+\lambda_{1}}\right)\prod_{l\neq -1,l<0}\left(1-\frac{2\lambda_{1}}{\lambda_l+\lambda_{1}}\right)\\
&=\prod_{l\neq 1,l>0}\left(\frac{\lambda_l-\lambda_{1}}{\lambda_l+\lambda_{1}}\right)\left(\frac{\lambda_l+\lambda_{1}}{\lambda_{l}-\lambda_{1}}\right)=1.
\end{align*}
Similarly, $\prod_{l\neq\pm 1}\left(1+\frac{2\lambda_{1}}{\lambda_l-\lambda_{1}}\right)=1$. 
After simplifying the equation, the coefficients $C_1$ and $C_2$ satisfy that
\begin{align*}
\psi_0(\lambda_{1})=C_1+C_2\Sigma_{\beta}(-2\lambda_{1}),\;
\psi_0(\lambda_{-1})=C_1\Sigma_{\beta}(2\lambda_{1})+C_2.    \end{align*}
Since $\Sigma_{\beta}(2\lambda_{1})=\frac{1}{\|\sigma_{\beta}\|_{L^1(-1,1)}}\int_{-1}^1\sigma_{\beta}(x)e^{-2\ii\nu_{\delta}(\beta) x\lambda_{1}}dx$, changing the variable $y=-x$, we obtain that
\begin{align*}
\frac{1}{\|\sigma_{\beta}\|_{L^1(-1,1)}}\int_{-1}^1\sigma_{\beta}(x)e^{-2\ii\nu_{\delta}(\beta) x\lambda_{1}}dx=\frac{1}{\|\sigma_{\beta}\|_{L^1(-1,1)}}\int_{-1}^1\sigma_{\beta}(-y)e^{2\ii\nu_{\delta}(\beta) y\lambda_{1}}dy.
\end{align*}
Using the fact that $\sigma_{\beta}(-y)=e^{-\frac{\beta}{1-y^2}}=\sigma_{\beta}(y)$, we know that $\Sigma_{\beta}(2\lambda_{1})=\Sigma_{\beta}(-2\lambda_{1})$. Therefore, we choose that $C_1=C_2=\frac{1}{1+\Sigma_{\beta}(2\lambda_{1})}$.
Then, we conclude that $\psi_0(\lambda_{1})=\psi_0(\lambda_{-1})=1$. This implies that
\begin{equation*}
\int_{-\frac{T}{2}}^{\frac{T}{2}}\vartheta_0(s)e^{-\ii \lambda_k s}ds=0,\forall k\neq 0, \pm 1,\;         \int_{-\frac{T}{2}}^{\frac{T}{2}}\vartheta_0(s)e^{-\ii s\lambda_{1}}ds=\int_{-\frac{T}{2}}^{\frac{T}{2}}\vartheta_0(s)e^{-\ii s\lambda_{-1}}ds=1.    
\end{equation*}
Now we check the uniform bound of $|C_1|$ and $|C_2|$. As we set before, $\nu_{\delta}(\beta)=\frac{T(1-\delta)}{2}$.  Since $\lim_{L\rightarrow2\pi}\lambda_{1}=0$, for $0<\delta<\frac{1}{2}$ sufficiently small, and $L\in [2\pi-\delta,2\pi+\delta]\setminus\{2\pi\}$, we obtain $0<|\lambda_{1}(L)|\leq T^{-1}$. Hence, let $s_0=s_0(\nu,\lambda_{1})=2\nu_{\delta}(\beta)\lambda_{1}$, thus $|s_0|< 2T^{-1}\frac{T}{2}=1$.
By definition of $\Sigma_{\beta}(2\lambda_{1})$,
\begin{equation*}
\Sigma_{\beta}(2\lambda_{1})=\frac{\int_{-1}^1\sigma_{\beta}(x)e^{-2\ii\nu_{\delta}(\beta) x\lambda_{1}}dx}{\|\sigma_{\beta}\|_{L^1(-1,1)}}= \frac{\int_{-1}^1\sigma_{\beta}(x)e^{-\ii x s_0}dx}{\|\sigma_{\beta}\|_{L^1(-1,1)}}= \frac{\int_{-1}^1\sigma_{\beta}(x)\cos{s_0x}dx-\ii\int_{-1}^1\sigma_{\beta}(x)\sin{s_0x}dx}{\|\sigma_{\beta}\|_{L^1(-1,1)}}.
\end{equation*}
Since $\sigma_{\beta}(x)$ is even and $\sin{s_0x}$ is odd, we know that $\int_{-1}^1\sigma_{\beta}(x)\sin{s_0x}dx=0$. Therefore,
\begin{align*}
\Sigma_{\beta}(2\lambda_{1})
&=\frac{\int_{-1}^1\sigma_{\beta}(x)\cos{s_0x}dx}{\|\sigma_{\beta}\|_{L^1(-1,1)}}
\geq \frac{\int_{-1}^1\sigma_{\beta}(x)(1-\frac{s_0^2x^2}{2})dx}{\|\sigma_{\beta}\|_{L^1(-1,1)}}
\geq 1-\frac{1}{2}\frac{\int_{-1}^1x^2\sigma_{\beta}(x)dx}{\|\sigma_{\beta}\|_{L^1(-1,1)}}.
\end{align*}
It is easy to see that $\int_{-1}^1x^2\sigma_{\beta}(x)dx\leq \int_{-1}^1\sigma_{\beta}(x)dx=\|\sigma_{\beta}\|_{L^1(-1,1)}$. Thus, we know that $\frac{1}{2}\leq \Sigma_{\beta}(2\lambda_{1})\leq 1$. 
We conclude that $\frac{1}{2}\leq |C_1|=|C_2|\leq \frac{2}{3}$. At the end, by the estimate \eqref{eq: growth-psi_0-2pi}, we obtain for $m\in\{0,\cdots,N\}$
\begin{equation*}
    \|\vartheta_0^{(m)}\|_{L^{\infty}(\R)}\leq \|x^m\psi_0(x)\|_{L^1(\R)}\leq \|\frac{K C_1 e^{\frac{K}{\sqrt{T}}}x^m}{1+|x-\lambda_{1}|^{2N}}\|_{L^1(\R)}+ \|\frac{K C_2 e^{\frac{K}{\sqrt{T}}}x^m}{1+|x-\lambda_{-1}|^{2N}}\|_{L^1(\R)}\leq C e^{\frac{K}{\sqrt{T}}}|\lambda_{1}|^m.
\end{equation*}
\end{proof}
\subsection{Compensate bi-orthogonal family for $2\sqrt{7}\pi$}\label{sec: app-bi-7pi}
\begin{proof}[Proof of Lemma \ref{lem: bi-orthogonal family-2pi-7}]
We only concentrate on the construction of $\vartheta_{\pm}$. Other statements are similar to Section \ref{sec: app-compensate-bi-2pi}. We begin by defining the function
\begin{equation}\label{eq: defi-psi-0-bi-orthogonal family-2pi}
\begin{aligned}
\psi_+(z)&:=C_1\prod_{k\neq1,2}\left(1-\frac{z-\lambda_1}{\lambda_k-\lambda_1}\right)\Sigma_{\beta}(-z+\lambda_1)+C_2\prod_{k\neq1,2}\left(1-\frac{z-\lambda_2}{\lambda_k-\lambda_2}\right)\Sigma_{\beta}(-z+\lambda_2),\\
\psi_-(z)&:=C_{-1}\prod_{k\neq-1,-2}\left(1-\frac{z-\lambda_{-1}}{\lambda_k-\lambda_{-1}}\right)\Sigma_{\beta}(-z+\lambda_{-1})+C_{-2}\prod_{k\neq-1,-2}\left(1-\frac{z-\lambda_{-2}}{\lambda_k-\lambda_{-2}}\right)\Sigma_{\beta}(-z+\lambda_{-2})
\end{aligned}
\end{equation}
where $\{C_m\}_{m=\pm1,\pm2}$ are coefficients to be chosen later. Then, applying Paley-Wiener's Theorem, we obtain a function $\vartheta_+\in L^2(\R)$, whose Fourier transform is $\psi_{\pm}$ and $\supp{\vartheta_{\pm}}\subset[-\frac{T}{2},\frac{T}{2}]$. We shall verify the uniform bound assumption at the end of the proof for well-chosen coefficients $C_m$. As a consequence of our choice of $\vartheta_{\pm}$, we know that $\int_{-\frac{T}{2}}^{\frac{T}{2}}e^{-\ii s\lambda_k}\vartheta_{\pm}(s)ds=\psi_{\pm}(\lambda_k)$. Now we focus on the properties for $\vartheta_+$ and $\psi_+$. $\vartheta_-$ and $\psi_-$ can be treated similarly. For $k\neq 1,2$, it is easy to verify that $\psi_+(\lambda_k)=0$. After simplifying the equation, the coefficients $C_1$ and $C_2$ satisfy that
\begin{align*}
\psi_+(\lambda_{1})&=C_1+C_{2}\Sigma_{\beta}(-\lambda_1+\lambda_{2})\prod_{k\neq1,2}\left(\frac{\lambda_k-\lambda_{1}}{\lambda_k-\lambda_{2}}\right),\\
\psi_+(\lambda_{2})&=C_1\Sigma_{\beta}(-\lambda_2+\lambda_{1})\prod_{k\neq1,2}\left(\frac{\lambda_k-\lambda_{2}}{\lambda_k-\lambda_{1}}\right)+C_{2}.
\end{align*}
Therefore, we choose that 
\begin{gather}
C_1=\frac{1}{1+\Sigma_{\beta}(-\lambda_2+\lambda_{1})},\;
C_2=\frac{\prod_{k\neq1,2}\left(\frac{\lambda_k-\lambda_{2}}{\lambda_k-\lambda_{1}}\right)}{1+\Sigma_{\beta}(-\lambda_2+\lambda_{1})}.
\end{gather}
Then, we conclude that $\psi_+(\lambda_{1})=1$, and $\psi_+(\lambda_{2})=\prod_{k\neq1,2}\left(\frac{\lambda_k-\lambda_{2}}{\lambda_k-\lambda_{1}}\right)$.  This implies that
\begin{equation*}
\int_{-\frac{T}{2}}^{\frac{T}{2}}\vartheta_+(s)e^{-\ii \lambda_k s}ds=0,\forall k\neq 0,1,2\;         \int_{-\frac{T}{2}}^{\frac{T}{2}}\vartheta_+(s)e^{-\ii s\lambda_{1}}ds=\int_{-\frac{T}{2}}^{\frac{T}{2}}\vartheta_+(s)e^{-\ii s\lambda_{2}}ds=\prod_{k\neq1,2}\left(\frac{\lambda_k-\lambda_{2}}{\lambda_k-\lambda_{1}}\right).    
\end{equation*}
Now we check the uniform bound of $|C_1|$ and $|C_{2}|$. As we presented before, since $\lim_{L\rightarrow2\pi\sqrt{7}}\lambda_{1}=\lim_{L\rightarrow2\pi\sqrt{7}}\lambda_{2}=\lambda_{c,1}$, $\Sigma_{\beta}(-\lambda_2+\lambda_{1})$ satisfies $\frac{1}{2}\leq \Sigma_{\beta}(-\lambda_2+\lambda_{1})\leq 1$. For $\prod_{k\neq1,2}\left(\frac{\lambda_k-\lambda_{2}}{\lambda_k-\lambda_{1}}\right)$, since $\lim_{L\rightarrow2\pi\sqrt{7}}\lambda_1=\lim_{L\rightarrow2\pi\sqrt{7}}\lambda_2$, we know that $\lim_{L\rightarrow2\pi\sqrt{7}}\prod_{k\neq1,2}\left(\frac{\lambda_k-\lambda_{2}}{\lambda_k-\lambda_{1}}\right)=1$. Thus, for $\delta$ sufficiently small, $L\in[2\pi\sqrt{7}-\delta,2\pi\sqrt{7}+\delta]$, $\frac{1}{4}\leq \prod_{k\neq1,2}\left(\frac{\lambda_k-\lambda_{2}}{\lambda_k-\lambda_{1}}\right)\leq 1$.

We conclude that $\frac{1}{2}\leq |C_1|\leq \frac{2}{3}$ and $\frac{1}{4}\leq |C_1|\leq \frac{2}{3}$. 
\end{proof}

\section{Proof for duality arguments}\label{sec: proof in duality arguments}
\begin{proof}[Proof of Proposition \ref{prop: stability to ob}]
This is a direct consequence of the energy estimates. 
Using the equation \eqref{eq: adjoint linear KdV-HUM} and integrating by parts, we obtain $\frac{d}{dt}E(w(t))=-2 |\p_xw(t,0)|^2$. 
This implies that $ E(w^0)-E(w(t))=2\int_0^t |\p_xw(s,0)|^2ds$. 
Since the system \eqref{eq: adjoint linear KdV-HUM} is exponentially stable, we know that
\begin{equation*}
2\int_0^t |\p_xw(s,0)|^2ds=E(w^0)-E(w(t)) \geq (C_1^{-1}e^{C_2t}-1)E(w(t))  
\end{equation*}
We choose $T_0$ such that $e^{C_2T_0}-C_1>0$. For $T>T_0$, set $C^2=\frac{2C_1}{e^{C_2T}-C_1}$. We obtain the quantitative observability inequality $\|S(T)w^0\|^2_{L^2(0,L)}\leq C^2\int_0^T|\p_xw(t,0)|^2dt$.
\end{proof}
\begin{proof}[Proof of Proposition \ref{prop: ob to stability}]
Without loss of generality, we set $T=1$. Thus, thanks to the quantitative observability, there exists an effectively computable constant $C$ such that the solution $w$ to the system \eqref{eq: adjoint linear KdV-HUM} satisfies $\|S(1)w^0\|^2_{L^2(0,L)}\leq C^2\int_0^1|\p_xw(t,0)|^2dt$. 
Using again $\frac{d}{dt}E(w(t))=-2 |\p_xw(t,0)|^2$, we know that $E(w(1))\leq C^2(E(w^0)-E(w(1)))$.
This is equivalent to $E(w(1))\leq \frac{C^2}{C^2+1}E(w^0)$. 
Combining with $E(w(s))\leq E(w(t))$, $\forall s\leq t$, we obtain $ E(w(t))\leq e^{-t\ln{\frac{C^2+1}{C^2}}}E(w^0)$. 
Here we choose $C_1=1$ and $C_2=\ln{\frac{C^2+1}{C^2}}$.
\end{proof}
First, we introduce the following inclusion lemma. For this lemma and its proof, one can refer to \cite[Pages 194-195]{dolecki-russell-1977}. 
\begin{lem}\label{lem: inclusion lemma}
Let $H_1$, $H_2$, and $H_3$ be three Hilbert spaces. Let $C_2$ be a continuous linear map from $H_2$ to $H_1$. Let $C_3$ be a densely defined closed linear operator from $D(C_3)\subset H_3$ to $H_1$. Then the following two statements are equivalent
\begin{enumerate}
    \item $C_2(H_2)\subset C_3(D(C_3))$.
    \item There exists a constant $M$ such that
    \begin{equation}\label{eq: inclusion inequality}
        \|C^*_2h_1\|_{H_2}\leq M\|C_3^*h_1\|_{H_3},\forall h_1\in H_1.
    \end{equation}
\end{enumerate}
Moreover, if the inclusion inequality holds, there exists a continuous linear map $C_1$ from $H_2$ to $H_3$ such that
\begin{equation}\label{eq: control map est}
C_1(H_2)\subset D(C_3),\; C_2=C_3C_1,\; \|C_1\|_{\mathscr{L}(H_2,H_3)}\leq M.    
\end{equation}
\end{lem}
\begin{proof}[Proof of Lemma \ref{lem: trajectoory controllability}]
Assume that the system \eqref{eq: linearized KdV system-control-intro} is null controllable. Thus, there exists a control function $f\in L^2(0,T)$ such that for $\forall y^0\in L^2(0,L)$, the solution $y$ to the system \eqref{eq: linearized KdV system-control-intro} satisfies that $\Pi_H y(T,x)=0$. Let $\Hat{y}(t,x)=S(t)y^0(x)-y(t,x)$. Then, $\Hat{y}(0,x)=y^0(x)-y^0(x)=0$. And $\Hat{y}$ is a solution to 
\begin{equation*}
\left\{
\begin{array}{lll}
    \p_t\Hat{y}+\p_x^3\Hat{y}+\p_x\Hat{y}=0 & \text{ in }(0,T)\times(0,L), \\
     \Hat{y}(t,0)=\Hat{y}(t,L)=0&  \text{ in }(0,T),\\
     \p_x\Hat{y}(t,L)=-f(t)&\text{ in }(0,T),\\
     \Hat{y}(0,x)=0&\text{ in }(0,L).
\end{array}
\right.   
\end{equation*}
Here $\Hat{f}(t)=-f(t)$ with $\|\Hat{f}\|_{L^2(0,T)}\leq C\|y^0\|_{L^2(0,L)}$. Moreover, $\Pi_H\Hat{y}(T,x)=\Pi_HS(T)y^0(x)$. On the other hand, the proof of the converse statement can be treated similarly using $\Hat{y}(t,x)=S(t)y^0(x)-y(t,x)$.

For $T>0$, let us define a linear operator $\mathscr{F}_T$ by 
\begin{equation}
\begin{aligned}
    \mathscr{F}_T: L^2(0,T)&\rightarrow H,\\
    f(t)&\mapsto \Pi_H \Hat{y}(T,\cdot),
\end{aligned}
\end{equation}
where $\Hat{y}$ is a solution to the system \eqref{eq: controlled linear KdV-HUM-zero-initial}. By energy estimates, it is easy to verify that
\begin{equation*}
    \|\mathscr{F}_T(f)\|^2_{L^2(0,L)}= \|\Pi_H \Hat{y}(T,\cdot)\|^2_{L^2(0,L)}\leq 2\|f\|^2_{L^2(0,T)}.
\end{equation*}
This implies that $\mathscr{F}_T$ is a bounded operator from $L^2(0,T)$ to $H$. Thus, the system \eqref{eq: controlled linear KdV-HUM} is null controllable if and only if $\Pi_HS(T)(L^2(0,L))\subset \mathscr{F}_T(L^2(0,T))$. Then the following theorem is a direct conclusion of Lemma \ref{lem: inclusion lemma}.
\end{proof}
\begin{proof}[Proof of Proposition \ref{prop: control and ob}]
Let $H_1=H$, $H_2=L^2(0,L)$, and $H_3=L^2(0,T)$. Furthermore, set $C_2=\Pi_HS(T): H_2\rightarrow H_1$ and $C_3=\mathscr{F}_T: H_3\rightarrow H_1$. Thus, $C_2\in\mathscr{L}(H_2,H_1)$ and $C_3\in \mathscr{L}(H_3,H_1)$. By Lemma \ref{lem: inclusion lemma} and Lemma \ref{lem: trajectoory controllability}, we obtain 
\begin{equation*}
\|(\Pi_HS(T))^*w^0\|_{L^2(0,L)}\leq C\|\mathscr{F}^*_Tw^0\|_{L^2(0,T)}, \forall w^0\in H    
\end{equation*}
We need to specify the term $\mathscr{F}^*_Tw^0$ and $(\Pi_HS(T))^*w^0$. We consider the system 
\begin{equation}
\left\{
\begin{array}{lll}
    \p_t\Hat{y}+\p_x^3\Hat{y}+\p_x\Hat{y}=0 & \text{ in }(0,T)\times(0,L), \\
     \Hat{y}(t,0)=\Hat{y}(t,L)=0&  \text{ in }(0,T),\\
     \p_x\Hat{y}(t,L)=\Hat{f}(t)&\text{ in }(0,T),\\
     \Hat{y}(0,x)=0&\text{ in }(0,L).
\end{array}
\right.
\end{equation} 
Using the solution $w$ as a test function and multiplying on both sides of the equation above, we apply the duality relation defined in \eqref{eq: duality relations} and integrate by parts,
\begin{align*}
0&=\poscalr{\p_t\Hat{y}+\p_x^3\Hat{y}+\p_x\Hat{y}}{w}_{(0,T)\times(0,L)}\\
&=\poscalr{\Hat{y}(T.\cdot)}{w^0}_{(0,L)}-\poscalr{\Hat{y}(0.\cdot)}{w(T,\cdot)}_{(0,L)}+\poscalr{\p_x\Hat{y}(\cdot,L)}{\p_x w(\cdot,0)}_{(0,T)}-\poscalr{\p_x\Hat{y}(\cdot,L)}{\p_x w(\cdot,L)}_{(0,T)}.
\end{align*}
Using that $\Hat{y}(0.\cdot)=0$, $\p_x\Hat{y}(\cdot,L)=\Hat{f}(t)$ and $\p_x w(\cdot,L)=0$, we obtain 
\begin{equation*}
\poscalr{\Hat{y}(T.\cdot)}{w^0}_{(0,L)}=- \poscalr{\Hat{f}}{\p_x w(\cdot,0)}_{(0,T)}.  
\end{equation*}
By the decomposition $\Hat{y}(T.\cdot)=\Pi_H\Hat{y}(T.\cdot)+(Id-H)\Hat{y}(T.\cdot)$ and $w_0\in H$, we know that
\begin{equation*}
\poscalr{\mathscr{F}_T\Hat{f}}{w^0}_{(0,L)}=\poscalr{\Pi_H\Hat{y}(T.\cdot)}{w^0}_{(0,L)}=-\poscalr{\Hat{f}}{\p_x w(\cdot,0)}_{(0,T)}.     
\end{equation*}
This implies that $\mathscr{F}^*_Tw^0=-\p_x w(\cdot,0)$. For the term $(\Pi_HS(T))^*w^0$, we consider the system 
\begin{equation}
\left\{
\begin{array}{lll}
    \p_t y-\A y=0 & \text{ in }(0,T)\times(0,L), \\
     y(0,x)=y^0(x)&\text{ in }(0,L),
\end{array}
\right.    
\end{equation}
tested by the solution $w$. Similarly, integrating by parts, we obtain 
\begin{align*}
0&=\poscalr{\p_t y-\A y}{w}_{(0,T)\times(0,L)}\\
&=\poscalr{y(T,\cdot)}{w^0}_{(0,L)}-\poscalr{y^0}{w(T,\cdot)}_{(0,L)}+\poscalr{\p_x y(\cdot,L)}{\p_x w(\cdot,0)}_{(0,T)}-\poscalr{\p_x y(\cdot,0)}{\p_x w(\cdot,L)}_{(0,T)}.
\end{align*}
Using that $\p_x w(\cdot,L)=\p_x y(\cdot,L)=0$, we obtain 
\begin{equation*}
\poscalr{y(T,\cdot)}{w^0}_{(0,L)}=\poscalr{y^0}{w(T,\cdot)}_{(0,L)}.    
\end{equation*}
Using again the decomposition $y(T,\cdot)=\Pi_HS(T)y^0+(Id-\Pi_H)S(T)y^0$, we know that
\begin{equation*}
\poscalr{\Pi_HS(T)y^0}{w^0}_{(0,L)}=\poscalr{y^0}{S(T)w^0}_{(0,L)}.    
\end{equation*}
This implies that $(\Pi_HS(T))^*w^0=S(T)w^0$t. In summary, we know that the null controllability is equivalent to 
\begin{equation*}
    \|S(T)w^0\|_{L^2(0,L)}\leq C\|\p_x w(t,0)\|_{L^2(0,T)}, \forall w^0\in H,
\end{equation*}
where $w$ is a solution to the system \eqref{eq: adjoint linear KdV-HUM}. Next, we prove the inequality \eqref{eq: control-cost-HUM} and the inequality \eqref{eq: ob-cost-HUM} share the same constant $C$. Assume that the observability \eqref{eq: ob-cost-HUM} holds. Applying Lemma \ref{lem: inclusion lemma}, there exists a linear continuous map $\mathscr{G}\in \mathscr{L}(L^2(0,L);L^2(0,T))$ such that $\Pi_HS(T)=\mathscr{F}_T\circ\mathscr{G}$, i.e. the control function $f=\mathscr{G}y^0$. Thanks to the estimate \eqref{eq: control map est}, we know that 
\begin{equation}
    \|f\|_{L^2(0,T)}=\|\mathscr{G}y^0\|_{L^2(0,T)}\leq C\|y^0\|_{L^2(0,L)}.
\end{equation}
For the converse, assume that the quantitative control estimate \eqref{eq: control-cost-HUM} holds. We consider the system 
\begin{equation*}
\left\{
\begin{array}{lll}
    \p_ty+\p_x^3y+\p_xy=0 & \text{ in }(0,T)\times(0,L), \\
     y(t,0)=y(t,L)=0&  \text{ in }(0,T),\\
     \p_xy(t,L)=f(t)&\text{ in }(0,T),\\
     y(0,x)=y^0(x)&\text{ in }(0,L),\\
     \Pi_Hy(T,\cdot)=0,
\end{array}
\right.
\end{equation*}
tested with the solution $w$. Using the duality relations and integrating by parts, we obtain 
\begin{align*}
0&=\poscalr{\p_ty+\p_x^3y+\p_xy}{w}_{(0,T)\times(0,L)}\\
&=\poscalr{y(T,\cdot)}{w^0}_{(0,L)}-\poscalr{y^0}{w(T,\cdot)}_{(0,L)}+\poscalr{\p_x y(\cdot,L)}{\p_x w(\cdot,0)}_{(0,T)}-\poscalr{\p_x y(\cdot,0)}{\p_x w(\cdot,L)}_{(0,T)}.
\end{align*}
Using that $\p_x w(\cdot,L)=0$ and $\p_x y(\cdot,L)=f$, we know that
\begin{align*}
0=\poscalr{y(T,\cdot)}{w^0}_{(0,L)}-\poscalr{y^0}{S(T)w^0}_{(0,L)}+\poscalr{f}{\p_x w(\cdot,0)}_{(0,T)}.
\end{align*}
By the decomposition $y(T,\cdot)=\Pi_Hy(T,\cdot)+(Id-\Pi_H)y(T,\cdot)=(Id-\Pi_H)y(T,\cdot)$, we deduce that $\poscalr{y(T,\cdot)}{w^0}_{(0,L)}=0$. Thus,
\begin{equation*}
\poscalr{S(T)w^0}{y^0}_{(0,L)}=\poscalr{\p_x w(\cdot,0)}{f}_{(0,T)}.    
\end{equation*}
Since for every $y^0\in L^2(0,L)$, $\poscalr{S(T)w^0}{y^0}_{(0,L)}$ defines a linear functional on $L^2(0,L)$, we can choose a specific $\Tilde{y}^0$ such that $\|\Tilde{y}^0\|_{L^2(0,L)}=1$ and $\poscalr{S(T)w^0}{\Tilde{y}^0}_{(0,L)}=\|S(T)w^0\|_{L^2(0,L)}$. Therefore, 
\begin{align*}
\|S(T)w^0\|_{L^2(0,L)}&=|\poscalr{\p_x w(\cdot,0)}{f}_{(0,T)}|\\
&\leq \|f\|_{L^2(0,T)}\|\p_x w(\cdot,0)\|_{L^2(0,T)}\\
&\leq C\|\Tilde{y}^0\|_{L^2(0,L)} \|\p_x w(\cdot,0)\|_{L^2(0,T)}\\
&\leq C\|\p_x w(\cdot,0)\|_{L^2(0,T)}.  
\end{align*}
\end{proof}

\section{Part I: A transition-stabilization method}
\subsection{Proof of Corollary \ref{coro: simplified estimates for iteration scheme} }\label{sec: Proof of Corollary coro: simplified estimates for iteration scheme}
\begin{proof}
We plug $\mu_2=2\mu_1$ into the estimate \eqref{eq: L^2-y_1-single}, then $|\mu_1-\mu_2|=\mu_1$ and 
\begin{equation*}
\|y_1(t,\cdot)\|_{L^2(0,L)}\leq \frac{K_2}{|L-L_0|}e^{\frac{2\sqrt{2}K}{\sqrt{T}}}\frac{C_1\mu_1+2^{\frac{5}{2}}C_2C^h_2\mu_1^{\frac{7}{2}}+2^{\frac{5}{2}}C_3C^h_2\mu_1^{\frac{5}{2}}+2C_2C^h_2\mu_1^{\frac{7}{2}}+C_3C^h_2\mu_1^{\frac{5}{2}}}{\mu_1T^3}\|y^0\|_{L^2(0,L)}.
\end{equation*}
Thanks to $\mu_1>1$, we know that
\begin{equation*}
C_1\mu_1+2^{\frac{5}{2}}C_2C^h_2\mu_1^{\frac{7}{2}}+2^{\frac{5}{2}}C_3C^h_2\mu_1^{\frac{5}{2}}+2C_2C^h_2\mu_1^{\frac{7}{2}}+C_3C^h_2\mu_1^{\frac{5}{2}}\leq \left(C_1+2^{\frac{5}{2}}C_2C^h_2+2^{\frac{5}{2}}C_3C^h_2+2C_2C^h_2+C_3C^h_2\right)\mu_1^{\frac{7}{2}}  
\end{equation*}
Define a new constant $\mathcal{C}_1:=K_2\left(C_1+2^{\frac{5}{2}}C_2C^h_2+2^{\frac{5}{2}}C_3C^h_2+2C_2C^h_2+C_3C^h_2\right)$. Then, we simplify the estimate \eqref{eq: L^2-y_1-single}
\begin{equation*}
\begin{aligned}
\|y_1(t,\cdot)\|_{L^2(0,L)}&\leq \frac{K_2}{|L-L_0|}e^{\frac{2\sqrt{2}K}{\sqrt{T}}}\frac{C_1\mu_1+2^{\frac{5}{2}}C_2C^h_2\mu_1^{\frac{7}{2}}+2^{\frac{5}{2}}C_3C^h_2\mu_1^{\frac{5}{2}}+2C_2C^h_2\mu_1^{\frac{7}{2}}+C_3C^h_2\mu_1^{\frac{5}{2}}}{\mu_1T^3}\|y^0\|_{L^2(0,L)}\\
&\leq \frac{K_2}{|L-L_0|}e^{\frac{2\sqrt{2}K}{\sqrt{T}}}\frac{\mathcal{C}_1\mu_1^{\frac{7}{2}}}{K_2\mu_1T^3}\|y^0\|_{L^2(0,L)}\\
&\leq \frac{\mathcal{C}_1}{|L-L_0|}e^{\frac{2\sqrt{2}K}{\sqrt{T}}}\frac{\mu_1^{\frac{5}{2}}}{T^3}\|y^0\|_{L^2(0,L)}.
\end{aligned}
\end{equation*}
Similarly, we define the following constants:
\begin{align*}
\mathcal{C}_2&=C^h_1(2C_2+C_3),\\
\mathcal{C}_3&=C^h_1\frac{C_2+C_3}{\sqrt{2}}.
\end{align*}
Using these constants above, we obtain
\begin{align*}
\|y_2(t,\cdot)\|_{L^2(0,L)}&\leq C^h_1e^{-\mu_1(t-\frac{T}{2})}\frac{2C_2\mu_{1}+C_3}{\mu_1^{\frac{3}{2}}T^3}\|y^0\|_{L^2(0,L)},\\
&\leq e^{-\mu_1(t-\frac{T}{2})}\frac{\mathcal{C}_2\mu_1}{\mu_1^{\frac{3}{2}}T^3}\|y^0\|_{L^2(0,L)}\\
&\leq \mathcal{C}_2 e^{-\mu_1(t-\frac{T}{2})}\frac{1}{\mu_1^{\frac{1}{2}}T^3}\|y^0\|_{L^2(0,L)}.
\end{align*}
\begin{align*}
\|y_3(t,\cdot)\|_{L^2(0,L)}&\leq C^h_1e^{-2\mu_1(t-\frac{T}{2})}\frac{C_2\mu_{1}+C_3}{\sqrt{2}\mu_1^{\frac{3}{2}}T^3}\|y^0\|_{L^2(0,L)}\\
&\leq e^{-2\mu_1(t-\frac{T}{2})}\frac{\mathcal{C}_3\mu_1}{\mu_1^{\frac{3}{2}}T^3}\|y^0\|_{L^2(0,L)}\\
&\leq \mathcal{C}_3 e^{-2\mu_1(t-\frac{T}{2})}\frac{1}{\mu_1^{\frac{1}{2}}T^3}\|y^0\|_{L^2(0,L)}.
\end{align*}
By $\mu_1(t-\frac{T}{2})\geq0$ for $t\in(\frac{T}{2},T)$, hence, 
\begin{align*}
\|y_3(t,\cdot)\|_{L^2(0,L)}\leq \mathcal{C}_3 e^{-\mu_1(t-\frac{T}{2})}\frac{1}{\mu_1^{\frac{1}{2}}T^3}\|y^0\|_{L^2(0,L)}.
\end{align*}
Then we look at the control cost for $w_j(t)$, $t\in[0,T],j=1,2,3$. For instance, we estimate $w_1$. We plug $\mu_2=2\mu_1$ into the estimate \eqref{eq: cost estimate-w-1}, 
\begin{align*}
\|w_1\|_{L^{\infty}(0,T)}\leq \frac{K_3}{|L-L_0|}e^{\frac{2\sqrt{2}K}{\sqrt{T}}}\frac{C_1\mu_1+2^{\frac{5}{2}}C_2C^h_2\mu_1^{\frac{7}{2}}+2^{\frac{5}{2}}C_3C^h_2\mu_1^{\frac{5}{2}}+2C_2C^h_2\mu_1^{\frac{7}{2}}+C_3C^h_2\mu_1^{\frac{5}{2}}}{\mu_1 T^3}\|y^0\|_{L^2(0,L)}.
\end{align*}
Define a constant $\mathcal{C}_4:=\frac{K_3}{K_2}\mathcal{C}_1$. Therefore, thanks to $\mu_1>1$, we obtain
\begin{align*}
\|w_1\|_{L^{\infty}(0,T)}&\leq \frac{K_3}{|L-L_0|}e^{\frac{2\sqrt{2}K}{\sqrt{T}}}\frac{C_1\mu_1+2^{\frac{5}{2}}C_2C^h_2\mu_1^{\frac{7}{2}}+2^{\frac{5}{2}}C_3C^h_2\mu_1^{\frac{5}{2}}+2C_2C^h_2\mu_1^{\frac{7}{2}}+C_3C^h_2\mu_1^{\frac{5}{2}}}{\mu_1 T^3}\|y^0\|_{L^2(0,L)}\\
&\leq \frac{K_3}{|L-L_0|}e^{\frac{2\sqrt{2}K}{\sqrt{T}}}\frac{\left(C_1+2^{\frac{5}{2}}C_2C^h_2+2^{\frac{5}{2}}C_3C^h_2+2C_2C^h_2+C_3C^h_2\right)\mu_1^{\frac{7}{2}}}{\mu_1 T^3}\|y^0\|_{L^2(0,L)}\\
&\leq \frac{K_3}{|L-L_0|}e^{\frac{2\sqrt{2}K}{\sqrt{T}}}\frac{\mathcal{C}_1\mu_1^{\frac{5}{2}}}{K_2 T^3}\|y^0\|_{L^2(0,L)}\\
&\leq \frac{\mathcal{C}_4}{|L-L_0|}e^{\frac{2\sqrt{2}K}{\sqrt{T}}}\frac{\mu_1^{\frac{5}{2}}}{T^3}\|y^0\|_{L^2(0,L)}.
\end{align*}
Similarly, define 
\begin{align*}
\mathcal{C}_5&=C_4(2C_2+C_3),\\
\mathcal{C}_6&=C_4(C_2+C_3).
\end{align*}
Using the two constants above, we know that
\begin{align*}
\|w_2\|_{L^{\infty}(0,T)}&\leq\mathcal{C}_5e^{-\frac{\mu_1^{\frac{1}{3}}}{2}L}\frac{1}{T^3}\|y^0\|_{L^2(0,L)},\\
\|w_3\|_{L^{\infty}(0,T)}&\leq \mathcal{C}_6e^{-\frac{\mu_2^{\frac{1}{3}}}{2}L}\frac{1}{T^3}\|y^0\|_{L^2(0,L)}\leq \mathcal{C}_6e^{-\frac{\mu_1^{\frac{1}{3}}}{2}L}\frac{1}{T^3}\|y^0\|_{L^2(0,L)}.
\end{align*}
In summary, define $\mathcal{C}:=\mathcal{C}_1+(\frac{L_0}{2}+1)\mathcal{C}_2+(\frac{L_0}{2}+1)\mathcal{C}_3+\mathcal{C}_4+(\frac{L_0}{2}+1)\mathcal{C}_5+(\frac{L_0}{2}+1)\mathcal{C}_6+1$. We obtain the estimates for the solution
\begin{equation}\label{eq: L^2-estimate for solution y-package}
\begin{aligned}
\|y(t,\cdot)\|_{L^2(0,L)}&\leq \|y^0\|_{L^(0,L)},t\in(0,\frac{T}{2}],\\
\|y(t,\cdot)\|_{L^2(0,L)}&\leq \|y_1(t,\cdot)\|_{L^2(0,L)}+\|y_2(t,\cdot)\|_{L^2(0,L)}+\|y_3(t,\cdot)\|_{L^2(0,L)}\\
&\leq \left(\frac{\mathcal{C}_1}{|L-L_0|}e^{\frac{2\sqrt{2}K}{\sqrt{T}}}\frac{\mu_1^{\frac{5}{2}}}{T^3}+\mathcal{C}_2 e^{-\mu_1(t-\frac{T}{2})}\frac{1}{\mu_1^{\frac{1}{2}}T^3}+\mathcal{C}_3 e^{-\mu_1(t-\frac{T}{2})}\frac{1}{\mu_1^{\frac{1}{2}}T^3}\right)\|y^0\|_{L^2(0,L)}\\
&\leq \frac{\mathcal{C}}{|L-L_0|}\frac{e^{\frac{2\sqrt{2}K}{\sqrt{T}}}\mu_1^{\frac{5}{2}}+ e^{-\mu_1(t-\frac{T}{2})}}{T^3} \|y^0\|_{L^2(0,L)},t\in (\frac{T}{2},T).
\end{aligned}    
\end{equation}
In particular, $y_1(T,\cdot)=0$. Therefore,
\begin{equation}\label{eq: L^2-estimate for y at T}
\begin{aligned}
\|y(T,\cdot)\|_{L^2(0,L)}&\leq \|y_2(T,\cdot)\|_{L^2(0,L)}+\|y_3(T,\cdot)\|_{L^2(0,L)}\\
&\leq \left(\mathcal{C}_2 e^{-\mu_1(t-\frac{T}{2})}\frac{1}{\mu_1^{\frac{1}{2}}T^3}+\mathcal{C}_3 e^{-\mu_1(t-\frac{T}{2})}\frac{1}{\mu_1^{\frac{1}{2}}T^3}\right)\|y^0\|_{L^2(0,L)}\\
&\leq \mathcal{C}\frac{ e^{-\mu_1(t-\frac{T}{2})}}{T^3} \|y^0\|_{L^2(0,L)}
\end{aligned}
\end{equation}
In addition, for the cost of the control function $w$, we obtain
\begin{equation}
\begin{aligned}
\|w\|_{L^{\infty}(0,T)}&\leq \|w_1\|_{L^{\infty}(0,T)}+ \|w_2\|_{L^{\infty}(0,T)}+\|w_3\|_{L^{\infty}(0,T)} \\
&\leq \left(\frac{\mathcal{C}_4}{|L-L_0|}e^{\frac{2\sqrt{2}K}{\sqrt{T}}}\frac{\mu_1^{\frac{5}{2}}}{T^3}+\mathcal{C}_5e^{-\frac{\mu_1^{\frac{1}{3}}}{2}L}\frac{1}{T^3}+\mathcal{C}_6e^{-\frac{\mu_1^{\frac{1}{3}}}{2}L}\frac{1}{T^3}\right)\|y^0\|_{L^2(0,L)}\\
&\leq \frac{\mathcal{C}}{|L-L_0|}\frac{e^{\frac{2\sqrt{2}K}{\sqrt{T}}}\mu_1^{\frac{5}{2}}+e^{-\frac{\mu_1^{\frac{1}{3}}}{2}L}}{T^3}\|y^0\|_{L^2(0,L)}.
\end{aligned}
\end{equation}
\end{proof}

\subsection{Choice of $Q$}\label{sec: app-choice-Q}
Now we choose a good constant $Q$ such that for $n>1$
\begin{equation}
\begin{aligned}
\frac{\mathcal{C}^{n-1}}{|L-L_0|^{n-1}}e^{-\frac{Q2^{\frac{n_0}{2}}(2^{\frac{n-1}{2}}-1)}{2(2-\sqrt{2})}}2^{3n_0(n-1)+\frac{3n(n-1)}{2}}&\leq 1,\\
\frac{\mathcal{C}^{n}}{|L-L_0|^{n}}e^{2\sqrt{2}K2^{\frac{n_0+n}{2}}-\frac{Q2^{\frac{n_0}{2}}(2^{\frac{n-1}{2}}-1)}{2(2-\sqrt{2})}}2^{\frac{15}{4}(n_0+n)+3n(n_0+1)+\frac{3n(n-1)}{2}}&\leq1,\\
\frac{\mathcal{C}^{n}}{|L-L_0|^{n}} e^{-\frac{Q2^{\frac{n_0}{2}}(2^{\frac{n-1}{2}}-1)}{2(2-\sqrt{2})}}2^{3n(n_0+1)+\frac{3n(n-1)}{2}}&\leq1,\\
\frac{\mathcal{C}^{n}}{|L-L_0|^{n}}e^{-\frac{Q^{\frac{1}{3}}2^{\frac{n_0+n}{2}}}{2}L}e^{-\frac{Q2^{\frac{n_0}{2}}(2^{\frac{n-1}{2}}-1)}{2(2-\sqrt{2})}}2^{3n(n_0+1)+\frac{3n(n-1)}{2}}&\leq1.
\end{aligned}
\end{equation}
Since $2^{-n_0}=T=\frac{\mathcal{K}^2}{\epsilon^2(\ln{|L-L_0|})^2}$, it suffices to choose $Q$ that satisfies the following conditions,
\begin{align*}
    \mathcal{C}^{n-1}e^{2^{\frac{n_0(n-1)}{2}}\frac{\mathcal{K}}{\epsilon}}e^{-\frac{Q2^{\frac{n_0}{2}}(2^{\frac{n-1}{2}}-1)}{2(2-\sqrt{2})}}2^{3n_0(n-1)+\frac{3n(n-1)}{2}}&\leq 1,\\
e^{2^{\frac{n_0n}{2}}\frac{\mathcal{K}}{\epsilon}}\mathcal{C}^{n}e^{2\sqrt{2}K2^{\frac{n_0+n}{2}}-\frac{Q2^{\frac{n_0}{2}}(2^{\frac{n-1}{2}}-1)}{2(2-\sqrt{2})}}2^{\frac{15}{4}(n_0+n)+3n(n_0+1)+\frac{3n(n-1)}{2}}&\leq1,\\
e^{2^{\frac{n_0n}{2}}\frac{\mathcal{K}}{\epsilon}}\mathcal{C}^{n} e^{-\frac{Q2^{\frac{n_0}{2}}(2^{\frac{n-1}{2}}-1)}{2(2-\sqrt{2})}}2^{3n(n_0+1)+\frac{3n(n-1)}{2}}&\leq1,\\
e^{2^{\frac{n_0n}{2}}\frac{\mathcal{K}}{\epsilon}}\mathcal{C}^{n}e^{-\frac{Q^{\frac{1}{3}}2^{\frac{n_0+n}{2}}}{2}L}e^{-\frac{Q2^{\frac{n_0}{2}}(2^{\frac{n-1}{2}}-1)}{2(2-\sqrt{2})}}2^{3n(n_0+1)+\frac{3n(n-1)}{2}}&\leq1.
\end{align*}
\cqfd

\noindent {\bf Acknowledgements} \; 
The authors would like to thank Ludovick Gagnon for valuable discussions during the preparation of this manuscript. Shengquan Xiang is partially supported by NSFC under grant 12571474. Part of this work was carried out while Jingrui Niu was visiting Peking University in August 2023 and July–August 2024. He would like to thank Peking University for its hospitality and financial support.

    \medskip
 \noindent\textbf{Conflict of interest} \; The authors have no competing interests to declare that are relevant to the content of this article.

    \medskip

    \noindent\textbf{Data availability} \; Data sharing not applicable to this article as no datasets were generated or analyzed during the current study.

\bibliographystyle{alpha}
\bibliography{ref}
\end{document}